\documentclass[11pt]{amsart}
\usepackage{amssymb}
\usepackage{amscd}
\usepackage{amsmath}
\usepackage{amsmath}  
\usepackage[all]{xy}
\usepackage{mathrsfs}
\usepackage{graphicx}
\usepackage{color}

\theoremstyle{plain}
\newtheorem{theorem}{Theorem}[section]
\newtheorem{lemma}[theorem]{Lemma}
\newtheorem{corollary}[theorem]{Corollary}
\newtheorem{proposition}[theorem]{Proposition}
\newtheorem{definition}[theorem]{Definition}
\newtheorem{hypothesis}[theorem]{Hypothesis}

\newtheorem{conjecture}[theorem]{Conjecture}

\newtheorem{example}[theorem]{Example}

\newtheorem{remark}[theorem]{Remark}
\newtheorem{remarks}[theorem]{Remark}

  1

\DeclareMathOperator{\Aut}{Aut}

\DeclareMathOperator{\Gal}{Gal}

\DeclareMathOperator{\Hom}{Hom}
\DeclareMathOperator{\End}{End}

\DeclareMathOperator{\N}{N}
\DeclareMathOperator{\K}{K}

\DeclareMathOperator{\im}{im}

\DeclareMathOperator{\Fr}{Fr}

\newcommand{\catname}[1]{\textnormal{{\textsf{#1}}}}

\newcommand{\Der}{\catname{D}}
\newcommand{\DR}{\catname{R}}
\newcommand{\DC}{\catname{C}}
\newcommand{\DL}{\catname{L}}

\newcommand{\cG}{\mathcal{G}}
\newcommand{\CC}{\mathbb{C}}
\newcommand{\bc}{\mathbb{C}}

\newcommand{\GG}{\mathbb{G}}

\newcommand{\QQ}{\mathbb{Q}}

\newcommand{\RR}{\mathbb{R}}
\newcommand{\br}{\mathbb{R}}
\newcommand{\ZZ}{\mathbb{Z}}

\newcommand{\cA}{\mathcal{A}}

\newcommand{\E}{\mathcal{E}}
\newcommand{\calE}{\mathcal{E}}

\newcommand{\calP}{\mathcal{P}}
\newcommand{\G}{\mathcal{G}}

\newcommand{\co}{\mathcal{O}}
\newcommand{\calV}{\mathcal{V}}
\newcommand{\KK}{\mathcal{K}}

\newcommand{\frp}{\mathfrak{p}}

\newcommand{\cok}{\mathrm{cok}}

\newcommand{\id}{\mathrm{id}}

\newcommand{\nr}{\mathrm{Nrd}}

\newcommand{\bz}{\mathbb{Z}}
\newcommand{\La}{\Lambda}
\newcommand{\bq}{\mathbb{Q}}
\def\bigcapp{\raise1ex\hbox{\rotatebox{180}{$\biguplus$}}}
 \def\bigcappd{\raise1ex\hbox{\rotatebox{180}{$\displaystyle\biguplus$}}}

\newcommand{\D}{\mathrm{d}}
\newcommand{\rgamma}{\DR\Gamma}
\newcommand{\rhom}{\DR\Hom}

\makeatletter
 
  \@addtoreset{equation}{subsection}
\makeatother

\begin{document}


\title[On non-commutative Iwasawa theory]{On non-commutative Iwasawa theory \\and derivatives of Euler systems
\endgraf \vskip 0.2truein \small{\it To the memory of Jan Nekov\'a\v r,\\
 friend and colleague}}


\author{David Burns and Takamichi Sano}

\begin{abstract} 
We use the theory of reduced determinant functors from \cite{bses} to give a new, computationally useful, description of the relative $\K_0$-groups of orders in finite dimensional separable algebras that need not be commutative. By combining this approach with a canonical  generalization to non-commutative algebras of the notion of `zeta element' introduced by Kato \cite{K1}, we then formulate, for each odd prime $p$, a natural 
 main conjecture of non-commutative $p$-adic Iwasawa theory for $\mathbb{G}_m$ over arbitrary number fields. This conjecture predicts a simple relation between a canonical Rubin-Stark non-commutative Euler system that we introduce and the compactly supported $p$-adic cohomology of $\ZZ_p$ and is shown to simultaneously extend both the higher rank
(commutative) main conjecture for $\mathbb{G}_m$ formulated by Kurihara and the present authors \cite{bks2} and the $K$-theoretical formalism of main conjectures in non-commutative Iwasawa theory developed by Ritter and Weiss \cite{RitterWeiss} and by Coates, Fukaya, Kato, Sujatha and Venjakob \cite{cfksv}. In particular, via these links we obtain strong evidence in support of the conjecture in the setting of Galois CM extensions of totally real fields. Our approach also leads to the formulation over arbitrary number fields of a precise conjectural `higher derivative formula' for the Rubin-Stark non-commutative Euler system  that is shown to recover upon appropriate specialisation the classical Gross-Stark Conjecture for Deligne-Ribet $p$-adic $L$-functions. We then show that this conjectural derivative formula can be combined with the main conjecture of non-commutative $p$-adic Iwasawa theory to give a strategy for obtaining evidence in support of the equivariant Tamagawa Number Conjecture for $\mathbb{G}_m$ over arbitrary finite Galois extensions of number fields, thereby obtaining a wide-ranging generalization of the main result of \cite{bks2}. 
\end{abstract}

\address{King's College London,
Department of Mathematics,
London WC2R 2LS,
U.K.}
\email{david.burns@kcl.ac.uk}


\address{Osaka Metropolitan University, 
Department of Mathematics, 
3-3-138 Sugimoto\\Sumiyoshi-ku\\Osaka\\558-8585, 
Japan}
\email{tsano@omu.ac.jp}


\maketitle

\tableofcontents
\section{Introduction}

\subsection{Background and main results} The theory of determinant functors for complexes of modules over commutative noetherian rings was developed by Knudsen and Mumford in \cite{knumum}, with later clarifications provided by Knudsen in \cite{knudsen}, in both cases following initial suggestions of Grothendieck.
 It was subsequently shown by Deligne in \cite{delignedet} that for any exact category (in the sense of Quillen \cite{quillen}) there exists a universal
 determinant functor that takes values in an associated category of  `virtual  objects'.

Such determinant functors have hitherto played a key role in the formulation with respect to non-commutative coefficient rings of arithmetic special value conjectures, including the equivariant Tamagawa Number Conjecture that was formulated in \cite{bf}, following the seminal work of Bloch and Kato \cite{bk} and of Fontaine and Perrin-Riou \cite{fpr}. 
 
In an earlier article \cite{bses}, we constructed a canonical family of extensions of the determinant functor of Grothendieck, Knudsen and Mumford in the setting of the derived categories of perfect complexes over orders in finite-dimensional separable algebras that need not be commutative. These `reduced determinant functors' are of a more explicit, and concrete, nature than are the virtual objects used in \cite{bf} and, after establishing certain necessary results in integral representation theory, we were able, as a first application, to use them to develop a theory of non-commutative Euler systems in the setting of $p$-adic representations that are defined over arbitrary number fields and are endowed with the action of an arbitrary Gorenstein $\ZZ_p$-order. 

In this article we consider the question of whether the general approach developed in  \cite{bses} can be expected to have concrete consequences for the study of special value conjectures relative to  non-commutative coefficient rings. In particular, we answer this question affirmatively in the special case of the equivariant Tamagawa Number Conjecture for $\mathbb{G}_m$, or ${\rm eTNC}(\mathbb{G}_m)$ for short in the rest of this introduction, relative to arbitrary Galois extensions of number fields. 

As some motivation for focusing on this case, we recall that ${\rm eTNC}(\mathbb{G}_m)$ is both known to strengthen many classical refinements of Stark's seminal conjectures on the leading terms at zero of Artin $L$-series (see, for example, \cite{dals}) and can also, via the philosophy described by Huber and Kings in \cite[\S3.3]{hk} and by Fukaya and Kato in \cite[\S2.3.5]{fukaya-kato}, be seen to  play an important role in the study of the most general case of the equivariant Tamagawa Number Conjecture (in fact, in \cite[Ch.\@ I, \S\,3.3]{K2} Kato even refers to this special case of the conjecture as `the universal case').  

However, in order to directly apply the constructions of \cite{bses} to the study of ${\rm eTNC}(\mathbb{G}_m)$, we must first prove several intermediate results which we feel may themselves be of independent interest. To better orientate the reader through the article, therefore, we shall separate the remainder of this discussion according to these intermediate steps.

\subsubsection{Zeta elements and relative $K$-theory}

In the first part of the article (comprising \S\ref{reduced det section} to \S\ref{rkt section}) we prove some results in relative $K$-theory. 

These results are established in the setting of an order $\mathcal{A}$ that is defined over a Dedekind domain $R$ and spans a (finite-dimensional) separable algebra $A$ over the quotient field of $R$. Recalling that the `Whitehead order' $\xi(\mathcal{A})$ of $\mathcal{A}$ is a canonical $R$-order in the centre of $A$ that is defined  in \cite{bses}, we shall here introduce natural notions of `locally-primitive' and `primitive' bases over $\xi(\mathcal{A})$ for the reduced determinant lattices of perfect complexes of $\mathcal{A}$-modules. 

By using these notions, we shall then give (in Theorem \ref{ltc2}) a concrete interpretation of the sort of equalities in relative $\K_0$-groups that have hitherto underpinned the formulation of refined special value conjectures relative to non-commutative coefficient rings. This interpretation can be computationally useful and, in particular, leads to an explicit reinterpretation of ${\rm eTNC}(\mathbb{G}_m)$ in terms of a natural non-commutative generalization of the notion of `zeta element' introduced by Kato in \cite{K1} (for details see Remark \ref{remark etnc}). 

We recall that the notion of zeta element plays a key role in the strategy to investigate ${\rm eTNC}(\mathbb{G}_m)$ over abelian extensions of arbitrary number fields that is developed by Kurihara and the present authors in \cite{bks, bks2}.  In a little more detail, then, one of the main aims of the present article will be to use the techniques introduced in \cite{bses} to extend the latter strategy to the setting of arbitrary  Galois extensions of number fields.

\subsubsection{Non-commutative Euler systems} After reviewing in \S\ref{hnase section} relevant properties of the canonical Selmer modules and modified \'etale  cohomology complexes that are constructed in  \cite{bks}, our next step will be to adapt, and in appropriate ways refine, aspects of the theory of non-commutative Euler systems that is developed in 
 \cite{bses1}. 
 
To this end,  in \S\ref{res section} we introduce a notion of  `(higher rank) non-commutative (pre-)Euler system for $\mathbb{G}_m$'  and use a detailed analysis of the compactly supported $p$-adic cohomology of $\ZZ_p$ to give an unconditional construction of such systems over any number field. In this way we are able, for example, to strengthen the main result of \cite{bses1} concerning the existence of `extended cyclotomic Euler systems' over $\QQ$ (for details see Theorem \ref{cyclotomiccor}).

More generally, for any Galois extension of number fields $E/K$, we use the values at zero of higher derivatives of the Artin $L$-series of complex characters of $\Gal(E/K)$ to unconditionally define a `non-commutative Rubin-Stark element' $\varepsilon^{\rm RS}_{E/K}$ for $E/K$. 

These elements belong, a priori, to the real vector spaces spanned by an appropriate reduced exterior power bidual of unit groups and we show that, for any fixed set $\Sigma$ of archimedean places of $K$, as $E$ varies over the finite Galois extensions of $K$ (in some fixed algebraic closure) in which all places in $\Sigma$ split completely, the elements $\varepsilon^{\rm RS}_{E/K}$ constitute a pre-Euler system  $\varepsilon_{K,\Sigma}^{\rm RS}$ of rank $|\Sigma|$ for $\mathbb{G}_m$ over $K$. 

Such systems $\varepsilon_{K,\Sigma}^{\rm RS}$ will be referred to as `Rubin-Stark pre-Euler systems for $K$' and play a fundamental role in our approach. In preparation for such applications, we formulate, as Conjecture \ref{MRSconjecture}, a natural generalization to non-abelian Galois extensions of the central conjecture formulated by Rubin in \cite{R} and show that this `non-commutative Rubin-Stark Conjecture' implies that the systems $\varepsilon_{K,\Sigma}^{\rm RS}$ are non-commutative Euler (rather than just pre-Euler)  systems of rank $|\Sigma|$ for $\mathbb{G}_m$ over $K$. 

\subsubsection{Non-commutative Iwasawa theory} In \S\ref{hrncit section} and  \S\ref{can res semi section} we then develop certain aspects of non-commutative Iwasawa theory that are necessary for our approach. 

As a first step, in \S\ref{hrncit section} we formulate  a `main
conjecture of higher rank non-commutative $p$-adic Iwasawa theory' for $\mathbb{G}_m$ over an arbitrary number field $K$. This 
prediction is stated as Conjecture \ref{hrncmc} and uses the formalism of primitive bases developed in \S\ref{rkt section} to express a precise connection between the non-commutative Rubin-Stark pre-Euler system $\varepsilon_{K,\Sigma}^{\rm RS}$ and the reduced determinant of the compactly supported $p$-adic cohomology of $\ZZ_p$ over a compact $p$-adic Lie extension of $K$ in which all places in $\Sigma$ split completely. 

Conjecture \ref{hrncmc} can be seen to
simultaneously generalize both the higher rank main conjecture of (commutative)
Iwasawa theory that is formulated by Kurihara and the present authors in
\cite{bks2} and the standard formulation of main conjectures in non-commutative
Iwasawa theory following the approaches of Ritter and Weiss in
\cite{RitterWeiss} and of Coates, Fukaya, Kato, Sujatha and Venjakob in \cite{cfksv}. 

In particular, 
in an important special case (in which $\Sigma$ is empty) our approach allows us to deduce the validity of Conjecture
\ref{hrncmc} by using the known validity, due independently to Ritter and Weiss \cite{RWnew}, and to Kakde \cite{kakde}, of the main conjecture of non-commutative
Iwasawa theory for totally real fields.

In \S\ref{can res semi section}, we then prove the existence over arbitrary compact $p$-adic Lie extensions of $K$ of a distinguished family of resolutions of the compactly supported $p$-adic cohomology of $\ZZ_p$. This family of resolutions has two  important roles in the present article and will also have further applications  elsewhere (cf.  \S\ref{other apps}). 

In the remainder of \S\ref{can res semi section} we use these resolutions  to introduce a natural notion of `semisimplicity' for the Selmer module of $\mathbb{G}_m$ over $p$-adic Lie extensions of $K$ of rank one. This notion provides our theory with an appropriate generalization of the hypothesis that a finitely generated torsion module over the classical Iwasawa algebra should be `semisimple at zero' (which is a standard assumption that arises in relation to descent computations in contexts in which the associated  $p$-adic $L$-functions have trivial zeroes). 

In addition, we show that, in all relevant cases, our notion of semisimplicity specializes to recover the generalization, due to Jaulent \cite{jaulent}, of the `Order of Vanishing Conjecture' for $p$-adic Artin $L$-series at zero that was originally independently conjectured (for CM fields) by Gross \cite{G0} and Kuz'min \cite{kuzmin}. 


\subsubsection{The Generalized Gross-Stark Conjecture}

As the final preparatory step in our approach, in \S\ref{IMRSsection} we formulate a generalization to arbitrary Galois extensions of number fields of the ($p$-adic) Gross-Stark Conjecture. We recall that the latter conjecture was originally formulated by Gross \cite{G0} in the setting of CM Galois extensions of totally real fields and has been unconditionally verified for all odd $p$ by Dasgupta, Kakde and Ventullo in \cite{dkv}. 

To formulate our conjecture we first use the resolutions constructed in \S\ref{can res semi section} to introduce, under appropriate hypotheses, a canonical notion of the `value of a higher derivative' of the non-commutative pre-Euler system $\varepsilon_{K,\Sigma}^{\rm RS}$. 
By comparing the reduced exterior products of certain natural  Bockstein maps with those of canonical `valuation' maps, we also construct a canonical  `$\mathscr{L}$-invariant' homomorphism between the exterior power biduals (of differing ranks) of unit groups. 

Then, in any given setting, the `Generalized Gross-Stark Conjecture' of  Conjecture \ref{GGSconjecture} predicts an explicit formula for the value of an appropriate higher derivative of $\varepsilon_{K,\Sigma}^{\rm RS}$ in terms of the image under the relevant $\mathscr{L}$-invariant map of a non-commutative Rubin-Stark element of an appropriately higher rank. 

This conjectural derivative formula encodes families of significant, and even sometimes explicit, integral relations between the non-commutative Rubin-Stark elements of differing ranks that are defined relative to finite Galois extensions of $K$. In this way, the conjecture therefore also encodes information about families of fine integral relations between the values at zero of higher derivatives (of different orders) of the Artin $L$-series defined over $K$. 

For example, in the setting of CM extensions of totally real fields, we can show that the `odd component' of the derivative formula in Conjecture \ref{GGSconjecture} precisely recovers the classical Gross-Stark Conjecture. In this special case, therefore, we can thereby derive the unconditional validity of Conjecture \ref{GGSconjecture} as a consequence of the main result of Dasgupta et al \cite{dkv}.

\subsubsection{${\rm eTNC}(\mathbb{G}_m)$} 

In \S\ref{newsection6} we finally establish a concrete link between the results obtained in earlier sections and the conjecture ${\rm eTNC}(\mathbb{G}_m)$. 

Before doing so, however, we start by reviewing what is presently known concerning eTNC$(\mathbb{G}_m)$ over non-abelian Galois extensions and, at the same time, clarify aspects of the results of the first author in \cite{burns2}. 

Under a suitable semisimplicity hypothesis, we are then able to establish over arbitrary finite Galois extensions of number fields 
 a precise link between the non-commutative
higher rank Iwasawa main conjecture (of Conjecture \ref{hrncmc}), the explicit derivative formula given by the Generalized Gross-Stark Conjecture (of Conjecture \ref{GGSconjecture}) 
 and the reinterpretation of eTNC$(\mathbb{G}_m)$ in terms of zeta elements that is presented in \S\ref{rkt section}. 
 
 The latter result
  is stated explicitly as Theorem \ref{strategytheorem} and constitutes our desired generalization to arbitrary Galois extensions of the main result of
  Kurihara and the present authors in \cite{bks2}.  

The result of Theorem \ref{strategytheorem} sheds new light on the essential nature of ${\rm eTNC}(\mathbb{G}_m)$ over arbitrary Galois extensions and also, more concretely, presents a strategy for obtaining evidence for it beyond the case of CM-extensions of totally real fields. It is therefore to be hoped that, in the same way the main result of \cite{bks2} has motivated subsequent work and led to significant arithmetic results (see, for example, the recent articles of Bley and Hofer \cite{blh} and of Bullach and Hofer \cite{bh}), the strategy presented here will lead to concrete new evidence for ${\rm eTNC}(\mathbb{G}_m)$ for families of non-abelian Galois extensions. For example, even in the case of non-abelian CM extensions of totally real fields, Theorem \ref{strategytheorem} already gives a significant simplification of the proofs of results in \cite{burns2} (see, for example, Remark \ref{last remark}).  

\subsubsection{Other connections}\label{other apps}
 The techniques developed here, and in the related articles \cite{bses} and \cite{bses1}, also have consequences for the formulation and study of special value conjectures relative to non-commutative coefficient rings beyond the special cases that we focus on in the present article. 

For example, in work of Puignau, Seo and the present authors \cite{bpss} it is shown that the distinguished family of resolutions constructed in \S\ref{can res semi section} can be used to define canonical `non-commutative Artin-Bockstein maps' that extend the classical reciprocity maps of local class field theory to non-abelian Galois extensions (of local fields) and thereby to formulate a generalization to arbitrary finite Galois extensions of the `refined class number formula conjecture for $\mathbb{G}_m$' (or, as it is also often known, the `Mazur-Rubin-Sano Conjecture'). We recall that the latter conjecture was  independently conjectured for abelian extensions by Mazur and Rubin in \cite{MR2} and by the second author in \cite{sano} and has played a key role in the study of ${\rm eTNC}(\mathbb{G}_m)$ over such extensions. We further note that its natural generalization to arbitrary Galois extensions (as formulated precisely in \cite[Conj. 5.1]{bpss}) essentially constitutes a refined version `at finite level' of the Iwasawa-theoretic Generalised Gross-Stark Conjecture that we formulate here as Conjecture \ref{GGSconjecture}. 

One can also change focus from $\mathbb{G}_m$ to abelian varieties. We recall that, in this direction, the article \cite{mct} of Macias Castillo and Tsoi already gives interesting applications of the algebraic constructions made in \cite{bses} to the study of the Hasse-Weil-Artin $L$-series of dihedral twists of elliptic curves over
general number fields. 

In addition, in \cite{bks3} Kurihara and the present authors have established an analogue of the main result of \cite{bks2} that is relevant to the study of the Birch-Swinnerton-Dyer Conjecture for elliptic curves over $\QQ$ (see, in particular, \cite[Th. 7.6 and Rem. 7.7]{bks3}). In light of this, it would be interesting to know if the general approach developed here can be adapted to shed further light on concrete relations between the Birch-Swinnerton-Dyer Conjecture and the ${\rm GL}_2$ Main Conjecture of Iwasawa theory for elliptic curves without complex multiplication of Coates et al \cite{cfksv}. 

In yet another direction, the article \cite{ffmcmm} of de Frutos-Fern\'andez, Macias Castillo and Martinez Marqu\'es uses aspects of our approach to study class number formulas for Drinfeld modules over global function fields.

We observe finally that, in a more general setting, the theory of non-commutative Euler systems developed here and in \cite{bses1} is also relevant to the strategies described by Huber and Kings in \cite[\S3.3]{hk} and by Fukaya and Kato in \cite[\S2.3.5]{fukaya-kato} to study the general case of the equivariant Tamagawa number conjecture. In particular, given a motive $M$ defined over $\QQ$, it is possible to `twist' in a natural sense the family of cyclotomic non-commutative Euler systems that is constructed unconditionally in Theorem \ref{cyclotomiccor} below in order to obtain families of classical, commutative Euler systems (of suitable rank) for lattices $T_p(M)$ in the various $p$-adic realisations of $M$. Then, by applying the general theory of higher rank Euler systems developed by Sakamoto and the present authors in \cite{bss} to these systems, one can study the Selmer modules of $T_p(M)$. This aspect of our  theory  will be discussed elsewhere. 

\subsection{General notation}For the reader's convenience, we finally review some general notation that will be used throughout the article.

For each ring $R$, we write $\zeta(R)$ for its centre and $R^{\rm op}$ for the corresponding opposite ring (so that $\zeta(R) = \zeta(R^{\rm op}))$. 

By an $R$-module we shall, unless explicitly stated otherwise, mean a left $R$-module. The transpose of a matrix $M$ over $R$ is denoted by $M^{\rm tr}$. 

We write $\ZZ_{(p)}$ for the localization of $\ZZ$ at a prime number $p$. We write  $A_p$ for the pro-$p$ completion of an abelian group, or complex of abelian groups, $A$  (so that $\ZZ_p$ is the ring of $p$-adic integers) and use similar notation for morphisms. We often abbreviate $E\otimes_\ZZ A$ to $E\cdot A$ for a field $E$.

We fix an algebraic closure $\QQ_p^c$ of the field of $p$-adic rationals $\QQ_p$ and a completion
 $\CC_p$ of $\QQ_p^c$. For a finite group $\Gamma$ we write ${\rm Ir}(\Gamma)$ and ${\rm Ir}_p(\Gamma)$ for the respective sets of
irreducible $\CC$-valued and $\CC_p$-valued characters of $\Gamma$. 

If $A$ is a $\Gamma$-module, then we endow linear duals such as $A^\vee:=\Hom_\ZZ(A,\QQ/\ZZ)$ and $\Hom_\ZZ(A,\ZZ)$ with the natural contragredient action of $\Gamma$.

For a Galois extension of fields $E/F$ we usually abbreviate $\Gal(E/F)$ to $G_{E/F}$. \\

\noindent{}{\bf Acknowledgements} This article is warmly dedicated to the memory of Jan Nekov\'a\v r, whose very generous encouragement and remarkably insightful suggestions regarding this project (and numerous others) made an enormous difference to us. 

It is also a great pleasure for us to thank Masato Kurihara for his continual encouragement and for many helpful suggestions concerning this, and closely related, projects. 

In addition, the first author would like to thank Oliver Braunling, Cornelius Greither, 
Daniel Macias Castillo,  Henri Johnston, Mahesh Kakde, Andreas Nickel, Daniel Puignau, Soogil Seo, Al Weiss and Malte Witte for stimulating discussions related to the problems discussed here. He would also like to acknowledge the wonderful working conditions provided by Yonsei University in Seoul when a key part of this project was undertaken. 

\section*{\large{Part I: Reduced determinants and relative $K$-theory}}
\medskip

The equivariant Tamagawa Number Conjecture is formulated in \cite{bf} in terms of Deligne's categories of virtual objects and takes the form of an equality in the relative algebraic $\K_0$-group of an appropriate extension of rings. 

Alternative approaches to the formulation of conjectures in such groups were subsequently developed by Breuning and the first author \cite{bb} using a theory of `equivariant Whitehead torsion', by Fukaya and Kato  \cite{fukaya-kato} using a theory of `localized  $\K_1$-groups', by Muro, Tonks and Witte \cite{mtw} using Waldhausen $K$-theory and by Braunling \cite{ob, ob2, ob4} using a theory of `equivariant Haar measures' (and see also the associated article \cite{obetal} of Braunling, Henrard and van Roosmalen).


In this first part of the article (comprising \S\ref{reduced det section} to \S\ref{rkt section}), our main aim is to explain how the theory of `reduced determinant functors' that is  developed in \cite{bses} can also be used to give a new 
approach to the formulation of such conjectures.

The key contribution that we shall actually make in these sections is to introduce a  notion of `(locally-)primitive basis' in the setting of reduced determinants and to show that this notion is both well-defined and functorially well-behaved on the reduced determinants that arise from objects in suitable derived categories (for details see \S\ref{plp section}).

To relate this construction to relative algebraic $K$-theory, we shall also introduce a natural generalization of the notion of `zeta element' that was first used in a (commutative) arithmetic setting by Kato in \cite{K1} (cf. Definition \ref{def zeta} and Example \ref{ex}).

Our main result in this regard (Theorem \ref{ltc2}) is then entirely $K$-theoretic in nature and perhaps of some independent interest. However, it also provides a natural, and very concrete, interpretation of the equalities that underlie several existing refined special value conjectures in arithmetic (cf.  Remark \ref{remark etnc}). In addition, and more importantly, it can also be computationally useful in such contexts. 

In particular, in later sections of this article, we shall find that this approach leads to a more direct formulation of main conjectures in non-commutative Iwasawa theory, to a natural generalization of the classical Gross-Stark Conjecture, to improvements in the descent formalism that relates such conjectures to the relevant cases of ${\rm eTNC}(\mathbb{G}_m)$ and thereby to the derivation of further concrete evidence in support of ${\rm eTNC}(\mathbb{G}_m)$ itself.

\section{Reduced determinants}\label{reduced det section}

For the convenience of the reader, we shall first review relevant facts concerning the theory of reduced determinant functors developed by the present authors in  \cite{bses}.

To do so we fix a Dedekind domain $R$ with field of fractions $F$ of characteristic zero. We also fix a finite dimensional separable $F$-algebra $A$ and an $R$-order $\mathcal{A}$ in $A$.


\subsection{Reduced exterior powers}\label{Exterior powers over semisimple rings}


We write the Wedderburn decomposition of $A$ as 
\[ A = {\prod}_{i \in I}A_i,\]
where $I$ is a finite index set and each algebra $A_i$ is of the form ${\rm M}_{n_i}(D_i)$ for a suitable natural number $n_i$ and a division ring $D_i$ with $F \subseteq \zeta(D_i)$.

We next choose a field extension $E$ of $F$ that, for every index $i$, contains $\zeta(D_i)$ and splits $D_i$ and then fix associated isomorphisms 
\[ D_i\otimes_{\zeta(D_i)}E \cong {\rm M}_{m_i}(E)\]
for a suitable natural number $m_i$.

For each index $i$ we also fix an indecomposable idempotent $e_i$ of ${\rm M}_{m_i}(E)$.
 Then the direct sum $V_i$ of $n_i$-copies of $e_i{\rm M}_{m_i}(E)$ is a simple left $A_{i,E}:=A_i\otimes_{\zeta(A_i)}E$-module.

Each finitely generated $A$-module $M$ decomposes as a direct sum 
\[ M = {\bigoplus}_{i \in I}M_i,\]
where $M_i$ denotes the $A_i$-module $A_i\otimes_AM$. 

For each non-negative integer $r$ we then define the $r$-th reduced exterior power of $M$ over $A$ by setting
\begin{equation}\label{non comm ext power} {\bigwedge}_A^r M:={\bigoplus}_{i \in I}{\bigwedge}_{E}^{rd_i}(V_i^\ast\otimes_{A_{i,E}}M_{i,E}),\end{equation}
with $d_i := \dim_E(V_i) = n_im_i$, $M_{i,E}:=M_i\otimes_{\zeta(A_i)}E$, and $V_i^\ast:=\Hom_E(V_i,E)$. This construction depends on $E$, but is independent of each space $V_i$ up to isomorphism.

To discuss linear duals we note that $\Hom_A(M,A)$ has a natural structure as left $A^{\rm op}$-module and we consider its  exterior powers over $A^{\rm op}$.
 We also note that, for each index $i$, the space $V_i^\ast$ is a simple left $A^{\rm op}_{i,E}$-module, and that its dual $V_i^{\ast \ast}$ is canonically isomorphic to $V_i$.
In this case, the definition above therefore gives
\begin{equation}\label{non comm ext power dual} {\bigwedge}_{A^{\rm op}}^r \Hom_A(M,A) = {\bigoplus}_{i \in I}{\bigwedge}_E^{rd_i} (V_i\otimes_{A_{i,E}^{\rm op}}\Hom_{A_{i,E}}(M_{i,E},A_{i,E})).\end{equation}

For each integer $s$ with $0\leq s\leq r$ there are then natural duality pairings
\begin{equation}\label{duality pairing}{\bigwedge}_{A}^rM\times{\bigwedge}_{A^{\rm op}}^s\Hom_{A}(M,A)\to{\bigwedge}_{A}^{r-s}M,\,\,\,\, (m,\varphi)\mapsto\varphi(m).\end{equation}


To make this pairing explicit we fix, for each index $i$, an $E$-basis $\{v_{ij}: 1\le j\le m_i\}$ of $e_i{\rm M}_{m_i}(E)$. If $A_i$ is commutative, and hence a field (so that $m_i = 1$), then we always take  $v_{i1}$ to be identity element of $E$.
 The (lexicographically-ordered) set 
 \[ \underline{\varpi}(i) := \{\varpi_{iaj}: 1\le a\le n_i, 1 \le j \le m_i\}\]
is then an $E$-basis of $V_i$, where  $\varpi_{iaj}$ is the element of $V_i$ that is equal to $v_{ij}$ in its $a$-th component and is zero in all other components.

For any subsets $\{m_a\}_{1\le a\le r}$ of $M$ and $\{\varphi_a\}_{1\le a\le r}$ of $\Hom_A(M,A)$ we then set
\begin{equation}\label{non-commutative wedge} \wedge_{a=1}^{a=r}m_a :=\bigl({\wedge}_{1\leq a \leq r}(
{\wedge}_{x\in \underline{\varpi}(i)}(x^\ast\otimes m_{ai}))\bigr)_{i \in I} \in {\bigwedge}_A^rM \end{equation}
%
and
\begin{equation} \label{non-commutative wedge 2}\wedge_{a=1}^{a=r}\varphi_a := \bigl({\wedge}_{1\leq a \leq r}({\wedge}_{x\in \underline{\varpi}(i)}(x \otimes \varphi_{ai}))\bigr)_{i \in I} \in {\bigwedge}_{A^{\rm op}}^r \Hom_A(M,A).\end{equation}
Here we write $m_{ai}$ and $\varphi_{ai}$ for the projections of $m_a$ to $M_{i,E}$ and of $\varphi_{a} $ to $\Hom_{A_{i,E}}(M_{i,E},A_{i,E})$ and $\{x^*: x \in \underline{\varpi}(i)\}$ for the basis of $V_i^\ast$ that is dual to $\underline{\varpi}(i)$.


These constructions clearly depend on the collection 
\[ \varpi := \{ \underline{\varpi}(i): i \in I\}\]
of ordered bases and so should strictly be written as
$`\wedge_{\varpi}$' rather than $`\wedge$'. However, for simplicity, we prefer not to indicate this dependence, since in practice it does  not cause difficulties.

For example, in the sequel we will often use the following fact (proved in \cite[Lem. 4.13]{bses}) that is independent of the choice of $\varpi$: if $M$ is a (non-zero) free $A$-module of rank $r$ with basis $\{b_a\}_{1\le a\le r}$, then for each $\varphi$ in $\End_A(M)$ one has
\begin{equation}\label{4.13} \wedge_{a=1}^{a= r}\varphi(b_a) =\nr_{\End_A(M)}(\varphi)\cdot( \wedge_{a=1}^{a= r}b_a) \in {\bigwedge}_A^rM.\end{equation}
In addition, for any natural number $r$, any subset $\{m_b\}_{1\le b\le r}$ of an $A$-module $M$ and any subset $\{\varphi_a\}_{1\le a\le r}$ of $\Hom_A(M,A)$, \cite[Lem. 4.10]{bses} implies that the pairing (\ref{duality pairing}) sends $\bigl(\wedge_{b=1}^{b=r}m_b,\wedge_{a=1}^{a=r}\varphi_a\bigr)$ to the element 
\begin{equation}\label{rn rel} (\wedge_{a=1}^{a=r}\varphi_a)(\wedge_{b=1}^{b=r}m_b)={\rm Nrd}_{A^{\rm op}}((\varphi_a(m_b))_{1\leq a,b\leq r})\in \zeta(A).\end{equation}
%
\vskip 0.1truein

\begin{remark}\label{group rings}{\em Let $\Gamma$ be a finite group and $F$ a subfield of $\CC$. For each  $\chi$ in ${\rm Ir}(\Gamma)$ we write $n_\chi$ for the exponent of the quotient group $\Gamma/\ker(\chi)$ and $F_\chi$ for the field generated over $F$ by a primitive $n_\chi$-th root of unity. Then, following Brauer \cite{brauer}, we may fix a  representation 
\[ \rho_\chi: \Gamma \to {\rm GL}_{\chi(1)}(F_\chi)\]
of character $\chi$. In particular, if $F'$ is any finite extension of $F$ that contains $F_\chi$ for every $\chi$ in ${\rm Ir}(\Gamma)$, then the induced $F'$-linear ring homomorphisms $\rho_{\chi,*}: F'[\Gamma]\to {\rm M}_{\chi(1)}(F')$ combine to give an isomorphism of $F'$-algebras
$$F'[\Gamma]\cong{\prod}_{\chi\in {\rm Ir}(\Gamma)}{\rm M}_{\chi(1)}(F').$$

This shows that $F'$ is a splitting field for $F[\Gamma]$, that the spaces 
\[ V_\chi:= (F')^{\chi(1)},\]
considered as the first columns of the respective $\chi$-components ${\rm M}_{\chi(1)}(F')$ of $F'[\Gamma]$, are a complete set of representatives of the isomorphism classes of simple $F'[\Gamma]$-modules and that the standard $F'$-basis of $(F')^{\chi(1)}$ constitutes an ordered $F'$-basis of $V_\chi$.

In this way, a fixed choice of representations 
\[ \{\rho_\chi\}_{\chi\in {\rm Ir}(\Gamma)}\] 
as above gives rise, for each finitely generated $F[\Gamma]$-module $M$, to a canonical normalization of the constructions (\ref{non comm ext power}), (\ref{non comm ext power dual}), (\ref{non-commutative wedge}) and (\ref{non-commutative wedge 2}). We further note that, for each non-negative integer $r$, the resulting $\zeta(F[\Gamma])$-module ${\bigwedge}_{F[\Gamma]}^rM$ is then finitely generated and behaves functorially under change of $F$.}\end{remark}

\subsection{Reduced Rubin lattices}\label{integral structures section}

\subsubsection{}In the sequel we write ${\rm Spec}(R)$ for the set of all prime ideals of $R$ and ${\rm Spm}(R)$ for the subset ${\rm Spec}(R)\setminus \{ (0)\}$ of maximal ideals. 

For $\mathfrak{p}$ in ${\rm Spec}(R)$ we write $R_{(\mathfrak{p})}$ for the localization of $R$ at $\mathfrak{p}$ and for an  $R$-module $M$ set $M_{(\mathfrak{p})} := R_{(\mathfrak{p})}\otimes_R M$. For $\mathfrak{p}$ in ${\rm Spm}(R)$ we also write $R_\mathfrak{p}$ for the completion of $R$ at $\mathfrak{p}$ and for an $R$-module $M$ set  $M_\mathfrak{p}:= R_\mathfrak{p}\otimes_R M$.

A finitely generated module $M$ over an $R$-order $\mathcal{A}$ is said to be `locally-free' if $M_{(\mathfrak{p})}$ is a free $\mathcal{A}_{(\mathfrak{p})}$-module for all $\mathfrak{p}$ in ${\rm Spec}(R)$. We recall that for $\mathfrak{p}$ in ${\rm Spm}(R)$, the $\mathcal{A}_{(\mathfrak{p})}$-module $M_{(\mathfrak{p})}$ is free if and only if the $\mathcal{A}_{\mathfrak{p}}$-module $M_{\mathfrak{p}}$ is free (this follows as an easy consequence of Maranda's Theorem - see, for example, \cite[Prop. (30.17)]{curtisr}). 

We write ${\rm Mod}^{\rm lf}(\mathcal{A})$ for the category of (finitely generated) locally-free $\mathcal{A}$-modules.
For $M$ in ${\rm Mod}^{\rm lf}(\mathcal{A})$ the rank of the free $\mathcal{A}_{(\mathfrak{p})}$-module $M_{(\mathfrak{p})}$ is independent of $\mathfrak{p}$ and equal to the rank of the free $A$-module $M_F:=M_{(0)}=F\otimes_R M$. We refer to this common rank as the `rank' of $M$ and denote it ${\rm rk}_\mathcal{A}(M)$.

\begin{remark}\label{loc free exam}{\em  Localization is an exact functor and so a  locally-free $\mathcal{A}$-module is projective. There are important cases in which the converse is also true.
\

\noindent{}(i) If $\mathcal{A}=R$, then every finitely generated torsion-free $\mathcal{A}$-module $M$ is locally-free, with ${\rm rk}_\mathcal{A}(M)$ equal to the dimension of the $F$-space spanned by $M$.\

\noindent{}(ii) If $G$ is a finite group for which no prime divisor of $|G|$ is invertible in $R$ and $\mathcal{A} = R[G]$ then, by a fundamental result of Swan \cite{swan0} (see also \cite[Th. (32.11)]{curtisr}), a finitely generated projective $\mathcal{A}$-module is locally-free. For any such module $M$ the product ${\rm rk}_{R[G]}(M)\cdot |G|$ is equal to the dimension of the $F$-space spanned by $M$. \

\noindent{}(iii) There are also several classes of order $\mathcal{A}$ for which a finitely generated projective 
$\mathcal{A}$-module is locally-free if and only if it spans a free $A$-module. This is the case, for example, if $\mathcal{A}$ is commutative (cf. \cite[Prop. 35.7]{curtisr}), or if $\mathcal{A}_{(\mathfrak{p})}$ is a maximal $R_{(\mathfrak{p})}$-order in $A$ for all $\mathfrak{p}$ in ${\rm Spm}(R)$ (cf. \cite[Th. 26.24(iii)]{curtisr}), or if $\mathcal{A} = R[G]$ for any finite group $G$ (cf. \cite[Th. 32.1]{curtisr}).}
\end{remark}

\subsubsection{}We use the  $R$-submodule of $\zeta(A)$ defined by 
\[ \xi(\mathcal{A}) := R\cdot \left\{{\rm Nrd}_A(M): M \in {\bigcup}_{n \ge 0}{\rm M}_n(\mathcal{A})\right\}.\]
For each $\mathfrak{p}$ in ${\rm Spm}(R)$, we similarly define an $R_{(\mathfrak{p})}$-submodule $\xi(\mathcal{A}_{(\mathfrak{p})})$ of $\zeta(A)$ and an $R_{\mathfrak{p}}$-submodule $\xi(\mathcal{A}_{\mathfrak{p}})$ of $\zeta(A_\mathfrak{p})$ and we recall that    
\[ \xi(\mathcal{A}) = {\bigcap}_{\mathfrak{p}\in {\rm Spm}(R)}\xi(\mathcal{A}_{(\mathfrak{p})})\]
by \cite[Lem. 2.2(iii)]{bpss}. Further, by \cite[Lem. 3.2]{bses}, one knows that $\xi(\mathcal{A})$ is an $R$-order in $\zeta(A)$ (referred to as the `Whitehead order' of $\mathcal{A}$ in loc. cit.), that $\xi(\mathcal{A}) = \mathcal{A}$ if and only if $\mathcal{A}$ is commutative and that, in general, one has 
\[ \xi(\mathcal{A})_{(\mathfrak{p})} = \xi(\mathcal{A}_{(\mathfrak{p})}) = \zeta(A)\cap \xi(\mathcal{A}_{\mathfrak{p}})\] 
for every $\mathfrak{p}$ in ${\rm Spm}(R)$.
%
%

For a finitely generated $\mathcal{A}$-module $M$, and a non-negative integer $r$, the $r$-th reduced Rubin lattice of the $\mathcal{A}$-module $M$ is then defined (in \cite[Def. 4.17, Rem. 4.18]{bses}) to be the $\xi(\mathcal{A})$-submodule of ${\bigwedge}_{A}^r M_F$ obtained by setting
\[ {\bigcap}_{\mathcal{A}}^r M :=\left\{ a \in {\bigwedge}_{A}^r M_F : (\wedge_{i=1}^{i= r}\varphi_i)(a) \in\xi(\mathcal{A}) \mbox{ for all } \varphi_1,\ldots,\varphi_r \in \Hom_\mathcal{A}(M,\mathcal{A})    \right\}.\]
%
%
In \cite[Th. 4.19]{bses}, it is shown that this $\xi(\mathcal{A})$-module is finitely generated, has good functorial properties and  varies in a natural way with the choices of ordered bases $\varpi$ used to normalise the definition (\ref{non-commutative wedge 2}) of reduced exterior products. 

We next note that the argument of \cite[Lem. 4.16]{bses} implies injectivity of the canonical
   `evaluation' homomorphism of $\xi(\mathcal{A})$-modules
\begin{equation}\label{ev map def} {\rm ev}^r_M: {\bigcap}_{\mathcal{A}}^{r}M \to
  {\prod}_{\underline{\varphi}}\xi(\mathcal{A}); \,\,\, x\mapsto
  \bigl(({\wedge}_{j=1}^{j=r}\varphi_j)(x)\bigr)_{\underline{\varphi}},\end{equation}
  where in the direct product $\underline{\varphi} = (\varphi_1, \dots, \varphi_r)$
  runs over all elements of $\Hom_{\mathcal{A}}(M,\mathcal{A})^{r}$.  

We further recall from \cite[Prop. 5.9]{bses} that if $M$ belongs to ${\rm Mod}^{\rm lf}(\mathcal{A})$, with $r := {\rm rk}_\mathcal{A}(M)$, then
${\bigcap}_\mathcal{A}^rM$ is an invertible $\xi(\mathcal{A})$-module with the property that   
\[ {\bigcap}_\mathcal{A}^rM = {\bigcap}_{\mathfrak{p}\in {\rm Spm}(R)} \bigl({\bigcap}_{\mathcal{A}}^rM\bigr)_{(\mathfrak{p})}\]
and so can be explicitly computed via its localizations as follows: if for each $\frp$ in ${\rm Spm}(R)$ one fixes an ordered $\mathcal{A}_{(\mathfrak{p})}$-basis $\underline{b}_{\frp} = \{b_{\frp,j}\}_{1\le j\le r}$ of $M_{(\mathfrak{p})}$, then 
\begin{equation}\label{local rl} \bigl({\bigcap}_\mathcal{A}^rM\bigr)_{(\mathfrak{p})} =  {\bigcap}_{\mathcal{A}_{(\mathfrak{p})}}^r M_{({\mathfrak{p}})} = \xi(\mathcal{A}_{(\mathfrak{p})})\cdot \wedge_{j=1}^{j=r}b_{\frp,j}.\end{equation}
%
%
%
%
In particular, if $M$ is a free $\mathcal{A}$-module, with ordered basis $\{b_{j}\}_{1\le j\le r}$, then one can take $b_{\frp,i} = b_i$ for all $i$ with $1\le i\le r$ and so the $\xi(\mathcal{A})$-module 
\begin{equation}\label{free rem} {\bigcap}^r_\mathcal{A}M = \xi(\mathcal{A})\cdot \wedge_{j=1}^{j=r}b_{j}\end{equation}
is free of rank one (with basis $\wedge_{j=1}^{j=r}b_{j}$).

\begin{remark}{\em For other examples of the explicit computation of Whitehead orders and reduced Rubin lattices see \cite[Exam. 3.4, Exam. 3.5 and Rem. 4.18]{bses}.}\end{remark}

\subsection{Reduced determinant functors}


\subsubsection{}We write $\Der(\mathcal{A})$ for the derived category of (left) $\mathcal{A}$-modules. We also write $\DC^{\rm lf}(\mathcal{A})$ for the category of bounded complexes of modules in ${\rm Mod}^{\rm lf}(\mathcal{A})$ and $\Der^{{\rm lf}}(\mathcal{A})$ for the full triangulated subcategory of $\Der(\mathcal{A})$ comprising complexes that are isomorphic to a complex in $\DC^{\rm lf}(\mathcal{A})$.

We write $\K_0^{\rm lf}(\mathcal{A})$ for the Grothendieck group of ${\rm Mod}^{\rm lf}(\mathcal{A})$ and recall that each object $C$ of $\Der^{{\rm lf}}(\mathcal{A})$ gives rise to a canonical `Euler characteristic' element $\chi_\mathcal{A}(C)$ in $\K^{\rm lf}_0(\mathcal{A})$ (for details see \cite[\S5.1.4]{bses}). 

We also write ${\rm SK}_0^{\rm lf}(\mathcal{A})$ for the kernel of the homomorphism $\K^{\rm lf}_0(\mathcal{A}) \to \ZZ$ that is induced by sending each $M$ in ${\rm Mod}^{\rm lf}(\mathcal{A})$ to ${\rm rk}_\mathcal{A}(M)$.

We then write $\DC^{{\rm lf},0}(\mathcal{A})$ for the subcategory of $\DC^{\rm lf}(\mathcal{A})$ comprising complexes $C$ for which $\chi_\mathcal{A}(C)$ belongs to ${\rm SK}_0^{\rm lf}(\mathcal{A})$ and $\Der^{{\rm lf},0}(\mathcal{A})$ for the full triangulated subcategory of $\Der^{\rm lf}(\mathcal{A})$ comprising complexes $C$ for which $\chi_\mathcal{A}(C)$ belongs to ${\rm SK}_0^{\rm lf}(\mathcal{A})$. (The latter condition is equivalent to requiring that $C$ is isomorphic in $\Der(\mathcal{A})$ to an object of $\DC^{{\rm lf},0}(\mathcal{A})$).


\begin{remark}{\em Assume $\mathcal{A}$ is such that, for all $\mathfrak{p}$ in ${\rm Spec}(R)$, the reduced projective class group ${\rm SK}_0(\mathcal{A}_{(\mathfrak{p})})$ of the $R_{(\mathfrak{p})}$-order $\mathcal{A}_{(\mathfrak{p})}$ vanishes (as is the case, for example, for all of the orders discussed in Remark \ref{loc free exam}). Then, in this case, $\Der^{{\rm lf},0}(\mathcal{A})$ is naturally equivalent to the full triangulated subcategory of the derived category $\Der^{\rm perf}(\mathcal{A})$ of perfect complexes of $\mathcal{A}$-modules that comprises all (perfect) complexes whose Euler characteristic in $\K_0(\mathcal{A})$ belongs to the subgroup ${\rm SK}_0(\mathcal{A})$. A proof of this fact is given in \cite[Lem. 5.2]{bses}. }\end{remark}

\subsubsection{}In the next result we record relevant properties of the reduced determinant functors constructed in \cite[\S5]{bses}.

To state this result we recall that, in terms of the notation used in \S\ref{Exterior powers over semisimple rings}, the `reduced rank' of a module $M$ in ${\rm Mod}^{\rm lf}(\mathcal{A})$ is defined to be the vector
\[ {\rm rr}_{\mathcal{A}}(M) := \bigl({\rm rk}_\mathcal{A}(M)\cdot d_i \bigr)_{i \in I},\]
where each natural number $d_i$ is defined as in \S \ref{Exterior powers over semisimple rings}. By using \cite[Lem. 5.1]{bses}, this vector is  regarded as a locally-constant function on ${\rm Spec}(\xi(\mathcal{A}))$.

We also write $\Der^{\rm lf}(\mathcal{A})_{\rm Isom}$ for the subcategory of $\Der^{\rm lf}(\mathcal{A})$ in which morphisms are restricted to be isomorphisms and $\mathcal{P}(\xi(\mathcal{A}))$ for the category of graded invertible $\xi(\mathcal{A})$-modules. We recall that each object of $\mathcal{P}(\xi(\mathcal{A}))$ is a pair 
\[ X = (X^{\rm u},{\rm gr}(X))\]
comprising an `ungraded part' $X^{\rm u}$ that is a locally-free, rank one, $\xi(\mathcal{A})$-module and a `grading' ${\rm gr}(X)$ that is a locally-constant function  on ${\rm Spec}(\xi(\mathcal{A}))$. 

Finally, we note that the concept of `extended determinant functor' originates in \cite{knumum} and is recalled precisely in \cite[Def. 5.13]{bses}.

Then the  following result is an immediate consequence of \cite[Th. 5.4]{bses}.

\begin{theorem}\label{ext det fun thm} For each set of ordered bases $\varpi$ as in \S\ref{Exterior powers over semisimple rings}, there exists a canonical extended determinant functor
\begin{equation*} {\rm d}_{\mathcal{A},\varpi}: \Der^{\rm lf}(\mathcal{A})_{\rm Isom} \to \mathcal{P}(\xi(\mathcal{A}))\end{equation*}
that has all of the following properties.

\begin{itemize}
\item[(i)] For each exact triangle 
\[ C' \to  C \to C'' \to  C'[1]\]
in $\Der^{\rm lf}(\mathcal{A})$ there exists a canonical isomorphism in $\mathcal{P}(\xi(\mathcal{A}))$
\[ {\rm d}_{\mathcal{A},\varpi}(C') \otimes {\rm d}_{\mathcal{A},\varpi}(C'') \xrightarrow{\sim} {\rm d}_{\mathcal{A},\varpi}(C)\]
 that is functorial with respect to isomorphisms of triangles.
\item[(ii)] If $C$ belongs to $\Der^{\rm lf}(\mathcal{A})$ is such that every module  $H^i(C)$ also belongs to 
$\Der^{\rm lf}(\mathcal{A})$, then there exists a canonical isomorphism in $\mathcal{P}(\xi(\mathcal{A}))$
    \[ {\rm d}_{\mathcal{A},\varpi}(C) \cong {\bigotimes}_{i \in \ZZ}{\rm d}_{\mathcal{A},\varpi}(H^i(C))^{(-1)^i}\]
that is functorial with respect to quasi-isomorphisms.
\item[(iii)] For each $P$ in ${\rm Mod}^{\rm lf}(\mathcal{A})$ one has
\[ {\rm d}_{\mathcal{A},\varpi}(P) = \bigl({\bigcap}^{{\rm rk}_{\mathcal{A}}(P)}_{\mathcal{A}}P,{\rm rr}_\mathcal{A}(P)\bigr),\]
where the reduced Rubin lattice ${\bigcap}^{{\rm rk}_{\mathcal{A}}(P)}_{\mathcal{A}}P$ is defined with respect to $\varpi$.
\item[(iv)] The restriction of ${\rm d}_{\mathcal{A},\varpi}$ to $\Der^{{\rm lf},0}(\mathcal{A})_{\rm Isom}$ is independent of the choice of $\varpi$.
\end{itemize}
\end{theorem}

\begin{remark}\label{comparison of categories0} {\em The approach of Deligne in \cite[\S4]{delignedet} constructs a `universal determinant functor' for the exact category ${\rm Mod}^{\rm lf}(\mathcal{A})$, with values in an associated commutative Picard category $\mathcal{V}^{\rm lf}(\mathcal{A})$ of `virtual objects'. In particular, in this way each determinant functor $ {\rm d}_{\mathcal{A},\varpi}$ constructed as in Theorem \ref{ext det fun thm} naturally induces a strongly monoidal functor 
\begin{equation*}\label{mon functor} \phi^{\rm lf}_{\mathcal{A},\varpi}: \mathcal{V}^{\rm lf}(\mathcal{A}) \to \mathcal{P}(\xi(\mathcal{A})).\end{equation*}
It is known that, in most cases, this functor $\phi^{\rm lf}_{\mathcal{A},\varpi}$ is not an equivalence of commutative Picard categories. For more details see \cite[Rems. 5.5 and 5.8]{bses}.}
\end{remark} 

\begin{remark}\label{exp reps}{\em For any free rank one $\zeta(A)$-module $W$ we set 
\[ W^1 := W\,\,\,\text{ and }\,\,\,  W^{-1} := \Hom_{\zeta(A)}(W,\zeta(A)),\]
with the linear dual regarded as a (free, rank one) $\zeta(A)$-module via the natural composition action. For each basis element $w$ of $W$ we set $w^1 := w$ and write $w^{-1}$ for the (unique) basis element of $W^{-1}$ that sends $w$ to $1$. For any invertible $\xi(\mathcal{A})$-module $\mathcal{L}$ we similarly define invertible $\xi(\mathcal{A})$-modules by setting  %
\[ \mathcal{L}^1 := \mathcal{L}\quad \text{ and }\quad \mathcal{L}^{-1} := \Hom_{\xi(\mathcal{A})}(\mathcal{L},\xi(\mathcal{A})).\]
For any complex $P^\bullet$ in $\DC^{\rm lf}(\mathcal{A})$ of the form
\begin{equation}\label{rep complex} \cdots \to P^a \xrightarrow{d^a} P^{a+1}\to \cdots ,\end{equation}
%
%
we then set 
\[ {\rm d}^\diamond_{\mathcal{A},\varpi}(P^\bullet) :=
{\bigotimes}_{a \in \ZZ}\bigl({\bigcap}_\mathcal{A}
^{{\rm rk}_\mathcal{A}(P^a)}P^a\bigr)^{(-1)^a},\]
where the tensor product is taken over $\xi(\mathcal{A})$, and
\begin{equation}\label{rr def} {\rm rr}_\mathcal{A}(P^\bullet) := {\sum}_{a \in \ZZ}(-1)^a\cdot {\rm rr}_\mathcal{A}(P^a) = \left({\sum}_{a \in \ZZ}(-1)^a\cdot {\rm rk}_\mathcal{A}(P^a)\right)\cdot (d_i)_{i \in I}
.\end{equation}
Then claims (i) and (iii) of Theorem \ref{ext det fun thm} combine to give a canonical identification
\[ {\rm d}_{\mathcal{A},\varpi}(P^\bullet) = \bigl( {\rm d}^\diamond_{\mathcal{A},\varpi}(P^\bullet),{\rm rr}_\mathcal{A}(P^\bullet)\bigr)\]
so that ${\rm d}_{\mathcal{A},\varpi}(P^\bullet)^{\rm u} = {\rm d}^\diamond_{\mathcal{A},\varpi}(P^\bullet).$}
\end{remark}
\vskip0.2truein
%
%
%

\section{Primitive and locally-primitive bases}\label{plp section}In this section we introduce natural notions of `primitive basis' and `locally-primitive basis' in the setting of the functors that are discussed in Theorem \ref{ext det fun thm}. 

In the sequel we always regard the set of ordered bases $\varpi$ that occurs in Theorem \ref{ext det fun thm} as fixed and use the following abbreviations for the associated reduced determinant functors
\[ {\rm d}_{\mathcal{A}}(-) := {\rm d}_{\mathcal{A},\varpi}(-)\quad \text{ and } \quad {\rm d}^\diamond_{\mathcal{A}}(-):= {\rm d}^\diamond_{\mathcal{A},\varpi}(-).\]

\subsection{Primitive bases}\label{pb section}

\subsubsection{}
%
%
%
%
We write ${\rm Mod}^{\rm f}(\mathcal{A})$ for the full subcategory of ${\rm Mod}^{\rm lf}(\mathcal{A})$ comprising (finitely generated) free $\mathcal{A}$-modules and $\DC^{\rm f}(\mathcal{A})$ for the full subcategory of $\DC^{\rm lf}(\mathcal{A})$ comprising complexes $P^\bullet$ in which every term $P^a$ belongs to ${\rm Mod}^{\rm f}(\mathcal{A})$.

Fix $P^\bullet$ in $\DC^{\rm f}(\mathcal{A})$ and for each integer $a$ set $r_a := {\rm rk}_\mathcal{A}(P^a)$. Then the  equality (\ref{free rem}) (with $M$ and $r$ taken to be each $P^a$ and $r_a$) implies that the $\xi(\mathcal{A})$-module ${\rm d}^\diamond_{\mathcal{A}}(P^\bullet)$ defined in Remark \ref{exp reps} is free of rank one. 

Further, if for each $a$ we  fix an ordered $A$-basis 
\[ \underline{b}_{a} = \{b_{a,j}\}_{1\le j\le r_a}\]
of $P_F^a$, then we obtain an element of the $\zeta(A)$-module ${\rm d}^\diamond_{\mathcal{A}}(P^\bullet)_F$ by setting
\begin{equation}\label{upsilon def} \Upsilon( \underline{b}_\bullet) := {\bigotimes}_{a\in \ZZ} (\wedge_{j=1}^{j=r_a}b_{a,j})^{(-1)^a}.\end{equation}

In particular, if each $\underline{b}_a$ is an $\mathcal{A}$-basis of $P^a$, then (\ref{free rem}) implies that 
$\Upsilon( \underline{b}_\bullet)$ is a basis of the $\xi(\mathcal{A})$-module ${\rm d}^\diamond_{\mathcal{A}}(P^\bullet)$. 

This fact motivates us to make the following definition.

\begin{definition}\label{concrete primitive}{\em For any complex $P^\bullet$ in $\DC^{\rm f}(\mathcal{A})$ we shall say that a basis element $b$ of the (free, rank one) $\xi(\mathcal{A})$-module ${\rm d}^\diamond_{\mathcal{A}}(P^\bullet)$ is `primitive' if it is equal to 
$\Upsilon( \underline{b}_\bullet)$ for some choice of ordered bases $\underline{b}_a$ of the $\mathcal{A}$-modules $P^a$.} \end{definition}

The key observation that we shall make about this definition is that it extends naturally to give a well-defined concept on objects of $\Der^{{\rm lf}}(\mathcal{A})$.

To state the precise result we recall that for any integer $d$ greater than or equal to the stable range ${\rm sr}(\mathcal{A})$ of $\mathcal{A}$, the natural homomorphism
\begin{eqnarray}\label{sst}
{\rm GL}_d(\cA) \rightarrow \K_1(\cA)
\end{eqnarray}
is surjective (cf. \cite[Th. (40.42)]{curtisr}).
We further recall Bass has shown that ${\rm sr}(\mathcal{A})$ is equal to one if $R$ is local, and hence $\mathcal{A}$ is semi-local, and that ${\rm sr}(\mathcal{A})$ is equal to two in all other cases. (For more details see \cite[Th. (40.31)]{curtisr} and \cite[Th. (40.41)]{curtisr} respectively.)


\begin{proposition}\label{concrete primitive lemma 20} Let $\lambda: P_1^\bullet\to P_2^\bullet$ be a quasi-isomorphism in $\DC^{\rm f}(\mathcal{A})$ and assume that $P_1^\bullet$ and $P_2^\bullet$ each have a term of rank at least ${\rm sr}(\mathcal{A})$. Set $r := {\rm rr}_\mathcal{A}(P^\bullet_1) = {\rm rr}_\mathcal{A}(P_2^\bullet).$

Then an element $b$ of ${\rm d}^\diamond_{\mathcal{A}}(P_1^\bullet)$ is a primitive basis of ${\rm d}^\diamond_{\mathcal{A}}(P_1^\bullet)$ if and only if the image of $(b,r)$ under ${\rm d}_{\mathcal{A}}(\lambda)$ is equal to $(b',r)$ for a primitive basis $b'$ of ${\rm d}^\diamond_{\mathcal{A}}(P_2^\bullet)$.\end{proposition}

The proof of this fact uses several results concerning the functorial behaviour of primitive bases with respect to the constructions that underlie the proof of \cite[Th. 5.4]{bses}. In this regard we recall that the arguments in loc. cit. adapt earlier arguments of Flach and the first author in \cite[\S2]{bf} and so rely on explicit constructions made by Knudsen and Mumford in \cite{knumum}.

\subsubsection{}We first establish several useful technical results.

\begin{lemma}\label{concrete primitive lemma 1} Let $P^\bullet$ be a  complex in $\DC^{\rm f}(\mathcal{A})$ of the form (\ref{rep complex}) for which there exists an integer $a$ with ${\rm rk}_{\mathcal{A}}(P^{a}) \ge {\rm sr}(\mathcal{A})$.

Let $b$ be a primitive basis of ${\rm d}^\diamond_{\mathcal{A}}(P^\bullet)$. Then any other element $b'$ of ${\rm d}^\diamond_{\mathcal{A}}(P^\bullet)$ is a primitive basis of ${\rm d}^\diamond_\mathcal{A}(P^\bullet)$ if and only if $b' = u\cdot b$ with $u$ in ${\rm Nrd}_A(\K_1(\mathcal{A}))$.
\end{lemma}

\begin{proof} In each degree $s$ set $r_s:= {\rm rk}_\mathcal{A}(P^s)$.  Then (\ref{4.13}) implies that for any choices of ordered $\mathcal{A}$-bases $\{b_{sj}\}_{1\le j\le r_s}$ and $\{b'_{sj}\}_{1\le j\le r_s}$ of $P^s$ there exists a matrix $U_s$ in ${\rm GL}_{r_s}(\mathcal{A})$ such that %
\[ \wedge_{j=1}^{j=r_s}b'_{sj} = {\rm Nrd}_A(U_s)\cdot \wedge_{j=1}^{j=r_s}b_{sj}.\]
Since ${\rm Nrd}_A(U_s)$ is a unit of $\xi(\mathcal{A})$, this observation (in each degree $s$) implies that the stated condition is necessary.


To prove sufficiency we assume that $b$ is a primitive basis of ${\rm d}^\diamond_{\mathcal{A}}(P^\bullet)$ and that $b' = u\cdot b$ with $u$ in ${\rm Nrd}_A(\K_1(\mathcal{A}))$.

Since, by assumption, $r_a \ge {\rm sr}(\mathcal{A})$ the homomorphism  (\ref{sst}) (with $d=r_a$) is surjective and so there exists a matrix $u_a = (u_{a,tw})_{1\le t,w\le r_a}$ in ${\rm GL}_{r_a}(\mathcal{A})$ with ${\rm Nrd}_A(u_a)^{(-1)^a} = u$. 

We then fix ordered bases $\underline{b}_s := \{b_{s,t}\}_{1\le t\le r_s}$ of the $\mathcal{A}$-modules $P^s$ such that $b = \Upsilon( \underline{b}_\bullet)$ and write $\underline{b}_s' := \{b'_{s,t}\}_{1\le t\le r_s}$ for the ordered basis of each $\mathcal{A}$-module $P^s$ that is obtained by setting $b'_{s,t} := b_{s,t}$ if $s\not= a$ and $b'_{a,t} := {\sum}_{w=1}^{w=r_a}u_{a,tw} b_{a,w}$. Then the equality (\ref{4.13}) implies that
\[ b' = u\cdot b = {\rm Nrd}_A(u_a)^{(-1)^a}\cdot \Upsilon( \underline{b}_\bullet) = \Upsilon(\underline{b}'_\bullet),\]
and hence that $b'$ is a primitive basis of ${\rm d}^\diamond_\mathcal{A}(P^\bullet)$, as required. \end{proof}


\begin{lemma}\label{concrete primitive lemma 2}
 Let 
\[ 0 \to P_1^\bullet \to P_2^\bullet \to P^\bullet_3 \to 0\]
be
 a short exact sequence in $\DC^{\rm f}(\mathcal{A})$ and write
\[\Delta: {\rm d}_{\mathcal{A}}(P_1^\bullet)\otimes{\rm d}_{\mathcal{A}}(P_3^\bullet)
\cong {\rm d}_{\mathcal{A}}(P_2^\bullet)\]
for the isomorphism induced by Theorem \ref{ext det fun thm}(i). Set $r_j := {\rm rr}_{\mathcal{A}}(P^\bullet_j)$ for $j = 1,2,3$. 

Then $r_1 + r_3 = r_2$ and, for any primitive bases $x_1$  of
${\rm d}^\diamond_{\mathcal{A}}(P^\bullet_1)$ and $x_3$ of
${\rm d}^\diamond_{\mathcal{A}}(P^\bullet_3)$, there exists a primitive
basis $x_2$ of ${\rm d}^\diamond_{\mathcal{A}}(P^\bullet_2)$ such that
\[ \Delta((x_1,r_1)\otimes (x_3,r_3)) = (x_2,r_2).\]
\end{lemma}

\begin{proof} The given exact sequence induces in each degree $a$ a (split) short exact
sequence in ${\rm Mod}^{\rm f}(\mathcal{A})$ of the form
\begin{equation}\label{ses i} 0 \to P_1^a \to P^a_2 \xrightarrow{\phi^a} P^a_3
\to 0.\end{equation}
These sequences imply that ${\rm rk}_\mathcal{A}(P_1^a) + {\rm rk}_\mathcal{A}(P_3^a) = {\rm rk}_\mathcal{A}(P_2^a)$ and so the claimed equality $r_1 + r_3 = r_2$ follows directly from the definition (\ref{rr def}) of each term ${\rm rr}_{\mathcal{A}}(P_j^\bullet)$.

In addition, if one sets $r_j^a = {\rm rr}_\mathcal{A}(P_j^a)$ for $j = 1,2,3$, then the above sequences combine with Remark \ref{exp reps} to show that the remaining claim is valid provided that for any primitive bases $x^a_1$
of ${\rm d}^\diamond_{\mathcal{A}}(P_1^a)$ and $x^a_3$ of
${\rm d}^\diamond_{\mathcal{A}}(P^a_3)$, there exists a primitive
basis $x^a_2$ of ${\rm d}^\diamond_{\mathcal{A}}(P^a_2)$ such that the isomorphism
\[ \Delta^a: {\rm d}_{\mathcal{A}}(P^a_1)\otimes {\rm d}_{\mathcal{A}}(P^a_3)
\cong {\rm d}_{\mathcal{A}}(P^a_2)\]
induced by (\ref{ses i}) sends $(x^a_1,r_1^a)\otimes (x^a_3,r_3^a)$ to $(x^a_2,r_2^a)$.

For this we fix an $\mathcal{A}$-module section $\sigma$ to $\phi^a$ and note that, if $x_j$ corresponds to the ordered $\mathcal{A}$-basis
 $\{b_{js}\}_{1\le s\le r^a_{j}}$ of $P^a_{j}$ for each $j=1,3$, then one obtains
  an ordered $\mathcal{A}$-basis $\{b^{\sigma}_{s}\}_{1\le s\le r^a_{2}}$ of $P_{2}^a$ by setting
\[ b^{\sigma}_{s} := \begin{cases} b^1_{s}, &\text{if $1\le s\le r^a_{1}$,}\\
\sigma(b^3_{s-r^a_{1}}), &\text{if $r^a_{1}< s\le r^a_{2}$.}\end{cases}\] 

Then 
$\wedge_{j=1}^{j=r^a_{2}}b^{\sigma}_{j}$ is a primitive
basis of ${\rm d}^\diamond_{\mathcal{A}}(P^a_2) = {\bigcap}_{\mathcal{A}}^{r^a_2}P^a_2$ and so the required result is
true because the isomorphism 
$\Delta^a$ is defined  by the condition that 
\[ \Delta^a\left((\wedge_{s=1}^{s=r^a_{1}}b^1_{s},r^a_1)\otimes (\wedge_{t=1}^{t=r^a_{3}}b^3_{t},r^a_3) \right) =(\wedge_{j=1}^{j=r^a_{2}}b^{\sigma}_{j},r^a_2)\]
(cf. the proof of \cite[Prop. 5.9(v)]{bses}). \end{proof}

\begin{lemma}\label{concrete primitive lemma 3} Let $P^\bullet$ be an acyclic complex in $\DC^{\rm f}(\mathcal{A})$ and $b$ a primitive basis of ${\rm d}^\diamond_\mathcal{A}(P^\bullet)$. Then ${\rm rr}_{\mathcal{A}}(P^\bullet) = 0$ and the canonical isomorphism
${\rm d}_\mathcal{A}(P^\bullet) \cong (\xi(\mathcal{A}),0)$ sends  $(b,0)$ to $(u,0)$ for some element $u$ of ${\rm Nrd}_A(\K_1(\mathcal{A}))$.
\end{lemma}

\begin{proof} The acyclicity of $P^\bullet$ implies that ${\sum}_{a\in \ZZ}(-1)^a\cdot {\rm rk}_\mathcal{A}(P^a) = 0$ and hence also that ${\rm rr}_\mathcal{A}(P^\bullet) = 0$, as claimed.

Regarding the second claim, we note that the `only if' part of Lemma \ref{concrete primitive lemma 1} (the argument for which does not require the existence of an integer $a$ with ${\rm rk}_\cA(P^a) \geq {\rm sr}(\cA)$) reduces us to proving the existence of a primitive basis $b$ of ${\rm d}^\diamond_\mathcal{A}(P^\bullet)$ such that the canonical isomorphism ${\rm d}_\mathcal{A}(P^\bullet) \cong (\xi(\mathcal{A}),0)$ sends $(b,0)$ to $(u,0)$ for some element $u$ of ${\rm Nrd}_A(\K_1(\mathcal{A}))$.


We first prove this in the special case that there exists an integer $a$ such that $P^i$ vanishes for all $i\notin \{a,a+1\}$. In this situation  the acyclicity of $P^\bullet$ implies 
it has the form $\mathcal{A}^t\xrightarrow{\theta} \mathcal{A}^t$
 for a suitable natural number $t$ and isomorphism of $\mathcal{A}$-modules
 $\theta$. In particular, if we write $b$ for the primitive basis of ${\rm d}^\diamond_\mathcal{A}(P^\bullet)$
 that corresponds to the standard basis of $\mathcal{A}^t$
 (in both degrees $a$ and $a+1$), then (\ref{4.13}) implies that the canonical isomorphism
${\rm d}_\mathcal{A}(P^\bullet) \cong (\xi(\mathcal{A}),0)$
sends $(b,0)$ to $({\rm Nrd}_A(M_\theta),0)$, where $M_\theta$ is the matrix of $\theta$
with respect to the standard basis of $\mathcal{A}^t$. The required result is therefore
 true since ${\rm Nrd}_A(M_\theta)$ belongs to ${\rm Nrd}_A(\K_1(\mathcal{A}))$.

Turning now to the general case, we write $a$ and $a'$ for the least and greatest integers $m$ for which $P^m$ is non-zero. Then, if necessary by taking the direct sum of $P^\bullet$ with a suitable collection of complexes of the form $\mathcal{A}\xrightarrow{{\rm id}}\mathcal{A}$ (and applying Lemma \ref{concrete primitive lemma 2}), we can assume that ${\rm rk}_{\mathcal{A}}(\ker(d^j)) \ge {\rm sr}(\mathcal{A})$ for each $j$ with $a < j < a'$.

Then, since $P^\bullet$ is acyclic, the Bass Cancellation Theorem (cf. \cite[Th. (41.20)]{curtisr}) combines with an easy downward induction on $j$ to imply, firstly that each $\mathcal{A}$-module $\im(d^j) = \ker(d^{j+1})$ is free and hence that there is an isomorphism of $\mathcal{A}$-modules $P^j\cong \ker(d^j)\oplus \im(d^j)$, and secondly that each module $P^j/\ker(d^j) \cong \im(d^j)$ is free.

Now write $P^\bullet_1$ for the complex $P^a\xrightarrow{d^a} \im(d^a)$ where the first term is placed in degree $a$, and $\iota$ for the natural inclusion of complexes $P_1^\bullet \to P^\bullet$. Then there is a tautological short exact sequence of acyclic complexes
\begin{equation}\label{acyclicity induction} 0 \to P^\bullet_1 \xrightarrow{\iota} P^\bullet \to \cok(\iota)\to 0\end{equation}
in $\DC^{\rm f}(\mathcal{A})$ and, by applying  Lemma \ref{concrete primitive lemma 2} to this sequence, one
 can use an induction on the number of non-zero modules $P^j$ to reduce to the special case (that $P^j$ vanishes except in two consecutive degrees) that was considered earlier. \end{proof}


\subsubsection{}We can now prove Proposition \ref{concrete primitive lemma 20}.

For $j\in \{1,2\}$ and each integer $a$ we set $r_{j}^a := {\rm rk}_\mathcal{A}(P_j^a)$. For each $j$ and $a$ we then fix an ordered $\mathcal{A}$-basis $\underline{b}_{j,a}:= \{b_{j,ak}\}_{1\le k\le r^a_{j}}$ of $P_j^a$ and write $x_j$ for the associated primitive basis $\Upsilon(\underline{b}_{j,\bullet})$ 
of ${\rm d}^\diamond_\mathcal{A}(P_j^\bullet)$.

Then Lemma \ref{concrete primitive lemma 1}
implies that the stated claim is true if and only if there exists an element $u$ of ${\rm Nrd}_A(\K_1(\mathcal{A}))$ such that
\begin{equation}\label{kmi}
 {\rm d}_\mathcal{A}(\lambda)((x_1,r)) = (u\cdot x_2,r)\end{equation}
with $r := {\rm rr}_\mathcal{A}(P_1^\bullet) = {\rm rr}_\mathcal{A}(P_2^\bullet)$. 
 To prove this we shall adapt an argument of Knudsen and Mumford from \cite[proof of Th. 1]{knumum} (see also \cite[Chap. III, \S 3, Lem.]{GM}).

For this purpose we recall that the mapping cylinder  of $\lambda$ is the complex $Z_\lambda^\bullet$ that has $Z_\lambda^a = P_1^a \oplus P_2^a\oplus P_1^{a+1}$ in each degree  $a$ and is such that, with respect to these decompositions, the differential in degree $a$ is represented by the matrix
\[ \begin{pmatrix}
d_1^a& 0 & -1\\
0& d_2^a & \lambda^{a+1}\\
0& 0 & -d_1^{a+1}
\end{pmatrix}\]
where $d_j^a$ denotes the differential of $P_j^\bullet$ in degree $a$.

Then, with this notation there are quasi-isomorphisms
$\lambda_1: P_1^\bullet\to Z^\bullet_\lambda$, $\lambda_2: P_2^\bullet \to Z^\bullet_\lambda$ and $\lambda_2': Z^\bullet_\lambda \to P^\bullet_2$ in $\DC^{\rm f}(\mathcal{A})$ with (in the obvious notation) 
\[ \lambda_1 =
\begin{pmatrix}
1\\
0\\
0\end{pmatrix},\,\,\,
\lambda_2 = \begin{pmatrix}
0\\
1\\
0\end{pmatrix}\,\,\,
\text{ and }\,\,\,
\lambda_2' = \begin{pmatrix}
\lambda\\
1\\
0\end{pmatrix}.\]
In addition, one checks that $\lambda_2'\circ \lambda_1 = \lambda$ and $\lambda_2'\circ \lambda_2 = {\rm id}_{P^\bullet_2}$ and that both of the complexes ${\rm cok}(\lambda_1)$ and ${\rm cok}(\lambda_2)$ are acyclic objects of $\DC^{\rm f}(\mathcal{A})$. It follows that, for each $i$ in $\{1,2\}$ there are natural isomorphisms
\begin{equation*}\label{knudsen mumford intermediate} {\rm d}_{\mathcal{A}}(\lambda_i)': {\rm d}_{\mathcal{A}}(P_i^\bullet) \cong  {\rm d}_{\mathcal{A}}(P_i^\bullet)\otimes {\rm d}_{\mathcal{A}}({\rm cok}(\lambda_i)) \cong {\rm d}_{\mathcal{A}}(Z_\lambda^\bullet),\end{equation*}
where the first map is induced by the acyclicity of ${\rm cok}(\lambda_i)$ and
the second by the tautological short exact sequence 
\[ 0 \to P_i^\bullet \xrightarrow{\lambda_i} Z^\bullet_\lambda \to {\rm cok}(\lambda_i) \to 0.\]

By applying Lemma \ref{concrete primitive lemma 3} to the first isomorphism in this composite, and then
 Lemma \ref{concrete primitive lemma 2} to the second, one deduces
 that there exists an element $u_i$ of ${\rm Nrd}_A(\K_1(\mathcal{A}))$ and a primitive basis $y_i$ of
${\rm d}^\diamond_{\mathcal{A}}(Z_\lambda^\bullet)$ such that
\[ {\rm d}_{\mathcal{A}}(\lambda_i)'((x_i,r)) = ((u_i\cdot y_i,r)).\]

On the other hand, the explicit construction of ${\rm d}_\mathcal{A}(\lambda)$ via 
\cite[Prop. 5.14]{bses} implies directly that
\[ {\rm d}_\mathcal{A}(\lambda) =({\rm d}_\mathcal{A}(\lambda_2)')^{-1}\circ {\rm d}_\mathcal{A}(\lambda_1)'\]
and hence that
\[ {\rm d}_\mathcal{A}(\lambda)((x_1,r)) =
((u_1u_2^{-1}u_3)\cdot x_2,r)\]
where $u_3$ is the element of ${\rm Nrd}_A(\K_1(\mathcal{A}))$
that is defined by (Lemma \ref{concrete primitive lemma 1} and)
the equality 
\[ y_1 = u_3\cdot y_2.\]

In particular, since the element 
 $u:= u_1u_2^{-1}u_3$ belongs to ${\rm Nrd}_A(\K_1(\mathcal{A}))$, this  proves the required equality (\ref{kmi}) (for this value of $u$), and hence completes the proof of Proposition \ref{concrete primitive lemma 20}.\qed

%
%
%
%

\subsection{Locally-primitive bases}

We can now make the key definitions of \S\ref{plp section}.

\begin{definition}\label{loc prim def}{\em Let $C$ be an object of $\Der^{{{\rm lf}}}(\mathcal{A})$ and $b$ an element of ${\rm d}_{\mathcal{A}}(C)_F$.\

\noindent{}(i) We say that $b$ is a `primitive basis' of the $\xi(\mathcal{A})$-module ${\rm d}_{\mathcal{A}}(C)$ if $C$ is isomorphic in $\Der^{\rm lf}(\mathcal{A})$ to a complex $P^\bullet$ in $\DC^{{\rm f}}(\mathcal{A})$ with the property that  in some degree $a$ one has ${\rm rk}_\mathcal{A}(P^a) \ge {\rm sr}(\mathcal{A})$ and, with respect to the induced identification ${\rm d}_\mathcal{A}(C) \cong ({\rm d}^\diamond_\mathcal{A}(P^\bullet),{\rm rr}_\mathcal{A}(P^\bullet))$, the element $b$ corresponds to $(b',{\rm rr}_\mathcal{A}(P^\bullet))$ for some primitive basis $b'$ of ${\rm d}^\diamond_\mathcal{A}(P^\bullet)$.\

\noindent{}(ii) We say that $b$ is a `locally-primitive basis' of the $\xi(\mathcal{A})$-module ${\rm d}_{\mathcal{A}}(C)$ if for all $\mathfrak{p}$ in ${\rm Spm}(R)$ the image of $b$ in the $\mathfrak{p}$-completion ${\rm d}_{\mathcal{A}}(C)_{F,\frp} = {\rm d}_{\mathcal{A}_{\frp}}( C_\frp)_{F_\frp}$ of ${\rm d}_{\mathcal{A}}(C)_F$ is a primitive basis of the $\xi(\mathcal{A}_\mathfrak{p})$-module ${\rm d}_{\mathcal{A}_{\frp}}(C_{\frp})$.

\noindent{}(iii) We say that $b$ is a `generically-primitive basis' of the $\xi(\mathcal{A})$-module 
${\rm d}_{\mathcal{A}}(C)$ if its image in ${\rm d}_{\mathcal{A}}(C)_F = {\rm d}_A(C_F)$ is a primitive basis of the $\zeta(A)$-module ${\rm d}_A(C_F)$.} \end{definition}

%

The arguments of \S\ref{pb section} imply that the
notion of (locally-)primitive basis is intrinsic to objects of
$\Der^{\rm lf}(\mathcal{A})$ and further that, in this setting, it has the following useful functorial properties.

\begin{proposition}\label{prim criterion 2}
Let $C$ be an object of $\Der^{{\rm lf}}(\mathcal{A})$.
\begin{itemize}
\item[(i)] If $C$ is acyclic, then the canonical isomorphism 
\[ {\rm d}_\mathcal{A}(C)
\cong (\xi(\mathcal{A}),0)\]
coming from Theorem \ref{ext det fun thm} sends any primitive basis of ${\rm d}_\mathcal{A}(C)$ to an element of $({\rm Nrd}_A(\K_1(\mathcal{A})),0)$.
\item[(ii)] Let $b$ be a primitive basis of ${\rm d}_\mathcal{A}(C)$. Then an element $b'$ of ${\rm d}_\mathcal{A}(C)_F$ is a primitive basis
of ${\rm d}_\mathcal{A}(C)$ if and only if $b' = u\cdot b$ for some $u$ in ${\rm Nrd}_{A}(\K_1(\mathcal{A}))$.
\item[(iii)] Let 
\[ C_1\to C\to C_3 \to C_1[1]\]
be an exact
triangle in $\Der^{{\rm lf}}(\mathcal{A})$. Then for any
primitive bases, respectively locally-primitive bases, $x_1$ of ${\rm d}_{\mathcal{A}}(C_1)$ and
$x_3$ of ${\rm d}_{\mathcal{A}}(C_3)$, the canonical isomorphism
\[ {\rm d}_{\mathcal{A}}(C_1)\otimes {\rm d}_{\mathcal{A}}(C_3) \cong
{\rm d}_{\mathcal{A}}(C)\]
coming from Theorem \ref{ext det fun thm} sends $x_1\otimes x_3$ to a primitive basis, respectively locally-primitive basis, of ${\rm d}_{\mathcal{A}}(C)$.
\end{itemize}
\end{proposition}

\begin{proof} The key point is Proposition \ref{concrete primitive lemma 20} implies that if $b$ is a primitive basis of ${\rm d}^\diamond_\mathcal{A}(P_1^\bullet)$ for any
complex $P_1^\bullet$ in $\DC^{\rm f}(\mathcal{A})$ that is both isomorphic
 in $\Der^{\rm lf}(\mathcal{A})$ to $C$ and such that in some degree $a$ one has
 ${\rm rk}_\mathcal{A}(P_1^a) \ge {\rm sr}(\mathcal{A})$, then it also
 corresponds to a primitive basis of ${\rm d}^\diamond_\mathcal{A}(P_2^\bullet)$
 for any other such complex $P^\bullet_2$ in $\DC^{\rm f}(\mathcal{A})$.

Given this fact, the assertions of claims (i), (ii) and (iii) follow  directly from the respective results of Lemmas \ref{concrete primitive lemma 3}, \ref{concrete primitive lemma 1} and \ref{concrete primitive lemma 2}.
\end{proof}


We next clarify the links between the varying notions of basis considered above.

\begin{proposition}\label{remarks primitive} Let $C$ be an object of $\Der^{{\rm lf}}(\mathcal{A})$ and $b$ an element of ${\rm d}_{\mathcal{A}}(C)_F$. Then the following claims are valid. 
\begin{itemize}
\item[(i)] If $b$ is a primitive basis of ${\rm d}_\mathcal{A}(C)$, then it is also both a locally-primitive and generically-primitive basis of ${\rm d}_\mathcal{A}(C)$. 

\item[(ii)] If ${\rm d}_\mathcal{A}(C)$ has a primitive basis, then the converse of claim (i) is true.
\item[(iii)] If $b$ is an locally-primitive basis of ${\rm d}_\mathcal{A}(C)$, then it is a $\xi(\mathcal{A})$-basis of  ${\rm d}_{\mathcal{A}}(C)^{\rm u}$. 
\item[(iv)] Assume ${\rm d}_\mathcal{A}(C)$ has a primitive basis. Then the $\xi(\mathcal{A})$-module ${\rm d}_{\mathcal{A}}(C)^{\rm u}$ is free and the following claims are valid.
\begin{itemize}
\item[(a)] Every basis of ${\rm d}_{\mathcal{A}}(C)^{\rm u}$ corresponds to a primitive basis of ${\rm d}_{\mathcal{A}}(C)$ if and only if $\xi(\mathcal{A})^\times = {\rm Nrd}_A(\K_1(\mathcal{A}))$.

\item[(b)] Every basis of ${\rm d}_{\mathcal{A}}(C)^{\rm u}$ corresponds to a locally-primitive basis of ${\rm d}_{\mathcal{A}}(C)$ if and only if $\xi(\mathcal{A})^\times$ is the full pre-image of the direct product of ${\rm Nrd}_{A_\mathfrak{p}}(\K_1(\mathcal{A}_\mathfrak{p}))$ over $\mathfrak{p}$ in ${\rm Spm}(R)$ under the diagonal map 
\[\zeta(A)^\times \to {\prod}_\mathfrak{p}\zeta(A_\mathfrak{p})^\times = {\prod}_\mathfrak{p}{\rm Nrd}_{A_\mathfrak{p}}(\K_1(A_\mathfrak{p})).\]  
\end{itemize}
\end{itemize}
\end{proposition}

\begin{proof} Claim (i) follows easily by a direct comparison of the respective definitions that are given in Definition \ref{loc prim def}. 

To prove claim (ii), we fix a primitive basis $b$ of ${\rm d}_{\mathcal{A}}(C)$. Then, in view of Proposition \ref{prim criterion 2}(ii), it is enough to show that if $b'$ is both a generically-primitive and locally-primitive basis of ${\rm d}_{\mathcal{A}}(C)$, then one has $b' = u\cdot b$ for some $u$ in ${\rm Nrd}_A(\K_1(\mathcal{A}))$.  

Now, as $b'$ is a generically-primitive basis, Proposition \ref{prim criterion 2}(ii) implies the existence of a (unique) element $v$ of ${\rm Nrd}_A(\K_1(A))$ such that $b' = v\cdot b$ in ${\rm d}_A(C_F)$. Then, since $b'$ is a locally-primitive basis, we can deduce from Proposition \ref{prim criterion 2}(ii) (and the uniqueness of $v$) that, for every $\mathfrak{p}$ in ${\rm Spm}(R)$, the image of $v$ in ${\rm Nrd}_{A_\mathfrak{p}}(\K_1(A_\mathfrak{p}))$ belongs to ${\rm Nrd}_{A_\mathfrak{p}}(\K_1(\mathcal{A}_\mathfrak{p}))$. But then, by a general result of $K$-theory (which, for convenience, we defer to Lemma \ref{last tech}(ii) below), this implies that $v$ belongs to ${\rm Nrd}_{A}(\K_1(\mathcal{A}))$, as required. 

To prove claim (iii) we note that, for each $\mathfrak{p}$ in ${\rm Spm}(R)$, the given element $b$ is a basis of the $\xi(\mathcal{A}_\mathfrak{p})$-module ${\rm d}_{\mathcal{A}_{\mathfrak{p}}}(C_\mathfrak{p})^{\rm u} = 
{\rm d}_{\mathcal{A}}(C)_\mathfrak{p}^{\rm u}$. This fact implies, firstly, that for every such $\mathfrak{p}$ one has 
\[ (\zeta(A)\cdot b\bigr)_\mathfrak{p} = \zeta(A_\mathfrak{p})\cdot b = \bigl(\xi(\mathcal{A}_\mathfrak{p})\cdot b)_{F_\mathfrak{p}} = \bigl({\rm d}_{\mathcal{A}}(C)_F^{\rm u}\bigr)_{\mathfrak{p}}\]
and hence that $\zeta(A)\cdot b = {\rm d}_{\mathcal{A}}(C)_F^{\rm u}$. Then, upon applying the general result of \cite[Prop. (4.21)(vi)]{curtisr}, one deduces the required equality by noting that
\begin{align*} {\rm d}_{\mathcal{A}}(C)^{\rm u} =&\, {\rm d}_{\mathcal{A}}(C)^{\rm u}_F \cap  {\bigcap}_{\mathfrak{p}}{\rm d}_{\mathcal{A}}(C)_\mathfrak{p}^{\rm u}\\
 =&\, (\zeta(A)\cdot b) \cap {\bigcap}_{\mathfrak{p}}(\xi(\mathcal{A}_\mathfrak{p})\cdot b)\\ 
 =&\, \bigl(\zeta(A)\cap {\bigcap}_{\mathfrak{p}} \xi(\mathcal{A}_\mathfrak{p})\bigr)\cdot b \\
 =&\, \xi(\mathcal{A})\cdot b.\end{align*}
Here, in each intersection $\mathfrak{p}$ runs over ${\rm Spm}(R)$, and the last equality is valid because 
\[ \zeta(A)\cap {\bigcap}_{\mathfrak{p}} \xi(\mathcal{A}_\mathfrak{p}) = {\bigcap}_{\mathfrak{p}} \bigl(\zeta(A)\cap \xi(\mathcal{A})_\mathfrak{p}\bigr) = {\bigcap}_{\mathfrak{p}} \xi(\mathcal{A})_{(\mathfrak{p})} = \xi(\mathcal{A}).\]

Finally, to prove claim (iv), we fix a primitive basis $b$ of ${\rm d}_{\mathcal{A}}(C)$ and an arbitrary element $b'$ of ${\rm d}_\mathcal{A}(C)_F$. Then one has $b' = x\cdot b$ for a unique element $x = x_{b',b}$ of $\zeta(A)$. 

Now, since $b$ is a basis of the $\xi(\mathcal{A})$-module ${\rm d}_{\mathcal{A}}(C)^{\rm u}$ one knows that $b'$ is a basis of this module if and only if $x \in \xi(\mathcal{A})^\times$. 

In a similar way, the result of Proposition \ref{prim criterion 2}(ii) implies that $b'$ is a primitive basis, respectively locally-primitive basis,  of ${\rm d}_{\mathcal{A}}(C)$ if and only if one has $x \in {\rm Nrd}_A(\K_1(\mathcal{A}))$, respectively the image of $x$ in $\zeta(A_\mathfrak{p})^\times = {\rm Nrd}_{A_\mathfrak{p}}(\K_1(A_\mathfrak{p}))$ belongs to ${\rm Nrd}_{A_\mathfrak{p}}(\K_1(\mathcal{A}_\mathfrak{p}))$ for every $\mathfrak{p}$ in ${\rm Spm}(R)$. 

The assertions in claim (iv) follow directly from these observations. \end{proof}


\subsection{Primitive bases and Euler characteristics}

The result of Proposition \ref{prim criterion 2}(iii) leaves open the problem of whether there are conditions on $\mathcal{A}$ under which the freeness of ${\rm d}_\mathcal{A}(C)^{\rm u}$ as a $\xi(\mathcal{A})$-module can itself imply the existence of a locally-primitive basis, or even a primitive basis, of ${\rm d}_\mathcal{A}(C)$. 

To shed light on this problem, in this section we reinterpret the conditions that ${\rm d}_\mathcal{A}(C)^{\rm u}$ is free, that ${\rm d}_\mathcal{A}(C)$ has a locally-primitive basis and that ${\rm d}_\mathcal{A}(C)$ has a primitive basis in terms of the Euler characteristic $\chi_\mathcal{A}(C)$.  

As necessary preparation for this result, we shall also make some general observations concerning class groups of orders.

\subsubsection{}We note first that the argument of \cite[Rem. (49.11)(iv)]{curtisr} shows
that ${\rm SK}_0^{\rm lf}(\mathcal{A})$ is naturally isomorphic to the
`locally-free classgroup' ${\rm Cl}(\mathcal{A})$ of $\mathcal{A}$, as
defined in \cite[(49.10)]{curtisr}.

We recall ${\rm Cl}(\mathcal{A})$ is finite, that it is equal to
the set of stable isomorphism classes $[I]$ of locally-free, rank one, $\mathcal{A}$-submodules
$I$ of $A$ and that its addition is defined by setting
\[ [I_1] + [I_2] := [I_3]\]
whenever there is an isomorphism
of $\mathcal{A}$-modules $I_1\oplus I_2 \cong I_3\oplus \mathcal{A}$. (Recall from \cite[Rem. (49.11)(i)]{curtisr} that $\cA$-modules $I_1$ and $I_2$ are stably-isomorphic if and only if $I_1\oplus \cA$ is isomorphic to $I_2\oplus \cA$.)

We recall further that if $\mathcal{A}$ is commutative, then
${\rm Cl}(\mathcal{A})$ is naturally isomorphic to the multiplicative
group of isomorphism classes of locally-free, rank one,  $\mathcal{A}$-submodules of $A$.

\begin{lemma}\label{explicit hom lemma} The association $P \mapsto {\bigcap}_\mathcal{A}^{{\rm rk}_\mathcal{A}(P)}P$ for each $P$ in ${\rm Mod}^{\rm lf}(\mathcal{A})$  induces a well-defined
 homomorphism of abelian groups 
 \[ \Delta_\mathcal{A}: {\rm SK}_0^{\rm lf}(\mathcal{A}) \to {\rm Cl}(\xi(\mathcal{A})).\]
\end{lemma}

\begin{proof} Since ${\rm SK}_0^{\rm lf}(\mathcal{A})$ is naturally isomorphic to ${\rm Cl}(\mathcal{A})$, this result is equivalent to the following two claims: firstly,
if $I_1$ and $I_2$ are any locally-free, rank one, $\mathcal{A}$-modules that are
stably-isomorphic, then the $\xi(\mathcal{A})$-modules
${\bigcap}_{\mathcal{A}}^1I_1$ and ${\bigcap}_{\mathcal{A}}^1I_2$
are isomorphic; secondly, if $I_1, I_2$ and $I_3$ are any locally-free, rank one, $\mathcal{A}$-modules for which the $\mathcal{A}$-modules $I_1\oplus I_2$ and $I_3\oplus\mathcal{A}$ are isomorphic, then the $\xi(\mathcal{A})$-modules  $({\bigcap}_{\mathcal{A}}^1I_1)\otimes_{\xi(\mathcal{A})}
({\bigcap}_{\mathcal{A}}^1I_2)$ and ${\bigcap}_{\mathcal{A}}^1I_3$ are  isomorphic.

To prove the first claim we note that if $I_1$ and $I_2$ are stably-isomorphic,
then there exists an isomorphism of $\mathcal{A}$-modules 
\[ I_1\oplus \mathcal{A} \cong I_2\oplus \mathcal{A}.\]
Then, since ${\bigcap}_{\mathcal{A}}^1\mathcal{A}$ is a free $\xi(\mathcal{A})$-module of rank one
(by (\ref{free rem})), the argument of Lemma
\ref{concrete primitive lemma 2} induces an
isomorphism of $\xi(\mathcal{A})$-modules of the required form
\begin{align*} {\bigcap}_{\mathcal{A}}^1I_1 \cong&\, ({\bigcap}_{\mathcal{A}}^1I_1)
\otimes_{\xi(\mathcal{A})}
({\bigcap}_{\mathcal{A}}^{1}\mathcal{A})\\
 \cong&\, 
({\bigcap}_{\mathcal{A}}^{1}I_2)\otimes_{\xi(\mathcal{A})}
({\bigcap}_{\mathcal{A}}^{1}\mathcal{A})\\
 \cong&\, {\bigcap}_{\mathcal{A}}^1I_2.\end{align*}

The second claim is proved in a similar way since the given isomorphism
$I_1\oplus I_2 \cong  I_3\oplus \mathcal{A}$ induces an
isomorphism of $\xi(\mathcal{A})$-modules
\begin{align*}({\bigcap}_{\mathcal{A}}^1I_1)\otimes_{\xi(\mathcal{A})}
({\bigcap}_{\mathcal{A}}^{1}I_2) \cong&\, {\bigcap}_{\mathcal{A}}^2(I_1\oplus I_2)\\
\cong&\,  {\bigcap}_{\mathcal{A}}^2(I_3\oplus \mathcal{A})\\
\cong&\, ({\bigcap}_{\mathcal{A}}^{1}I_3)\otimes_{\xi(\mathcal{A})}({\bigcap}_{\mathcal{A}}^1\mathcal{A})\\
 \cong&\, {\bigcap}_{\mathcal{A}}^1I_3. \end{align*}
\end{proof}

\begin{remark}\label{comparison of categories} {\em The homomorphism $\Delta_\mathcal{A}$ constructed above has a conceptual interpretation. To explain this we recall from Remark \ref{comparison of categories0} that the reduced determinant functor ${\rm d}_{\mathcal{A},\varpi}$ induces a monoidal functor $\phi^{\rm lf}_{\mathcal{A},\varpi}: \mathcal{V}^{\rm lf}(\mathcal{A}) \to \mathcal{P}(\xi(\mathcal{A}))$. The latter functor in turn induces a homomorphism of abelian groups $\pi_0(\phi^{\rm lf}_{\mathcal{A},\varpi})$ from the Grothendieck group $\K_0^{\rm lf}(\mathcal{A})$ to the Picard group ${\rm Pic}(\xi(\mathcal{A}))$ of the commutative ring $\xi(\mathcal{A})$ (cf. \cite[Rem. 5.5]{bses}). This Picard group is canonically isomorphic to ${\rm Cl}(\xi(\mathcal{A}))$ and, with respect to this identification, $\Delta_\mathcal{A}$ is equal to the composite 
\[ {\rm SK}_0^{\rm lf}(\mathcal{A}) \to \K_0^{\rm lf}(\mathcal{A}) \xrightarrow{\pi_0(\phi^{\rm lf}_{\mathcal{A},\varpi})} 
{\rm Pic}(\xi(\mathcal{A})) \cong {\rm Cl}(\xi(\mathcal{A}))\]
in which the first arrow is the natural inclusion.}
\end{remark}

We shall next define a canonical subgroup $\ker(\Delta_\mathcal{A})^{\rm lp}$ of the kernel of $\Delta_\mathcal{A}$ comprising `locally-primitive classes'.

 To do this we write $J_f(A)$ for the group of finite ideles of the $F$-algebra $A$ and
 $\{z_1,z_2\}$ for the standard basis of $A^2$. For each locally-free, rank one, $\mathcal{A}$-submodule $I$ of $A$ we then write $M(I)$ for the coset of
 ${\prod}_\mathfrak{p}{\rm GL}_2(\mathcal{A}_\mathfrak{p})$ in ${\rm GL}_2(J_f(A))$
 comprising matrices $M = (M_\mathfrak{p})_\mathfrak{p}$
 with the property that for all $\mathfrak{p}$ (in ${\rm Spm}(R)$) the set
 $\{ M_\mathfrak{p} z_1, M_{\mathfrak{p}}z_2\}$ is an
 $\mathcal{A}_\mathfrak{p}$-basis of
  $I_\mathfrak{p}\oplus \mathcal{A}_\mathfrak{p}$.

 We write ${\rm Nrd}_{J_f(A)}$ for the homomorphism
 ${\rm GL}_2(J_f(A)) \to J_f(\zeta(A))$ that is induced by the product map
 ${\prod}_{\mathfrak{p}}{\rm Nrd}_{A_\mathfrak{p}}$. We also identify $\zeta(A)^\times$ with its image under the natural  diagonal embedding $\zeta(A)^\times \to  J_f(\zeta(A))$. 

\begin{lemma}\label{subgroup lemma} Let
${\rm ker}(\Delta_\mathcal{A})^{\rm lp}$ denote the subset of
${\rm Cl}(\mathcal{A})$ comprising elements $[I]$ for which the set $M(I)$ contains a
matrix $M$ with ${\rm Nrd}_{J_f(A)}(M)\in \zeta(A)^\times$. Then ${\rm ker}(\Delta_\mathcal{A})^{\rm lp}$ is a well-defined subgroup of ${\rm ker}(\Delta_\mathcal{A})$.
\end{lemma}

\begin{proof} The set $X := {\rm ker}(\Delta_\mathcal{A})^{\rm lp}$ is
well-defined since if $[I] = [J]$, then the $\mathcal{A}$-modules $I\oplus \mathcal{A}$ and $J\oplus \mathcal{A}$
are isomorphic and so there exists a matrix $M^*$ in ${\rm GL}_2(A)$ such that
$N$ belongs to $M(I)$ if and only if $M^*\cdot N$ belongs to $M(J)$.

To show $X$ is contained in ${\rm ker}(\Delta_\mathcal{A})$ it is enough to prove that
${\bigcap}_{\mathcal{A}}^{1}I$ is a free $\xi(\mathcal{A})$-module whenever $[I]$
belongs to $X$. This is true since if we fix $M$ in $M(I)$ with ${\rm Nrd}_{J_f(A)}(M) \in
 \zeta(A)^\times$, then
the $\xi(\mathcal{A})$-module isomorphisms
\begin{align*} {\bigcap}_{\mathcal{A}}^{1}I \cong&\, {\bigcap}_{\mathcal{A}}^{1}I\otimes_{\xi(\mathcal{A})}{\bigcap}_{\mathcal{A}}^{1}\mathcal{A}\\ 
\cong&\, {\bigcap}_{\mathcal{A}}^{2}(I\oplus \mathcal{A})\\
 =&\, \xi(\mathcal{A})\cdot ({\rm Nrd}_{J_f(A)}(M)\cdot
(z_1\wedge_A z_2))\end{align*}
imply ${\bigcap}_{\mathcal{A}}^{1}I$ is isomorphic to $\xi(\mathcal{A})$. 

Finally, to show $X$ is a subgroup of ${\rm ker}(\Delta_\mathcal{A})$ it is enough to prove the following claims:
\begin{itemize}
\item[(i)] If $[I_1]$ belongs to $X$ and $I_2$ is any locally-free, rank one, $\mathcal{A}$-module for which there exists an isomorphism of $\mathcal{A}$-modules $I_1\oplus I_2 \cong \mathcal{A}\oplus \mathcal{A}$, then $[I_2]$ belongs to $X$.
\item[(ii)] If
 $[I_1]$ and $[I_2]$ belong to $X$, then any isomorphism of
 $\mathcal{A}$-modules $I_1\oplus I_2 \cong  I_3\oplus\mathcal{A}$
 implies $[I_3]$ belongs to $X$.
\end{itemize}

Claim (i) is true since the isomorphism $I_1\oplus I_2 \cong \mathcal{A}\oplus \mathcal{A}$ implies the existence of a matrix $M_*$ in ${\rm GL}_4(A)$ such that
if $M_1$ and $M_2$ belong to $M(I_1)$ and $M(I_2)$, then in ${\rm GL}_4(J_f(A))$ one has
\[ M_* \cdot \begin{pmatrix}
M_1& 0 \\
0& M_2\end{pmatrix} = {\rm Id}_4,\]
where ${\rm Id}_n$ is the identity matrix in
${\rm GL}_n(J_f(A))$, and so 
\[ {\rm Nrd}_{J_f(A)}(M_1)\cdot{\rm Nrd}_{J_f(A)}(M_2)={\rm Nrd}_{A}(M_*)^{-1}\]
belongs to $\zeta(A)^\times.$

In a similar way, the above claim (ii) follows from the fact that
 the induced isomorphism
\[ (I_1\oplus \mathcal{A})\oplus (I_2 \oplus \mathcal{A}) \cong (I_3 \oplus
  \mathcal{A})\oplus (\mathcal{A}\oplus \mathcal{A})\]
of $\mathcal{A}$-modules implies the existence of a matrix $M'_*$ in ${\rm GL}_4(A)$ such that
if $M_1$ and $M_2$ belongs to $M(I_1)$ and $M(I_2)$, then there exists
a matrix $M_3$ in $M(I_3)$ such that
\[ M'_* \cdot \begin{pmatrix}
M_1& 0 \\
0& M_2\end{pmatrix} = \begin{pmatrix}
M_3& 0\\
0& {\rm Id}_2
\end{pmatrix}\]
in ${\rm GL}_4(J_f(A))$, and hence
\[ {\rm Nrd}_{J_f(A)}(M_3) = {\rm Nrd}_{A}(M'_*)\cdot {\rm Nrd}_{J_f(A)}(M_1)\cdot {\rm Nrd}_{J_f(A)}(M_2).\]
 \end{proof}

\subsubsection{}We can now state the main result of this section. This result provides explicit criteria in terms of Euler characteristics that determine whether primitive or locally-primitive bases exist. 

\begin{theorem}\label{prim criterion} Let $C$ be an object of $\Der^{{{\rm lf},0}}(\mathcal{A})$. Then 
$\chi_\mathcal{A}(C)$ belongs to ${\rm SK}_0^{\rm lf}(\mathcal{A})$ and the following claims are valid.

\begin{itemize}
\item[(i)] ${\rm d}_{\mathcal{A}}(C)^{\rm u}$ is a free $\xi(\mathcal{A})$-module if
and only if $\chi_\mathcal{A}(C)$ belongs to $\ker(\Delta_\mathcal{A})$.
\item[(ii)] ${\rm d}_{\mathcal{A}}(C)$ has a locally-primitive basis if and only if $\chi_\mathcal{A}(C)$ belongs to $\ker(\Delta_\mathcal{A})^{\rm lp}$.
\item[(iii)] ${\rm d}_{\mathcal{A}}(C)$ has a primitive basis if and only if
 $\chi_\mathcal{A}(C)$ vanishes.
\end{itemize}
\end{theorem}

\begin{proof} We first fix a complex $P^\bullet$ in $\DC^{\rm lf}(\mathcal{A})$
that is isomorphic in $\Der^{\rm lf}(\mathcal{A})$ to $C$. Then, by a
standard construction of homological algebra (as, for example, in \cite[Rapport, Lem. 4.7]{del}), there exists a quasi-isomorphism
of complexes of finitely generated $\mathcal{A}$-modules of the form 
\[\theta: Q^\bullet \to P^\bullet,\]
where $Q^\bullet$ is bounded and has the property that if $a$ is the lowest
degree of a non-zero module $Q^a$, then $P^j$ vanishes for all $j < a+2$ and
$Q^j$ is a free $\mathcal{A}$-module for all $j > a$. We set $r_j := {\rm rk}_\mathcal{A}(Q^j)$ in each degree $j$ and note that,
if necessary after replacing  $Q^\bullet$ by the direct sum of $Q^\bullet$ and
the (acyclic) complex 
\[ \mathcal{A} \xrightarrow{{\rm id}} \mathcal{A},\]
where the first term is placed in degree $a$, we can assume that $r_a \ge 2$, and
hence also that $r_a \ge {\rm sr}(\mathcal{A})$.

Now the mapping cone $D^\bullet$ of $\theta$ is an acyclic complex for which in each degree $j$ one has $D^j = P^j\oplus Q^{j+1}$. In particular, since  $D^j$ belongs to ${\rm Mod}^{\rm lf}(\mathcal{A})$ for all $j \not= a-1$, the acyclicity of $D^\bullet$ combines with the Krull-Schmidt-Azumaya Theorem (as in the argument of \cite[Lem. 5.2]{bses}) to imply $Q^a$ belongs to ${\rm Mod}^{\rm lf}(\mathcal{A})$, and hence that $Q^\bullet$ belongs to $\DC^{\rm lf}(\mathcal{A})$.

To prove claim (i) we now use \cite[Prop. (49.3)]{curtisr} to choose an isomorphism
of $\mathcal{A}$-modules of the form 
\[ Q^a \cong (I\oplus \mathcal{A})\oplus N,\]
where $I$ is locally-free of rank one and $N$ is free of rank $r_a-2$. Then, since each of the modules
$Q^j$ for $j \not= a$ is free, the $\xi(\mathcal{A})$-module ${\rm d}_\mathcal{A}(C)^{\rm u}$ is isomorphic to
\begin{align}\label{compo iso}{\rm d}^\diamond_\mathcal{A}(Q^\bullet) = &{\bigotimes}_{j\in \ZZ}({\bigcap}^{r_j}_{\mathcal{A}}Q^j)^{(-1)^j}\\
 \cong &({\bigcap}^2_{\mathcal{A}}(I\oplus \mathcal{A}))^{(-1)^a}
 \otimes_{\xi(\mathcal{A})}(({\bigcap}^{r_a-2}_{\mathcal{A}}N)^{(-1)^a}
 \otimes_{\xi(\mathcal{A})}{\bigotimes}_{j\in \ZZ\setminus
 \{a\}}({\bigcap}^{r_j}_{\mathcal{A}}Q^j)^{(-1)^j})\notag\\
 \cong &({\bigcap}^2_{\mathcal{A}}(I\oplus \mathcal{A}))^{(-1)^a}.
 \notag\end{align}

In particular, since the $\xi(\mathcal{A})$-module ${\bigcap}^2_{\mathcal{A}}(I\oplus \mathcal{A})$ is isomorphic to
\[ ({\bigcap}^1_{\mathcal{A}}I)\otimes_{\xi(\mathcal{A})}
 ({\bigcap}^1_{\mathcal{A}}\mathcal{A}) \cong ({\bigcap}^1_{\mathcal{A}}I),\]
the equivalence in claim (i) is a consequence of the fact that $\Delta_\mathcal{A}$ sends the element %
\begin{align}\label{euler sum}\chi_\mathcal{A}(C) = &\, \chi_\mathcal{A}(Q^\bullet) \\
= &\, {\sum}_{j\in \ZZ}(-1)^j(Q^j)\notag\\ 
= &\, (-1)^a(I) + ((-1)^a(r_a-1)+ {\sum}_{j\in \ZZ\setminus \{a\}}(-1)^jr_j)(\mathcal{A})\notag\end{align}
to $(-1)^a$ times the isomorphism class of ${\bigcap}^1_{\mathcal{A}}I$. This proves claim (i).

Next we note that the isomorphism (\ref{compo iso}) implies
 ${\rm d}_\mathcal{A}(C)$ has a locally-primitive basis if and only if the module ${\rm d}^\diamond_\mathcal{A}((I\oplus \mathcal{A})[0]) = {\bigcap}^2_{\mathcal{A}}(I\oplus \mathcal{A})$ has a locally-primitive basis. To prove claim (ii) it thus enough to show that ${\bigcap}^2_{\mathcal{A}}(I\oplus \mathcal{A})$ has a locally-primitive basis if and only if the set $M(I)$ contains a matrix $M$ such that ${\rm Nrd}_{J_f(A)}(M)$ belongs to $\zeta(A)^\times$.

To prove this we fix a matrix $M = (M_\mathfrak{p})_\mathfrak{p}$ in $M(I)$ and for each $\mathfrak{p}$  set $b_{1,\mathfrak{p}} := M_\mathfrak{p}\cdot z_1$ and $b_{2,\mathfrak{p}} := M_\mathfrak{p}\cdot z_2$. 
Then for each $\mathfrak{p}$ (in ${\rm Spm}(R)$) the equality (\ref{4.13}) implies %
\begin{equation}\label{local equal} b_{1,\mathfrak{p}}\wedge b_{2,\mathfrak{p}} = {\rm Nrd}_{A_{\mathfrak{p}}}(M_{\mathfrak{p}})\cdot (z_1\wedge z_2) = {\rm Nrd}_{J_f(A)}(M)_\mathfrak{p}\cdot (z_1\wedge z_2)\end{equation}
and so the description (\ref{local rl}) implies that  %
\[ \bigl({\bigcap}^2_{\mathcal{A}}(I\oplus \mathcal{A})\bigr)_{\mathfrak{p}} = \xi(\mathcal{A})_{\mathfrak{p}}\cdot
({\rm Nrd}_{J_f(A)}(M)_{\mathfrak{p}}\cdot (z_1\wedge z_2)).\]

Assume now that $\mu := {\rm Nrd}_{J_f(A)}(M)$ belongs to $\zeta(A)^\times$. Then $\mu\cdot (z_1\wedge z_2)$ is a generator of the $\zeta(A)$-module $\bigl({\bigcap}^2_{\mathcal{A}}(I\oplus \mathcal{A})\bigr)_F$ and one has $\mu = {\rm Nrd}_{J_f(A)}(M)_\mathfrak{p}$ for all $\mathfrak{p}$ in ${\rm Spm}(R)$. By the argument of Proposition \ref{remarks primitive}(iii), the above displayed equalities therefore imply that $\mu\cdot (z_1\wedge z_2)$ is a locally-primitive basis of ${\bigcap}^2_{\mathcal{A}}(I\oplus \mathcal{A})$, as required. 

To prove the converse we assume ${\bigcap}^2_\mathcal{A}(I\oplus \mathcal{A})$ has a locally-primitive basis $b$ that corresponds to a choice of ordered $\mathcal{A}_{\mathfrak{p}}$-basis $\underline{b}_\mathfrak{p}$ of $I_{\mathfrak{p}}\oplus \mathcal{A}_{\mathfrak{p}}$ for each $\mathfrak{p}$ in ${\rm Spm}(R)$.
 Then, in this case, the equalities (\ref{local equal}) imply the transition matrices $M_\mathfrak{p}$ from $\{z_1,z_2\}$ to $\underline{b}_\mathfrak{p}$ combine to give a matrix $M = (M_{\mathfrak{p}})_{\mathfrak{p}}$ in ${\rm GL}_2(J_f(A))$ with the property that 
\[ b = {\rm Nrd}_{J_f(A)}(M)_\mathfrak{p}\cdot (z_1\wedge z_2)\]
for every $\mathfrak{p}$. These equalities combine to imply that ${\rm Nrd}_{J_f(A)}(M)$ belongs to $\zeta(A)^\times$, and this completes the proof of claim (ii).

Turning to claim (iii), we note Proposition \ref{concrete primitive lemma 20} implies that ${\rm d}_\mathcal{A}(C)$ has a primitive basis if and only if $C$ is isomorphic in $\Der(\mathcal{A})$ to a complex $K^\bullet$ that belongs to both $\DC^{{\rm lf},0}(\mathcal{A})$ and $\DC^{{\rm f}}(\mathcal{A})$ and also has the property that  ${\rm rk}_\mathcal{A}(K^a) \ge {\rm sr}(\mathcal{A})$ in some degree $a$. 

In particular, if such a complex $K^\bullet$ exists, then it is clear that the Euler characteristic $\chi_\mathcal{A}(C) = \chi_\mathcal{A}(K^\bullet)$ vanishes. Conversely, if $\chi_\mathcal{A}(C)$ vanishes, then
the sum (\ref{euler sum}) also vanishes and so the $\mathcal{A}$-module $Q^a \cong (I\oplus\mathcal{A})\oplus N$ 
is stably-free. Then, since ${\rm rk}_\mathcal{A}(Q^a)\ge 2$, the Bass Cancellation Theorem  implies $Q^a$ is a free $\mathcal{A}$-module of rank at least ${\rm sr}(\mathcal{A})$ and hence that the complex $Q^\bullet$ implies ${\rm d}_\mathcal{A}(C)$ has a primitive  basis, as required.
\end{proof}

\vskip0.2truein

The following result describes two useful, and concrete, consequences of Theorem \ref{prim criterion}. 
\vskip0.2truein

\begin{corollary}\label{abelian primitive}\
\begin{itemize}
\item[(i)] If $\mathcal{A}$ is commutative, then, for any object $C$ of $\Der^{{\rm lf},0}(\mathcal{A})$, an element of ${\rm d}_\mathcal{A}(C)_F$ is a primitive basis of ${\rm d}_\mathcal{A}(C)$  if and only if it is a basis of 
${\rm d}_\mathcal{A}(C)^{\rm u}$ as a $\xi(\mathcal{A})$-module.

\item[(ii)] If $|{\rm Cl}(\mathcal{A})|$ does not divide $|{\rm Cl}(\xi(\mathcal{A}))|$, then there exist objects $C$ of 
$\Der^{{\rm lf},0}(\mathcal{A})$ for which ${\rm d}_\mathcal{A}(C)^{\rm u}$ is a free $\xi(\mathcal{A})$-module but 
${\rm d}_\mathcal{A}(C)$ has no primitive basis. 
\end{itemize}
\end{corollary}

\begin{proof} For claim (i) we note that, if $\mathcal{A}$ is commutative, then $\xi(\mathcal{A}) = \mathcal{A}$ and the map ${\Delta}_\mathcal{A}$ identifies with the identity automorphism of ${\rm Cl}(\mathcal{A}) = {\rm Cl}(\xi(\mathcal{A}))$. The group $\ker({\Delta}_\mathcal{A})$ therefore vanishes and so the conditions in Theorem  \ref{prim criterion}(i), (ii) and (iii) are equivalent.

In particular, in this case, if ${\rm d}_\mathcal{A}(C)^{\rm u}$ is a free $\xi(\mathcal{A})$-module, then ${\rm d}_\mathcal{A}(C)$ has a primitive basis $b$. In addition, any basis $b'$ of the $\xi(\mathcal{A})$-module ${\rm d}_\mathcal{A}(C)^{\rm u}$ must differ from $b$ by multiplication by an element of $\mathcal{A}^\times = {\rm Nrd}_A(\mathcal{A}^\times)$ and so Proposition \ref{prim criterion 2}(ii)
implies $b'$ is also a primitive basis of ${\rm d}_\mathcal{A}(C)$. This proves claim (i).

The hypothesis of claim (ii) implies that the group $\ker({\Delta}_\mathcal{A})$ is not trivial. Fix a non-zero element $c$ of this group and a locally-free, rank one, $\mathcal{A}$-module $I$ that corresponds to $c$ under the isomorphism ${\rm Cl}(\mathcal{A}) \cong {\rm SK}_0^{\rm lf}(\mathcal{A})$. 

Then the complex $C := I[0] \oplus \mathcal{A}[-1]$ belongs to $\Der^{{\rm lf},0}(\mathcal{A})$ and its Euler characteristic $\chi_\mathcal{A}(C)$ is equal to $c$. From Theorem \ref{prim criterion}(i) and (iii) it therefore follows that the $\xi(\mathcal{A})$-module ${\rm d}_\mathcal{A}(C)^{\rm u}$ is free and that  
${\rm d}_\mathcal{A}(C)$ has no primitive basis. This proves claim (ii). 
\end{proof}


\section{Relative $K$-theory and zeta elements}\label{rkt section}

In this section we fix a finite extension $F$ of either $\QQ$ or $\QQ_p$ for some prime $p$. If $F$ is a finite extension of $\QQ$, then we write $\mathcal{O}_F$ for its ring of integers. 

We fix a Dedekind domain $R$ with field of fractions $F$. We also fix an $R$-order $\mathcal{A}$ in a finite dimensional separable $F$-algebra $A$ and, 
%
for any extension field $\mathcal{F}$ of $F$, we consider the (finite dimensional, separable) $\mathcal{F}$-algebra $A_\mathcal{F} := \mathcal{F}\otimes_FA.$

\subsection{Relative $K$-theory}\label{recall LTC}

\subsubsection{}We write $\K_0(\mathcal{A},A_\mathcal{F})$ for the relative algebraic $K$-group of the ring inclusion $\mathcal{A}\subset A_\mathcal{F}$.

We recall from \cite[p. 215]{swan} that this group can be described as a quotient (by certain explicit relations) of the free abelian group on elements $ (P,g,Q)$ where $P$ and $Q$ are finitely generated projective $\mathcal{A}$-modules and $g$ an isomorphism of $A_\mathcal{F}$-modules $\mathcal{F}\otimes_R P \cong \mathcal{F}\otimes_R Q$.

We further recall that for any extension field $\mathcal{F}'$ of $\mathcal{F}$ there exists a commutative diagram
\begin{equation} \label{E:kcomm}
\begin{CD} \K_1(\mathcal{A}) @> >> \K_1(A_{\mathcal{F}'}) @> \partial_{\mathcal{A},\mathcal{F}'} >> \K_0(\mathcal{A},A_{\mathcal{F}'}) @> \partial'_{\mathcal{A},\mathcal{F}'} >> \K_0(\mathcal{A})\\
@\vert @A\iota_{A,\mathcal{F},\mathcal{F}'} AA @A\iota_{\mathcal{A},\mathcal{F},\mathcal{F}'} AA @\vert\\
\K_1(\mathcal{A}) @> >> \K_1(A_\mathcal{F}) @> \partial_{\mathcal{A},\mathcal{F}}  >> \K_0(\mathcal{A},A_\mathcal{F}) @> \partial'_{\mathcal{A},\mathcal{F}}  >> \K_0(\mathcal{A})
\end{CD}
\end{equation}
in which the upper and lower rows are the long exact sequences in relative $K$-theory of the inclusions  $\mathcal{A}\subset A_{\mathcal{F}'}$ and $\mathcal{A}\subset A_\mathcal{F}$ and the homomorphisms $\iota_{A,\mathcal{F},\mathcal{F}'}$ and $\iota_{\mathcal{A},\mathcal{F},\mathcal{F}'}$ are injective and induced by the inclusion $A_\mathcal{F} \subseteq A_{\mathcal{F}'}$. (For more details see \cite[Th. 15.5]{swan}.)


For each prime ideal $\mathfrak{p}$ of $R$ we regard the group $\K_0(\mathcal{A}_\mathfrak{p},A_\mathfrak{p})$ as a subgroup of $\K_0(\mathcal{A},A_\mathcal{F})$ by means of the canonical composite injective homomorphism
\begin{equation}\label{decomp}
\bigl({\bigoplus}_{\mathfrak{p}\in {\rm Spm}(R)} \K_0(\mathcal{A}_\mathfrak{p},A_\mathfrak{p})\bigr) \cong \K_0(\mathcal{A},A)\xrightarrow{\iota_{\mathcal{A},F,\mathcal{F}}} \K_0(\mathcal{A},A_\mathcal{F}),
\end{equation}
in which the isomorphism is  described in the discussion following \cite[(49.12)]{curtisr}.

We write $p(\mathfrak{p})$ for the residue characteristic of each $\mathfrak{p}$ in ${\rm Spm}(R)$. If $F$ is a finite extension of $\QQ$ in $\CC$, then for each $\mathfrak{p}$ in ${\rm Spm}(R)$, we write ${\rm Isom}_\mathfrak{p}$ for the set of field isomorphisms $j: \CC\cong \CC_{p(\mathfrak{p})}$ for which the induced embedding $F \subset \CC \to \CC_p$ induces the prime ideal $\mathfrak{p}$.

\subsubsection{}Since $A_\mathcal{F}$ is semisimple, one can compute in the Whitehead group $\K_1(A_\mathcal{F})$ by using the injective homomorphism 
\[ {\rm Nrd}_{A_\mathcal{F}}: \K_1(A_\mathcal{F}) \to \zeta(A_\mathcal{F})^\times\]
that is induced by taking reduced norms (see \cite[\S45.A]{curtisr}). The image of ${\rm Nrd}_{A_\mathcal{F}}$ is described explicitly by the Hasse-Schilling-Maass Norm Theorem (cf. \cite[(7.48)]{curtisr}). We recall, in particular, that this map is bijective if $\mathcal{F}$ is either algebraically closed or a subfield of the completion $\CC_p$ of an algebraic closure of $\QQ_p$ for any prime $p$, but that, in general, its cokernel is of exponent $2$. 

We further recall that $A$ is said to be `ramified' at an archimedean place $v$ of $F$ if there exists a simple component $A'$ in the Wedderburn decomposition of $A$ that is ramified, as a simple central $\zeta(A')$-algebra, at a place of $\zeta(A')$ lying above $v$.  

In the following result we construct, for suitable fields $\mathcal{F}$, a canonical `extended boundary homomorphism' $\zeta(A_\mathcal{F})^\times \to \K_0(\mathcal{A},A_\mathcal{F})$ (that extends the special case considered by Flach and the first author in \cite[\S4.2, Lem. 9]{bf}). 

\begin{proposition}\label{ext bound hom} Fix an embedding of fields $F \to \mathcal{F}$ with the following property:
\begin{itemize}
\item[$\bullet$] if $F$ is a number field, then $\mathcal{F} \subseteq \CC$ and the chosen embedding induces an archimedean place $v$ of $F$ in such a way that $\mathcal{F} = F_v$. 
\end{itemize}
Then there exists a canonical homomorphism of abelian groups 
\begin{equation*}\label{ext bound hom equation} \delta_{\mathcal{A},\mathcal{F}}: 
\zeta(A_\mathcal{F})^\times \to \K_0(\mathcal{A},A_\mathcal{F})\end{equation*}
that has all of the following properties. 

\begin{itemize} 
\item[(i)] The connecting homomorphism $\partial_{\mathcal{A},\mathcal{F}}$ in (\ref{E:kcomm}) is equal to $\delta_{\mathcal{A},\mathcal{F}}\circ {\rm Nrd}_{A_\mathcal{F}}$.
\item[(ii)] The kernel of $\delta_{\mathcal{A},\mathcal{F}}$ comprises all elements of $\zeta(A)^\times$ whose image in $\zeta(A_\mathfrak{p})^\times$ belongs to ${\rm Nrd}_{A_{\mathfrak{p}}}(\K_1(\mathcal{A}_{\mathfrak{p}}))$ for every $\mathfrak{p}$ in ${\rm Spm}(R)$.
\item[(iii)] For any extension $\mathcal{E}$ of $F$ in $\mathcal{F}$, the full pre-image of 
${\rm im}(\iota_{\mathcal{A},\mathcal{E},\mathcal{F}})$ under $\delta_{\mathcal{A},\mathcal{F}}$ is equal to $\zeta(A_\mathcal{E})^\times$.
\item[(iv)] If $F$ is a number field, then for all $\mathfrak{p} \in {\rm Spm}(R)$ and $j \in {\rm Isom}_\mathfrak{p}$, there exists a commutative diagram of the form 
\begin{equation*}\label{p-completion-delta}\begin{CD}
\zeta(A_\mathcal{F})^\times @> \delta_{\mathcal{A},\mathcal{F}}>> \K_0(\mathcal{A},A_\mathcal{F})\\
@Vj'_* VV @VV j_{*}V\\
\zeta(\CC_{p(\mathfrak{p})}\otimes_{F,j}A)^\times @> \delta_{\mathcal{A}_{{\tiny \mathfrak{p}}},
\CC_{\tiny{\tiny p(\mathfrak{p})}}} >> \K_0(\mathcal{A}_{\mathfrak{p}},\CC_{p(\mathfrak{p})}\otimes_{F,j}A).\end{CD}\end{equation*}
Here $j'_*$ denotes the embedding induced by the restriction of $j$ to $\mathcal{F}$ and the homomorphism $j_{*}$ is  induced by sending each tuple $(P,g,Q)$ to $(P_{\mathfrak{p}},\CC_{p(\mathfrak{p})}\otimes_{\mathcal{F},j}g,Q_{\mathfrak{p}})$.
\item[(v)] Assume $F$ is a number field, $A$ is unramified at all archimedean places other than $v$ and ${\rm Spec}(R)$ is open in ${\rm Spec}(\mathcal{O}_F)$. Then, in terms of the notation in claim (iv), for each $x$ in $\zeta(A_\mathcal{F})$, the element $\delta_{\mathcal{A},\mathcal{F}}(x)$ is uniquely determined by the elements 
$\{\delta_{\mathcal{A}_{{\tiny \mathfrak{p}}},\CC_p}(j'_*(x))\}_{\mathfrak{p},j}$, as $\mathfrak{p}$ ranges over ${\rm Spm}(R)$ and $j$ over ${\rm Isom}_\mathfrak{p}$. In particular, in this case, the map $\delta_{\mathcal{A},\mathcal{F}}$ is uniquely determined by the commutativity of all diagrams in claim (iv).  
\end{itemize}
\end{proposition} 

\begin{proof} If either $F$ is a finite extension of $\QQ_p$, or $\mathcal{F} = \CC$, then the map ${\rm Nrd}_{A_\mathcal{F}}$ is bijective and we set 
\begin{equation}\label{ext bound def local} \delta_{\mathcal{A},\mathcal{F}} := \partial_{\mathcal{A},\mathcal{F}}\circ ({\rm Nrd}_{A_\mathcal{F}})^{-1}.\end{equation}
Then, with this definition, claim (i) is obvious and claims (ii) and (iii) both follow from the exactness of the relevant case of (\ref{E:kcomm}) and the fact that ${\rm Nrd}_{A_\mathcal{E}}(\K_1(A_\mathcal{E})) = \zeta(A_\mathcal{E})^\times$. 

Hence, in the rest of the argument, we assume $F$ is a number field and $\mathcal{F} = \RR$ arises as the completion of $F$ at a place $v$. 

Then, writing ${\prod}_{i \in I}F_i$ for the Wedderburn decomposition of $\zeta(A)$ and $\Sigma(i,v)$ for the set of places of each field $F_i$ above $v$, there is a natural decomposition
\begin{equation}\label{wedd units} \zeta(A_\mathcal{F})^\times = {\prod}_{i\in I} (F_i\otimes_F F_v)^\times = {\prod}_{i\in I}{\prod}_{w \in \Sigma(i,v)}F_{i,w}^\times.\end{equation}
For any element $x = (x_i)_{i \in I}$ of $\zeta(A_\mathcal{F})^\times$, and each index $i$, we can therefore use the weak approximation theorem to choose an element $\lambda_{i,x}$ 
of $F_i^\times$ with the property that for each $w$ in $\Sigma(i,v)$ for which $F_{i,w} = \RR$, one has 
$(\lambda_{i,x}x_i)_w >0$. Then, writing $\lambda_x$ for the element $(\lambda_{i,x})_{i \in I}$ of $\zeta(A)^\times$, the Hasse-Schilling-Maass Norm Theorem implies that the product $\lambda_xx $ belongs to ${\rm Nrd}_{A_\mathcal{F}}(\K_1(A_\mathcal{F}))$. 

We may therefore define an element of $\K_0(\mathcal{A},A_\mathcal{F})$ by means of the sum 
\begin{equation}\label{ext bound def} \delta_{\mathcal{A},\mathcal{F}}(x) := \partial_{\mathcal{A},\mathcal{F}}\bigl(({\rm Nrd}_{A_\mathcal{F}})^{-1}(\lambda_x x)\bigr) - {\sum}_{\mathfrak{p}\in {\rm Spm}(R)} \delta_{\mathcal{A}_\mathfrak{p},F_\mathfrak{p}}(\lambda_x ).\end{equation}
Here the sum is regarded as an element of $\K_0(\mathcal{A},A_\mathcal{F})$ via the embedding (\ref{decomp}) (this makes sense since the corresponding cases of the long exact sequence in (\ref{E:kcomm}) implies that each element of $\zeta(A)^\times$ belongs to the kernel of $\delta_{\mathcal{A}_\mathfrak{p},F_\mathfrak{p}}$ for almost all $\mathfrak{p}$ in ${\rm Spm}(R)$).

Note that if $\lambda'_x = (\lambda'_{i,x})_{i \in I}$ is any other element of $\zeta(A)^\times$ chosen as above (with respect to the same element $x$), then for each $w$ in $\Sigma(i,v)$ for which $F_{i,w} = \RR$, one has 
\[ (\lambda'_{i,x}(\lambda_{i,x})^{-1})_w = (\lambda'_{i,x} x_i)_w \cdot (\lambda_{i,x} x_i)^{-1}_w >0\]
and hence $\lambda'_x (\lambda_x)^{-1} \in {\rm Nrd}_{A}(\K_1(A))$. This fact implies that $\delta_{\mathcal{A},\mathcal{F}}(x)$ is independent of the choice of $\lambda_x$. It is also easily seen that the assignment $x \mapsto \delta_{\mathcal{A},\mathcal{F}}(x)$ is a group homomorphism. 

Further, with this explicit definition, the property in claim (i) is clear since for each $x$ in ${\rm im}({\rm Nrd}_{A_\mathcal{F}})$ one can compute $\delta_{\mathcal{A},\mathcal{F}}(x)$ by taking $\lambda_x = 1$ in (\ref{ext bound def}).

 Claim (iii) is also true since for each $x$ in $\zeta(A_\mathcal{F})$, and any $\lambda_x$ in $\zeta(A)^\times$ as fixed in (\ref{ext bound def}), one has 
\begin{align*} \delta_{\mathcal{A},\mathcal{F}}(x) \in \im(\iota_{\mathcal{A},\mathcal{E},\mathcal{F}}) \Longleftrightarrow&\, \partial_{\mathcal{A},\mathcal{F}}\bigl(({\rm Nrd}_{A_\mathcal{F}})^{-1}(\lambda_x x)\bigr) \in \im(\iota_{\mathcal{A},\mathcal{E},\mathcal{F}})\\
\Longleftrightarrow &\, ({\rm Nrd}_{A_\mathcal{F}})^{-1}(\lambda_x x) \in \im(\iota_{A,\mathcal{E},\mathcal{F}})\\
\Longleftrightarrow &\, \lambda_x x \in {\rm Nrd}_{A_\mathcal{F}}\bigl(\im(\iota_{A,\mathcal{E},\mathcal{F}})\bigr) = {\rm Nrd}_{A_\mathcal{E}}(\K_1(A_\mathcal{E}))\\
\Longleftrightarrow &\, x \in \zeta(A_\mathcal{E})^\times.\end{align*}
Here the first equivalence is true since each term $\delta_{\mathcal{A}_\mathfrak{p},F_\mathfrak{p}}(\lambda )$ in (\ref{ext bound def}) belongs to the subgroup $\im(\iota_{\mathcal{A},F,\mathcal{F}})$ of $\im(\iota_{\mathcal{A},\mathcal{E},\mathcal{F}})$, the second follows from the exact commutative diagram (\ref{E:kcomm}) (with $\mathcal{F}$ and $\mathcal{F}'$ replaced by $\mathcal{E}$ and $\mathcal{F}$), the third is clear and the fourth is true since $\lambda_x\in \zeta(A)^\times$. 

To prove claim (iv) we abbreviate $p(\mathfrak{p})$ to $p$ and write $\iota_j$ for the scalar extension map $\K_1(A_\mathcal{F}) \to \K_1(\CC_p\otimes_{F,j}A)$ induced by $j$. Then the given diagram commutes since for each $x \in \zeta(A_\mathcal{F})^\times$, and each $\lambda$ fixed as in (\ref{ext bound def}), one has  
\begin{align*} j_*\bigl(\delta_{\mathcal{A},\mathcal{F}}(x) \bigr) =&\, j_*\bigl(\partial_{\mathcal{A},\mathcal{F}}\bigl(({\rm Nrd}_{A_\mathcal{F}})^{-1}(\lambda_x x)\bigr)\bigr) - j_*\bigl(\delta_{\mathcal{A}_{\mathfrak{p}},F_\mathfrak{p}}(\lambda_x )\bigr)\\
=&\, \partial_{\mathcal{A}_{\mathfrak{p}},\CC_p}\left(\iota_j({\rm Nrd}_{A_\mathcal{F}})^{-1}(\lambda_x x)\right)- \delta_{\mathcal{A}_{\mathfrak{p}},\CC_p}(\lambda_x)\\
=&\, \delta_{\mathcal{A}_{\mathfrak{p}},\CC_p}\left(j'_*(\lambda_x x)\right)- \delta_{\mathcal{A}_{\mathfrak{p}},\CC_p}(\lambda)\\
=&\, \delta_{\mathcal{A}_{\mathfrak{p}},\CC_p}\left(j'_*(x)\right). \end{align*}
Here the first equality follows from the formula (\ref{ext bound def}) and the fact that $j_*(\delta_{\mathcal{A}_\mathfrak{q},F_\mathfrak{q}}(\lambda_x )) =0$ for all $\mathfrak{q}\in {\rm Spm}(R)\setminus \{\mathfrak{p}\}$, the second from commutativity of the relevant case of (\ref{E:kcomm}), the third from the definition (\ref{ext bound def local}) of $\delta_{\mathcal{A}_{\mathfrak{p}},\CC_p}$ and the compatibility of reduced norms under scalar extension and the last equality is clear. 

Claim (v) follows directly from the commutativity of the diagrams in claim (iv) and the general result of Lemma \ref{injectivity result} below. 

In a similar way, since $\zeta(A)^\times$ is the full pre-image under 
$\delta_{\mathcal{A},\mathcal{F}}$ of $\im(\iota_{\mathcal{A},A,\mathcal{F}})$ (by claim (iii)), the assertion of claim (ii) in the number field case follows from Lemma \ref{injectivity result} and the exactness of the relevant cases of (\ref{E:kcomm}).

This completes the proof of the proposition.  \end{proof} 


In the next result we assume the notation and hypotheses of Proposition \ref{ext bound hom}(v). For each $\mathfrak{p}$ in 
${\rm Spm}(R)$ and $j \in {\rm Isom}_\mathfrak{p}$ we also set $A_j^c := \CC_{p(\mathfrak{p})}\otimes_{F,j}A$.

\begin{lemma}\label{injectivity result} Assume $F$ is a number field, $A$ is unramified at all archimedean places other than $v$ and ${\rm Spec}(R)$ is open in ${\rm Spec}(\mathcal{O}_F)$. Then the natural diagonal map
\[ \K_0(\mathcal{A},A_\mathcal{F}) \xrightarrow{(j_{*})_{\mathfrak{p},j}} {\prod}_{\mathfrak{p}\in {\rm Spm}(R)}
{\prod}_{j\in {\rm Isom}_\mathfrak{p}} \K_0(\mathcal{A}_{\mathfrak{p}},A_j^c)\]
is injective.
\end{lemma}

\begin{proof} We consider the exact sequences that are given by the lower row of (\ref{E:kcomm}) with the pair $(R,\mathcal{F})$ taken to be $(R,F)$, $(R,\mathcal{F})$, $(R_{\mathfrak{p}},F_{\mathfrak{p}})$ and
$(R_{\mathfrak{p}},\CC_{p(\mathfrak{p})})$ and the maps between these sequences that are
induced by the obvious inclusions and by an isomorphism $j$ in ${\rm Isom}_\mathfrak{p}$. Then an easy diagram chase gives a
commutative diagram of short exact sequences
\begin{equation*}
\xymatrix{
0 \ar[r] & \K_0(\mathcal{A},A) \ar[r] \ar[d] & \K_0(\mathcal{A},A_\mathcal{F}) \ar[r] \ar[d] &
\K_1(A_\mathcal{F})/\K_1(A) \ar[r] \ar[d] & 0 \\
0 \ar[r] & \K_0(\mathcal{A}_{\mathfrak{p}},A_{\mathfrak{p}}) \ar[r] & \K_0(\mathcal{A}_{\mathfrak{p}},A_j^c) \ar[r] &
\K_1(A_j^c)/\K_1(A_{\mathfrak{p}}) \ar[r] & 0.
}
\end{equation*}
In view of the isomorphism in (\ref{decomp}) it is therefore enough to show that the natural diagonal map
\begin{equation*}
\label{equation_K_injectivity_right}
\K_1(A_\mathcal{F})/\K_1(A)\to {\prod}_{\mathfrak{p}\in {\rm Spm}(R)}
{\prod}_{j\in {\rm Isom}_\mathfrak{p}}
\K_1(A_j^c)/\K_1(A_{\mathfrak{p}})
\end{equation*}
is injective. To do this we fix $x$ in $\K_1(A_\mathcal{F})$ with $j_*(x)\in \K_1(A_{\mathfrak{p}})\subseteq \K_1(A_j^c)$ for all $\mathfrak{p}$ and all $j \in {\rm Isom}_\mathfrak{p}$ and must show that $x$ belongs to the subgroup $\K_1(A)$ of $\K_1(A_\mathcal{F})$. 
%
%

We use the (injective) maps ${\rm Nrd}_{A_\mathcal{F}}$ and ${\rm Nrd}_{A}$ to identify $\K_1(A_\mathcal{F})$ and $\K_1(A)$ with ${\rm Nrd}_{A_\mathcal{F}}(\K_1(A_\mathcal{F}))$ and ${\rm Nrd}_{A}(\K_1(A))$ respectively. We then fix an $F$-basis $\{a_\omega:\omega \in \Omega\}$ of $\zeta(A)$ so that $x = {\sum}_{\omega \in \Omega} c_\omega a_\omega$ with each $c_\omega$ in $\mathcal{F}$ and for every $j\in {\rm Isom}_\mathfrak{p}$ one has
\begin{equation}
\label{equation_lemma_K_injectivity}
j'_*(x)={\sum}_{\omega \in \Omega} j(c_\omega)a_\omega\in\zeta(A_{\mathfrak{p}})^\times ,
\end{equation}
where $j'_*$ denotes the inclusion $\zeta(A_\mathcal{F})^\times \to \zeta(A_{\mathfrak{p}}^c)^\times$ induced by $j$.

We now fix $\omega$ in $\Omega$ and consider the
coefficient $c_\omega$. If $c_\omega$ was transcendental over $F$, then there would exist an isomorphism $j$ in ${\rm Isom}_\mathfrak{p}$ such that
$j(c_\omega)\not\in F_{\mathfrak{p}}$, thereby contradicting
(\ref{equation_lemma_K_injectivity}). Therefore $c_\omega$ is algebraic
over $F$. The fact that $j(c_\omega)$ belongs to $F_{\mathfrak{p}}$ for all $j \in {\rm Isom}_\mathfrak{p}$ then implies that $\mathfrak{p}$ is completely split in the extension $F(c_\omega)/F$. Hence, since ${\rm Spec}(R)$ is open in ${\rm Spec}(\mathcal{O}_F)$, the Tchebotarov Density Theorem implies $F(c_\omega)=F$.

At this stage, we know that $x$ belongs to both ${\rm Nrd}_{A_\mathcal{F}}(\K_1(A_\mathcal{F}))$ and $\zeta(A)^\times$ and so it suffices to show that, under the stated hypotheses, one has  
\[ {\rm Nrd}_{A_\mathcal{F}}(\K_1(A_\mathcal{F}))\cap \zeta(A)^\times = {\rm Nrd}_{A}(\K_1(A)).\]

To verify this we use the decomposition (\ref{wedd units}) and, for each $i \in I$, we write $\Sigma(i,v)'$ for the subset of $\Sigma(i,v)$ comprising (archimedean) places at which the algebra $A_i\otimes_{F_i}F_{i,w}$ is ramified (and so $F_{i,w} = \mathbb{R}$). It is then enough to note that the Hasse-Schilling-Maass Norm Theorem implies both that  
\begin{multline*} {\rm Nrd}_{A_\mathcal{F}}(\K_1(A_\mathcal{F}))\\ = \{(x_{i,w})_{i,w} \in {\prod}_{i \in I}{\prod}_{w \in \Sigma(i,v)}F_{i,w}^\times = \zeta(A_\mathcal{F})^\times : x_{i,w} > 0 \,\,\text{ for all } w \in \Sigma(i,v)'\},\end{multline*}
and, as $A$ is unramified at all archimedean places of $F$ other than $v$, also 
\[ {\rm Nrd}_{A}(\K_1(A)) = \{(x_i)_{i} \in {\prod}_{i \in I}F_i^\times = \zeta(A)^\times : x_{i,w} > 0 \,\,\text{ for all } w \in \Sigma(i,v)'\}.\]
\end{proof}

\subsubsection{}In the following result we describe another useful consequence of the long exact sequence of relative $K$-theory. 

\begin{lemma}\label{last tech} For each element $x$ of ${\rm Nrd}_A(\K_1(A))$ the following claims are valid. 
\begin{itemize}
\item[(i)] For each $\mathfrak{p}$ in ${\rm Spm}(R)$ one has $x \in {\rm Nrd}_{A}(\K_1(\mathcal{A}_{(\mathfrak{p})}))$ if and only if the image of $x$ in ${\rm Nrd}_{A_\mathfrak{p}}(\K_1(A_\mathfrak{p}))$ belongs to 
${\rm Nrd}_{A_\mathfrak{p}}(\K_1(\mathcal{A}_\mathfrak{p}))$.
\item[(ii)] The following conditions are equivalent: 
\begin{itemize}
\item[(a)] $x \in {\rm Nrd}_A(\K_1(\mathcal{A}))$.
\item[(b)] For all $\mathfrak{p} \in {\rm Spm}(R)$, one has $x \in {\rm Nrd}_A(\K_1(\mathcal{A}_{(\mathfrak{p})}))$. 
\item[(c)] For all $\mathfrak{p} \in {\rm Spm}(R)$, the image of $x$ in ${\rm Nrd}_{A_\mathfrak{p}}(\K_1(A_\mathfrak{p}))$ belongs to ${\rm Nrd}_{A_\mathfrak{p}}(\K_1(\mathcal{A}_\mathfrak{p}))$. 
\end{itemize}
\end{itemize}
\end{lemma}

\begin{proof} The relevant cases of the exact sequence (\ref{E:kcomm}) give rise to an commutative diagram of abelian groups 
in which all rows are exact

\[\begin{CD} \K_1(\mathcal{A}) @> \iota_\mathcal{A} >> \K_1(A) @> >> \K_0(\mathcal{A},A)\\
@V VV @V VV @VV \kappa V \\
{\prod}_\mathfrak{p} \K_1(\mathcal{A}_{(\mathfrak{p})}) @> (\iota_{(\mathfrak{p})})_\mathfrak{p} >> {\prod}'_\mathfrak{p}\K_1(A) @> >> {\bigoplus}_\mathfrak{p}\K_0(\mathcal{A}_{(\mathfrak{p})},A)\\
@V VV @V VV @VV (\kappa_{\mathfrak{p}})_\mathfrak{p} V \\
{\prod}_\mathfrak{p} \K_1(\mathcal{A}_{\mathfrak{p}}) @> (\iota_{\mathfrak{p}})_\mathfrak{p} >> {\prod}''_\mathfrak{p}
\K_1(A_\mathfrak{p}) @> >> {\bigoplus}_\mathfrak{p}\K_0(\mathcal{A}_{\mathfrak{p}},A_\mathfrak{p})
.\end{CD}\]
Here $\iota_{(\mathfrak{p})}$ and $\iota_\mathfrak{p}$ denote the scalar extension maps $\K_1(\mathcal{A}_{(\mathfrak{p})}) \to \K_1(A)$ and $\K_1(\mathcal{A}_{\mathfrak{p}}) \to\K_1(A_\mathfrak{p})$, ${\prod}'_\mathfrak{p}$ the restricted direct product over $\mathfrak{p}$ in ${\rm Spm}(R)$ of $\K_1(A)$ with respect to the subgroups $\im(\iota_{(\mathfrak{p})})$ and ${\prod}''_\mathfrak{p}$ the restricted direct product of the groups 
$\K_1(A_\mathfrak{p})$ with respect to $\im(\iota_{\mathfrak{p}})$. In addition, the upper vertical arrows are the natural diagonal maps, the first and second lower vertical maps are induced by the scalar extension maps 
$\K_1(\mathcal{A}_{(\mathfrak{p})}) \to \K_1(\mathcal{A}_{\mathfrak{p}})$ and $\K_1(A) \to \K_1(A_{\mathfrak{p}})$ and $\kappa_\mathfrak{p}$ denotes the scalar extension map $\K_0(\mathcal{A}_{(\mathfrak{p})},A) \to \K_0(\mathcal{A}_{\mathfrak{p}},A_\mathfrak{p})$. 

We recall that each map $\kappa_\mathfrak{p}$ is bijective since both groups $\K_0(\mathcal{A}_{(\mathfrak{p})},A)$ and $\K_0(\mathcal{A}_{\mathfrak{p}},A_\mathfrak{p})$ identify with the Grothendieck group of finitely generated $\mathfrak{p}$-torsion $\mathcal{A}$-modules of finite projective dimension over $\mathcal{A}$ (cf. the discussion in \cite[Rem. (40.19)]{curtisr}).
 From the commutativity of the lower part of the diagram, one can therefore deduce that an element $y$ of $\K_1(A)$ belongs to $\im(\iota_{(\mathfrak{p})})$ if and only if its image in $\K_1(A_\mathfrak{p})$ belongs to $\im(\iota_\mathfrak{p})$. This implies claim (i) since the maps ${\rm Nrd}_{A}$ and ${\rm Nrd}_{A_\mathfrak{p}}$ are both bijective and, by definition, one has both ${\rm Nrd}_{A}(\K_1(\mathcal{A})) = {\rm Nrd}_{A}(\im(\iota_{\mathcal{A}}))$ and ${\rm Nrd}_{A}(\K_1(\mathcal{A}_{(\mathfrak{p})})) = {\rm Nrd}_{A}(\im(\iota_{(\mathfrak{p})}))$. 
 
In a similar way, the result of claim (i) combines with the injectivity of ${\rm Nrd}_A$ to reduce the proof of claim (ii) to showing that an element $y$ of $\K_1(A)$ belongs to $\im(\iota_\mathcal{A})$ if and only if, for every $\mathfrak{p}$ in ${\rm Spm}(R)$, it belongs to $\im(\iota_{(\mathfrak{p})})$. It is thus enough to note that this property follows directly from the injectivity of $\kappa$ and the commutativity of the upper part of the diagram. 
%
%
%
\end{proof}

\subsection{Virtual objects and zeta elements} 

\subsubsection{}\label{virtual sec}In this section we recall the construction of Euler characteristics that underlies the formulation of a range of refined special value conjectures in the literature.

To do this we recall first that, as already mentioned  in a special case in Remark \ref{comparison of categories0}, in \cite[\S4]{delignedet} Deligne constructs for any category $\mathcal{E}$ that is exact in the sense of Quillen \cite[p. 91]{quillen} a universal determinant functor of the form
\[ [-]: \mathcal{E}_{\rm Isom} \to \calV(\calE).\]
Here $\mathcal{E}_{\rm Isom}$ denotes the subcategory of $\calE$ in which morphisms are restricted to isomorphisms and $\calV(\calE)$ is the Picard category of `virtual objects' associated to $\mathcal{E}$.
%

In the case that $\mathcal{E}$ is the category ${\rm Mod}^{\rm proj}(\Lambda)$ of finitely generated projective left modules over a ring $\Lambda$, we write $\calV(\calE)$ as $\calV(\Lambda)$.

If now $\Lambda\to \Sigma$ is a ring homomorphism, then the functor ${\rm Mod}^{\rm proj}(\Lambda) \to {\rm Mod}^{\rm proj}(\Sigma)$ that sends each $P$ to $\Sigma\otimes_\Lambda P$  is exact and so, by \cite[\S4.11]{delignedet}, induces a monoidal functor 
\[ \calV(\Lambda)\to\calV(\Sigma), \quad  L\mapsto L_\Sigma\]
that is unique up to natural isomorphism.

Let $\calP_0$ be the Picard category with unique object
${\bf 1}_{\calP_0}$ and trivial automorphism group $\Aut_{\calP_0}({\bf 1}_{\calP_0})$. Then the fibre product diagram
\begin{equation*}
\xymatrix{
\calV(\Lambda,\Sigma) \ar[r] \ar[d] & \calP_0 \ar[d]^{{\bf 1}_{\calP_0}\mapsto {\bf 1}_{\calV(\Sigma)}} \\
\calV(\Lambda) \ar[r]^{L \mapsto L_\Sigma} & \calV(\Sigma)
}
\end{equation*}
defines a Picard category $\calV(\Lambda,\Sigma)$ in which objects are pairs $(L,t)$ with $L$ in
$\calV(\Lambda)$ and $t$ an isomorphism $L_\Sigma\to {\bf 1}_{\calV(\Sigma)}$ in $\calV(\Sigma)$.
%
 %

 It is shown by Breuning and the first author in \cite[Lem. 5.1]{bb} (following an argument of \cite[\S2.8, Prop. 2.5]{bf}) that there exists a canonical isomorphism of abelian groups
\begin{equation*}\tau_{\Lambda,\Sigma}: K_0(\Lambda,\Sigma)\to\pi_0(\calV(\Lambda,\Sigma)).\end{equation*}
This map sends each element $(P,g,Q)$ in $\K_0(\Lambda,\Sigma)$
 to the isomorphism class of the pair comprising $[P]\otimes [Q]^{-1}$ and the composite isomorphism
\begin{equation*}
([P]\otimes [Q]^{-1})_\Sigma\longrightarrow [\Sigma\otimes_\Lambda P]\otimes [\Sigma\otimes_\Lambda Q]^{-1}
\xrightarrow{[g]\otimes\id} [\Sigma\otimes_\Lambda Q]\otimes [\Sigma\otimes_\Lambda Q]^{-1} \longrightarrow {\bf 1}_{\calV(\Sigma)}.
\end{equation*}
In particular, if $\alpha$ is an automorphism of the $\Sigma$-module $\Sigma\otimes_\Lambda P$, and $\langle \alpha\rangle$ its class in $\K_1(\Sigma)$, then one has
\begin{equation}\label{virtual iso}
\tau_{\Lambda,\Sigma}( \partial_{\Lambda,\Sigma}(\langle\alpha\rangle)) = \tau_{\Lambda,\Sigma}( (P,\alpha,P)) = [{\bf 1}_{\mathcal{V}(\Lambda)},u_\alpha].\end{equation}
Here $\partial_{\Lambda,\Sigma}: \K_1(\Sigma)\to \K_0(\Lambda,\Sigma)$ is the canonical connecting homomorphism as in (\ref{E:kcomm}), $u_\alpha$ denotes the image of $\langle\alpha\rangle$ under the canonical identification 
\[ \K_1(\Sigma)\cong \pi_1(\mathcal{V}(\Sigma)) :={\rm Aut}_{\calV(\Sigma)}({\bf 1}_{\calV(\Sigma)}),\]
and we write $[L,t]$ for the isomorphism class of a pair $(L,t)$ in $\mathcal{V}(\Lambda,\Sigma)$.

The following definition of Euler characteristic underlies the constructions that are made in  \cite{bb} and \cite{bf}.

\begin{definition}\label{rec def}{\em Fix $C$ in $\Der^{\rm perf}(\mathcal{A})$ and a morphism
$t: [C_\mathcal{F}] \to {\bf 1}_{\mathcal{V}(A_\mathcal{F})}$ in $\mathcal{V}(A_\mathcal{F})$. Then
$\chi_{\mathcal{A},\mathcal{F}}(C,t)$ denotes the element of $\K_0(\mathcal{A},A_\mathcal{F})$
that $\tau_{\mathcal{A},A_\mathcal{F}}$ sends to $[[C],t]$.}
\end{definition}

\begin{remark}\label{remark euler}
{\em The homomorphism
$\partial_{\mathcal{A},\mathcal{F}}'$ in (\ref{E:kcomm}) sends each element $\chi_{\mathcal{A},\mathcal{F}}(C,t)$ to the classical 
Euler characteristic of $C$
in $\K_0(\mathcal{A})$. It is for this reason that the elements $\chi_{\mathcal{A},\mathcal{F}}(C,t)$
are sometimes referred to as `refined Euler characteristics'.}\end{remark}

\subsubsection{}\label{exp comp zeta}

As in \S \ref{plp section}, we shall in the sequel abbreviate the functors ${\rm d}_{A_\mathcal{F},\varpi}$ and ${\rm d}_{\cA,\varpi}$ to ${\rm d}_{A_\mathcal{F}}$ and  ${\rm d}_\cA$ respectively.

The following definition is a natural analogue in our setting of the `zeta elements' that were introduced (in an arithmetic setting) by Kato in \cite{K1}.

\begin{definition}\label{def zeta}
{\em Let $C$ be an object of $\Der^{{\rm lf},0}(\mathcal{A})$ and $\lambda$ an isomorphism in $\mathcal{P}(\zeta(A_\mathcal{F}))$ of the form ${\rm d}_{A_\mathcal{F}}(C_\mathcal{F}) \cong (\zeta(A_\mathcal{F}),0)$.

Then, for each element $x$ of $\zeta(A_\mathcal{F})^\times$, the `zeta element' associated to the pair $(\lambda,x)$ is the unique element $z_{\lambda,x}$ of ${\rm d}_{A_\mathcal{F}}(C_\mathcal{F})$ that satisfies 
\[ \lambda(z_{\lambda,x}) = (x,0)\]
in $(\zeta(A_\mathcal{F}),0)$.}\end{definition}

\begin{example}\label{ex}{\em
 There are several important examples of zeta elements in the literature.\

\noindent{}(i) Assume $F=\QQ_p$, $R=\ZZ_p$ and $\mathcal{F}=\CC_p$. Let $M$ be a motive defined over a number field $K$ with coefficients in a finite dimensional semisimple $\QQ$-algebra $B$, and $T$ a Galois-stable lattice in the $p$-adic realization of $M$ that is projective over some $\ZZ_p$-order $\cA$ in $A := \QQ_p\otimes_\QQ B$. Fix a finite set $\Pi$ of places of $K$ containing all archimedean and $p$-adic places and all places at which $M$ has bad reduction. Then the $p$-adic \'etale cohomology complex $C(M):=\rhom_{\cA}(\rgamma_c(\mathcal{O}_{K,\Pi},T), \cA[-2])$ belongs to $\Der^{{\rm f},0}(\cA)$ and, after fixing an isomorphism of fields $\CC \cong \CC_p$,  there exists a canonical `period-regulator isomorphism' of the form $\lambda: {\rm d}_{A_{\CC_p}}(C(M)_{\CC_p}) \xrightarrow{\sim} (\zeta(A_{\CC_p}),0)$ (cf. Remark \ref{remark etnc} below). In particular, if we write $x$ for the leading term $L^\ast(M,0)$ in $\zeta(A_\CC)^\times \cong \zeta(A_{\CC_p})^\times$ of the $L$-function of $M$ at $s=0$, then the element $z_{\lambda,x}$ defined above 
 generalizes the zeta elements defined for abelian extensions by Kato in \cite{K1}. 
 
\noindent{}(ii) Assume $F=\QQ$, $R=\ZZ$ and $\mathcal{F}=\RR$. Let $L/K$ be a finite Galois extension of number fields with group $G$ and set $\cA:=\ZZ[G]$ (so that $A_\mathcal{F}=\RR[G]$). Then, for a suitable choice of complex $C$, isomorphsm $\lambda$ and element $x$, the element $z_{\lambda,x}$ defined above generalizes the `zeta elements' $z_{L/K,\Pi}$ defined for abelian extensions $L/K$ by Kurihara and the present authors in \cite[Def. 3.5]{bks} (following Kato \cite{K1}). For details, see Remark \ref{bks remark} below.
 }\end{example}

In the next result we interpret zeta elements in terms of the Euler characteristic construction in Definition \ref{rec def}.

Before stating this result we note that every object of the category ${\rm Mod}^{\rm fg}(A_\mathcal{F})$ of finitely generated $A_\mathcal{F}$-modules is projective (as $A_\mathcal{F}$ is semisimple) and so the construction of Deligne \cite{delignedet} gives a determinant functor 
\[ [-]: {\rm Mod}^{\rm fg}(A_\mathcal{F})_{\rm Isom} \to \mathcal{V}(A_\mathcal{F}).\]
In particular, the universal nature of this functor implies that, for any choice of ordered bases $\varpi$ as in \S\ref{Exterior powers over semisimple rings}, there exists a canonical  additive functor
\begin{eqnarray*}\label{nu}
 \nu = \nu_{A_\mathcal{F},\varpi}: \mathcal{V}(A_\mathcal{F})\to \mathcal{P}(\zeta(A_\mathcal{F}))
 \end{eqnarray*}
(cf. \cite[\S2.6, Lem. 2]{bf}). 

This functor sends the virtual object $[M]$, for each $M$ in ${\rm Mod}^{\rm fg}(A_\mathcal{F})$, to ${\rm d}_{A_\mathcal{F},\varpi}(M)$, and $[\alpha]$, for each $\alpha$ in $\pi_1(\mathcal{V}(A_\mathcal{F})) \cong \K_1(A_\mathcal{F})$, to the  automorphism of $(\zeta(A_\mathcal{F}),0)$ that is given by multiplication by ${\rm Nrd}_{A_\mathcal{F}}(\alpha)$.



\begin{theorem}\label{ltc2} We assume to be given data of the following sort: 
\begin{itemize}
\item[$\bullet$] an $R$-order $\mathcal{A}$ in a finite dimensional separable $F$-algebra $A$;
\item[$\bullet$] an embedding of fields $F \to \mathcal{F}$, as in Proposition \ref{ext bound hom};
\item[$\bullet$] an object $C$ of $\Der^{{\rm lf},0}(\mathcal{A})$; 
\item[$\bullet$] a morphism $t: [C_\mathcal{F}] \to {\bf 1}_{\mathcal{V}(A_\mathcal{F})}$ in $\mathcal{V}(A_\mathcal{F})$.
\end{itemize}
Then $\nu(t)$ is an isomorphism ${\rm d}_{A_\mathcal{F}}(C_\mathcal{F}) \cong (\zeta(A_\mathcal{F}),0)$ in 
 $\mathcal{P}(\zeta(A_\mathcal{F}))$ and, for each element $x$ of $\zeta(A_\mathcal{F})^\times$, one can consider the possible equality 
\begin{equation}\label{etnc interpret} \delta_{\mathcal{A},\mathcal{F}}(x) \stackrel{?}{=} \chi_{\mathcal{A},\mathcal{F}}(C,t)
\end{equation}
in $\K_0(\mathcal{A},A_\mathcal{F})$. In any such situation, the following claims are valid. 

\begin{itemize}
\item[(i)] If (\ref{etnc interpret}) is valid, then $z_{\nu(t),x}$ is a basis of the $\xi(\mathcal{A})$-module 
${\rm d}_\mathcal{A}(C)^{\rm u}$.

\item[(ii)] Assume $F$ is a finite extension of $\QQ_p$, $R$ is its valuation ring and $\mathcal{F}$ is an extension of $F$ in $\CC_p$. Then (\ref{etnc interpret}) is valid if and only if $z_{\nu(t),x}$ is a primitive basis of ${\rm d}_\mathcal{A}(C)$.
\end{itemize}

In the remaining claims we assume that $F$ is a number field, that $A$ is unramified at all archimedean places of $F$ other than that corresponding to the fixed embedding $F \to \mathcal{F}\subseteq \CC$ and that ${\rm Spec}(R)$ is open in ${\rm Spec}(\mathcal{O}_F)$. 

\begin{itemize}
\item[(iii)] The equality (\ref{etnc interpret}) is valid if and only if $z_{\nu(t),x}$ is a locally-primitive basis of ${\rm d}_\mathcal{A}(C)$.

\item[(iv)] If (\ref{etnc interpret}) is valid, then ${\rm d}_\mathcal{A}(C)$ has a primitive basis if and only if there exists an element $x'$ of ${\rm Nrd}_{A_\mathcal{F}}(\K_1(A_\mathcal{F}))$ of $\zeta(A_\mathcal{F})^\times$ such that, for all 
$\mathfrak{p}$ in ${\rm Spm}(R)$, and all field isomorphisms $j$ in ${\rm Isom}_\mathfrak{p}$, one has  
$j'_*(x\cdot x')\in {\rm Nrd}_{A_{\mathfrak{p}}}(\K_1(\mathcal{A}_{\mathfrak{p}}))$.
\item[(v)] If (\ref{etnc interpret}) is valid, then $z_{\nu(t),x}$ is a primitive basis of ${\rm d}_\mathcal{A}(C)$ if and only if one has $x\in {\rm Nrd}_{A_\mathcal{F}}(\K_1(A_\mathcal{F}))$.
\end{itemize}
\end{theorem}

\begin{remark}\label{remark etnc}
{\em The `equivariant Tamagawa Number Conjecture' that is formulated in \cite[Conj. 4(iv)]{bf}, and hence also various associated `equivariant leading term conjectures' in the literature, are equalities of the form (\ref{etnc interpret}) for suitable choices of data $\mathcal{A},\mathcal{F},x, C$ and $t$. We briefly mention two concrete applications of Theorem \ref{ltc2} in this setting. 

\noindent{}(i) If, in the setting of Example \ref{ex}(i), we consider the composite morphism
$$t: [C(M)_{\CC_p}]=[\CC_p\otimes_{\ZZ_p}^{\DL}\rgamma_c(\mathcal{O}_{K,S}, T)]^{-1} \cong \Xi(M)_{\CC_p}^{-1} \cong {\bf 1}_{\mathcal{V}(A_{\CC_p})},$$
where $\Xi(M)$ is the `fundamental line' defined in \cite[(29)]{bf} and the first and the second isomorphisms are respectively induced by the morphisms $\vartheta_p(M,S)$ and $\vartheta_\infty$ in \cite[\S3.4]{bf}, then the equivariant Tamagawa Number Conjecture for $(M,\cA)$ is formulated as the equality $\delta_{\mathcal{A},\CC_p}(L^\ast(M,0)) = \chi_{\mathcal{A},\CC_p}(C(M),t)$. Theorem \ref{ltc2}(ii) therefore implies that this conjecture is valid if and only if the zeta element $z_{\nu(t),L^\ast(M,0)}$ is a primitive basis of ${\rm d}_\cA(C(M))$.\

\noindent{}(ii) In the setting of Example \ref{ex}(i), Theorem \ref{ltc2}(iii) provides a similar reinterpretation of the   `lifted root number conjecture' of Gruenberg, Ritter and Weiss \cite{grw}. For details see Remark \ref{bks remark} below. 
}\end{remark}

\begin{remark}\label{full ltc2}{\em Assume the setting of Theorem \ref{ltc2}(ii). Then, in this case, the result of Theorem \ref{ltc2} combines with Proposition \ref{prim criterion 2}(ii) to give an equivalence
\[ \delta_{\mathcal{A},\mathcal{F}}(x) = \chi_{\mathcal{A},\mathcal{F}}(C,t)\Longleftrightarrow {\rm Nrd}_{A}(\K_1(\mathcal{A}))\cdot z_{\nu(t),x} =  {\rm d}_{\mathcal{A}}(C)^{\rm pb},\]
where ${\rm d}_{\mathcal{A}}(C)^{\rm pb}$ denotes the subset of ${\rm d}_{\mathcal{A}}(C)$ comprising all primitive-basis elements.
} \end{remark}

\begin{remark}{\em Assume $\mathcal{A}$ is such that a finitely generated $\mathcal{A}$-module is locally-free if and only if it is both projective and spans a free $A$-module (cf. Remark \ref{loc free exam}(iii)). Then, in this case, it is easily seen that every element of $\K_0(\mathcal{A},A_\mathcal{F})$ is of the form $\chi_{\mathcal{A},\mathcal{F}}(C,t)$ for a suitable choice of data $C$ and $t$ as in Theorem \ref{ltc2}.}\end{remark}

The proof of Theorem \ref{ltc2} will occupy the rest of \S\ref{rkt section}.

\subsubsection{}To prove claim (i) it is enough to show that, for every $\mathfrak{p}$ in ${\rm Spm}(R)$,  the validity of the image of (\ref{etnc interpret}) under the natural map $\K_0(\mathcal{A},A_\mathcal{F}) \to 
K_0(\mathcal{A}_{(\mathfrak{p})},A_\mathcal{F})$ implies an equality 
\[ \xi(\mathcal{A}_{(\mathfrak{p})})\cdot z_{\nu(t),x} = 
{\rm d}_{\mathcal{A}_{(\mathfrak{p})}}(C_{(\mathfrak{p})})^{\rm u}.\] 

Thus, after fixing $\mathfrak{p}$ and replacing $\mathcal{A}$ by $\mathcal{A}_{(\mathfrak{p})}$ we will assume that $R$ is local (with maximal ideal $\mathfrak{p}$) and $C$ belongs to $\DC^{\rm f}(\mathcal{A})$. For each integer $a$ we then set $r_a := {\rm rk}_{\mathcal{A}}(C^a)$ and fix an ordered $\mathcal{A}$-basis $\{b_{as}\}_{1\le s\le r_a}$ of $C^a$. Taken together, these choices determine an isomorphism in $\mathcal{V}(\mathcal{A})$ of the form
\[ \kappa: [C] \cong {\bigotimes}_{a\in \ZZ}[C^a]^{(-1)^a} \cong {\bigotimes}_{a\in \ZZ}[\mathcal{A}^{r_a}]^{(-1)^a}\cong [\mathcal{A}]^{{\sum}_{a\in \ZZ}(-1)^ar_a} \cong {\bf 1}_{\mathcal{V}(\mathcal{A})}\]
where the last map is induced by the fact ${\sum}_{a\in \ZZ}(-1)^ar_a = 0$ (as $C$ belongs to $\DC^{{\rm lf},0}(\mathcal{A})$).

This isomorphism in turn induces an isomorphism in $\mathcal{V}(\mathcal{A},A_\mathcal{F})$
\[ ([C],t) \cong (\kappa([C]),t\circ \kappa^{-1}) = ({\bf 1}_{\calV(\mathcal{A})},t\circ \kappa^{-1})\]
which combines with (\ref{virtual iso}) to imply that
%
%
\[ \chi_{\mathcal{A},\mathcal{F}}(C,t) = \delta_{\mathcal{A},\mathcal{F}}(\epsilon_{t,\kappa}),\]
with $\epsilon_{t,\kappa} := {\rm Nrd}_{A_\mathcal{F}}(t\circ \kappa^{-1})\in \zeta(A_\mathcal{F})^\times$.
%

The validity of (\ref{etnc interpret}) is therefore equivalent to an equality 
$\delta_{\mathcal{A},\mathcal{F}}(\epsilon_{t,\kappa}) = \delta_{\mathcal{A},\mathcal{F}}(x)$, and hence to a containment 
\begin{equation}\label{last dis} x\cdot \epsilon_{t,\kappa}^{-1} \in \ker(\delta_{\mathcal{A},\mathcal{F}}).\end{equation}
%
%
In particular, if this containment is valid, then Proposition \ref{ext bound hom}(ii) implies that $x\cdot \epsilon_{t,\kappa}^{-1}$ belongs to both $\zeta(A)^\times$ and ${\rm Nrd}_{A_\mathfrak{p}}(\K_1(\mathcal{A}_{\mathfrak{p}})) \subseteq \xi(\mathcal{A}_{\mathfrak{p}})^\times$. In this situation it would therefore follow that $x\cdot \epsilon_{t,\kappa}^{-1}$ belongs to $\zeta(A)\cap \xi(\mathcal{A}_\mathfrak{p})^\times = \xi(\mathcal{A})^\times$.  
%

On the other hand, the ordered bases $\{b_{as}\}_{1\le s\le r_a}$ fixed above (for each $a \in \ZZ$) together determine a 
primitive basis $z'$ of ${\rm d}_{\mathcal{A}}(C)$ that $\nu(\kappa)$ sends to the element $(1,0)$ of $(\zeta(A),0)$.
 One therefore has $\nu(t)(z') = (\epsilon_{t,\kappa},0)$
 and so the definition of $z_{\nu(t),x}$ implies that 
\begin{equation}\label{different bases} z_{\nu(t),x} = (x\cdot \epsilon_{t,\kappa}^{-1})\cdot z'.\end{equation}

Thus, if $x\cdot \epsilon_{t,\kappa}^{-1}$ belongs to $\xi(\mathcal{A})^\times$, then one has 
\begin{align*} \xi(\mathcal{A})\cdot z_{\nu(t),x} &=\, \xi(\mathcal{A})\cdot ((x\cdot \epsilon_{t,\kappa}^{-1})\cdot z')\\
 &=\, (\xi(\mathcal{A})\cdot(x\cdot \epsilon_{t,\kappa}^{-1}))\cdot z'\\
  &=\, \xi(\mathcal{A})\cdot z' \\
  &=\, {\rm d}_{\mathcal{A}}(C)^{\rm u},\end{align*}
as required to prove claim (i). 

Claim (ii) of Theorem \ref{ltc2} also follows directly from the above argument and the fact that, under the given hypotheses, the equality (\ref{different bases}) combines with Proposition \ref{prim criterion 2}(ii) to imply that $z_{\nu(t),x}$ is a primitive-basis of ${\rm d}_{\mathcal{A}}(C)$ if and only if $x\cdot \epsilon_{t,\kappa}^{-1}$ belongs to ${\rm Nrd}_A(\K_1(\mathcal{A})) = \ker(\delta_{\mathcal{A},\mathcal{F}})$.

\begin{remark}{\em Assume $F$ is a finite extension of $\QQ$ and fix a prime ideal $\mathfrak{p}$ in ${\rm Spm}(R)$. Then the above argument also shows that if the image of (\ref{etnc interpret}) under the natural localization map $\K_0(\mathcal{A},\mathcal{A}_\mathcal{F}) \to K_0(\mathcal{A}_{(\mathfrak{p})},\mathcal{A}_\mathcal{F})$ is valid, then $z_{\nu(t),x}$ is a basis of the $\xi(\mathcal{A}_{(\mathfrak{p})})$-module 
${\rm d}_{\mathcal{A}}(C)_{(\mathfrak{p})}^{\rm u} = {\rm d}_{\mathcal{A}_{(\mathfrak{p})}}(C_{(\mathfrak{p})})^{\rm u}$.}\end{remark}

\subsubsection{}In the remainder of the argument we assume that $F$ is a number field, that $A$ is unramified at all archimedean places of $F$ other than the place $v$ corresponding to $\mathcal{F}$ and that ${\rm Spec}(R)$ is open in ${\rm Spec}(\mathcal{O}_F)$.   

To prove claim (iii) we also use the following notation: for each $\mathfrak{p}$ in ${\rm Spm}(R)$ and each isomorphism $j$ in ${\rm Isom}_\mathfrak{p}$ we set $\mathcal{A}_j := \mathcal{A}_{\mathfrak{p}}$, $A_j := A_{\mathfrak{p}}$, $A_j^c := \CC_{p(\mathfrak{p})}\otimes_{F,j}A$ and  $C_j := C_{\mathfrak{p}}$. 

Then, in this case, Proposition \ref{ext bound hom}(v) implies that the equality (\ref{etnc interpret}) is valid as stated if and only if for all $\mathfrak{p}$ in ${\rm Spm}(R)$ and all $j$ in ${\rm Isom}_\mathfrak{p}$ it is valid with $\mathcal{A},\mathcal{F}, x, C$ and $t$  respectively replaced by $\mathcal{A}_j, A_j^c, j_*(x), C_j$ and $t_j := \CC_{p(\mathfrak{p})}\otimes_{\mathcal{F},j}t$. 

In addition, the argument of claim (ii) implies that (\ref{etnc interpret}) is valid for any such collection of data 
 if and only if $z_{\nu(t_j),j_*(x)}$ is a primitive-basis of ${\rm d}_{\mathcal{A}_j}(C_j)$. To deduce claim (iii) it is therefore enough to note $z_{\nu(t_j),j_*(x)}$ is equal to the image of $z_{\nu(t),x}$ under the natural map 
\[ {\rm d}_{A_\mathcal{F}}(C_\mathcal{F}) \to \CC_{p(\mathfrak{p})}\otimes_{\mathcal{F},j}{\rm d}_{A_\mathcal{F}}(C_\mathcal{F}) = \zeta(A_j^c)\otimes_{\xi(\mathcal{A}_j)}{\rm d}_{\mathcal{A}_j}(C_j).\]

To prove claim (iv) we note Theorem \ref{prim criterion}(iii) implies that ${\rm d}_{\mathcal{A}}(C)$ has a primitive-basis if and only if $\chi_\mathcal{A}(C)$ vanishes. In addition, since Remark \ref{remark euler} implies
\[ \chi_\mathcal{A}(C) =  \partial_{\mathcal{A},\mathcal{F}}'(\chi_{\mathcal{A},\mathcal{F}}(C,t)) = \partial_{\mathcal{A},\mathcal{F}}'(\delta_{\mathcal{A},\mathcal{F}}(x)),\]
the result of Proposition \ref{ext bound hom}(i) combines with the exactness of the lower row of (\ref{E:kcomm}) to imply $\chi_\mathcal{A}(C)$ vanishes if and only if there exists an element $x'$ of ${\rm Nrd}_{A_\mathcal{F}}(\K_1(A_\mathcal{F}))$ such that
 $x\cdot x'$ belongs to $\ker(\delta_{\mathcal{A},\mathcal{F}})$.  The result of claim (iv) therefore follows directly from the description of $\ker(\delta_{\mathcal{A},\mathcal{F}})$ given in Proposition \ref{ext bound hom}(ii). 

Finally, to prove claim (v) we abbreviate $z_{\nu(t),x}$ to $z$. We first assume $z$ is a primitive-basis of ${\rm d}_{\mathcal{A}}(C)$. In this
 case we can assume  $C$ belongs to $\DC^{\rm f}(\mathcal{A})$ and hence that $z$ arises from a choice of (ordered) $\mathcal{A}$-bases of each module $C^a$.  Then, just as in the proof of claim (i), this choice of bases determines an isomorphism
  $\kappa: [C] \cong {\bf 1}_{\mathcal{V}(\mathcal{A})}$ in $\mathcal{V}(\mathcal{A})$ with the property that $\nu(\kappa)$ sends $z$ to the element $(1,0)$ of $(\zeta(A_\mathcal{F}),0)$. The definition of $z$ therefore implies that
\begin{align*} (x,0) = &\, \nu(t)(z)\\
 = &\, \nu(t\circ \kappa^{-1})(\nu(\kappa)(z))\\
  = &\, {\rm Nrd}_{A_\mathcal{F}}(t\circ \kappa^{-1})\cdot (1,0)\\
   = &\, ({\rm Nrd}_{A_\mathcal{F}}(t\circ \kappa^{-1}),0)\end{align*}
and hence that $x$ belongs to ${\rm Nrd}_{A_\mathcal{F}}(\K_1(A_\mathcal{F}))$, as required.

%

To prove the converse we assume $x$ belongs to ${\rm Nrd}_{A_\mathcal{F}}(\K_1(A_\mathcal{F}))$.
 In this case, the result of claim (iv) (with $x'$ taken to be $x^{-1}$) implies that ${\rm d}_{\mathcal{A}}(C)$ has a 
 primitive-basis $z'$.

Since $z$ is assumed to be a locally-primitive basis of ${\rm d}_{\mathcal{A}}(C)$, Proposition \ref{prim criterion 2}(ii) therefore implies that $z= y\cdot z'$ for some element $y$ of $\xi(\mathcal{A})^\times$ that belongs to ${\rm Nrd}_A(\K_1(\mathcal{A}_{\mathfrak{p}}))$ for all $\mathfrak{p}$ in ${\rm Spm}(R)$ and further that $z$ is a primitive-basis of ${\rm d}_{\mathcal{A}}(C)$ if and only if $y$ belongs to ${\rm Nrd}_A(\K_1(\mathcal{A}))$. It therefore follows from Lemma \ref{last tech}(ii) that $z$ is a primitive-basis of ${\rm d}_{\mathcal{A}}(C)$ if $y$ belongs to ${\rm Nrd}_{A}(\K_1(A))$.

In addition, the same argument as used above shows that the image under $\nu(t)$ of the primitive-basis $z'$ is equal to $(u,0)$ for some $u$ in ${\rm Nrd}_{A_\mathcal{F}}(\K_1(A_\mathcal{F}))$ and hence that

\begin{align*} (y,0) = &\, \nu(t)(z)\cdot
\nu(t)(z')^{-1}\\
 = &\, x\cdot \nu(t)(z')^{-1}\\
  = &\, (x\cdot u^{-1},0).\end{align*}
In particular, because $x$ is now assumed to belong to ${\rm Nrd}_{A_\mathcal{F}}(\K_1(A_\mathcal{F}))$, it follows that the element $y = x\cdot u^{-1}$ belongs to both $\zeta(A)^\times$ and ${\rm Nrd}_{A_\mathcal{F}}(\K_1(A_\mathcal{F}))$. 

To deduce that $y$ belongs to ${\rm Nrd}_{A}(\K_1(A))$, and hence complete the proof of claim (v), it is thus enough to note that, under the present hypotheses, the Hasse-Schilling-Maass Norm Theorem implies (via the argument at the end of the proof of Lemma \ref{injectivity result}) that $\zeta(A)^\times \cap {\rm Nrd}_{A_\mathcal{F}}(\K_1(A_\mathcal{F})) = {\rm Nrd}_{A}(\K_1(A))$. 
%
%

This then completes the proof of Theorem \ref{ltc2}.

\section*{\large{Part II: Arithmetic over non-abelian Galois extensions}}

In the remainder of the article we shall make some technical improvements to the theory of non-commutative Euler systems introduced by the present authors  in \cite{bses1} and then combine these strengthened results with the $K$-theoretic techniques  developed in \S\ref{plp section} and \S\ref{rkt section} to improve aspects of the theory of leading term conjectures over arbitrary Galois extensions. 

In particular, in this way we shall formulate both a natural main conjecture of higher-rank non-commutative Iwasawa theory for $\mathbb{G}_m$ over arbitrary number fields and a precise `derivative formula' for the `non-commutative Rubin-Stark Euler system' that generalizes to arbitrary Galois extensions of number fields the classical Gross-Stark Conjecture. 

We shall also obtain strong evidence in support of both of these conjectures in important special cases and  establish a precise link between them and  the equivariant Tamagawa Number Conjecture for $\mathbb{G}_m$ over arbitrary Galois extensions, thereby generalising the main result of Kurihara and the present authors in \cite{bks2}.

\section{Integral arithmetic cohomology and Selmer modules}\label{hnase section}


%


As a convenience for the reader, in this section we shall first recall some basic facts about the arithmetic modules and complexes  that will play a key role in our theory. 

Throughout, we fix a finite Galois extension $L/K$ of global fields. 

\subsection{Selmer modules}

\subsubsection{}For a finite set of places $\Pi$ of $K$
 and an extension $E$ of $K$ we write $\Pi_E$ for the
set of places of $E$ lying above those in $\Pi$, $Y_{E,\Pi}$ for the free abelian group on the set $\Pi_E$ and $X_{E,\Pi}$ for the submodule of $Y_{E,\Pi}$ comprising elements whose coefficients sum to zero.

For each place $v$ of $K$ we fix a place $w_v$ of $L$ above $v$ and for each intermediate field $E$ of $L/K$ we write $w_{v,E}$ for the restriction of $w_v$ to $E$. For each non-archimedean place $w$ of $E$ we write $\kappa_w$ for its residue field and ${\rm N}w$ for its absolute norm.

We write $S^\infty_K$ for the set of archimedean places of $K$ (so that $S^\infty_K = \emptyset$ unless $K$ is a number field). For an extension $E$ of $K$ we write $S_{\rm ram}(E/K)$ for the set of primes of $K$ that ramify in $E$. 

If $\Pi$ is non-empty and (in the number field case) contains  $S^\infty_K$, then we write $\co_{E,\Pi}$ for the subring of
$E$ comprising elements integral at all places outside $\Pi_E$ and $\co_{E,\Pi}^\times$ for the unit group of $\co_{E,\Pi}$. (If $\Pi = S^\infty_K$, then we abbreviate $\co_{E,\Pi}$ to $\co_E$.)

In this case, for any
finite set of places $\Pi'$ of $K$ that is disjoint from
 $\Pi$, we write $\co^\times_{E,\Pi,\Pi'}$ for the (finite
index) subgroup of $\co_{E,\Pi}^\times$ consisting of those elements
congruent to $1$ modulo all places in $\Pi'_E$. In addition, we write ${\rm Cl}_\Pi^{\Pi'}(E)$ for the ray class group of $\mathcal{O}_{E,\Pi}$ modulo ${\prod}_{w\in \Pi'_E}w$ (that is, the quotient of the group of fractional $\mathcal{O}_E$-ideals
whose supports are coprime to all places in $(\Pi\cup \Pi')_E$ by the subgroup
of principal ideals with a generator congruent to $1$ modulo all
places in $\Pi'_{E}$).

If $E/K$ is Galois, we set $G_{E/K} := \Gal(E/K)$ and note each of the groups $Y_{E,\Pi}, X_{E,\Pi},$ $ \co_{E,\Pi}^\times$, $\co^\times_{E,\Pi,\Pi'}$ and ${\rm Cl}_{\Pi}^{\Pi'}(E)$ has a natural action of $G_{E/K}$. In this case, for a non-archimedean place $v$ of $K$ we also fix a lift $\Fr_{v}$ to $G_{E/K}$ of the Frobenius automorphism of $w_{v,E}$.

\subsubsection{}If $\Pi$ contains $S_K^\infty$, then the `($\Pi$-relative $\Pi'$-trivialized) integral dual Selmer group for $\GG_m$ over $E$' is defined in \cite[Def. 2.1]{bks} (where the  notation $\mathcal{S}_{\Pi,\Pi'}(\GG_m/E)$ is used) by setting
\begin{eqnarray}
{\rm Sel}_{\Pi}^{\Pi'}(E) :=
{\rm cokernel}({\prod}_{w \notin (\Pi\cup \Pi')_E}\ZZ \longrightarrow \Hom_\ZZ(E_{\Pi'}^\times,\ZZ) ), \nonumber
\end{eqnarray}
where $E_{\Pi'}^\times$ is the group $\{ a\in E^\times : {\rm ord}_w(a-1)>0 \text{ for all } w\in \Pi'_E \}$ and the arrow denotes the homomorphism that sends $(x_w)_w$ to the map $(a \mapsto {\sum}_{w\notin (\Pi\cup\Pi')_E}{\rm ord}_w(a)x_w)$.


We recall from loc. cit. that there exists a canonical exact sequence
\begin{equation}\label{selmer lemma seq} 0 \to {\rm Cl}_{\Pi}^{\Pi'}(E)^\vee \to {\rm Sel}_{\Pi}^{\Pi'}(E) \to \Hom_{\ZZ}(\mathcal{O}^\times_{E,\Pi,\Pi'},\ZZ)\to 0,\end{equation}
and a canonical transpose ${\rm Sel}_{\Pi}^{\Pi'}(E)^{\rm tr}$ to ${\rm Sel}_{\Pi}^{\Pi'}(E)$ (in the sense of Jannsen's homotopy theory of modules \cite{jannsen}) that lies in a canonical exact sequence
\begin{equation}\label{selmer lemma seq2} 0 \longrightarrow {\rm Cl}_{\Pi}^{\Pi'}(E) \longrightarrow
{\rm Sel}_{\Pi}^{\Pi'}(E)^{{\rm tr}}\xrightarrow{\varrho_{E,\Pi}} X_{E,\Pi} \longrightarrow 0.\end{equation}

\subsection{Modified \'etale cohomology complexes}
%
%

We set $G:= G_{L/K}$ and write $\Der(\ZZ[G])$ for the derived category of $G$-modules. We also write
$\Der^{{\rm lf},0}(\ZZ[G])$ for its full triangulated subcategory comprising complexes isomorphic to a bounded complex of finitely generated locally-free $G$-modules $C$ with the property that the Euler characteristic of $\QQ\otimes_\ZZ C$ in $\K_0(\QQ[G])$ vanishes.

The complexes that are used in the next result are described in terms of the complexes $\DR\Gamma_{c,\Pi'}((\mathcal{O}_{L,\Pi})_{\mathcal{W}}, \ZZ)$ introduced by Kurihara and the current authors in \cite[Prop. 2.4]{bks}. We recall, in particular, that the latter complexes can be naturally interpreted in terms of the Weil-\'etale cohomology theory that Lichtenbaum has constructed for global function fields \cite{lichtenbaum} and conjectured to exist for number fields \cite{lichtenbaum1} (see \cite[Rem. 2.5]{bks} for more details).

\begin{lemma}\label{complex construction} Let $\Pi$ be a finite non-empty set of places of $K$ containing $S_K^\infty \cup S_{\rm ram}(L/K)$, and let $\Pi'$ be a finite set of places of $K$ that is disjoint from $\Pi$.  Then  the complex
\[ C_{L,\Pi,\Pi'} := {\DR}\Hom_\ZZ(\DR\Gamma_{c,\Pi'}((\mathcal{O}_{L,\Pi})_{\mathcal{W}}, \ZZ),\ZZ)[-2]\]
belongs to $\Der^{{\rm lf},0}(\ZZ[G])$ and has the following properties.

\begin{itemize}
\item[(i)] $C_{L,\Pi,\Pi'}$ is acyclic outside degrees zero and one and there are canonical identifications $H^0(C_{L,\Pi,\Pi'})$ and $H^1(C_{L,\Pi,\Pi'})$ with $\mathcal{O}_{L,\Pi,\Pi'}^\times$ and ${\rm Sel}_{\Pi}^{\Pi'}(L)^{{\rm tr}}$ respectively.
\item[(ii)] If $\Pi_1$ is any finite set of places of $K$ that contains $\Pi$ and is disjoint from $\Pi'$, then there is a canonical exact triangle in $\Der^{{\rm lf},0}(\ZZ[G])$ of the form
$$C_{L,\Pi,\Pi'}\to C_{L,\Pi_1,\Pi'}\to\bigl({\bigoplus}_{w\in(\Pi_1\setminus \Pi)_L}\DR\Hom_\ZZ(\DR\Gamma((\kappa_w)_{\mathcal{W}},\ZZ),\ZZ)\bigr)[-1]\to C_{L,\Pi,\Pi'}[1],$$
where each complex $\DR\Gamma((\kappa_w)_{\mathcal{W}},\ZZ)$ is as defined in \cite[Prop. 2.4(ii)]{bks}.
\item[(iii)] If $\Pi'_1$ is any finite set of places of $K$ that contains $\Pi'$ and is disjoint from $\Pi$, then there is a canonical exact triangle in $\Der^{{\rm lf},0}(\ZZ[G])$ of the form
$$C_{L,\Pi,\Pi_1'}\to C_{L,\Pi,\Pi'}\to\bigl({\bigoplus}_{w\in(\Pi'_1\setminus \Pi')_L}\kappa_w^\times\bigr)[0]\to C_{L,\Pi,\Pi'_1}[1].$$
\item[(iv)] For any normal subgroup $H$ of $G$ there is a canonical `projection formula' isomorphism in $\Der^{{\rm lf},0}(\ZZ[G/H])$ 
\[ \ZZ[G/H]\otimes_{\ZZ[G]}^{\DL}C_{L,\Pi,\Pi'} \cong C_{L^H,\Pi,\Pi'},\]
and hence also a canonical isomorphism of $\ZZ[G/H]$-modules
\[ \ZZ[G/H]\otimes_{\ZZ[G]}{\rm Sel}_{\Pi}^{\Pi'}(L)^{\rm tr} \cong {\rm Sel}_{\Pi}^{\Pi'}(L^H)^{\rm tr}.\]
\item[(v)] If $\Pi$ contains every $p$-adic place of $K$, then there exists a canonical exact triangle in $\Der(\ZZ_p[G])$ of the form 
$$C_{L,\Pi,\Pi',p} \to {\DR}\Hom_{\ZZ_p}(\DR\Gamma_{c}(\mathcal{O}_{L,\Pi}, \ZZ_p),\ZZ_p)[-2] \to\bigl(\ZZ_p\otimes_\ZZ{\bigoplus}_{w\in \Pi'_L}\kappa_w^\times\bigr)[0]\to C_{L,\Pi,\Pi',p}[1]$$
in which $\DR\Gamma_{c}(\mathcal{O}_{L,\Pi}, \ZZ_p)$ denotes the compactly-supported $p$-adic cohomology of $\ZZ_p$ over the scheme ${\rm Spec}(\mathcal{O}_{L,\Pi})$.\end{itemize}
\end{lemma}

\begin{proof} 
The descriptions in claim (i) follow directly from \cite[Def. 2.6 and Rem. 2.7]{bks}. In addition, since $\Pi$ is assumed to contain all places which ramify in $L/K$, the fact that $C_{L,\Pi,\Pi'}$ belongs to $\Der^{{\rm lf},0}(\ZZ[G])$ follows from the argument used to prove \cite[Lem. 2.8]{bks}.

The canonical exact triangle in claim (ii), resp. (iii), results directly from applying the functor $C\mapsto \DR\Hom_\ZZ(C,\ZZ)[-2]$ to the triangle given by the right-hand column of the diagram in claim (i), resp. the exact triangle in claim (ii), of Proposition 2.4 in loc. cit.

The first displayed isomorphism in claim (iv) follows by combining the construction of $C_{L,\Pi,\Pi'}$ in \cite{bks} with the canonical projection formula isomorphism in \'etale cohomology 
\[ \ZZ[G/H]\otimes_{\ZZ[G]}^\DL \DR\Gamma_c((\mathcal{O}_{L,\Pi})_{{\rm\acute e t}},\ZZ) \cong \DR\Gamma_c((\mathcal{O}_{L^H,\Pi})_{{\rm\acute e t}},\ZZ).\]
The claimed isomorphism of $\ZZ[G/H]$-modules then follows directly from this isomorphism and the explicit description of cohomology groups given in claim (i). 

Lastly we note that the existence of a canonical triangle as in claim (v) follows from the discussion in \cite[\S2.2]{bks2}.
\end{proof}


\begin{remark}\label{exp EF}{\em If, in the setting of Lemma \ref{complex construction}(ii), $v$ is any place in $\Pi_1\setminus \Pi$, then the direct sum of $\DR\Hom_\ZZ(\DR\Gamma((\kappa_w)_{\mathcal{W}},\ZZ),\ZZ)[-1]$ over places $w$ of $L$ above $v$ is a complex of (left) $G$-modules that identifies with 
\[ \ZZ[G]\xrightarrow{x\mapsto x(1-{\rm Fr}_{v}^{-1})}\ZZ[G],\]
where the first term is placed in degree zero and ${\rm Fr}_v$ is the Frobenius automorphism in $G$ of some fixed place of $L$ above $v$.}\end{remark}

\begin{remark}\label{non-tranpose remark}{\em  If the group $\mathcal{O}_{L,\Pi,\Pi'}^{\times}$ is torsion-free, then Lemma \ref{complex construction}(i) implies that the complex $C_{L,\Pi,\Pi'}^\ast := \DR\Gamma_{c,\Pi'}((\mathcal{O}_{L,\Pi})_{\mathcal{W}}, \ZZ)$ is acyclic outside degrees one and two. Since $H^2(C_{L,\Pi,\Pi'}^\ast)$ identifies with ${\rm Sel}_{\Pi}^{\Pi'}(L)$ (by \cite[Prop. 2.4(iii)]{bks}), a similar argument to that in Lemma \ref{complex construction}(iv) implies the existence in this case, for any normal subgroup $H$ of $G$, of a canonical isomorphism of $\ZZ[G/H]$-modules $\ZZ[G/H]\otimes_{\ZZ[G]}{\rm Sel}_{\Pi}^{\Pi'}(L) \cong {\rm Sel}_\Pi^{\Pi'}(L^H).$}
\end{remark}

\section{Non-commutative Euler systems for $\mathbb{G}_m$}\label{res section}

\subsection{Hypotheses and definitions}\label{hyp def section}

In this section we fix a number field $K$, with algebraic closure $K^c$, and set $G_K:={\rm Gal}(K^c/K)$.

We write ${\rm Ir}(K)$ for the set of irreducible $\QQ^c$-valued characters of $G_K$ that have open kernel. 

For each character $\chi$ in ${\rm Ir}(K)$ we fix an associated (finite-dimensional) representation $V_\chi$ of $G_K$ over $\QQ^c$ and we assume that all reduced exterior powers occurring in the sequel are defined relative to these fixed representations (cf. Remark \ref{group rings}).

\subsubsection{}\label{ESdefinitions}

We fix a Galois extension $\KK$ of $K$ in $K^c$ and a finite set $S$ of places of $K$ with 
\[ S^\infty_K \subseteq S.\]
We write $\Sigma_S(\mathcal{K})$ for the subset of $S$ comprising places that split completely in $\mathcal{K}$ and set 
\[ r_S
 = r_{S,\mathcal{K}} := |\Sigma_S(\mathcal{K})|.\]

We assume that there exists a  prime number $p$ and a (possibly empty) finite set of places $T$ of $K$ that is disjoint from $S$ and such that the following condition is satisfied.

\begin{hypothesis}\label{tf hyp}\
{\em
\begin{itemize}\item[(i)] $\KK$ is unramified at all places of $T$, and
\item[(ii)] no element of $\KK^\times$ of order $p$ is congruent to $1$ modulo all places in $T_\KK$.
\end{itemize}}
\end{hypothesis}

\begin{remark}\label{tf hyp rem}{\em Hypothesis \ref{tf hyp} is widely satisfied: for example, if $\KK^\times$ contains no element of order $p$,  then one can take $T$ to be empty.}\end{remark} 

We write $\Omega(\KK) = \Omega(\mathcal{K}/K)$ for the set of finite {\em ramified} Galois extensions of $K$ in $\mathcal{K}$. 
 For each $F$ in $\Omega(\KK)$ we set 
\[ S(F):=S\cup S_{\rm ram}(F/K) \quad \text{and}\quad \mathcal{G}_F:= G_{F/K}\]
and we identify
 ${\rm Ir}(\G_F)$ with the subset of ${\rm Ir}(K)$ comprising characters that factor through the restriction map $G_K\to\G_F$. 

We write $S_K^{\rm all}$ for the set of all places of $K$ and fix an ordering %
\begin{equation}\label{ordering} S_K^{\rm all} = \{v(i)\}_{i \in \mathbb{N}}\end{equation}
in such a way that 
\begin{equation}\label{ordering 2}\Sigma_S(\mathcal{K}) = \{v(i)\}_{i \in [r_S]}.\end{equation}

In the sequel we use, for each set of places of $K$, the ordering that is induced by (\ref{ordering}). In particular, in all of the (exterior product) constructions that are made in the sequel, we regard the sets $S(F)$ to be ordered in this way.



For each place $v$ of $K$ we fix a place $w_v$ of $\KK$ and, by abuse of notation, also write $w_v=w_{F,v}$ for the restriction of $w_v$ to any field $F$ in $\Omega(\KK)$.
If $v$ is not in $S(F)$ for a given $F$ in $\Omega(\KK)$ then we denote by $\G_{F,v}$ the decomposition subgroup of $\G_F$ relative to $w_v$ and write 
\[ \Fr_v=\Fr_{F,v}\in\G_{F,v}\]
for the Frobenius automorphism relative to $w_v$.

\subsubsection{}The functoriality of reduced exterior powers implies that for $F$ and $F'$ in $\Omega(\KK)$ with $F\subseteq F'$, and any non-negative integer $r$, the norm map
$\N_{F'/F}:(F')^\times \to F^\times$
induces a homomorphism of $\zeta(\QQ[\G_{F'}])$-modules
$$\N^r_{F'/F}:{\bigwedge}_{\QQ[\G_{F'}]}^r (\QQ\cdot\co_{F',S(F')}^\times)\to {\bigwedge}_{\QQ[\G_F]}^r(\QQ\cdot\co_{F,S(F')}^\times).$$
Since $\co^\times_{F',S(F'),T,p}$ is torsion-free (as a consequence of Hypothesis \ref{tf hyp}(ii)), the general result of \cite[Lem. 2.10(ii)]{bses1} implies $\N^r_{F'/F}$ restricts to give a homomorphism of $\xi(\ZZ_p[\G_{F'}])$-modules
\begin{equation}\label{bidual transition} {\bigcap}^r_{\ZZ_p[\G_{F'}]}\co^\times_{F',S(F'),T,p}\to{\bigcap}^r_{\ZZ_p[\G_{F}]}
\co^\times_{F,S(F'),T,p}.\end{equation}
This fact helps motivate the following definition.

\begin{definition}\label{pES}{\em Let $r$ be a non-negative integer. Then a  `pre-Euler system of rank $r$' for $\mathbb{G}_m$ with respect to the data $\KK/K, S$ and $p$ is a family of elements
$$(c_F)_F \in {\prod}_{F \in \Omega(\KK)}\CC_p\otimes_{\QQ}{\bigwedge}_{\QQ[\G_{F}]}^r (\QQ\cdot\co_{F,S(F)}^\times)$$
with the property that for every $F$ and $F'$ in $\Omega(\KK)$ with $F\subset F'$ one has
\begin{equation}\label{distribution}\N^r_{F'/F}(c_{F'})=\left({\prod}_{v\in S(F')\setminus S(F)}{\rm Nrd}_{\QQ[\G_F]}(1-{\rm Fr}_{F,v}^{-1})\right)(c_F)\end{equation}
in $\CC_p\otimes_{\QQ}{\bigwedge}_{\QQ[\G_{F}]}^r (\QQ\cdot\co_{F,S(F')}^\times)$.

An `Euler system of rank $r$' for $\mathbb{G}_m$ with respect to the data $\KK/K, S, T$ and $p$ is a pre-Euler system  $(c_F)_F$ for $\KK/K, S$ and $p$ with the additional property that 
\[ c_F\in {\bigcap}_{\ZZ_p[\G_F]}^r\co_{F,S(F),T,p}^\times\]
for every $F$ in $\Omega(\KK)$.

We write ${\rm pES}_r(\KK/K,S,p)$ and ${\rm ES}_r(\KK/K,S,T,p)$ for the respective collections of all such pre-Euler systems and Euler systems.
}\end{definition}


It is clear that ${\rm ES}_r(\KK/K,S,T,p)$ is an abelian group that is endowed with a natural action of the algebra
$$\xi_p(\KK/K) :={\varprojlim}_{F\in\Omega(\KK)}\xi(\ZZ_p[\G_F]),$$
where the transition morphisms are induced by the projection maps $\ZZ_p[\G_{F'}]\to\ZZ_p[\G_F]$ for $F\subseteq F'$ (and are surjective by \cite[Lem. 3.2(v)]{bses}). 

\begin{remark}\label{new ES}{\em There are useful relations between different modules of (pre-)Euler systems 
 of any given rank. \

\noindent{}(i) If $\KK'$ is a Galois extension of $K$ in $\KK$, then the  restriction of a system to the subset $\Omega(\KK')$ of $\Omega(\KK)$ defines a homomorphism of $\xi_p(\KK/K)$-modules
\[ {\rm ES}_r(\KK/K,S,T,p)\to
{\rm ES}_{r}(\KK'/K,S,T,p).\]
We refer to the image of $\varepsilon$ under this homomorphism as  
the `restriction of $\varepsilon$ to $\KK'$'.

\noindent{}(ii) Let $v$ be a (non-archimedean) place of $K$ outside $S\cup T$ and $\sigma$ an element of $G_K$ that acts as the inverse of the Frobenius automorphism of a place above $v$ on every $F$ in $\Omega(\KK)$ in which $v$ is unramified. Then there exists a homomorphism of $\xi_p(\KK/K)$-modules
\[ {\rm ES}_r(\KK/K,S,T,p)\to
{\rm ES}_{r}(\KK/K,S\cup \{v\},T,p)\]
that sends each $\varepsilon$ to the system $\varepsilon_\sigma = (\varepsilon_{\sigma,F})_F$ specified at each $F$ in $\Omega(\KK)$ by
\[ \varepsilon_{\sigma,F} := \begin{cases} ({\rm Nrd}_{\QQ[\G_F]}(1-\sigma))(\varepsilon_F), &\text{ if $v$ is unramified in $F$,}\\
 \varepsilon_F, &\text{otherwise.}\end{cases}\]
In such a case we say that the system $\varepsilon$ is a `refinement' of the system $\varepsilon_\sigma$.
}\end{remark}

\subsection{Euler systems and $L$-series} \label{RSsystem}

In this section we define a canonical family of pre-Euler systems in terms of the leading terms at zero of Artin $L$-series  and discuss conditions under which this family comprises Euler systems.   

\subsubsection{}\label{MRSstatement}  
%
%

We must first discuss some necessary preliminaries in the general setting of \S\ref{hnase section}. 

For this we fix a finite Galois extension $L/K$ of global fields of group $G$, a finite non-empty set of places $\Pi$ of $K$ containing $S_K^\infty\cup S_{\rm ram}(L/K)$ and a finite set $\Pi'$ of places of $K$ that is disjoint from $\Pi$. For each intermediate field $F$ of $L/K$ we then set  
%
%
\[ \Sigma(F) = \Sigma_{\Pi}(F) := \{v\in \Pi: v \,\,\text{splits completely in}\,\, F\}.\]
%

%
%
%


%
%

We write
\begin{equation}\label{lambdadef} R_{L,\Pi} : \br \cdot \mathcal{O}^\times_{L,\Pi} \rightarrow \br\cdot X_{L,\Pi}\end{equation}
for the isomorphism of $\br[G]$-modules that, for every $u$ in $\mathcal{O}^\times_{L,\Pi}$, satisfies 
\[ R_{L,\Pi}(u)  = -{\sum}_{w\in \Pi_L}\mathrm{log}|u|_w \cdot
w,\]
where $|\cdot |_w$ denotes the absolute value at
$w$ (normalized as in \cite[Chap. 0, 0.2]{tate}).
 
 Then, for each
%
 non-negative integer $a$, the map $R_{L,\Pi}$ induces an isomorphism of $\zeta(\RR[G])$-modules
%
\[ \lambda^a_{L,\Pi}: {{\bigwedge}}_{\RR[G]}^a (\RR\cdot \mathcal{O}_{L,\Pi}^\times) \stackrel{\sim}{\to} {{\bigwedge}}_{\RR[G]}^a (\RR\cdot X_{L,\Pi}).\]
%
%

For each such $a$, the `$a$-th derived Stickelberger function' of the data $L/K,\Pi$ and $\Pi'$ is defined to be the $\zeta(\CC[G])$-valued meromorphic function of a complex variable $z$
\[ \theta_{L/K,\Pi,\Pi'}^{a}(z) := {\sum}_{\chi\in {\rm Ir}(G)} (z^{-a\chi(1)}L_{\Pi,\Pi'}(\check\chi,z))\cdot e_\chi,\]
where $L_{\Pi,\Pi'}(\check\chi,z)$ is the $\Pi$-truncated $\Pi'$-modified Artin $L$-function for the contragredient $\check\chi$ of $\chi$ and we use the primitive central idempotent 
\[ e_\chi := \chi(1)|G|^{-1}{\sum}_{g \in G}\chi(g)g^{-1}\]
of $\CC[G]$. In the case $a=0$, we set 
\begin{equation}\label{0 stick def} \theta_{L/K,\Pi,\Pi'}(z) := \theta^0_{L/K,\Pi,\Pi'}(z)\end{equation}
and refer to this function as the `Stickelberger function' of $L/K, \Pi$ and $\Pi'$. 

An explicit analysis of the functional equation of Artin $L$-functions (as in \cite[Chap. I, Prop. 3.4]{tate}) shows that 
for each $\chi$ in ${\rm Ir}(G)$ one has 
\begin{equation}\label{tate fe} {\rm ord}_{z=0}L_{\Pi}(\chi,z) = 
\chi(1)^{-1}\cdot {\rm dim}_{\CC}\bigl(e_\chi (\CC\cdot X_{L,\Pi})\bigr).\end{equation}
This formula implies, in particular, that if $\Pi$ contains a proper subset of $a$ elements that split completely in $L$, then the function $\theta_{L/K,\Pi,\Pi'}^{a}(z)$ is holomorphic at $z=0$ and it is then easily checked that its value $\theta_{L/K,\Pi,\Pi'}^{a}(0)$ at $z=0$ belongs to the subring $\zeta(\RR[G])$ of $\zeta(\CC[G])$. 

This observation allows us to make the following definition.


\begin{definition}\label{rubin-stark context} {\em  Let $\Sigma$ be a subset of $\Sigma(L)$ with $\Sigma\not= \Pi$ and fix $v' \in \Pi\setminus \Sigma$. Then the `(non-commutative) Rubin-Stark element' associated to $L/K, \Pi, \Pi'$ and 
$\Sigma$ is the unique element $\varepsilon_{L/K,\Pi,\Pi'}^{\Sigma}$ of ${{\bigwedge}}_{\RR[G]}^{|\Sigma|}(\RR\cdot \mathcal{O}_{L,\Pi}^\times)$ that satisfies 
\[\lambda_{L,\Pi}^{|\Sigma|}(\varepsilon_{L/K,\Pi,\Pi'}^{\Sigma}) = \theta_{L/K,\Pi,\Pi'}^{|\Sigma|}(0) \cdot \wedge_{v \in \Sigma}(w_v-w_{v'}),\]
where the exterior product is defined with respect to the ordering of $\Sigma$ induced by (\ref{ordering}). 
}\end{definition}

\begin{remark}\label{ind v_0}\  {\em

\noindent{}(i) The condition $\Sigma \not= \Pi$ is automatically satisfied if $\Sigma(L) \not= \Pi$ and hence, for example, if $L/K$ is ramified. 

\noindent{}(ii) Since every place $v$ in $\Sigma$ splits completely in $L$, the elements $\{ w_v - w_{v'}\}_{v \in \Sigma}$ span a free $G$-module of rank $|\Sigma|$. Given this, the explicit definition (via (\ref{non comm ext power})) of reduced exterior powers implies that 
 the element $e_\chi (\wedge_{v \in \Sigma}(w_v-w_{v'}))$ is non-zero  
for every $\chi$ in ${\rm Ir}(G)$. In particular, since $\lambda_{L,\Pi}^{|\Sigma|}$ is injective, the equality defining   $\varepsilon_{L/K,\Pi,\Pi'}^{\Sigma}$ implies, for each $\chi$, that 
\[ e_\chi(\varepsilon_{L/K,\Pi,\Pi'}^{\Sigma}) \not= 0\Longleftrightarrow e_\chi\cdot \theta_{L/K,\Pi,\Pi'}^{|\Sigma|}(0)\not= 0.\]

\noindent{}(iii) If $\Sigma \not=\Sigma(L)$, then $|\Sigma|< |\Sigma(L)|$ and, in this case, (\ref{tate fe}) combines with the observation in (ii) to imply that either $\varepsilon_{L/K,\Pi,\Pi'}^{\Sigma}$ vanishes or both $L=K$ (so that $\Sigma(L) = \Pi$) and $|\Sigma| = |\Sigma(L)|-1$. By a similar argument one checks that, if $\Sigma(L)\not= \Pi$ (so that $L\not= K$), then $\theta_{L/K,\Pi,\Pi'}^{|\Sigma(L)|}(0)\cdot (w_{v'_1}-w_{v'}) = 0$ for every $v_1'$ in $\Pi\setminus \Sigma(L)$ and so the element 
\begin{equation}\label{gen RS def} \varepsilon_{L/K,\Pi,\Pi'} := \varepsilon_{L/K,\Pi,\Pi'}^{\Sigma(L)}\end{equation}
depends only on the data $L/K, \Pi$ and $\Pi'$. 
%

\noindent{}(iv) If $\Sigma$ is empty, then the map $\lambda_{L,\Pi}^{|\Sigma|} = \lambda_{L,\Pi}^{0}$ identifies with the identity automorphism of the space 
\[ {{\bigwedge}}_{\CC[G]}^{0}(\CC\cdot \mathcal{O}_{L,\Pi}^\times) = \zeta(\CC[G]) = {{\bigwedge}}_{\CC[G]}^{0}(\CC\cdot X_{L,\Pi})\]
and $\wedge_{v \in \Sigma}(w_v-w_{v'})$ with the element $1$ of $\zeta(\CC[G])$. In this case, therefore,  $\varepsilon_{L/K,\Pi,\Pi'}^{\Sigma}$ coincides with the (non-commutative) `Stickelberger element' 
\[ \theta_{L/K,\Pi,\Pi'}(0) := {\sum}_{\chi\in {\rm Ir}(G)} L_{\Pi,\Pi'}(\check\chi,0)\cdot e_\chi \in \zeta(\QQ[G])\]
that was first studied by Hayes in \cite{hayes}. (We note here that, whilst $\theta_{L/K,\Pi,\Pi'}(0)$ belongs, a priori, only to $\zeta(\CC[G])$, a classical result of Siegel \cite{siegel} combines with Brauer induction to imply, via \cite[Th. 1.2, p. 70]{tate}, that it belongs to $\zeta(\QQ[G])$, as indicated above.)
}\end{remark}

\begin{remark}\label{bks remark}{\em We now have all of the ingredients that are required to justify the observations in Example \ref{ex}(ii) and Remark \ref{remark etnc}(ii). To do this, we fix data $L/K, \Pi, \Pi'$ as above and use the complex $C = C_{L,\Pi,\Pi'}$ defined in Lemma  \ref{complex construction}. Then $C$ belongs to $\Der^{{\rm lf},0}(\ZZ[G])$ and the isomorphism $R_{L,\Pi}$ from (\ref{lambdadef}) combines with the explicit descriptions of cohomology in Lemma \ref{complex construction}(i) to induce a canonical isomorphism $\lambda: {\rm d}_{\RR[G]}(C_{\RR}) \xrightarrow{\sim} (\zeta(\RR[G]),0)$ in $\mathcal{P}(\zeta(\RR[G]))$. In this setting, the zeta element $z_{\lambda,x}$ associated by Definition \ref{def zeta} to the leading term $x =\theta_{L/K,\Pi,\Pi'}^\ast(0)$ at $z=0$ of $\theta_{L/K,\Pi,\Pi'}(z)$ generalizes to non-abelian extensions the `zeta elements' $z_{L/K,\Pi,\Pi'}$ defined in \cite[Def. 3.5]{bks}. Further, an argument similar to Example \ref{ex}(i) shows Theorem \ref{ltc2}(iii) implies that the `lifted root number conjecture' for $L/K$ of Gruenberg, Ritter and Weiss \cite{grw} is equivalent to asserting that the element $z_{\lambda,x}$ is a locally-primitive basis of ${\rm d}_{\ZZ[G]}(C)$. This latter condition directly extends the leading term conjecture `${\rm LTC}(L/K)$' formulated for abelian $L/K$ in \cite[Conj. 3.6]{bks}. In addition, if $z_{\lambda,x}$ is a locally-primitive basis of ${\rm d}_{\ZZ[G]}(C)$, then Theorem \ref{ltc2}(v) combines with the Hasse-Schilling-Maass Norm Theorem to imply that $z_{\lambda,x}$ is a primitive basis of ${\rm d}_{\ZZ[G]}(C)$ if and only if, for every irreducible complex symplectic character $\chi$ of $G$, the leading term at $z=0$ of $L_\Pi(\chi,z)$ is a strictly positive real number. }
\end{remark}

\subsubsection{}\label{The non-commutative Rubin-Stark Conjecture} 

We next state a precise conjecture concerning the `integral' properties of the Rubin-Stark elements from Definition \ref{rubin-stark context}. 

In particular, in this conjecture we fix data $L/K, \Pi$ and $\Pi'$ as in \S\ref{MRSstatement}. 
We also fix a prime $p$ and
an isomorphism of fields $\CC\cong\CC_p$ and identify  each element of the form $\varepsilon^\Sigma_{L/K,\Pi,\Pi'}$ with its image under the induced embedding 
of $\zeta(\RR[G])$-modules 
\[ {\bigwedge}_{\RR[G]}^{|\Sigma|}(\RR\cdot\co^\times_{L,\Pi}) \to {\bigwedge}_{\CC_p[G]}^{|\Sigma|}(\CC_p\cdot\co^\times_{L,\Pi,p}).\]
\


\begin{conjecture}\label{MRSconjecture}{\em (Non-commutative Rubin-Stark Conjecture) } Let $\Sigma$ be a subset of $\Sigma(L)$ with $\Sigma \neq \Pi$. Then, for every prime $p$ for which the group $\co^\times_{L,\Pi,\Pi',p}$ is torsion-free, one has
\[ \varepsilon_{L/K,\Pi, \Pi'}^{\Sigma} \in {\bigcap}_{\ZZ_p[G]}^{|\Sigma|} \co_{L,\Pi, \Pi',p}^\times.\]
\end{conjecture}

\begin{remarks}\label{signsandintegral}{\em \

%
%
%

\noindent{}(i) If Tate's formulation \cite[Chap. I, Conj. 5.1]{tate} of Stark's principal conjecture is valid for $L/K$ (as is automatically the case if $K$ has positive characteristic), then the validity of Conjecture \ref{MRSconjecture} can be shown to be independent of the choice of isomorphism $j: \CC\cong \CC_p$. For this reason, we do not explicitly indicate the choice of $j$ either in the statement of Conjecture \ref{MRSconjecture} or in the arguments that follow.

\noindent{}(ii) If $\Sigma=\Sigma(L)$ is empty, then ($\Sigma\not= \Pi$ and)   $\varepsilon^\Sigma_{L/K,\Pi,\Pi'} = \theta_{L/K,\Pi,\Pi'}(0)$ (cf. Remark \ref{ind v_0}(iv)) and also $|\Sigma| = 0$ so ${\bigcap}^{|\Sigma|}_{\ZZ_p[G]}\co^\times_{L,\Pi,\Pi',p}=\xi(\ZZ_p[G])$ (by \cite[Th. 4.19(i)]{bses}). In this case, therefore, Conjecture \ref{MRSconjecture} asserts that $\theta_{L/K,\Pi,\Pi'}(0)$ belongs to $\xi(\ZZ_p[G])$ for every prime $p$ for which $\mathcal{O}_{L,\Pi,\Pi',p}^\times$ is torsion-free. Recent results of Ellerbrock and Nickel \cite[Th. 1, Th. 2]{en} provide evidence in support of this prediction. 

\noindent{}(iii) If $\co^\times_{L,\Pi,\Pi'}$ is itself torsion-free, then
Conjecture \ref{MRSconjecture} (for all $p$) combines
with \cite[Th. 4.19(iii)]{bses} to predict that the element $\varepsilon_{L/K,\Pi,\Pi'}$ (from (\ref{gen RS def})) belongs to $
{\bigcap}^{|\Sigma(L)|}_{\ZZ[G]}\co^\times_{L,\Pi,\Pi'}$. This prediction is a natural generalisation to arbitrary Galois extensions of the Rubin-Stark Conjecture that is formulated for abelian extensions in \cite{R}. For this reason, we shall in the sequel refer to Conjecture \ref{MRSconjecture} as the `Rubin-Stark Conjecture' for the data $(L/K,\Pi,\Pi',p)$.
 }\end{remarks}

\subsubsection{}As a final preliminary step, we establish some useful properties of a family of central idempotents of $\QQ[G]$ that will play an important role in the sequel. 

To do this we note that, for any proper subset $\Pi_1$ of the set of places $\Pi$ of $K$ fixed above, one obtains an idempotent of $\zeta(\QQ[G])$ by setting
\begin{equation}\label{key idem def} e_{L/K,\Pi,\Pi_1} := {\sum}_{\chi}e_\chi\end{equation}
where in the sum $\chi$ runs over all characters in ${\rm Ir}(G)$ for which
$e_\chi(\QQ^c\otimes_\ZZ X_{L,\Pi\setminus \Pi_1})$ vanishes.

In the following result, we fix, for each $\chi$ in ${\rm Ir}(G)$, an associated $\QQ^c$-representation $V_\chi$ (cf. the beginning of \S\ref{hyp def section}).

\begin{lemma}\label{idem lemma} Fix $L/K, \Pi$ and $\Pi_1$ as above. Then, for every character $\chi$ in ${\rm Ir}(G)$, 
the following conditions are equivalent.

\begin{itemize}
\item[(i)] $e_\chi\cdot e_{L/K,\Pi,\Pi_1} \not= 0$.
\item[(ii)] The projection map $e_\chi(\QQ^c\otimes_\ZZ X_{L,\Pi}) \to
 e_\chi(\QQ^c\otimes_\ZZ Y_{L,\Pi_1})$ is bijective.
\item[(iii)] The multiplicity of $V_\chi$ in the decompositions of both $\QQ^c\otimes_\ZZ\mathcal{O}^\times_{L,\Pi}$ and $\QQ^c\otimes_\ZZ X_{L,\Pi}$ is equal to ${\sum}_{v \in \Pi_1}{\rm dim}_{\QQ^c}(V_\chi^{G_{v}})$, where $G_v$ is the decomposition subgroup in $G$ of any fixed place $w$ of $L$ above $v$.
\item[(iv)] The order of vanishing of $L_{\Pi}(\chi,z)$ at $z=0$ is equal to 
 ${\sum}_{v \in \Pi_1}{\rm dim}_{\QQ^c}(V_\chi^{G_{v}})$.
\item[(v)] If $\chi$ is non-trivial, then $V_\chi^{G_{v}}$ vanishes for each $v$ in
$\Pi\setminus \Pi_1$. If $\chi$ is trivial, then
 $|\Pi\setminus \Pi_1| = 1$.\end{itemize}
\end{lemma}

\begin{proof} The definition (\ref{key idem def}) of $e_{L/K,\Pi,\Pi_1}$ ensures
that $e_\chi\cdot e_{L/K,\Pi,\Pi_1} \not= 0$ if and only if the space 
 $e_\chi(\QQ^c\otimes_\ZZ X_{L,\Pi\setminus \Pi_1})$ vanishes.
 The equivalence of conditions (i) and (ii) is therefore a consequence of the
 natural exact sequence 
 \[ 0 \to X_{L,\Pi\setminus \Pi_1} \to X_{L,\Pi} \to Y_{L,\Pi_1} \to 0.\]

Conditions (ii) and (iii) are equivalent since $X_{L,\Pi}$ and $\mathcal{O}_{L,\Pi}^\times$ span isomorphic $\QQ^c[G]$-modules  and for each place $v$ in $\Pi_1$ the space
  $e_\chi(\QQ^c\otimes_\ZZ Y_{L,\{v\}})$ is isomorphic to the
  $G_{v}$-invariants of $V_\chi$.

The latter fact also combines with the explicit formula (\ref{tate fe}) to imply  equivalence of the conditions (ii) and (iv) and with the exact sequence 
\[ 0 \to X_{L,\Pi\setminus \Pi_1} \to Y_{L,\Pi\setminus \Pi_1} \to \ZZ \to 0\]
to imply equivalence of the conditions (ii) and (v).
  \end{proof}

\begin{remark}\label{rs stable} {\em If the set $\Pi_1$ in  Lemma \ref{idem lemma} is the subset $\Sigma$ of $\Sigma(L)$ that occurs in Definition \ref{rubin-stark context}, then $G_v$ is the trivial group for each $v$ in $\Pi_1 = \Sigma$ and so, for every $\chi$ in ${\rm Ir}(G)$, one has 
\[ {\sum}_{v \in \Pi_1}{\rm dim}_{\QQ^c}(V_\chi^{G_{v}}) = {\sum}_{v \in \Sigma}{\rm dim}_{\QQ^c}(V_\chi) = |\Sigma|\cdot \chi(1).\]
In this case, therefore, the equivalence of conditions (i) and (iv) in the above result implies, for any auxiliary set of places  $\Pi'$ as in \S\ref{MRSstatement}, that 
\begin{align*} e_\chi\cdot e_{L/K,\Pi,\Sigma} \not= 0 \Longleftrightarrow&\, e_\chi\cdot \theta_{L/K,\Pi,\Pi'}^{|\Sigma|}(0) \not= 0\\
\Longleftrightarrow&\, e_\chi (\varepsilon_{L/K,\Pi,\Pi'}^{\Sigma}), \end{align*}
where the second equivalence follows from Remark \ref{ind v_0}(ii). In the space ${{\bigwedge}}_{\RR[G]}^{|\Sigma|}(\RR\cdot \mathcal{O}_{L,\Pi}^\times)$, one therefore has an equality
\[ \varepsilon_{L/K,\Pi,\Pi'}^{\Sigma} = e_{L/K,\Pi,\Sigma}(\varepsilon_{L/K,\Pi,\Pi'}^{\Sigma}). \]
}\end{remark}

\subsubsection{} We now return to the setting of \S\ref{ESdefinitions}. 

For every $F$ in $\Omega(\KK)$ the set $\Sigma_S(\mathcal{K})$ is a proper subset of $S(F)$  (since $S_{\rm ram}(F/K)\not=\emptyset$) and so Remark \ref{ind v_0}(iii) implies that the non-commutative Rubin-Stark element
\begin{equation}\label{lazy RS def} \varepsilon_{F,S,T} := \varepsilon^{\Sigma_S(\mathcal{K})}_{F/K,S(F),T}\end{equation}
depends only on the data $F/K, S, T$ and $\KK$.


\begin{lemma}\label{ncRS} The collection
\[ \varepsilon_{\KK,S,T}^{\rm RS} := (\varepsilon_{F,S,T})_{F\in\Omega(\KK)}\]
has the following properties. 
\begin{itemize}
\item[(i)] $\varepsilon_{\KK,S,T}^{\rm RS}$ belongs to ${\rm pES}_{r_S}(\KK/K,S,p)$. 

\item[(ii)] If, for every $F$ in $\Omega(\mathcal{K})$, Conjecture \ref{MRSconjecture} is valid for the data $F/K$, $S(F)$, $T$ and $p$, then $\varepsilon_{\KK,S,T}^{\rm RS}$  belongs to ${\rm ES}_{r_S}(\KK/K,S,T,p)$. 
\end{itemize}
\end{lemma}

\begin{proof} Set $r:= r_S = |\Sigma_S(\KK)|$. Then, if Conjecture \ref{MRSconjecture} is valid for the data $F/K$, $S(F)$, $T$ and $p$ for every $F$ in $\Omega(\KK)$, each element 
\[ (\varepsilon_{\KK,S,T}^{\rm RS})_F = \varepsilon_{F,S,T} = \varepsilon^{\Sigma_S(\mathcal{K})}_{F/K,S(F),T}\]
belongs to ${\bigcap}_{\ZZ_p[\G_F]}^r\co_{F,S(F),T,p}^\times$ and so claim (ii) is a consequence of claim (i). 

To prove claim (i) it suffices to show that the distribution relation (\ref{distribution}) is valid for all pairs $F$ and $F'$ with $F \subset F'$ if we set $c_F := \varepsilon^{\Sigma_S(\mathcal{K})}_{F/K,S(F),T}$. 

For each finite set of places $S'$ of $K$ that contains $S(F)$, and each $\chi$ in ${\rm Ir}(\cG_F)$, there is an equality of functions
\begin{equation}\label{changeS}
        L_{S'}(\check\chi, s) = \bigl({\prod}_{v\in S'\setminus S(F)}{\rm det}
        \bigl(1-{\rm Fr}_{F,v}^{-1}\cdot ({\rm N}v)^{-s}\mid V_\chi\bigr)\bigr)\cdot L_{S(F)}(\check\chi,s).\end{equation}
Taken together with the inflation invariance of Artin $L$-series this implies that 
\[ \pi_{F'/F}\bigl(\theta_{F'/K,S(F'),T}^{r}(0)\bigr) =  \bigl({\prod}_{v\in S(F')\setminus S(F)}{\rm Nrd}_{\QQ[\G_F]}
        \bigl(1-{\rm Fr}_{F,v}^{-1}\bigr)\bigr)\cdot \theta_{F/K,S(F),T}^{r}(0),  \]
where $\pi_{F'/F}$ is the natural projection map $\zeta(\CC[\cG_{F'}]) \to \zeta(\CC[\G_F])$.  
Given the explicit definition (in Definition \ref{rubin-stark context}) of each element $c_F = \varepsilon^{\Sigma_S(\mathcal{K})}_{F/K,S(F),T}$, the validity of (\ref{distribution}) in this case will therefore follow if there exists a commutative diagram of $\zeta(\RR[\cG_{F'}])$-modules 
\[ \begin{CD} {\bigwedge}_{\RR[\G_{F'}]}^r(\RR\cdot\co_{F',S(F')}^\times) @> \lambda_{F',S(F')}^r >> {{\bigwedge}}_{\RR[\G_{F'}]}^r (\RR\cdot X_{F',S(F')})\\
@V \N^r_{F'/F} VV @V \theta VV \\
{\bigwedge}_{\RR[\G_F]}^r(\RR\cdot\co_{F,S(F')}^\times) @> \lambda_{F,S(F')}^r >> {{\bigwedge}}_{\RR[\G_F]}^r (\RR\cdot X_{F,S(F')}),
\end{CD} \]
in which one has 
\[ \theta\bigl(\wedge_{v \in \Sigma_S(\KK)}(w_{v,F'}-w_{v',F'})\bigr) = \wedge_{v \in \Sigma_S(\KK)}(w_{v,F}-w_{v',F})\]
for any choice of place $v'$ in
$S(F)\setminus \Sigma_S(\KK)$. The existence of such a diagram is in turn an easy  consequence of the fact that the following diagram commutes

\[ \begin{CD}  \RR\cdot\co_{F',S(F')}^\times @> R_{F',S(F')} >> \RR\cdot X_{F',S(F')}\\
@A AA @AA A\\
\RR\cdot\co_{F,S(F')}^\times @> R_{F,S(F')} >> \RR\cdot X_{F,S(F')},\end{CD}\]
where the left hand vertical map is the natural inclusion and the right hand vertical map is induced by sending each place $w_{v,F}$ to ${\sum}_{g \in \Gal(F'/F)}g(w_{v,F'})$ (cf. \cite[bottom of p. 29]{tate}).
\end{proof}

\begin{definition}\label{ncrs def}{\em The `Rubin-Stark (non-commutative) Euler system' relative to the data $\mathcal{K}/K, S, T$ and $p$ is the element $\varepsilon_{\KK,S,T}^{\rm RS}$ of ${\rm pES}_{r_S}(\KK/K,S,p)$ described in Lemma \ref{ncRS}.}\end{definition}

\begin{remark}\label{r_S=0 example}{\em \

\noindent{}(i) If $\KK/K$ is abelian, then $\varepsilon_{\KK,S,T}^{\rm RS}$ coincides with the pre-Euler system considered by Rubin in \cite[\S6]{R}.\

\noindent{}(ii) If $\Sigma_S(\KK)$ is empty, then, for every $F$ in $\Omega(\KK)$, Remark \ref{signsandintegral}(ii) implies that 
\[ (\varepsilon_{\KK,S,T}^{\rm RS})_F = \theta_{F/K,S(F),T}(0).\]
In addition, Conjecture \ref{MRSconjecture} predicts that  $\theta_{F/K,S(F),T}(0)$ belongs to $\xi(\ZZ_p[\cG_F])$ for every $p$ for which $\mathcal{O}_{F,S(F),T,p}^\times$ is torsion-free (see Remark \ref{signsandintegral}(ii)).\
}\end{remark} 

\subsection{Euler systems and Galois cohomology}\label{BESsection}

The `reduced determinant' functor constructed in \cite[\S5]{bses}
can be combined with the complexes constructed in Lemma \ref{complex construction} to
give an unconditional construction of non-commutative Euler systems. 

In this section we shall describe this construction and then use it to strengthen one of the main results of
\cite{bses1}.

To do so, we introduce the following convenient notation: in the sequel, for each natural number $d$, we consider the ordered set  
\[ [d] := \{i \in \ZZ: 1\le i\le d\}.\]
For each finite group $\Gamma$, we then write 
$\{b_{\Gamma,i}\}_{i \in [d]}$ for the standard (ordered) $\ZZ[\Gamma]$-basis of $\ZZ[\Gamma]^d$ and $\{b^{\ast}_{\Gamma,i}\}_{i \in [d]}$ for the corresponding dual basis of $\Hom_{\ZZ[\Gamma]}(\ZZ[\Gamma]^d,\ZZ[\Gamma])$. 

In the case that $\Gamma = \G_L$ for some $L$ in $\Omega(\KK_\infty)$ we shall also use the abbreviations
\begin{equation}\label{standard basis notation} b_{L,i} := b_{\G_L,i}\quad \text{ and }\quad b^\ast_{L,i} := b^\ast_{\G_L,i}.\end{equation}

\subsubsection{}We start by proving a version of the construction made in \cite[Lem. 4.8]{bses1} that is relevant to our setting. 

As usual, we endow each set of places of $K$ with the ordering induced by (\ref{ordering}). 

\begin{lemma}\label{can rep lemma} Fix data $L/K, G, \Pi$ and $\Pi'$ as in  \S\ref{MRSstatement}.

Fix a place $v'$ in $\Pi$, set $n := |\Pi\setminus \{v'\}| = |\Pi|-1$ and let $p$ be a prime for which $\mathcal{O}_{L,\Pi,\Pi',p}^\times$ is torsion-free. Then there exists a natural number $d$ with $d \ge n$ and a
canonical family of complexes of $\ZZ_p[G]$-modules $C(\phi)$ of the form 
\[ \ZZ_p[G]^d \xrightarrow{\phi} \ZZ_p[G]^d\]
in which the first term is placed in degree zero and the
following claims are valid. 


\begin{itemize}
\item[(i)] There exists an isomorphism $\kappa: C(\phi) \to C_{L,\Pi,\Pi',p}$ in $\Der^{\rm perf}(\ZZ_p[G])$ with the following  property: for $i \in [d]$, the image of $b_{G,i}$ under the composite map 
\[ \ZZ_p[G]^d \to {\rm cok}(\phi) \xrightarrow{H^1(\kappa)} {\rm Sel}_{\Pi}^{\Pi'}(L)_p^{\rm tr} \xrightarrow{\varrho_{L,\Pi,p}} X_{L,\Pi,p},\]
where $\varrho_{L,\Pi}$ is as in (\ref{selmer lemma seq2}), is equal to $w_{v_i}-w_{v'}$, with $v_i$ the $i$-th element of $\Pi \setminus\{v'\}$, if $i \in [n]$ and is equal to $0$ otherwise.  

\item[(ii)] If $C(\phi')$ is any complex in the family, then $\phi' =
\eta\circ \phi \circ (\eta')^{-1}$ where $\eta$ and $\eta'$ are
automorphisms of $\ZZ_p[G]^d$ and $\eta$ is represented,
with respect to $\{b_{G,i}\}_{i \in [d]}$, by a block matrix 
\begin{equation}\label{block} \left(
\begin{array} {c|c} I_n& \ast\\
                        \hline
                              0 & M_\eta\end{array}\right),\end{equation}
where $I_n$ is the $n\times n$ identity matrix and $M_\eta$ belongs to ${\rm GL}_{d-n}(\ZZ_p[G])$.
\end{itemize}
\end{lemma}

\begin{proof} We first fix a projective cover of $\ZZ_p[G]$-modules $\varpi': P \to \ker(\varrho_{L,\Pi,p})\cong {\rm Cl}_\Pi^{\Pi'}(L)_p$ and a module $P'$ of minimal rank such that $P\oplus P'$ is a free $\ZZ_p[G]$-module. With $d_0$ denoting the rank of $P\oplus P'$, we fix an identification $P\oplus P' = \ZZ_p[G]^{d_0}$ and write $\varpi_1 = (\varpi', 0_{P'})$ for the induced surjective map $\ZZ_p[G]^{d_0}\to\ker(\varrho_{L,\Pi,p})$. We finally set $d := n+ d_0$ and $\mathfrak{S}_{L} := {\rm Sel}_\Pi^{\Pi'}(L)_p^{\rm tr}$ and write
$\varpi: \ZZ_p[G]^d \to \mathfrak{S}_{L}$ for the map of $\ZZ_p[G]$-modules that
sends $b_{G,i}$ to a choice of pre-image of $w_{v_i}-w_{v'}$ under
$\varrho_{L,\Pi,p}$ if $i \in [n]$ and to $\varpi_1(b_{G,i-n})$ if $i \in [d]\setminus [n]$. Then $\varpi$ is surjective and such that 
\begin{equation}\label{resolution} \varrho_{L,\Pi,p}(\varpi(b_{G,i})) = \begin{cases} w_{v_i}-w_{v'}, &\text{if $i \in [n]$}\\
 0, &\text{if $i\in [d]\setminus [n]$.}
\end{cases}\end{equation}

Now, since $C = C_{L,\Pi,\Pi'}$ belongs to $\Der^{{\rm lf},0}(\ZZ[G])$, and the module 
$U_L = H^0(C_p)$ is torsion-free, a standard argument
(as in \cite[Prop. 3.2]{omac}) shows the existence of an isomorphism $\kappa: C(\phi) \to C_p$ in $\Der(\ZZ_p[G])$, where 
 $C(\phi)$ has the required form $\ZZ_p[G]^d\xrightarrow{\phi}\ZZ_p[G]^d$ and $H^1(\kappa)$ is induced by the map $\varpi$ constructed above. In view of (\ref{resolution}) this construction has the properties described in claim (i).

 To verify claim (ii), it is also enough to note that if $\kappa': C(\phi') \to C_p$ is any alternative set of data constructed as above, then the argument of \cite[Prop. 3.2(iv)]{omac} shows the existence of
automorphisms $\eta'$ and $\eta$ of $\ZZ_p[G]^d$ that lie in an
exact commutative diagram of the form 
\begin{equation}\label{very important diagram}\begin{CD}
0 @> >> \mathcal{O}_{L,\Pi,\Pi',p}^\times @> H^0(\kappa)^{-1} >> \ZZ_p[G]^d @> \phi >> \ZZ_p[G]^d 
@> H^1(\kappa)^\dagger >> \mathfrak{S}_{L} @> >> 0\\
@. @\vert @V \eta' VV @VV \eta V @\vert \\
0 @> >> \mathcal{O}_{L,\Pi,\Pi',p}^\times @> H^0(\kappa')^{-1} >> \ZZ_p[G]^d @> \phi' >> \ZZ_p[G]^d @> H^1(\kappa')^\dagger >>
\mathfrak{S}_{L} @> >> 0.\end{CD}\end{equation}
Here $H^1(\kappa)^\dagger$ and $H^1(\kappa')^\dagger$ are the composites of the respective tautological maps $\ZZ_p[G]^d \to {\rm cok}(\phi)$ and $\ZZ_p[G]^d \to {\rm cok}(\phi')$ with $H^1(\kappa)$ and $H^1(\kappa')$, and $\eta$ is represented with respect to the basis $\{b_{G,i}\}_{i \in [d]}$ by a block matrix
 of the required sort (\ref{block}).\end{proof} 

\subsubsection{}We next derive a useful consequence of Lemma \ref{can rep lemma} in the setting of \S\ref{ESdefinitions}.


To state the result we use the following notation: if $\Gamma$ is a finite group and $X$ a finitely generated $\xi(\ZZ_p[\Gamma])$-lattice, then we regard $X$ as a subset of $\QQ_p\otimes_{\ZZ_p}X$ in the natural way and, for each element $a$ of $\zeta(\QQ_p[\Gamma])$, we define a $\xi(\ZZ_p[\Gamma])$-submodule of $X$ by setting
\[ X[a] := \{x \in X: a\cdot x = 0\} = \{x \in X: x = (1-a)\cdot x\}.\]

We also note that, for every $L$ in $\Omega(\KK)$, the set $S(L)\setminus \Sigma_S(L)$ is non-empty (since $L/K$ is assumed to be ramified). 
%
%

\begin{proposition}\label{generate rubin} Fix $L$ in $\Omega(\KK)$, a place $v'$ in $S(L)\setminus \Sigma_S(L)$ and a subset $\Sigma$ of $\Sigma_S(L)$ of cardinality $a$ (so that $a \le |\Sigma_S(L)| < |S(L)|$). Set $G := \G_L$ and write $e$ for the idempotent $e_{L/K,S(L),\Sigma}$ of $\zeta(\QQ[G])$ defined in (\ref{key idem def}).

Fix a prime $p$ and a finite set of places $T$ of $K$ with  $T\cap S(L)= \emptyset$ and such that $U_{L,p} := \mathcal{O}_{L,S(L),T,p}^\times$ is torsion-free. Let $ \ZZ_p[G]^d \xrightarrow{\phi} \ZZ_p[G]^d$ be a representative  of $C_{L,S(L),T,p}$ of the form constructed in Lemma \ref{can rep lemma} (with respect to the place $v'$) and write 
\[ \iota_* : {\bigcap}_{\ZZ_p[G]}^{a}U_{L,p} \to
{\bigcap}_{\ZZ_p[G]}^{a}\ZZ_p[G]^d\]
for the injective homomorphism of $\xi(\ZZ_p[G])$-modules that is induced by the isomorphism $U_{L,p} \cong H^0(C_{L,S(L),T,p}) \cong \ker(\phi)$. Then, with $\{b_{i}\}_{i \in [d]}$ denoting
the standard basis of $\ZZ_p[G]^d$, there exists an (ordered) subset $I = I_\Sigma$ of $[d]$ of cardinality $d-a$ such that the element
\begin{equation}\label{element constuction}x_L := ({\wedge}_{i\in I}(b_{i}^\ast\circ\phi))({\wedge}_{j\in [d]}b_{j}) \end{equation}
of ${\bigcap}_{\ZZ_p[G]}^a\ZZ_p[G]^d$ has all of the following properties.
\begin{itemize}
\item[(i)] There exists a unique element $\varepsilon_L$ of ${\bigcap}_{\ZZ_p[G]}^aU_{L,p}$ that is independent of the choice of $v'$ and such that $\iota_*(\varepsilon_L) = x_L$.
    \item[(ii)] $\zeta(\QQ_p[G])\cdot \varepsilon_L = \QQ_p\otimes_{\ZZ_p}\bigl({\bigcap}_{\ZZ_p[G]}^aU_{L,p}\bigr)[1-e]$.
\item[(iii)] For each $\lambda$ in $\zeta(\QQ_p[G])$ the following conditions are equivalent.
\begin{itemize}
\item[(a)] $\lambda\cdot \varepsilon_L \in {\bigcap}_{\ZZ_p[G]}^aU_{L,p}$.
\item[(b)] $\lambda\cdot x_L \in {\bigcap}_{\ZZ_p[G]}^a\ZZ_p[G]^d$.
\item[(c)] $\lambda\cdot {\rm Nrd}_{\QQ_p[G]}(M)\in \xi(\ZZ_p[G])$ for all matrices $M = (M_{ij})$ in ${\rm M}_d(\ZZ_p[G])$ with the property that $M_{ij} = (b_{i}^\ast\circ\phi)(b_j)$ for $i \in I$ and $j \in [d]$.
\end{itemize}
\item[(iv)] Let $\varepsilon_L'$ be an element of ${\bigcap}_{\ZZ_p[G]}^aU_{L,p}$ obtained by making the above constructions with respect to another choice of representative  of $C_{L,S(L),T,p}$ of the form constructed in Lemma \ref{can rep lemma}. Then there exists an element   $\mu$ of the subgroup ${\rm Nrd}_{\QQ_p[G]}(\K_1(\ZZ_p[G]))$ of  
$\xi(\ZZ_p[G])^\times$ such that  $\varepsilon_L' = \mu\cdot\varepsilon_L$. 
\end{itemize}
\end{proposition}

\begin{proof} Set $R := \ZZ_p[G]$ and $A := \QQ_p[G]$.

Write $I'= I'_\Sigma$ for the subset of $[n]$ comprising integers $i$ for which the $i$-th place $v_i$ in $S(L)\setminus \{v'\}$ belongs to $\Sigma$ and set $I := [d]\setminus I'$. By relabelling if necessary, we can (and will) assume in the rest of the argument that $I' = [a]$, and hence $I = [d]\setminus [a]$.

Then, since every place in $\Sigma$ splits completely in $L$, the construction of $\phi$ combines with the properties of the surjective map fixed in (\ref{resolution}) to imply  $b_i^*\circ \phi = 0$ for each $i\in [a]$. In addition, the group ${\rm Ext}^1_{R}({\rm im}(\phi),R)$ vanishes since $\im(\phi)$ is torsion-free.

Given these observations, and the fact that the idempotent $e$ is defined via the condition (iii) in Lemma \ref{idem lemma}, claims (i) and (ii) are obtained directly upon applying the general result of \cite[Prop. 4.21(i)]{bses}, with the matrix $M$ in the latter result taken to be the submatrix $\bigl( (b_i^{\ast}\circ\phi)(b_j)\bigr)_{a< i \le d, 1\le j \le d}$ of the matrix of $\phi$ with respect to the basis $\{b_i\}_{1\le i\le d}$. We note, in particular, that the independence assertion in claim (i) is true since claim (ii) implies $\varepsilon_L = e(\varepsilon_L)$, whilst Lemma \ref{idem lemma}(iii) implies $e(v') = 0$.

Next we note that the vanishing of ${\rm Ext}^1_{R}({\rm im}(\phi),R)$ also combines with \cite[Th. 4.19(iv)]{bses} to imply that the conditions (a) and (b) in claim (iii) are equivalent. Finally, we note that condition (b) is equivalent to condition (c) since for every subset $\{\theta_i\}_{1\le i\le a}$ of $\Hom_{R}(R^d,R)$ one has
\begin{align*}({\wedge}_{i=1}^{i=a}\theta_i)(\lambda\cdot x_L) =&\,  \lambda\cdot ({\wedge}_{i=1}^{i=a}\theta_i)\bigl(({\wedge}_{i=a+1}^{i=d}(b_{i}^\ast\circ\phi))
({\wedge}_{j=1}^{j=d}b_{j})\bigr)\\
 =&\, \lambda\cdot{\rm Nrd}_{A}(-1)^{a(d-a)}\cdot \bigl(({\wedge}_{i=1}^{i=a}\theta_i)\wedge ({\wedge}_{i=a+1}^{i=d}(b_{i}^\ast\circ\phi))\bigr)({\wedge}_{j=1}^{j=d}b_{j})\\
=&\, \lambda\cdot{\rm Nrd}_{A}(-1)^{a(d-a)}\cdot{\rm Nrd}_{A}(M).\end{align*}
Here $M$ is the matrix in ${\rm M}_d(R)$ that satisfies

\[ M_{ij} = \begin{cases} \theta_i(b_j), &\text{ if $1\le i\le a$, $1\le j\le d$}\\
(b_{i}^\ast\circ\phi)(b_j), &\text{ if $a < i\le d$, $1\le j\le d$,}\end{cases}\]
and so the third equality is valid by \cite[Lem. 4.10]{bses}.

To prove claim (iv) we assume to be given a commutative diagram (\ref{very important diagram}) and for any map of $\xi(R)$-modules $\theta$ we set $\theta_{\QQ_p} := \QQ_p\otimes_{\ZZ_p}\theta$. 

Then, using the automorphisms $\eta'$ and $\eta$ from  (\ref{very important diagram}), we compute 
%
\begin{align*} ({\wedge}^{a}_{A}\eta'_{\QQ_p})(x_L) =&\,  ({\wedge}^{a}_{A}\eta'_{\QQ_p})\bigl( ({\wedge}_{j=a+1}^{j=d}(b_j^\ast\circ\phi))({\wedge}_{i=1}^{i=d}b_{i})\bigr)\\
 =&\,  ({\wedge}_{j=a+1}^{j=d}(b_j^\ast\circ\phi\circ (\eta')^{-1}))\bigl(({\wedge}^{d}_{A}\eta'_{\QQ_p})({\wedge}_{i=1}^{i=d}b_{i})\bigr)\notag\\
 =&\, {\rm Nrd}_A(\eta'_{\QQ_p})\cdot ({\wedge}_{j=a+1}^{j=d}(b_{j}^\ast\circ \eta^{-1}\circ\phi'))({\wedge}_{i=1}^{i=d}b_{i})\\
 =&\, {\rm Nrd}_A(\eta'_{\QQ_p})\cdot
 ({\rm Nrd}_A(\eta_{\QQ_p})^{\#})^{-1}\cdot  ({\wedge}_{j=a+1}^{j=d}(b_j^\ast\circ\phi'))({\wedge}_{i=1}^{i=d}b_{i}),\notag
 \end{align*}
where we write $x\mapsto x^\#$ for the $\QQ_p$-linear involution of $\zeta(A)$ that is induced by inverting elements of $\G_L$. Before justifying this computation, we note that, if true, it implies the validity of claim (iv) (with $\mu = {\rm Nrd}_A(\eta'_{\QQ_p})\cdot ({\rm Nrd}_A(\eta_{\QQ_p})^{\#})^{-1}$) since the commutativity of the first square in (\ref{very important diagram}) implies that the composite map $({\wedge}^{a}_{A}(\eta'_{\QQ_p}))\circ \iota_*$ coincides with the embedding ${\bigcap}_{\ZZ_p[G]}^{a}U_{L,p} \to
{\bigcap}_{\ZZ_p[G]}^{a}\ZZ_p[G]^d$ induced by $\hat\iota'$.

Now the first and second equalities in the above computation are clear and the third is true since 
\[ ({\wedge}^{d}_{A}\eta'_{\QQ_p})({\wedge}_{i\in [d]}b_{i}) = {\rm Nrd}_A(\eta'_{\QQ_p})\cdot {\wedge}_{i\in [d]}b_{i}\]
and because the commutativity of the second square in (\ref{very important diagram}) implies that
\[  b_{j}^\ast\circ \phi\circ (\eta')^{-1} = b_{j}^\ast\circ \eta^{-1}\circ\phi'.\]
To verify the fourth equality it suffices (by virtue of the injectivity of the map ${\rm ev}^a_{\ZZ_p[G]^d}$ in (\ref{ev map def})) to show that
\[ ({\wedge}_{j=1}^{j=a}\theta_j)\wedge ({\wedge}_{j=a+1}^{j=d}(b_{j}^\ast\circ \eta^{-1}\circ\phi')) =
  ({\rm Nrd}_A(\eta_{\QQ_p})^{\#})^{-1}\cdot\bigl(({\wedge}_{j=1}^{j=a}\theta_j)\wedge ({\wedge}_{j=a+1}^{j=d}(b_{j}^\ast\circ \phi'))\bigr)\]
for every subset $\{\theta_j\}_{1\le j\le a}$ of $\Hom_{R}(R^d,R)$.

To prove this we write $\tilde\eta$ for the automorphism of $R^d$ that is represented with respect to the basis $\{b_i\}_{1\le i\le d}$ by the matrix %
\begin{equation*} \left(
\begin{array} {c|c} I_a& 0\\
                        \hline
                              0 & M'_\eta\end{array}\right),\end{equation*}
where $M'_\eta$ agrees with the corresponding $(d-a)\times (d-a)$ minor of the matrix (\ref{block}) that represents $\eta$. Then it is clear that
\[ {\rm Nrd}_A(\eta_{\QQ_p}) = {\rm Nrd}_A(M_\eta) = {\rm Nrd}_A(M'_\eta) = {\rm Nrd}_A(\tilde\eta_{\QQ_p})\]
and also that 
\[  \eta^{-1}\circ\phi' = \tilde\eta^{-1}\circ\phi'\] 
since $\im(\phi') \subseteq R\cdot \{b_i\}_{a< i\le d}$. Thus, if for any given subset $\{\theta_i\}_{1\le i\le a}$ of $\Hom_{R}(R^d,R)$ we write $\varphi$ for the (unique) map in $\Hom_R(R^d,R^d)$ with
\[ b^\ast_j\circ \varphi = \begin{cases} \theta_j, &\text{if $1\le j\le a$}\\
                                         b^\ast_j\circ \phi', &\text{if $a< j \le d$,}\end{cases} \]
then for each $j$ with $a < j \le d$ one has 
\[ b^\ast_j\circ \eta^{-1} \circ \phi' = (b^\ast_j\circ \tilde\eta^{-1}) \circ \phi' = (b^\ast_j\circ \tilde\eta^{-1}) \circ \varphi\]
and hence
\begin{align*} ({\wedge}_{j=1}^{j=a}\theta_i)\wedge ({\wedge}_{j=a+1}^{j=d}(b_{j}^\ast\circ \eta^{-1}\circ\phi')) =&\, ({\wedge}_{j=1}^{j=a}(b^\ast_j\circ \tilde\eta^{-1}\circ \varphi))\wedge
({\wedge}_{j=a+1}^{j=d}(b^\ast_j\circ \tilde\eta^{-1}\circ \varphi))\\
=&\, \bigl(({\wedge}_{j=1}^{j=a}(b^\ast_j\circ \tilde\eta^{-1}))\wedge({\wedge}_{j=a+1}^{j=d}(b^\ast_j\circ \tilde\eta^{-1}))\bigr) \circ {\wedge}_A^d(\varphi_{\QQ_p})\\
=&\, ({\rm Nrd}_A(\tilde\eta_{\QQ_p})^{\#})^{-1}\cdot \bigl(({\wedge}_{j=1}^{j=a}b^\ast_j)\wedge ({\wedge}_{j=a+1}^{j=d}b^\ast_j)\bigr)\circ {\wedge}_A^d(\varphi_{\QQ_p})
\\
=&\, ({\rm Nrd}_A(\eta_{\QQ_p})^{\#})^{-1}\cdot \bigl(({\wedge}_{j=1}^{j=a}(b^\ast_j \circ \varphi))\wedge ({\wedge}_{j=a+1}^{j=d}(b^\ast_j\circ \varphi))\bigr)\\
=&\, ({\rm Nrd}_A(\eta_{\QQ_p})^{\#})^{-1}\cdot \bigl(({\wedge}_{j=1}^{j=a}\theta_j)\wedge ({\wedge}_{j=a+1}^{j=d}(b_j^*\circ\phi'))\bigr),
\end{align*}
as required. Here the third equality follows from \cite[Lem. 4.13]{bses} and the fact that the reduced norm of the automorphism of $\Hom_R(R^d,R)$ that sends each $b_i^\ast$ to $b_i^\ast\circ \tilde\eta^{-1}$ is the inverse of
${\rm Nrd}_A(\tilde\eta_{\QQ_p})^{\#}$, and all other equalities are clear. \end{proof}


\begin{remark}\label{rs stable2}{\em Claim (ii) of Proposition \ref{generate rubin} implies that there is an equality $\varepsilon_L = e(\varepsilon_L)$ in $\QQ_p\otimes_{\ZZ_p}\bigl({\bigcap}_{\ZZ_p[G]}^aU_{L,p}\bigr)$. Taken in conjunction with Remark \ref{rs stable}, this shows that the elements $\varepsilon_L$ have the same invariance properties as do the relevant Rubin-Stark elements.}\end{remark}

\subsubsection{}For each $F$ in $\Omega(\KK)$ we now set %
\[ \Xi(F) := \D_{\ZZ[\G_{F}]}(C_{F,S(F),T})\]
 and for $F'$ in $\Omega(\KK)$ with $F\subseteq F'$ we consider 
%
the composite surjective homomorphism of (graded) $\xi(\ZZ[\G_{F'}])$-modules
\begin{align}\label{nuF'F} \nu_{F'/F}: \Xi(F') \to&\,\,
  \D_{\ZZ[\G_F]}(\ZZ[\G_F]\otimes^{\DL}_{\ZZ[\G_{F'}]}C_{F',S(F'),T})\\
\to&\,\, \D_{\ZZ[\G_F]}(C_{F,S(F),T})\otimes{\bigotimes}_{v\in S(F')\setminus S(F)}\D_{\ZZ[\G_F]}(C_{v})\notag\\
\to&\,\, \Xi(F),\notag
\end{align}
where we set 
\[ C_v := {\bigoplus}_w\DR\Hom_\ZZ(\DR\Gamma((\kappa_w)_{\mathcal{W}},\ZZ),\ZZ)[-1]\] 
with $w$ running over all places of $F$ above $v$. Here the first map in (\ref{nuF'F}) is induced by the standard base-change isomorphism (from \cite[Th. 5.4(ii)]{bses}), the second is the isomorphism obtained by combining Lemma \ref{complex construction}(iv) with an application of \cite[Th. 5.4(i)]{bses} to the exact triangle in Lemma \ref{complex construction}(ii) and the final map uses the canonical isomorphisms
\begin{equation}\label{nuF'Fiso} \D_{\ZZ[\G_F]}(C_v)\cong (\xi(\ZZ[\G_F]),0)\end{equation}
that are induced by the descriptions in Remark \ref{exp EF} (as per \cite[(4.1.1)]{bses1}).

We can therefore define a graded $\xi_p(\KK/K)$-module of `vertical systems' for the data $\KK/K,S,T$ and $p$ by means of the inverse limit
$${\rm VS}(\KK/K,S,T,p) :=\varprojlim_{F\in\Omega(\KK)}\Xi(F)_p,$$
where the transition morphism for 
 $F\subseteq F'$ is $\nu_{F'/F,p}$. The Hermite-Minkowski theorem implies
 that the ungraded part of ${\rm VS}(\KK/K,S,T,p)$ is a free $\xi_p(\KK/K)$-module
 of rank one but we make no use of this fact (and so do not give a proof).

\begin{theorem}\label{VStoES} There exists a canonical homomorphism of $\xi_p(\K/K)$-modules
\[ \Theta_{\KK/K,S,T,p}: {\rm VS}(\KK/K,S,T,p) \to {\rm ES}_{r_S}(\KK/K,S,T,p).\]
This homomorphism is non-zero if and only if $\theta_{F/K,S(F)}^{r_S}(0)\not= 0$ for some $F$ in $\Omega(\KK)$.
\end{theorem}


\begin{proof} 
%
%
Set $\Sigma := \Sigma_S(\mathcal{K})$, $r := r_S (= |\Sigma|)$ and $U_F := \co_{F,S(F),T}^\times$ and $C_F := C_{F,S(F),T}$ for every $F$ in $\Omega(\KK)$. Then, since each transition morphism $\nu_{F'/F,p}$ is surjective, to construct a map of the claimed sort it is enough to construct for each $F$ a canonical homomorphism of $\zeta(\QQ_p[\cG_F])$-modules
\[ \Theta^\Sigma_{F,S,T,p}: \QQ_p\otimes_{\ZZ_p}\Xi(F)_p \to \QQ_p\otimes_{\ZZ_p}{\bigcap}_{\ZZ_p[\cG_F]}^{r}U_{F,p}\]
with all of the following properties:-

\begin{itemize}
\item[(a)] $\Theta^\Sigma_{F,S,T,p}(\Xi(F)_p) \subseteq {\bigcap}_{\ZZ_p[\cG_F]}^{r}U_{F,p}$.
\item[(b)] $\Theta^\Sigma_{F,S,T,p}$ is the zero map if and only if $\theta_{F/K,S(F)}^{r}(0)= 0$.
\item[(c)] For all $F'$ in $\Omega(\KK)$ with $F \subset F'$ and all $x$ in $\Xi(F')_p$ the equality (\ref{distribution}) is valid with 
\[ c_{F'} = \Theta^\Sigma_{F',S,T,p}(x)\quad \text{ and }\quad c_F = \Theta^\Sigma_{F,S,T,p}(\nu_{F'/F,p}(x)).\]
\end{itemize}

We write $e_F$ for the idempotent $e_{F/K,S(F),\Sigma}$ of $\zeta(\QQ[\cG_F])$ defined in (\ref{key idem def}) 
and note that the space $e_F(\QQ\otimes_\ZZ\ker(\alpha))$ vanishes, where $\alpha$ is the
 composite
surjective homomorphism
\[ H^1(C_F) \to X_{F,S(F)} \to Y_{F,\Sigma} \cong \ZZ[\cG_F]^r.\]
Here the first map is induced by Lemma \ref{complex construction}(i) and the
exact sequence (\ref{selmer lemma seq2}), the second is the natural
projection and the isomorphism is induced by sending the chosen set
of places  $\{w_{v,F}\}_{v \in \Sigma}$ to the standard basis of
$\ZZ[\G_F]^r$. 

We then  define $\Theta^\Sigma_{F,S,T,p}$ to be the scalar
extension of the following composite homomorphism of $\zeta(\QQ[\G_F])$-modules
\begin{align}\label{Theta'}&\,\,\D_{\QQ[\G_F]}(\QQ\cdot C_{F})\\ \notag
\cong&\,\,\D^\diamond_{\QQ[\G_F]}(\QQ\cdot U_F)\otimes \D^\diamond_{\QQ[\G_F]}(\QQ\cdot H^1(C_F))^{-1}\\ \notag
\to &\,\,e_F\bigl(\QQ\cdot {\bigcap}_{\ZZ[\G_F]}^r U_F\bigr)\otimes_{\zeta(\QQ[\G_F])}e_F\bigl(\Hom_{\zeta(\QQ[\G_F])}\bigl(
 \QQ\cdot {\bigcap}_{\ZZ[\G_F]}^r \ZZ[\G_F]^r,\zeta(\QQ[\G_F])\bigr)\bigr)\\ \notag
\cong&\,\,e_F\bigl(\QQ\cdot {\bigcap}_{\ZZ[\G_F]}^r U_F\bigr).
\end{align}
Here the first map is induced by the standard `passage to cohomology' isomorphism
(cf. \cite[Prop. 5.17(i)]{bses}) and the descriptions in
Lemma \ref{complex construction}(i), the second is induced by
multiplication by $e_F$, the isomorphism $e_F(\QQ\otimes_\ZZ \alpha)$
and the argument of
\cite[Lem. 4.7(ii)]{bses1} and the last map uses the  isomorphism of $\xi(\ZZ[\G_F])$-modules ${\bigcap}_{\ZZ[\G_F]}^r \ZZ[\G_F]^r \cong \xi(\ZZ[\G_F])$ induced by the standard basis of $\ZZ[\G_F]^r$ (and \cite[Prop. 5.9(i)]{bses}).

To verify that this definition of $\Theta^\Sigma_{F,S,T,p}$ has the required
properties, we fix a place $v'$ in $S(F)\setminus \Sigma_S(F)$ (so $v' \notin \Sigma$) and note that condition (\ref{ordering 2}) ensures that $\Sigma$ corresponds to the first $r$ elements of the (ordered) set $S(F)\setminus \{v'\}$. 

We use Hypothesis \ref{tf hyp} to fix a representative of
$C_{F,p}$ of the form 
\[ \ZZ_p[\G_F]^d \xrightarrow{\phi_F} \ZZ_p[\G_F]^d\]
used in Proposition \ref{generate rubin}. We write $\varepsilon_F$ for the element of ${\bigcap}_{\ZZ_p[\G_F]}^rU_{F,p}$ that is obtained in this case via the formula (\ref{element constuction}) (noting that the present hypotheses imply the set $I$ in the latter formula is equal to $[d]\setminus [r]$). We write $\{ b_{F,i}\}_{i \in [d]}$ for
the standard basis of $\ZZ_p[\G_F]^d$. Then, setting
\[ \beta_{F} := (({\wedge}_{i\in [d]}b_{F,i})\otimes ({\wedge}_{i\in [d]}b_{F,i}^\ast),0),\]
the argument establishing \cite[(4.2.6)]{bses1} shows that the restriction of $\Theta^\Sigma_{F,S,T,p}$ to $\Xi(F)_p$ coincides with the composite
\begin{equation}\label{exp proj} \Xi(F)_p = {\rm d}_{\ZZ_p[\G_F]}(C_{F,p}) \xrightarrow{\sim}  \xi(\ZZ_p[\G_F])\!\cdot\!\beta_F \to {\bigcap}_{\ZZ_p[\G_F]}^{r}U_{F,p}\end{equation}
where the first map is the isomorphism induced by the given representative of $C_{F,p}$ and the second map sends $\beta_F$ to $\varepsilon_F$.

Given this explicit description of $\Theta^\Sigma_{F,S,T,p}$, property (a) follows from Proposition \ref{generate rubin}(i) and property (b) upon combining the results of Proposition \ref{generate rubin}(ii) and Lemma \ref{idem lemma}. 

Finally, the fact $\Theta^\Sigma_{F,S,T,p}$ has property (c) follows directly from the argument proving  \cite[(4.2.4)]{bses1}, after making the following changes: the map ${\rm Cor}^r_{F'/F}$ and element $P_v$ that are used in loc. cit. are replaced by ${\rm N}^r_{F'/F}$ and ${\rm Nrd}_{\QQ[\G_F]}(1-\Fr_{F,v}^{-1})$; the exact triangle in \cite[Lem. 4.1(iii)]{bses1} and the isomorphisms occurring in \cite[(4.1.2)]{bses1} are replaced by the exact triangle in Lemma \ref{complex construction}(ii) and the isomorphisms (\ref{nuF'Fiso}) 
that occur in  the composite homomorphism (\ref{nuF'F}). \end{proof}


\subsubsection{}
 We now use Theorem \ref{VStoES} to strengthen the construction of \cite[\S1, Th. B]{bses1}. 
 To state the result we fix an embedding $\sigma:\QQ^c\to\CC$ and for each subfield $F$ of $\QQ^c$ write $w_{F,\sigma}$ for the archimedean place of $F$ corresponding to $\sigma$.
 For each natural number $n$ we write $\zeta_n$ for the unique primitive $n$-th root of unity in $\QQ^c$ that satisfies $\sigma(\zeta_n)=e^{2\pi i/n}$. 
 We write $\QQ^{c,+}$ for the maximal totally real extension of $\QQ$ in $\QQ^c$.

For a finite group $\Gamma$ we use the ideal $\delta(\ZZ[\Gamma])$ of $\zeta(\ZZ[\Gamma])$ introduced in \cite[Def. 3.6]{bses}. Following \cite[Def. 3.18]{bses}, we then define the `central pre-annihilator' of a $\Gamma$-module $M$ by setting 
\[ {\rm pAnn}_{\ZZ[\Gamma]}(M) := \{x \in \zeta(\ZZ[\Gamma]): x\cdot \delta(\ZZ[\Gamma]) \subseteq {\rm Ann}_{\ZZ[\Gamma]}(M)\}.\]
We note, in particular, that this lattice is a $\xi(\ZZ[\Gamma])$-submodule of $\zeta(\QQ[\Gamma])$ and is equal to the annihilator of $M$ in $\ZZ[\Gamma]$ if $\Gamma$ is abelian. 

\begin{theorem}\label{cyclotomiccor}
 For each odd prime $p$, there exists an Euler system 
\[ \varepsilon^{\rm cyc} = (\varepsilon^{\rm cyc}_F)_F\]
in ${\rm ES}_1(\QQ^{c,+}/\QQ,\{\infty\},\emptyset, p)$ that has the following properties at every $F$ in $\Omega(\QQ^{c,+}/\QQ)$.
%

\begin{itemize}
\item[(i)] If $F/\QQ$ is abelian, and of conductor $f(F)$, then
\[ \varepsilon^{\rm cyc}_F = {\rm Norm}_{\QQ(\zeta_{f(F)})/F}(1-\zeta_{f(F)}).\]

\item[(ii)] For $\varphi$ in $\Hom_{\cG_F}(\mathcal{O}_{F,S(F)}^\times, \ZZ[\cG_F])$, and every prime $\ell$ that ramifies in $F$, one has
%
 \[  \bigl({\bigwedge}_{\QQ[\cG_F]}^1\varphi\bigr)(\varepsilon^{\rm cyc}_F)\in {\rm pAnn}_{\ZZ[\cG_F]}({\rm Cl}(\mathcal{O}_{F}[1/\ell]))_p.\]
\item[(iii)] For every $\chi$ in ${\rm Ir}(\cG_F)$ there exists a non-zero element $u_{F,\chi}$ of $\CC_p$ that satisfies both
\[ \bigl({\bigwedge}^{\chi(1)}_{\CC_p}V^\ast_\chi\otimes_{\RR[\cG_F]} {\rm Reg}_{F,S(F)}\bigr)(e_\chi\cdot\varepsilon^{\rm cyc}_{F}) = u_{F,\chi}\cdot L_{S(F)}^{\chi(1)}(\check\chi,0)\cdot e_\chi({\bigwedge}^1_{\CC_p[\cG_{F}]}(w_{F,\sigma}-w_{F,p})),\]
and 
\[ {\prod}_{\omega\in \Gal(\QQ(\chi)/\QQ)}u_{F,\chi^\omega} \in \ZZ_p^\times.\]
\end{itemize}
\end{theorem}


\begin{proof} We set $\KK := \QQ^{c,+}$ and write $\K^{\rm ab}$ for the maximal absolutely abelian subfield of $\KK$. 
 Then, since $p$ is odd, Hypothesis \ref{tf hyp} is satisfied for $\KK$ with $S = \{\infty\}$ and $T = \emptyset$ (cf. Remark \ref{tf hyp rem}) and so Theorem \ref{VStoES} constructs a canonical map $\Theta := \Theta_{\KK/\QQ,\{\infty\},\emptyset,p}$. 

In addition, the known validity of the equivariant Tamagawa Number Conjecture in the relevant case implies
(via the argument of \cite[Lem. 5.4]{bses1}) that the (free, rank one) $\xi_p(\KK^{\rm ab}/\QQ)$-module ${\rm VS}(\KK^{\rm ab}/\QQ,\{\infty\},\emptyset,p)$ has a basis element $\eta^{\rm ab}$ with the following property: for each $F$ in $\Omega(\KK^{\rm ab})$ the image $\eta^{\rm ab}_F$ of $\eta^{\rm ab}$ in $\Xi(F)_p$ is sent by the homomorphism $\Theta_F := \Theta_{F,\{\infty\},\emptyset,p}^{\{\infty\}}$ constructed in the proof of Theorem \ref{VStoES} to ${\rm Norm}_{\QQ(\zeta_{f(F)})/F}(1-\zeta_{f(F)})$.

 Since the projection map
$\xi_p(\KK/\QQ)^\times \to \xi_p(\KK^{\rm ab}/\QQ)^\times$ is surjective (by \cite[Lem. 5.5]{bses1}), we can then fix a basis element $\eta$ of the $\xi_p(\KK/\QQ)$-module ${\rm VS}(\KK^{\rm ab}/\QQ,\{\infty\},\emptyset,p)$ that the projection map 
\[ {\rm VS}(\KK/\QQ,\{\infty\},\emptyset,p) \to {\rm VS}(\KK^{\rm ab}/\QQ,\{\infty\},\emptyset,p)\]
sends to $\eta^{\rm ab}$. We then obtain a system in ${\rm ES}_1(\KK/\QQ,\{\infty\},\emptyset,p)$ by setting
\[ \varepsilon^{\rm cyc}:= \Theta(\eta).\]

This construction ensures directly that for every $F$ in $\Omega(\KK^{\rm ab})$ one has 
\[ \varepsilon^{\rm cyc}_F = \Theta_F(\eta_F^{\rm ab}) = {\rm Norm}_{\QQ(\zeta_{f(F)})/F}(1-\zeta_{f(F)}),\]
as stated in claim (i). 

On the other hand, the properties in claims (ii) and (iii) are verified by precisely mimicking the arguments in \cite[\S5.2]{bses1}. \end{proof}

\begin{remark}\label{formulaiii}{\em This result strengthens that of \cite[\S1, Th. B]{bses1} since $\varepsilon^{\rm cyc}$ is a strict refinement (in the sense of Remark \ref{new ES}(ii)) of the system in ${\rm ES}_1(\KK/\QQ,\{\infty\}\cup \{p\},\emptyset,p)$ that is constructed in loc. cit.}\end{remark}

\begin{remark}{\em The displayed containment in Theorem \ref{cyclotomiccor}(ii) can be interpreted as a special case of the annihilation results relating to Selmer modules of $p$-adic representations that are obtained by Macias Castillo and Tsoi in \cite{mct}.}\end{remark}

\section{Higher rank non-commutative Iwasawa theory}\label{hrncit section}

In this section we use the constructions made in \S\ref{res section} to
formulate an explicit main conjecture of non-commutative Iwasawa theory for $\mathbb{G}_m$ over arbitrary number fields.

We then show that this conjecture simultaneously extends
both the higher rank main conjecture of (commutative) Iwasawa theory
formulated by Kurihara and the present authors in \cite{bks2} and
the general formalism of main conjectures in non-commutative Iwasawa theory following the approaches of Ritter and Weiss in  \cite{RitterWeiss} and
of Coates et al in \cite{cfksv}, and thereby deduce its validity in important special cases.

In the sequel we shall regard the prime $p$ as fixed, write $L^{\rm cyc}$ for
the cyclotomic $\ZZ_p$-extension of each
number field $L$ and set $\Gamma_L := \Gal(L^{\rm cyc}/L)$.

We also fix a rank one $p$-adic Lie extension $\KK_\infty$ of $K$, set $\G_\infty := \Gal(\KK_\infty/K)$ and for any infinite subquotient $\mathcal{G}$ of $\G_\infty$ we write $\Lambda(\mathcal{G})$ for the Iwasawa algebra $\ZZ_p[[\mathcal{G}]]$. 

We recall that the total quotient ring $Q(\G_\infty)$ of $\Lambda(\G_\infty)$ is a semisimple algebra and hence that there exists a reduced norm homomorphism 
\[ {\rm Nrd}_{Q(\G_\infty)}: \K_1(Q(\G_\infty)) \to \zeta(Q(\G_\infty))^\times.\]

Finally we set
\[ \QQ_p[[\G_\infty]] := {\varprojlim}_{L \in \Omega(\KK_\infty)}\QQ_p[\G_L],\]
where the transition morphisms are the natural projection maps.

\subsection{Whitehead orders in Iwasawa theory} To help set the context for our conjecture, we first clarify the link between the subrings $\xi_p(\KK_\infty/K)$ and $\zeta(\Lambda(\G_\infty))$ of $\zeta(\QQ_p[[\G_\infty]])$.

\begin{lemma}\label{relate xi zeta} Set $\mathcal{R} := \zeta(\La(\G_\infty))$ and $\mathcal{R}' := \xi_p(\KK_\infty/K)$. Fix a central open subgroup $\mathcal{Z}$ of $\G_\infty$ that is topologically isomorphic to $\ZZ_p$ and for each natural number $n$ write  $\mathcal{R}_n$ for the subring $\La(\mathcal{Z}^{p^n})$ of $\mathcal{R}$. Then the following claims are valid for each natural number $n$. 
\begin{itemize}
\item[(i)] The $\ZZ_p$-module $\mathcal{R}'/\bigl(\mathcal{R}'\cap \mathcal{R}\bigr)$ has finite exponent. 
\item[(ii)] There exists a natural number $t$ such that the element $p^t\cdot {\rm Nrd}_{Q(\G_\infty)}(p)$ belongs to $\mathcal{R}\cap \mathcal{R}'$ and any sufficiently large power of it annihilates $\mathcal{R}/\bigl(\mathcal{R}\cap (\mathcal{R}'\cdot\mathcal{R}_n)\bigr)$.
\item[(iii)] There are inclusions
\[ \mathcal{R}'\subseteq \mathcal{R}\left[p^{-1}\right] \subseteq (\mathcal{R}'\cdot \mathcal{R}_n)
\left[p^{-1},{\rm Nrd}_{Q(\G_\infty)}(p)^{-1}\right].\]
\end{itemize}
\end{lemma}

\begin{proof} It is clearly enough to prove claims (i) and (ii) and to do this we use the fact (proved in \cite[Lem. 4.13]{bses1}) that for any matrix $B$ in
${\rm M}_d(\La(\G_\infty))$ the reduced norm ${\rm Nrd}_{Q(\mathcal{G}_\infty)}( B)$ belongs to $\mathcal{R}'$ and is equal to
\begin{equation}\label{rn description}
{\rm Nrd}_{Q(\mathcal{G}_\infty)}( B) =
\bigl( {\rm Nrd}_{\QQ_p[\cG_L]}(B_L)\bigr)_{L \in \Omega(\KK_\infty)}\end{equation}
where $B_L$ is the image of $B$ in ${\rm M}_d(\ZZ_p[\cG_L])$.

To prove claim (i) we fix an element $x = (x_L)_L$ of $\mathcal{R}'$. Then, for each $L$ in $\Omega(\KK_\infty)$, there exists a finite index set $I_L$ and, for each $i \in I_L$, an element $a_i$ of $\ZZ_p$, a natural number $d_i$ and a matrix $M_{L,i}$ in ${\rm M}_{d_i}(\ZZ_p[\G_L])$ such that in $\xi(\ZZ_p[\G_L])$ one has 
\[ x_L = {\sum}_{i\in I_L}a_i\cdot {\rm Nrd}_{\QQ_p[\G_L]}(M_{L,i}).\]

For each index $i$ in $I_L$ we fix a pre-image $M^\infty_{L,i}$ of $M_{L,i}$ under the natural (surjective) projection map
${\rm M}_{d_i}(\La(\G_\infty)) \to {\rm M}_{d_i}(\ZZ_p[\G_L])$. Then ${\rm Nrd}_{Q(\G_\infty)}(M^\infty_{L,i})$ belongs to the integral closure of $\La(\mathcal{Z})$ in $\zeta(Q(\G_\infty))$ (this follows, for example, from the observation of Ritter and Weiss in \cite[\S5, Rem. (H)]{RitterWeiss}) and therefore also to any choice of a maximal $\La(\mathcal{Z})$-order in $Q(\G_\infty)$. Hence, by the central conductor formula of Nickel \cite[Th. 3.5]{NickelDoc}, there exists a natural number $N$ (that is independent of both $L$ and $M_{L,i}$) such that $p^N\cdot {\rm Nrd}_{Q(\G_\infty)}(M^\infty_{L,i})$ belongs to $\La(\G_\infty)$.

The latter containment combines with (\ref{rn description}) to imply  $p^N\cdot {\rm Nrd}_{\QQ_p[\G_L]}(M_{L,i})\in \ZZ_p[\G_L]$ for each $i$ and hence that $p^N\cdot x_L \in \ZZ_p[\G_L]$. It follows that $p^N\cdot x = (p^N\cdot x_L)_L$ belongs to
\[ \zeta(\QQ_p[[\G_\infty]]) \cap {\prod}_L\ZZ_p[\G_L] = \mathcal{R},\]
and hence that $\mathcal{R}'/\bigl(\mathcal{R}'\cap \mathcal{R}\bigr)$ is annihilated by a fixed power $p^N$ of $p$, as required.

To prove claim (ii) we fix a maximal $\mathcal{R}_n$-order $\mathfrak{M}$ in $Q(\G_\infty)$ that contains $\La(\G_\infty)$. We write $\mathcal{M}_n$ for the maximal $\mathcal{R}_n$-order in $\zeta(Q(\G_\infty))$ and claim that the $\mathcal{R}_n$-order generated by
 the elements ${\rm Nrd}_{Q(\G_\infty)}(M)$ as $M$ runs over $\bigcup_{n \ge 1}{\rm M}_n(\mathfrak{M})$ has finite index in $\mathcal{M}_n$.
 To show this we note  $\mathfrak{M}$ is a finitely generated $\mathcal{R}_n$-module and hence that it is enough (by the structure theory of $\mathcal{R}_n$-modules) to show that the localization of $\mathcal{M}_n$ at every height one prime ideal $\wp $ of $\mathcal{R}_n$ is generated over $\mathcal{R}_{n,\wp}$ by ${\rm Nrd}_{Q(\G_\infty)}(M)$ as $M$ runs over $\bigcup_{n \ge 1}{\rm M}_n(\mathfrak{M}_\wp)$. In addition, since each such ring $\mathcal{R}_{n,\wp}$ is a discrete valuation ring and $\mathfrak{M}_\wp$ is a maximal $\mathcal{R}_{n,\wp}$-order in $\zeta(Q(\G_\infty))$, this follows in a straightforward fashion from the arithmetic of local division algebras (as in the proof
 of \cite[Prop. (45.8)]{curtisr}).

 We can therefore fix a natural number $t_1$ such that for each $x = (x_L)_L$ in $\mathcal{R}$, there exists a finite index set $I_x$, and for each $i$ in $I_x$ an element $y_i$ of $\mathcal{R}_n$, a natural number $n_i$ and a matrix $M_{x,i}$ in 
 ${\rm M}_{n_i}(\mathfrak{M})$ such that 
\[ p^{t_1}\cdot x = {\sum}_{i \in I_x} {\rm Nrd}_{Q(\G_\infty)}(M_{x,i})\cdot y_i.\]
In addition, by another application of the central conductor formula \cite[Th. 3.5]{NickelDoc}, there exists a natural number $t$ that is greater than or equal to the integer $N$ fixed above and is such that for every $i$ in $I_x$ the matrix 
 $p^{t}\cdot M_{x,i}$ belongs to 
 ${\rm M}_{n_i}(\La(\G_\infty))$. These facts combine to imply that for every $x$ in $\mathcal{R}$ one has 
 
 \[ p^{t_1}\cdot {\rm Nrd}_{Q(\G_\infty)}(p)^{t}\cdot x = {\sum}_{i \in I_x} {\rm Nrd}_{Q(\G_\infty)}(p^{t}\cdot M_{x,i})\cdot y_i\in \mathcal{R}'\cdot\mathcal{R}_n.\]

In particular, since ${\rm Nrd}_{Q(\G_\infty)}(p)\in \mathcal{R}'$, this containment implies $(p^{t}\cdot {\rm Nrd}_{Q(\G_\infty)}(p))^{t'} \cdot x$ belongs to $\mathcal{R}'\cdot\mathcal{R}_n$ for any integer $t'$ that is greater than both $t_1/t$ and $t$. To complete the proof of claim (ii) it is therefore enough to note that, since $t \ge N$, the proof of claim (i) above implies that $p^t\cdot {\rm Nrd}_{Q(\G_\infty)}(p)$ belongs to $\mathcal{R}'\cap \mathcal{R}$. \end{proof}


\subsection{A main conjecture of higher rank non-commutative Iwasawa theory for $\mathbb{G}_m$}

\subsubsection{}\label{pb section2} For each object $C$ of $\Der^{\rm perf}(\La(\G_\infty))$ we  define a $\xi_p(\KK_\infty/K)$-module by setting
\[ {\rm d}_{\Lambda(\G_\infty)}(C) := {\varprojlim}_{L\in \Omega(\KK_\infty)}{\rm d}_{\ZZ_p[\G_L]}(\ZZ_p[\G_L]\otimes^{\DL}_{\La(\G_\infty)}C) \]
where the transition morphism are induced by \cite[Th. 5.4(ii)]{bses}. This $\xi_p(\KK_\infty/K)$-module is free of rank one and, following the approach of \S\ref{plp section}, we now introduce a canonical set of basis elements.

To do this we fix an isomorphism in $\Der^{\rm perf}(\La(\G_\infty))$ of the form 
\begin{equation}\label{fixed resolution} P^\bullet \cong C,\end{equation}
in which $P^\bullet$ is a bounded complex of finitely generated free $\La(\G_\infty)$-modules. 

In each degree $a$ we write $r_a$ for the rank of $P^a$ and fix an ordered $\La(\G_\infty)$-basis $\underline{b}_{a} = \{b_{a,j}\}_{1\le j\le r_a}$ of $P^a$. Then in every degree $a$ and for every field $L$ in $\Omega(\KK_\infty)$ the image $\underline{b}_{L,a} = \{b_{L,a,j}\}_{1\le j\le r_a}$ of $\underline{b}_{a}$ under the projection map $P^a \to P^a_L:= H_0(\Gal(\KK_\infty/L),P^a)$ is an ordered $\ZZ_p[\G_L]$-basis of $P^a_L$. The element 
\[ x(\underline{b}_\bullet)_L := ({\bigotimes}_{a\in \ZZ} (\wedge_{j\in [r_a]}b_{L,a,j})^{(-1)^a},0)\]
is then a basis of the (graded) $\xi(\ZZ_p[\G_L])$-module ${\rm d}_{\ZZ_p[\G_L]}(\ZZ_p[\G_L]\otimes_{\La(\G_\infty)}P^\bullet)$ that is  compatible with the natural transition morphisms as $L$ varies and so the tuple
\[ x(\underline{b}_\bullet) = \bigl(x(\underline{b}_\bullet)_L\bigr)_{L\in \Omega(\KK_\infty)}\]
is a $\xi_p(\KK_\infty/K)$-basis of ${\rm d}_{\Lambda(\G_\infty)}(P^\bullet)$.

We use this construction to define an Iwasawa-theoretic analogue of the notion of primitive-basis. 

\begin{definition}\label{pbe def} {\em Fix $C$ in $\Der^{\rm perf}(\La(\G_\infty))$. Then an element $x$ of ${\rm d}_{\Lambda(\G_\infty)}(C)$ is said to be a `primitive basis element' if, for every resolution of $C$ of the form (\ref{fixed resolution}), there exists a collection $\underline{b}_\bullet$ of ordered bases of the modules $P^a$ such that the induced isomorphism ${\rm d}_{\Lambda(\G_\infty)}(P^\bullet)\cong {\rm d}_{\Lambda(\G_\infty)}(C)$ sends $x(\underline{b}_\bullet)$ to $x$. We write ${\rm d}_{\Lambda(\mathcal{G}_\infty)}(C)^{\rm pb}$ for the subset of
${\rm d}_{\Lambda(\mathcal{G}_\infty)}(C)$ comprising all primitive basis elements.}\end{definition}

\begin{remark}{\em The fact that the elements $x(\underline{b}_\bullet)$ are defined as inverse limits of the corresponding elements $x(\underline{b}_\bullet)_L$ over finite extensions $L$ of $K$ in $\KK_\infty$ is important. Specifically, this property combines with the argument of Proposition \ref{concrete primitive lemma 20} to imply that, in order to show $x$ belongs to ${\rm d}_{\Lambda(\mathcal{G}_\infty)}(C)^{\rm pb}$ it is sufficient to check, for any {\em fixed} resolution (\ref{fixed resolution}) of $C$, that there exists a collection $\underline{b}_\bullet$ of ordered bases  such that the induced isomorphism ${\rm d}_{\Lambda(\G_\infty)}(P^\bullet)\cong {\rm d}_{\Lambda(\G_\infty)}(C)$ sends $x(\underline{b}_\bullet)$ to $x$.}\end{remark}

\subsubsection{}\label{ncmc section}The set $S_{\rm ram}(\KK_\infty/K)$ of places of $K$ that ramify in $\KK_\infty$ is finite and we define
\begin{equation}\label{can places} S_{\KK_\infty/K} := S_K^\infty\cup S_K^p \cup S_{\rm ram}(\KK_\infty/K),\end{equation}
where $S_K^p$ denotes the set of all $p$-adic places of $K$. We then fix a finite set $S$ of places of $K$ with the property that 
\[ S_{\KK_\infty/K} \subseteq S.\]

We note that, since $S = S(L)$ for every $L$ in $\Omega(\KK_\infty)$, the construction in Lemma \ref{complex construction} gives rise to an object
\[ C_{\KK_\infty,S,T} := {\varprojlim}_{L\in \Omega(\KK_\infty)} C_{L,S,T,p}\]
of $\Der^{\rm perf}(\La(\G_\infty))$, where the transition morphisms for $L \subset L'$ are induced by the morphisms in Lemma \ref{complex construction}(iv). We further note that, since $S$ contains $S_K^p$, Lemma \ref{complex construction}(v) implies that this object can be naturally interpreted in terms of the compactly-supported $p$-adic cohomology of $\ZZ_p$.

We also set
\[ \mathcal{O}_{\KK_\infty,S,T}^\times := {\varprojlim}_{L\in \Omega(\KK_\infty)} \co_{L,S,T,p}^\times\]
where the transition morphisms for $L \subseteq L'$ are induced by the field-theoretic norm maps $(L')^\times \to L^\times$.
For each non-negative integer $a$ we then define a $\xi_p(\KK/K)$-module by setting
\begin{equation*}{\bigwedge}^{a}_{\CC_p\cdot\Lambda(\cG_\infty)}(\CC_p\cdot\co^\times_{\KK_\infty,S,T}):=
{\varprojlim}_{L\in \Omega(\KK_\infty)}{\bigwedge}_{\CC_p[\G_{L}]}^{a}(\CC_p\cdot\co_{L,S,T,p}^\times)\end{equation*}
and a submodule 
\begin{equation}\label{infinite bidual}{\bigcap}^{a}_{\Lambda(\cG_\infty)}\co^\times_{\KK_\infty,S,T}:=
{\varprojlim}_{L\in \Omega(\KK_\infty)}{\bigcap}_{\ZZ_p[\G_{L}]}^{a}\co_{L,S,T,p}^\times\end{equation}
%
where, in both cases, the limits are taken with respect to (the scalar extensions of) the transition morphisms (\ref{bidual transition}).

We set $\Sigma := \Sigma_S(\KK_\infty)$ and $r := r_{S,\KK_\infty} (= |\Sigma|)$. Then, since $S = S(L)$ for every $L$ in $\Omega(\KK_\infty)$, there are identifications
\[ {\rm d}_{\Lambda(\G_\infty)}(C_{\KK_\infty,S,T}) =
{\rm VS}(\KK_\infty/K,S,T,p)\]
and 

\[ {\rm ES}_r(\KK_\infty/K,S,T,p) = {\bigcap}^{r}_{\Lambda(\cG_\infty)}\co^\times_{\KK_\infty,S,T},\]
and so the constructions in \S\ref{BESsection} gives a canonical  homomorphism of $\xi_p(\KK_\infty/K)$-modules
\[ \Theta^\Sigma_{\KK_\infty,S,T}: {\rm d}_{\Lambda(\G_\infty)}(C_{\KK_\infty,S,T})\to
{\bigcap}^{r}_{\Lambda(\cG_\infty)}\co^\times_{\KK_\infty,S,T}.\]
%

For the same reason, the distribution relation (\ref{distribution}) gives rise to an element   
\begin{equation}\label{limit RS element} \varepsilon^{\rm RS}_{\KK_\infty,S,T} := (\varepsilon^\Sigma_{F/K,S,T})_{F \in \Omega(\KK_\infty)}\end{equation}
of ${\bigwedge}^{r}_{\CC_p\cdot\Lambda(\cG_\infty)}(\CC_p\cdot\co^\times_{\KK_\infty,S,T})$.

We can now formulate an explicit main conjecture of non-commutative $p$-adic Iwasawa theory for $\mathbb{G}_m$ relative to $\KK_\infty/K$.   

\begin{conjecture}\label{hrncmc}{\em (Higher Rank Non-commutative Main Conjecture for $\mathbb{G}_m$) } Fix $\KK_\infty/K$ and $S$ as above and set $\Sigma := \Sigma_S(\KK_\infty)$ and $r := |\Sigma|$. 
Then one has 
\[ {\rm Nrd}_{Q(\mathcal{G}_\infty)}(\K_1(\Lambda(\mathcal{G}_\infty)))\cdot \varepsilon^{\rm RS}_{\KK_\infty,S,T} = \Theta^\Sigma_{\KK_\infty,S,T}({\rm d}_{\Lambda(\mathcal{G}_\infty)}(C_{\KK_\infty,S,T})^{\rm pb})\]
in
${\bigcap}_{\Lambda(\mathcal{G}_\infty)}^r\mathcal{O}_{\KK_\infty,S,T}^\times$. 
\end{conjecture}

\begin{remark}\label{mc to rs}{\em The validity of Conjecture \ref{hrncmc} combines with the definition (\ref{infinite bidual}) (with $a=r$) to imply that for every $F$ in $\Omega(\KK_\infty)$ the element $\varepsilon^\Sigma_{F/K,S,T}$ belongs to ${\bigcap}_{\ZZ_p[\G_{F}]}^{r}\co_{F,S,T,p}^\times$ and hence implies the validity of Conjecture \ref{MRSconjecture} for $F/K,S$ and $T$. (Here we recall, from Remark \ref{ind v_0}(iii), that for any field $F$ in $\Omega(\KK_\infty)$ for which $\Sigma\not= \Sigma_S(F)$ 
the element $\varepsilon^\Sigma_{F/K,S,T}$ vanishes.) }\end{remark}

\begin{remark}{\em If $\KK_\infty/K$ is abelian, then \cite[Th. 3.5]{bks2} implies Conjecture \ref{hrncmc} is equivalent to the `higher rank main conjecture' of Iwasawa theory formulated by Kurihara and the present authors in \cite[Conj. 3.1]{bks2}. In particular, the argument of  \cite[Cor. 5.6]{bks2} shows that if $K = \QQ$ and $\KK_\infty/K$ is abelian, then the validity of Conjecture \ref{hrncmc} follows as a consequence of the classical Iwasawa main conjecture in this setting, as proved by Mazur and Wiles.}\end{remark}

\begin{remark}\label{general mc remark}{\em Following the general approach of Coates et al in \cite{cfksv}, we let $\KK'_\infty$ be any compact $p$-adic Lie extension of $K$ in $\KK$ that is ramified at only finitely many places and also contains an intermediate field $\K_\infty$ that is Galois over $K$ and such $\Gal(K_\infty/K)$ is topologically isomorphic to $\ZZ_p$. Then $\KK'_\infty$ is equal to the union of all compact $p$-adic Lie extensions $\KK_\infty$ of rank one of $K$ in $\KK'_\infty$ and hence, by taking the limit of Conjecture \ref{hrncmc} over all such extensions $\KK_\infty/K$, one can formulate a `main conjecture' for $\mathbb{G}_m$ relative to the extension $\KK'_\infty/K$. } \end{remark}


\subsection{Evidence for the rank zero case}\label{rk 0 section} In this section we show that Conjecture \ref{hrncmc} generalizes to arbitrary rank (of Euler systems) the standard formulation of main conjectures in non-commutative Iwasawa theory, 
and thereby deduce the validity of an appropriate component of Conjecture \ref{hrncmc} for an important class of extensions.

\subsubsection{}Before stating the next result we recall that any object $C$ of the category  $\Der^{\rm perf}(\Lambda(\G_\infty))$ for which  $Q(\G_\infty)\otimes_{\La(\G_\infty)}C$ is acyclic gives rise to a canonical Euler characteristic element $\chi_{\La(\G_\infty)}^{\rm ref}(C)$ in $\K_0(\La(\G_\infty),Q(\G_\infty))$.

\begin{proposition}\label{translate prop} Let $\epsilon$ be a
non-zero idempotent  of $\zeta(\Lambda(\G_\infty))$ such that the cohomology
groups of $\epsilon\cdot C_{\KK_\infty,S,T}$ are torsion $\Lambda(\G_\infty)$-modules.

Then $\Sigma_S(\KK_\infty)$ is empty (so that $r = 0$) and the $\epsilon$-component
of Conjecture \ref{hrncmc} is valid if and only if there exists an
element $\lambda$ of $\K_1(Q(\G_\infty)\epsilon)$ with the following two
properties:
\begin{itemize}
\item[(i)] the canonical connecting homomorphism
$\K_1(Q(\G_\infty)\epsilon) \to
\K_0(\Lambda(\G_\infty)\epsilon,Q(\G_\infty)\epsilon)$ sends $\lambda$ to  
$\chi_{\Lambda(\G_\infty)\epsilon}^{\rm ref}(\epsilon\cdot C_{\KK_\infty,S,T})$;
\item[(ii)] the reduced norm map of the semisimple algebra $Q(\G_\infty)\epsilon$ sends $\lambda$ to
$\epsilon\cdot\theta_{\KK_\infty,S,T}(0)$.
\end{itemize}
\end{proposition}

\begin{proof} Lemma \ref{complex construction} implies, under Hypothesis \ref{tf hyp}, the complex
$C_{\KK_\infty,S,T}$ is isomorphic in $\Der^{\rm perf}(\La(\G_\infty))$ to a complex $P^\bullet$ of
the form 
\[\La(\G_\infty)^d \xrightarrow{\phi} \La(\G_\infty)^d,\]
where the first term is placed in degree zero (for a detailed construction of such a complex see, for example, the proof of Proposition \ref{limitomac} below). We write
\[ \underline{b}_\bullet := \{b_{i}\}_{i\in [d]}\]
for the standard basis of $\La(\G_\infty)^d$.

Then, since
the cohomology groups of $C := \epsilon\cdot C_{\KK_\infty,S,T}$
are assumed to be torsion $\Lambda(\G_\infty)$-modules,
the matrix 
\[ M := \bigl(\epsilon\cdot(b_i^\ast\circ \phi)(b_j)\bigr)_{i,j \in [d]}\]
belongs to ${\rm M}_{d}(\La(\G_\infty))\cap {\rm GL}_d(Q(\G_\infty)\epsilon)$ and its class $\langle M\rangle$ in $\K_1(Q(\G_\infty)\epsilon)$ is a pre-image of $\chi^{\rm ref}_{\La(\G_\infty)\varepsilon}(C)$ under the connecting homomorphism in claim (i).

The long exact sequence of relative $K$-theory therefore implies that the stated conditions (i) and (ii) are equivalent to asserting that
\[ {\rm Nrd}_{Q(\mathcal{G}_\infty)\epsilon}(\K_1(\Lambda(\mathcal{G}_\infty)\epsilon))\cdot
\theta_{\KK_\infty,S,T}(0) = \{ {\rm Nrd}_{Q(\mathcal{G}_\infty)\epsilon}(u \cdot \langle M\rangle): u \in \K_1(\La(\G_\infty)\epsilon)\}. \]

Remark \ref{r_S=0 example}(ii) implies that the left hand side of this equality is equal to the $\epsilon$-component of the left hand side of the equality in Conjecture \ref{hrncmc}. To complete the proof it is thus enough to show that
\[ \{ {\rm Nrd}_{Q(\mathcal{G}_\infty)\epsilon}(u \cdot \langle M\rangle): u \in \K_1(\La(\G_\infty)\epsilon)\} = \Theta({\rm d}_{\Lambda(\mathcal{G}_\infty)\epsilon}(C)^{\rm pb}),\]
where we write $\Theta$ in place of $\Theta^\Sigma_{\KK_\infty,S,T}$.
Since the description (\ref{rn description}) of the reduced norm of $Q(\G_\infty)$ combines with the argument of 
\S\ref{pb section} to imply
\[ \Theta({\rm d}_{\Lambda(\mathcal{G}_\infty)\epsilon}( C)^{\rm pb})
 = {\rm Nrd}_{Q(\mathcal{G}_\infty)\epsilon}(\K_1(\La(\G_\infty)\epsilon)) \cdot\Theta(x(\underline{b}_\bullet)),\]
where the tuple $x(\underline{b}_\bullet)$ is constructed using $\underline{b}$ as the ordered basis of both non-zero terms of $P^\bullet$, it is therefore enough to show  $\Theta(x(\underline{b}_\bullet) )$ is equal to ${\rm Nrd}_{Q(\mathcal{G}_\infty)\epsilon}(\langle M\rangle)$.

To prove this we first apply (\ref{rn description}) to the matrix $M$ to deduce that
\[ {\rm Nrd}_{Q(\mathcal{G}_\infty)\epsilon}(\langle M\rangle) = \bigl( {\rm Nrd}_{\QQ_p[\cG_L]\epsilon_L}(M_L)\bigr)_{L \in \Omega(\KK_\infty)}\]
with 
\[ M_L:= \bigl(\epsilon_L\cdot(b_{L,i}^\ast\circ \phi_L)(b_{L,j})\bigr)_{i,j \in [d]}\in {\rm M}_d(\ZZ_p[\cG_L]),\]
where the standard basis elements $b_{L,i}^\ast$ and $b_{L,j}$ are as specified in (\ref{standard basis notation}), $\epsilon_L$ is the image of $\epsilon$ in
$\zeta(\ZZ_p[\cG_L])$ 
and $\phi_L$ is the endomorphism of $\ZZ_p[\G_L]^d$
induced by $\phi$. In addition, for $L$ in $\Omega(\KK_\infty)$ the equality (\ref{rn rel}) implies
\begin{equation*}\label{finite level} \epsilon_L\cdot\bigl(\wedge_{i\in [d]}(b_{L,i}^\ast\circ \phi_L)\bigr)\bigl(\wedge_{j\in [d]}b_{L,j}\bigr)= {\rm Nrd}_{\QQ_p[\cG_L]\epsilon_L}(M_L).\end{equation*}

It is thus enough to show the left hand side of this expression is equal to the image $y_L$ of $\Theta(\epsilon\cdot x(\underline{b}_\bullet))$ under the projection map 
\[{\bigcap}_{\Lambda(\mathcal{G}_\infty)}^0\mathcal{O}_{\KK_\infty,S,T}^\times \to {\bigcap}^0_{\ZZ_p[\G_L]}\mathcal{O}_{L,S,T,p}^\times = \xi(\ZZ_p[\cG_L]).\]
To verify this we note $\epsilon_L\cdot C_{L,S,T,p}$ is isomorphic in $\Der^{\rm perf}(\ZZ_p[\G_L])$ to the complex
\[ P^\bullet_L := \epsilon_L(\ZZ_p[\G_L]\otimes_{\La(\G_\infty)}P^\bullet)\]
 and hence that
 $(y_L,0)$ is the image of $\epsilon_L\cdot x(\underline{b}_\bullet)_L$ under the canonical morphism
 \[ {\rm d}_{\ZZ_p[\G_L]\epsilon_L}(P^\bullet_L) \subset {\rm d}_{\QQ_p[\G_L]\epsilon_L}(\QQ_p\otimes_{\ZZ_p}P^\bullet_L) \cong {\rm d}_{\QQ_p[\G_L]\epsilon_L}(0) = (\zeta(\QQ_p[\G_L]),0)\]
 induced by the acyclicity of $\QQ_p\otimes_{\ZZ_p}P^\bullet_L$. The claimed result is thus true because the latter morphism sends the element 
 \[ \epsilon_L\cdot x(\underline{b}_\bullet)_L = \epsilon_L\cdot ( ({\wedge}_{j\in [d]}b_{L,j})\otimes ({\wedge}_{i\in [d]}b_{L,i}^\ast),0)\]
to $(\epsilon_L\cdot z_L,0)$ with
 \[z_L := ({\wedge}_{i\in [d]}b_{L,i}^\ast)({\wedge}_{j\in [d]}(\phi_L(b_{L,j}))) = ({\wedge}_{i\in [d]}(b_{L,i}^\ast\circ \phi_L))({\wedge}_{j\in [d]}b_{L,j}).\]
  \end{proof}

 %
%
%

\subsubsection{}In the sequel we write $L^+$ for the maximal totally real subfield of a number field $L$ and $\mu_p(L)$ for the $p$-adic cyclotomic $\mu$-invariant of $L$.

\begin{corollary}\label{hrmc true} Assume $K$ is totally real, $L$ is CM and
$\KK_\infty = L^{\rm cyc}$. Write  $\epsilon$ for the idempotent $(1-\tau)/2$ of $\zeta(\La(\G_\infty))$, where $\tau$ is the (unique) non-trivial element of $\Gal(\KK_\infty/\KK_\infty^+)$.

Then, if $\mu_p(L)$ vanishes, the $\epsilon$-component of Conjecture \ref{hrncmc} is valid for the data $\KK_\infty/K$, $S = S_{\KK_\infty/K}$ and any auxiliary set of places $T$.\end{corollary}

\begin{proof} Set $\G := \G_\infty$ and write $\Lambda(\mathcal{G})^\#(1)$ for the (left) $\Lambda(\mathcal{G})$-module $\Lambda(\mathcal{G})$ endowed
with the action of $G_K$ whereby each element $\sigma$ acts as right
multiplication by $\chi_K(\sigma)\cdot\bar{\sigma}^{-1}$ where $\chi_K: G_K\to \ZZ_p^\times$ is the cyclotomic character and $\bar{\sigma}$
the image of $\sigma$ in $\mathcal{G}$. We then write $C'$ and $C$ for the respective complexes $\epsilon\cdot \DR\Gamma_{{\rm \acute e
t},T}(\mathcal{O}_{F,S},\Lambda(\mathcal{G})^\#(1))$ and $\epsilon\cdot C_{\KK_\infty,S,T}$.

Then, since $\epsilon(Y_{F,S_K^\infty,p})$ vanishes for all $F$ in $\Omega(\KK_\infty)$, the complex $Q(\G)\otimes_{\La(\G)}C$ is acyclic and the Artin-Verdier Duality theorem implies the existence of a canonical isomorphism $C\cong C'[1]$ in 
$\Der^{\rm perf}(\La(\G))$ and hence of an equality  in $\K_0(\La(\G)\epsilon,Q(\G)\epsilon)$
\[ \chi_{\La(\G)\epsilon}^{\rm ref}(C) = - \chi_{\La(\G)\epsilon}^{\rm ref}(C').\]

We next note that $K^{\rm cyc}\subseteq \KK_\infty$ and
use a fixed choice of topological generator $\gamma_K$
of $\Gamma_K$ to identify (via the association $\gamma_K -1 \leftrightarrow t$) the Iwasawa algebra $\ZZ_p[[\Gamma_K]]$ with a power series ring in one variable  $\ZZ_p[[t]]$. We also write $A(\G)$ for the set of irreducible $\QQ_p^c$-valued characters of $\G$ that have open kernel.

We recall that for each $\chi$ in $A(\G)$ Ritter and Weiss  have in \cite[Prop. 6]{RitterWeiss} constructed a canonical homomorphism
\[ j_\chi: \zeta(Q(\G))^\times \to (\mathbb{Q}_p^c\otimes_{\QQ_p} Q(\ZZ_p[[t]]))^\times.\]
The Weierstrass Preparation Theorem combines with \cite[Prop. 5(3)]{RitterWeiss} to imply an element $x$ of $\zeta(Q(\G))^\times$ is uniquely determined by the value $j_\chi(x)(0)$ of  $j_\chi(x)$ at $t=0$ for every $\chi$ in $A(\G)$.
  In addition, if $V_\chi$ is a $\QQ_p^c[[\G]]$-module of character $\chi$ and $x$ belongs to $\zeta(Q(\G))\cap  \zeta(\QQ_p[[\G]])$, then $j_\chi(x)(0)$ is equal to the $\chi$-component of the image of $x$ in $\zeta(\QQ_p[\Gal(\KK_\infty^{\ker(\chi)}/K)])$. (This can be verified by an explicit computation and relies on the fact that the elements $\gamma_\chi$ and $e_\chi$ occurring in \cite[Prop. 6]{RitterWeiss} act trivially on $V_\chi$.) In particular,
  since $\theta_{\KK_\infty,S,T}$ belongs to $\zeta(Q(\G))\cap  \zeta(\QQ_p[[\G]])$ (as a consequence, for example, of \cite[Prop. 11 and the proof of Th. 8]{RitterWeiss}), one finds that if $-\tau$ acts as the identity on $V_\chi$, then
\begin{equation}\label{stick interp} j_\chi(\epsilon\cdot\theta_{\KK_\infty,S,T})(0) = L_{S,T}(\check\chi,0) = L_{p,S,T}(\check\chi\cdot\omega_K,0),\end{equation}
where $\omega_K$ is the Teichm\"uller character of $K$ and $L_{p,S,T}(\check\chi\cdot\omega_K,z)$ the $S$-truncated $T$-modified Deligne-Ribet $p$-adic Artin $L$-series of $\check\chi\cdot\omega_K$, as discussed by Greenberg in \cite{greenberg}.

Next we note that $j_\chi$ is related to the map 
\[ \Phi_{\G,\chi}: \K_1(Q(\G)) \to (\QQ_p^c\otimes_{\QQ_p}Q(\ZZ_p[[t]]))^\times\]
defined by Coates et al in \cite{cfksv} by virtue of the fact (proved in \cite[Lem. 3.1]{mcrc}) that for every element $\lambda$ of $\K_1(Q(\G))$ one has
\[\Phi_{\G,\chi}(\lambda) = j_\chi({\rm Nrd}_{Q(\G)}(\lambda)).\]

Given this, the above observations combine to imply an element $\lambda$ of $\K_1(Q(\G))$ satisfies the conditions (i) and (ii) in
 Proposition \ref{translate prop} if and only if the connecting homomorphism $\K_1(Q(\G)) \to \K_0(\La(\G),Q(\G))$
sends $\lambda$ to $-\chi_{\La(\G)\epsilon}^{\rm ref}(C')$ and, in addition, for every $\chi$ in $A(\G)$ one has 
\[\Phi_{\G,\chi}(\lambda)(0) = L_{p,S,T}(\check\chi\cdot\omega_K,0).\]%

In view of Proposition \ref{translate prop}, it is therefore enough to note that if $\mu_p(L)$ vanishes, then the existence of such an element $\lambda$ is deduced in \cite[Prop. 7.1]{burns2} from the proof (under the given assumption on $\mu_p(L)$) of the main conjecture of non-commutative Iwasawa theory for totally real fields, due to Ritter and Weiss \cite{RWnew} and, independently, Kakde \cite{kakde}.
\end{proof}

\begin{remark}\label{jn remark} {\em In \cite{JN3} Johnston and Nickel identify  families of (non-abelian) Galois extensions $L/K$ for which one can prove the main conjecture of non-commutative $p$-adic Iwasawa theory for $L^{\rm cyc}/K$ without assuming $\mu_p(L)$ vanishes (or that $p$ does not divide $[L:K]$). In all such cases the argument of Corollary \ref{hrmc true} shows that the $(1-\tau)/2$-component of Conjecture \ref{hrncmc} is valid for the data $L^{\rm cyc}/K$, $S = S_{\KK_\infty/K}$ and any auxiliary set of places $T$. }\end{remark}

\section{Canonical resolutions and semisimplicity in Iwasawa theory}\label{can res semi section}

\subsection{Canonical resolutions in Iwasawa theory}\label{can res}

In this section we fix a compact $p$-adic Lie extension
$\KK_\infty$ of $K$ of strictly positive rank for which $S_{\rm ram}(\KK_\infty/K)$ is finite and non-empty, and set $\G_\infty := \Gal(\KK_\infty/K)$. 

We also fix a finite set $S$ of places of $K$ that contains the set $S_{\KK_\infty/K}$ specified in (\ref{can places}) 
 and a finite set of places $T$ of $K$
that is disjoint from $S$ and such that Hypothesis \ref{tf hyp} is satisfied (with $\KK$ taken to be $\KK_\infty$).

Then, since $S_{\rm ram}(\KK_\infty/K)$ is non-empty, one has $\Sigma_S(\KK_\infty)\not= S$ and so we can fix a place $v'$ in $S\setminus \Sigma_S(\KK_\infty)$. We then set 
\[ S' := S\setminus \{v'\}.\] 
%

\subsubsection{}\label{Selmer modules} We first describe an important aspect of the descent
properties of transpose Selmer modules.

\begin{lemma}\label{selmer lemma} Fix $L$ in $\Omega(\KK_\infty)$, with $G := \G_L$, and a
normal subgroup $H$ of $G$ with 
$E := L^H$. 
%
%
%
Consider the composite surjective homomorphism of $G/H$-modules
$$\beta_{E,S}: {\rm Sel}_S^T(E)^{{\rm tr}}\xrightarrow{\varrho_{E,S}} X_{E,S}\xrightarrow{\alpha_{E,S,S'}} Y_{E,S'},$$
in which $\varrho_{E,S}$ comes from (\ref{selmer lemma seq2}) and
$\alpha_{E,S,S'}$ is the natural projection map. 

Then the image of $\ker(\beta_{L,S})$ under the composite homomorphism
\[ {\rm Sel}_S^T(L)^{{\rm tr}} \to \ZZ[G/H]\otimes_{\ZZ[G]}{\rm Sel}_S^T(L)^{{\rm tr}}\cong
{\rm Sel}_S^T(E)^{{\rm tr}},\]
in which the isomorphism is as in Lemma \ref{complex construction}(iv), is equal to $\ker(\beta_{E,S})$.
\end{lemma}

\begin{proof} 

The claimed result is obtained directly by applying the Snake Lemma
 to the following exact commutative diagram
\begin{equation*}
\begin{CD}
& & H_0(H,\ker(\beta_{L,S})) @>  >>  H_0(H,{\rm Sel}_S^T(L)^{\rm tr}) @>  \beta_{L,S} >> H_0(H,Y_{L,S'}) @>  >> 0\\
@.  @V VV @V VV @V VV @.\\
 0 @> >> \ker(\beta_{E,S}) @>  >> {\rm Sel}_S^T(E)^{\rm tr} @>  \beta_{E,S} >> Y_{E,S'} @> >> 0.
\end{CD}\end{equation*}
Here the lower row is the tautological exact sequence, the upper row is obtained
by taking $H$-coinvariants of the analogous tautological
exact sequence, the right-hand vertical arrow is the canonical isomorphism that
 maps the class of each place $w$ in $S'_L$ to the restriction of $w$ to $E$, the
 middle vertical arrow is the canonical isomorphism coming from Lemma \ref{complex construction}(iv) and the left hand
 vertical arrow is induced by the commutativity of the second square. \end{proof}

\subsubsection{}
%
In the sequel we will use the $\La(\G_\infty)$-modules that are defined by the inverse limits 
\[ \mathcal{S}_S^T(\KK_\infty)^{\rm tr} := {\varprojlim}_{F \in \Omega(\KK_\infty)}{\rm Sel}_S^T(F)^{\rm tr}_p\]
and, if Hypothesis \ref{tf hyp} is satisfied, also 
\[ \mathcal{S}_S^T(\KK_\infty) := {\varprojlim}_{F \in \Omega(\KK_\infty)}{\rm Sel}_S^T(F)_p.\]
Here the respective transition morphisms for $F\subset F'$ are the composite maps 
\[ {\rm Sel}_S^T(F')^{\rm tr}_p \to \ZZ_p[\G_F]\otimes_{\ZZ_p[\G_{F'}]}{\rm Sel}_S^T(F')^{\rm tr}_p\cong
 {\rm Sel}_S^T(F)^{\rm tr}_p,\]
and 
\[ {\rm Sel}_S^T(F')_p \to \ZZ_p[\G_F]\otimes_{\ZZ_p[\G_{F'}]}{\rm Sel}_S^T(F')_p\cong
 {\rm Sel}_S^T(F)_p,\]
 where the first map in both cases is the obvious projection and the second is the isomorphism given by Lemma \ref{complex construction}(iv), respectively Remark \ref{non-tranpose remark} (if Hypothesis \ref{tf hyp} is satisfied).

In the following result we also use the object $C_{\KK_\infty,S,T}$ of $\Der^{\rm perf}(\La(\G_\infty))$
defined in \S\ref{ncmc section}.

\begin{proposition}\label{limitomac} Set $n := |S|-1$. Then there exists a natural number $d$ with $d \ge n$ and a
canonical family of complexes of $\La(\G_\infty)$-modules $C(\phi)$ of the form
\[ \La(\G_\infty)^d \xrightarrow{\phi} \La(\G_\infty)^d\]
in which the first term is placed in degree zero and the
following properties are satisfied.

\begin{itemize}
\item[(i)] $C(\phi)$ is isomorphic in $\Der^{\rm perf}(\La(\G_\infty))$ to $C_{\KK_\infty,S,T}$.

\item[(ii)] If $C(\tilde\phi)$ is any other complex in the family, then $\tilde\phi =
\eta\circ \phi \circ (\eta')^{-1}$ where $\eta$ and $\eta'$ are
automorphisms of $\La(\G_\infty)^d$ and $\eta$ is represented,
with respect to the standard basis of $\La(\G_\infty)^d$, by a block matrix of the form (\ref{block}).

\item[(iii)] For each $L$ in $\Omega(\KK_\infty)$ we set $\mathfrak{S}_L := {\rm Sel}_S^T(L)^{\rm tr}_p$ and write $\phi_L$, $\iota_L$ and $\varpi_L$ respectively for the endomorphism of $\ZZ_p[\G_L]^d$ induced by $\phi$ and
  the embedding $\mathcal{O}_{L,S,T,p}^\times \cong \ker(\phi_L) \subseteq
 \ZZ_p[\G_L]^d$ and surjection $\ZZ_p[\G_L]^d \to {\rm cok}(\phi_L) \cong \mathfrak{S}_L$ that are induced by the descent isomorphism
\[ C_{L,S,T,p}\cong \ZZ_p[\G_L]\otimes^{\DL}_{\La(\G_\infty)}C_{\KK_\infty,S,T}
\cong \ZZ_p[\G_L]\otimes_{\La(\G_\infty)}C(\phi).\]
Then there exists a natural number $d_L$ with $n \le d_L \le d$ and a commutative diagram of $\ZZ_p[\G_L]$-modules of the form
\begin{equation}\label{limitomacdiagram}\minCDarrowwidth0.2em\begin{CD}
0 @> >> \mathcal{O}_{L,S,T,p}^\times @> \iota_{L} >> \ZZ_p[\G_L]^d
@> \phi_L >> \ZZ_p[\G_L]^d @> \varpi_L >> \mathfrak{S}_L @> >> 0\\
@. @\vert @VV \kappa_L' V @VV \kappa_L V @\vert\\
0 @> >> \mathcal{O}_{L,S,T,p}^\times @> (\hat\iota_L,0) >> \ZZ_p[\G_L]^{d_L+ (d-d_L)} @> (\hat\phi_L,{\rm id}) >> \ZZ_p[\G_L]^{d_L + (d-d_L)} @> (\hat\pi_L,0) >> \mathfrak{S}_L @> >> 0.\end{CD}
\end{equation}
Here $\kappa_L'$ and $\kappa_L$ are bijective, the matrix of $\kappa_L$ with respect to the standard basis of $\ZZ_p[\G_L]^d = \ZZ_p[\G_L]^{d_L}\oplus \ZZ_p[\G_L]^{(d-d_L)}$ has the form
 $\left(
\begin{array} {c|c} I_n& 0\\
                        \hline
                              0 & \ast\end{array}\right)$
and the exact sequence
\begin{equation}\label{critical} 0 \to \mathcal{O}_{L,S,T,p}^\times \xrightarrow{\hat\iota_L} \ZZ_p[\G_L]^{d_L} \xrightarrow{ \hat\phi_L} \ZZ_p[\G_L]^{d_L} \xrightarrow{\hat\pi_L} \mathfrak{S}_L \to 0\end{equation}
is constructed as in diagram (\ref{very important diagram}).
\end{itemize}
\end{proposition}

\begin{proof} We set $\Omega_\infty := \Omega(\KK_\infty)$, $R_\infty := \La(\G_\infty)$ and 
$\mathfrak{S}_\infty := \mathcal{S}_S^T(\KK_\infty)^{\rm tr}$. For $v \in S$  we set 
$ Y_{\infty,v} := {\varprojlim}_{F\in\Omega_\infty}Y_{F,\{v\},p},$ where the
transition maps are the natural projection maps. We then set
$Y_{\infty,S'} := {\bigoplus}_{v \in S'}Y_{\infty,v}$ and consider the homomorphism
\[ \beta_{\infty}:= {\varprojlim}_{F\in \Omega_\infty}\beta_{F,S,p}:
\mathfrak{S}_\infty\to Y_{\infty,S'},\]
where $\beta_{F,S}$ is the map of $\G_F$-modules
 in Lemma \ref{selmer lemma}. We note $\beta_\infty$
 is surjective since each map $\beta_{F,S,p}$ is surjective
  and each module
 $\ker(\beta_{F,S,p})$ is compact.

For $i \in [n]$ we write $v_i$ for the $i$-th element of $S'$ (with respect to the ordering induced by (\ref{ordering})). For each such $i$, we fix a place $w_{i,\infty}$ of $\KK_\infty$ above $v_i$ and
write $\pi_{1,\infty}: R_\infty^n \to Y_{\infty,S'}$
for the
 surjective map of $R_\infty$-modules that sends the $i$-th element
  in the standard basis of $R_\infty^n$ to $w_{i,\infty}$.
 We choose a lift
 \[ \pi'_{1,\infty}: R_\infty^n \to
\mathfrak{S}_\infty\]
of
$\pi_{1,\infty}$ through $\beta_{\infty}$.

The algebra $R_\infty$ is semiperfect since it is both
semilocal and complete with respect to its Jacobson radical (as $\G_\infty$ is a compact $p$-adic Lie group). In view of
 \cite[Th. (6.23)]{curtisr}, we may therefore fix a projective cover
\[ \pi_{2,\infty}: P\to\ker(\beta_{\infty})\]
of $\ker(\beta_{\infty}) =
 \varprojlim_{F \in \Omega_\infty}\ker(\beta_{F,S,p})$.

We next choose a $R_\infty$-module $P'$ such that $P\oplus P'$ is a
 free module of minimal rank, $n'$ say, fix an isomorphism $j: P\oplus P' \cong
 R_\infty^{n'}$ and write $\pi'_{2,\infty}$ for the map $(\pi_{2,\infty},0)\circ j^{-1}$ on
 $R_\infty^{n'}$. We set $d := n + n'$ and consider the
 homomorphism of $R_\infty$-modules
\begin{equation}\label{infty map} \pi_\infty: R_\infty^d = R_\infty^n \oplus R_\infty^{n'}
\xrightarrow{(\pi'_{1,\infty},\pi'_{2,\infty})}  \mathfrak{S}_\infty.\end{equation}

For $L$ in $\Omega_\infty$ we set 
\[ R_L := \ZZ_p[\G_L],\,\, U_L := \mathcal{O}_{L,S,T,p}^\times \quad\text{and}\quad \mathcal{H}_L := \Gal(\KK_\infty/L).\] 
Then the $\mathcal{H}_L$-coinvariants of $\pi_\infty$
gives a surjective homomorphism $\varpi_L:R_L^d \to \mathfrak{S}_L$ and, for each $L$ and $L'$ in $\Omega_\infty$ with $L\subset L'$, the argument of \cite[Prop. 3.2]{omac}
allows us to fix an exact commutative diagram of the form
\begin{equation}\label{very important diagram2}\begin{CD}
0 @> >> U_{L'} @> \hat\iota_{L'} >> R_{L'}^d @>
\phi_{L'} >> R_{L'}^d @> \varpi_{L'} >> \mathfrak{S}_{L'} @> >> 0\\
@. @. @VV\omega^0_{L'/L}V @VV \omega^1_{L'/L} V @VV \omega_{L'/L}V\\
0 @> >> U_L @> \hat\iota_{L} >> R_L^d
@> \phi_L >> R_L^d @> \varpi_L >> \mathfrak{S}_L @> >> 0.\end{CD}\end{equation}
Here $\omega^1_{L'/L}$ and $\omega_{L'/L}$ are the natural projection maps, $\omega^0_{L'/L}$ sends each element
$b_{L',i}$ in the standard basis of $R_{L'}^d$ to any choice of element $x_i$ with 
\[ \phi_L(x_i) =
\omega^1_{L'/L}(\phi_{L'}(b_{L',i}))\]
and the following property is satisfied: if $F$ denotes either $L$ or $L'$, then the complex $C(\phi_{F})$ given by
$R_F^d \xrightarrow{\phi_F} R_F^d$,
where the first module is placed in degree zero and the cohomology groups are identified
with $U_F$ and $\mathfrak{S}_F$ by means of the maps in
the respective row of the diagram, then there exists an isomorphism
$C(\phi_F) \cong  C_{F,S,T,p}$ in $\Der^{\rm perf}(R_F)$ that induces the
identity map on both
cohomology groups. With these identifications the canonical descent isomorphism
 $R_L\otimes^\DL_{R_{L'}}C_{L',S,T,p}\cong C_{L,S,T,p}$ implies that the morphism
  $R_L\otimes_{R_{L'}}C(\phi_{L'})\cong C(\phi_L)$ induced
   by the maps $\omega^0_{L'/L}$ and $\omega^1_{L'/L}$ is a
   quasi-isomorphism. Since $\omega^1_{L'/L}$ is surjective, this in turn implies that the
    map $\omega^0_{L'/L}$ is surjective and hence that the limit
    $\varprojlim_{F \in \Omega_\infty}R_F^d$ with respect to the transition morphisms
    are $\omega^0_{F'/F}$ is isomorphic to $R_\infty^d$.

The limit of the complexes $\{C(\phi_L)\}_{L\in \Omega(\KK_\infty)}$
with respect to the morphisms in (\ref{very important diagram2}) is therefore
a complex of the form
 $R_\infty^d\xrightarrow{\phi} R_\infty^d$ that is isomorphic in $\Der^{\rm perf}(R_\infty)$
 to $C_{\KK_\infty,S,T,p}$. Thus, denoting this complex by $C(\phi)$, claim (i) is clear. 
 

Turning to claim (ii), we note that if $\tilde\pi_\infty$ is any map defined in the same way as $\pi_\infty$ but with respect to a
  different choices either of projective cover $\pi_{2,\infty}$, isomorphism $j$
  or lift $\pi_{1,\infty}'$ of $\pi_{1,\infty}$, then an easy exercise shows that
   $\pi_\infty = \tilde\pi_{\infty}\circ \eta$, where $\eta$
   is an automorphism of $R_\infty^d$
   that is represented with respect to the standard basis $\{b_i\}_{i\in [d]}$
 of $R_\infty^d$ by a block
   matrix of the form (\ref{block}). (Here, with respect to the decomposition of
   $R_\infty^d$ used in (\ref{infty map}) we identify $b_i$ for
   $i\in [d]\setminus [n]$
 with the $(i-n)$-th element of the standard basis of $R_\infty^{n'}$.)

Then, if $C(\tilde\phi)$ is any complex obtained
in the same way by using $\tilde\pi_\infty$ rather than $\pi_\infty$, there exists a
quasi-isomorphism $\xi: C(\tilde\phi) \cong C(\phi)$ that is represented by an exact
commutative diagram of the form
\begin{equation}\label{limit independence}\begin{CD}
0 @> >> \mathcal{O}_{\KK_\infty,S,T}^\times @> \iota_\infty >> R_\infty^d @> \phi >> R_\infty^d @> \pi_\infty >> \mathfrak{S}_\infty @> >> 0\\
@. @\vert @VV \eta' V @VV \eta V @\vert\\
0 @> >> \mathcal{O}_{\KK_\infty,S,T}^\times @> >> R_\infty^d @> \tilde\phi
>> R_\infty^d @> \tilde\pi_\infty >> \mathfrak{S}_\infty @> >> 0\end{CD}
\end{equation}
To deduce claim (ii) it is thus sufficient to note that, since $\eta$ is
bijective, the Five Lemma implies that $\eta'$ is also bijective.

Finally, to prove claim (iii), we fix $L$ in $\Omega_\infty$ and set $\mathcal{H} := \mathcal{H}_L$. We note that the projection map $\ker(\beta_\infty)_\mathcal{H} \to \ker(\beta_{L,S,p})$ is surjective (by Lemma \ref{selmer lemma}(ii)) and hence that there exists a direct sum decomposition 
\[ P_\mathcal{H} = P_L \oplus Q_L\]
of $R_L$-modules so that $(\pi_{2,\infty})_\mathcal{H}$ is zero on $Q_L$ and restricts to $P_L$ to give a projective cover of $\ker(\beta_{L,S,p})$. We fix a projective $R_L$-module $P'_L$ of minimal rank so $P_L\oplus P'_L$ is a free $R_L$-module, write $n_L'$ for the rank of the latter module and set $d_L := n + n'_L$ and $\delta_L := n'- n_L'$. Then $d = d_L + \delta_L$ and the Krull-Schmidt Theorem implies $\delta_L \ge 0$ and that, for any choice of an isomorphism of $R_L$-modules $j': R_L^{n'_L} \cong P_L\oplus P'_L$, there exists an isomorphism
\[ \iota_L: Q_L\oplus P'_\mathcal{H}\to P'_L \oplus R_L^{\delta_L}\]
of $R_L$-modules and a commutative diagram of the form
\[ \begin{CD}
R_L^n\oplus R_L^{n'} @> ({\rm id},j_\mathcal{H}) >>  R_L^{n} \oplus (P\oplus P')_\mathcal{H}
@> (\pi'_{1,\infty},(\pi_{2,\infty},0))_\mathcal{H} >>   (\mathfrak{S}_\infty)_\mathcal{H}\\
@V \kappa_L VV @V ({\rm id},\iota'_L) VV @V VV \\
R_L^{n}\oplus R_L^{n_L'}\oplus R_L^{\delta_L} @> ({\rm id},j',{\rm id}) >> R_L^{n} \oplus (P_L\oplus P'_L) \oplus R_L^{\delta_L} @> ((\pi'_{1,\infty})_\mathcal{H},((\pi_{2,\infty})_\mathcal{H},0),0) >> \mathfrak{S}_L.\end{CD}\]
Here $j$ is the isomorphism $R_\infty^{n'} \cong P\oplus P'$ fixed just before (\ref{infty map}), $\iota'_L$ is the isomorphism

\[ (P\oplus P')_\mathcal{H} = P_L \oplus (Q_L \oplus P'_\mathcal{H}) \xrightarrow{({\rm id},\iota_L)} P_L\oplus P'_L\oplus R_L^{\delta_L},\]
$\kappa_L = ({\rm id},\tilde\kappa_L)$ with $\tilde\kappa_L$ the automorphism of $ R_L^{n'} = R_L^{n_L'}\oplus R_L^{\delta_L}$ given by $(j',{\rm id})^{-1}\circ \iota'_L\circ j_\mathcal{H}$ and the right hand vertical arrow is the isomorphism induced by Lemma \ref{selmer lemma}(i).

The upper and lower composite horizontal maps in this diagram are respectively equal to the map $\varpi_L$ in diagram (\ref{very important diagram2}) and to $\hat\pi\oplus 0$, where $\hat\pi$ is a surjective map  $R_L^{d_L}\to \mathfrak{S}_L$ constructed as in diagram (\ref{very important diagram}). Hence, if we fix an embedding $\hat\iota: U_L \to R_L^{d_L}$ and an endomorphism $\hat\phi$ of $R_L^{d_L}$ as in the upper row of (\ref{very important diagram}) (for this choice of $\hat\pi$), then there exists a commutative diagram of $R_L$-modules of the form
\begin{equation*}\begin{CD}
0 @> >> U_L @> \hat\iota_{L} >> R_L^d
@> \phi_L >> R_L^d @> \varpi_L >> \mathfrak{S}_L @> >> 0\\
@. @\vert @VV \kappa_L' V @VV \kappa_L V @\vert\\
0 @> >> U_L @> (\hat\iota,0) >> R_L^{d_L+ \delta_L} @> (\hat\phi,{\rm id}) >> R_L^{d_L\oplus \delta_L} @> (\hat\pi,0) >> \mathfrak{S}_L @> >> 0\\
@. @\vert @AA A @AA A @\vert\\
0 @> >> U_L @> \hat\iota >> R_L^{d_L} @> \hat\phi >> R_L^{d_L} @> \hat\pi >> \mathfrak{S}_L @> >> 0.
\end{CD}
\end{equation*}
Here the unlabelled vertical maps are the natural inclusions and so the construction of the sequence in (\ref{very important diagram}) implies that the upper two rows of this diagram represent the same element of the Yoneda Ext-group ${\rm Ext}_{R_L}^2(\mathfrak{S}_L,U_L)$. Given this, the existence of a map $\kappa_L'$ that makes the first and second upper squares commute (and hence is bijective) follows from the fact that $\kappa_L$ is bijective and that the upper third square commutes.
 Finally, we note that the upper part of the above diagram satisfies all of the assertions in  claim (iii).
\end{proof}

\begin{remark}\label{exp interp hrncmc}{\em If $\G_\infty$ has rank one, then the resolution $C(\phi)$ of $C_{\KK_\infty,S,T}$ constructed in Proposition \ref{limitomac} leads to the following explicit interpretation of
Conjecture \ref{hrncmc}. For $L$ in $\Omega(\KK_\infty)$ the complex
$C_{L,S,T,p}$ is isomorphic in $\Der^{\rm perf}(\ZZ_p[\G_L])$ to
\[ \ZZ_p[\G_L]^d\xrightarrow{\phi_L}\ZZ_p[\G_L]^d,\] 
where the first term is placed in degree zero, and the
 alternative description of $\Theta^\Sigma_{L,S,T,p}$ given in
 (\ref{exp proj}) implies that the image
  of $\Theta^\Sigma_{\KK_\infty,S,T}(x(\underline{b}_\bullet) )$ under the natural projection map
 ${\bigcap}_{\Lambda(\mathcal{G}_\infty)}^r\mathcal{O}_{\KK_\infty,S,T}^\times \to
 {\bigcap}^r_{\ZZ_p[\G_L]}\mathcal{O}_{L,S,T,p}^\times$ is 
 \[ \Theta^\Sigma_{L,S,T,p}(x(\underline{b}_\bullet)_L) =
 (\wedge_{i=r+1}^{i=d}(b_{L,i}^\ast\circ\phi_{L}))
 (\wedge_{j\in [d]}b_{L,j}).\]
The equality predicted in Conjecture \ref{hrncmc} is therefore valid
if and only if there exists
 an element $u$ in $\K_1(\La(\G_\infty))$ such that in ${\bigcap}_{\Lambda(\mathcal{G}_\infty)}^r\Lambda(\mathcal{G}_\infty)^d$ one has
\begin{equation}\label{fund eq}  \iota_{\infty,*}(\varepsilon^{\rm RS}_{\KK_\infty,S,T}) =
{\rm Nrd}_{Q(\G_\infty)}(u)\cdot\bigl( (\wedge_{i=r+1}^{i=d}(b_{L,i}^\ast\circ\phi_{L})(\wedge_{j\in [d]}b_{L,j})\bigr)_{L \in \Omega(\KK_\infty)},
 \end{equation}
 where $\iota_{\infty,\ast}$ is the homomorphism of $\xi_p(\KK_\infty/K)$-modules $ {\bigcap}_{\Lambda(\mathcal{G}_\infty)}^r\mathcal{O}_{\KK_\infty,S,T}^\times \to {\bigcap}_{\Lambda(\mathcal{G}_\infty)}^r\Lambda(\mathcal{G}_\infty)^d\,$ that is 
 induced by the injective map of $\Lambda(\mathcal{G}_\infty)$-modules 
 $\iota_\infty: \mathcal{O}_{\KK_\infty,S,T}^\times \to \Lambda(\mathcal{G}_\infty)^d$.}
 \end{remark}

\begin{remark}{\em The detailed descent properties established in Proposition \ref{limitomac}(iii) are finer than is strictly necessary for the present article. However, we have included the detail since it is required in the article \cite{bpss} of Puignau, Seo and the present authors in order to discuss non-commutative generalizations of the `refined class number formulas' conjectured independently by Mazur and Rubin in \cite{MR2} and by the second author in \cite{sano}. }\end{remark}

\subsection{Semisimplicity for Selmer modules}\label{semisection} The hypothesis that a finitely generated torsion module over the classical Iwasawa algebra be `semisimple at zero' allows one to make explicit descent computations even in the presence of trivial zeroes (see, for example, \cite{burns2}, though the notion arises in many earlier articles). 

In this section we introduce a generalization of the notion of semisimplicity in the context of Selmer modules that will play an important role in later sections.  

To do this we fix data $\KK_\infty/K$, $S$ and $T$ as in \S\ref{can res}. We assume throughout this section that $\G_\infty$ has rank one and fix a field $E$ in $\Omega(\KK_\infty)$ for which there is an isomorphism of topological groups 
\[ \mathcal{H}_\infty:= \Gal(\KK_\infty/E)\cong \ZZ_p.\]
We then also fix a subset $\Sigma$ of $\Sigma_S(E)$ such that   
\begin{equation}\label{sigma_S assumption} \begin{cases} \Sigma  = \Sigma_S(E), &\text{if $\Sigma_S(E) \not= S$},\\
\Sigma_S(\KK_\infty) \subseteq \Sigma\,\,\text{and}\,\, |\Sigma| = |S|-1, &\text{if $\Sigma_S(E) = S$}\end{cases}\end{equation}
and set 
\[ r := |\Sigma_S(\KK_\infty)|\,\,\text{ and }\,\, r' := |\Sigma|\]
(so that $r' \ge r$). 

\begin{remark}{\em The restrictions on $\Sigma$ given by (\ref{sigma_S assumption}) are motivated by the observations made in Remark \ref{ind v_0}(iii). If $\Sigma_S(E) = S$ (which occurs, for example, if $\KK_\infty$ is a $\ZZ_p$-extension of $E=K$), then there exists a choice of $\Sigma$ as above since $\Sigma_S(\KK_\infty) \not= S$ and the definitions and results in the rest of this section are independent of this choice.}\end{remark}

For any element $\gamma$
 of $\mathcal{G}_\infty$, we define an element of $\xi_p(\KK_\infty/K)$ by setting
 \[ \lambda(\gamma) := {\rm Nrd}_{Q(\G_\infty)}(\gamma-1)\in \xi_p(\KK_\infty/K).\]
 We also fix a
 topological generator $\gamma_E$
 of $\mathcal{H}_\infty$, and define an ideal of $\xi_p(\KK_\infty/K)$ by setting 
\[ I_E(\G_\infty):= \xi_p(\KK_\infty/K)\cdot \lambda(\gamma_E).\]

Finally, we write $x \mapsto x^\#$ for the $\QQ_p$-linear involution of $Q(\G_\infty)$ that is induced by inverting elements of $\G_\infty$. 

\begin{remark}\label{hash remark} {\em We record two important properties of the ideal $I_E(\G_\infty)$. \

\noindent{}(i) As the notation suggests, $I_E(\G_\infty)$ is independent of the choice of $\gamma := \gamma_E$ (and hence only depends on $\G_\infty$ and $E$). To see this note that any other topological generator of $\mathcal{H}_\infty$ is equal to $\gamma^a$ for some $a \in \ZZ_p^\times$ and that the corresponding quotient $x_a := (\gamma^a-1)/(\gamma-1)$ belongs to $\La(\mathcal{H}_\infty)^\times \subseteq \La(\G_\infty)^\times$. From the explicit description of reduced norm given in  (\ref{rn description}), it then follows that ${\rm Nrd}_{Q(\G_\infty)}(x_a)$ belongs to $\xi_p(\KK_\infty/K)^\times$, and hence that $\lambda(\gamma^a) = {\rm Nrd}_{Q(\G_\infty)}(x_a)\cdot\lambda(\gamma)$ generates $I_E(\G_\infty)$ over $\xi_p(\KK_\infty/K)$. 

\noindent{}(ii) One has $I_E(\G_\infty) = I_E(\G_\infty)^\#$. To see this one can combine (\ref{rn description}) with \cite[(3.4.1)]{bses} to deduce that, for each matrix $M = (M_{ij})$ in ${\rm M}_d(\La(\G_\infty))$ there is an equality 
\[ {\rm Nrd}_{Q(\G_\infty)}(M)^\# = {\rm Nrd}_{Q(\G_\infty)}(M^\#),\]
where $M^\#$ denotes the matrix $(M_{ij}^\#)$. These equalities imply that  $\xi_p(\KK_\infty/K) = \xi_p(\KK_\infty/K)^\#$. Since $(\gamma_E-1)^\# = (-\gamma_E^{-1})(\gamma_E-1)$ and ${\rm Nrd}_{Q(\G_\infty)}(-\gamma_E^{-1})\in \xi_p(\KK_\infty/K)^\times$, it then also follows that $I_E(\G_\infty) = I_E(\G_\infty)^\#$, as claimed. }\end{remark}

\subsubsection{}We start by defining an analogue for Iwasawa-theoretic Selmer modules of the
higher non-commutative Fitting invariants introduced in \cite[\S3.4]{bses}.

To do this, for each endomorphism $\phi$ of $\La(\G_\infty)^d$ we write $\mathfrak{G}_r(\phi)$ for the subset of
${\rm M}_{d}(\La(\G_\infty))$ comprising all matrices that are obtained by replacing the
elements in any selection of $r$ columns of the matrix of $\phi$ with respect to the standard basis of
 $\La(\G_\infty)^d$ by arbitrary elements of $\La(\G_\infty)$.

\begin{definition}\label{hfi def} {\em For each endomorphism $\phi$ of $\La(\G_\infty)^d$ constructed as in Proposition \ref{limitomac}, we define an ideal of $\xi_p(\KK_\infty/K)$ by setting }
\[ {\rm Fit}_{\La(\G_\infty)}^r(\mathcal{S}_S^T(\KK_\infty)) := 
\xi_p(\KK_\infty/K)\cdot \{{\rm Nrd}_{Q(\G_\infty)}(M): M\in  \mathfrak{G}_r(\phi)\}^\#.\] 
\end{definition}

The basic properties of this ideal are described in the following result. 

\begin{lemma}\label{prel semi} For each $\phi$ as above, the  following claims are valid.

\begin{itemize}
\item[(i)] ${\rm Fit}_{\La(\G_\infty)}^r(\mathcal{S}_S^T(\KK_\infty))$ depends only on the
 $\La(\G_\infty)$-module $\mathcal{S}_S^T(\KK_\infty)$.

\item[(ii)] ${\rm Fit}_{\La(\G_\infty)}^r(\mathcal{S}_S^T(\KK_\infty))$ is contained in
$I_E(\G_\infty)^{r'-r}$.
\end{itemize}
\end{lemma}

\begin{proof} The commutative diagram (\ref{limit independence}) shows that the collection of endomorphisms $\phi$ that are constructed via the approach in Proposition 
 \ref{limitomac} constitutes a distinguished family of free resolutions of the $\La(\G_\infty)$-module $\mathcal{S}_S^T(\KK_\infty)^{\rm tr}$. These endomorphisms are uniquely determined by their linear duals $\Hom_{\La(\G_\infty)}(\phi,\La(\G_\infty))$ which (Remark \ref{non-tranpose remark} implies) in turn constitute a distinguished family of resolutions of $\La(\G_\infty)$-module $\mathcal{S}_S^T(\KK_\infty)$. 
 
 To prove claim (i) it is therefore enough to show that ${\rm Fit}_{\La(\G_\infty)}^r(\mathcal{S}_S^T(\KK_\infty))$ is unchanged if one replaces $\phi$ by any other endomorphism constructed via Proposition \ref{limitomac}. 
 
To do this we write $M(\phi)$ for the matrix of $\phi$ with respect to the standard basis $\{b_i\}_{i \in [d]}$ of $\La(\G_\infty)^d$.
 Then, in view of the property (\ref{ordering 2}) (with $\KK$ taken to be $\KK_\infty$), the construction of the map (\ref{infty map}) implies that the $i$-th column of $M(\phi)$ is zero for every $i$ in $[r]$.
 This means that all non-zero contributions to the ideal ${\rm Fit}_{\La(\G_\infty)}^r(\mathcal{S}_S^T(\KK_\infty))$ arise from the reduced norms of block matrices of the form
  $\left(\!\!\!\begin{array}{c|c}N\!&M(\phi)^\dagger\end{array}\!\!\!\right)$, where
  $N$ is an arbitrary matrix in ${\rm M}_{d,r}(\La(\G_\infty))$ and for any
  matrix $M$ in ${\rm M}_d(\La(\G_\infty))$ we write
  $M^\dagger$ for the matrix in ${\rm M}_{d,d-r}(\La(\G_\infty))$ given by the last
  $d-r$ columns of $M$.

 Now if $\tilde\phi$ is any other endomorphism constructed as in Proposition \ref{limitomac}, then claim (ii) of that result implies $U\cdot M(\tilde\phi) = M(\phi)\cdot V$
 where $U$ and $V$ are matrices in ${\rm GL}_d(\La(\G_\infty))$ and $V$ is a
 block matrix of the form (\ref{block}). In particular, one has
 $U\cdot M(\tilde\phi)^\dagger = (M(\phi)\cdot V)^\dagger$ and, for any $N$ in
 ${\rm M}_{d,r}(\La(\G_\infty))$, also
\begin{align*}  U\cdot \left(\!\!\!\begin{array}{c|c}N\!&M(\tilde\phi)^\dagger\end{array}\!\!\!\right)
 =&\,\, \left(\!\!\!\begin{array}{c|c}U\cdot N\!&U\cdot M(\tilde\phi)^\dagger\end{array}\!\!\!\right)\\
  =&\,\, \left(\!\!\!\begin{array}{c|c}U\cdot N\!&(M(\phi)\cdot V)^\dagger\end{array}\!\!\!\right)\\
  =&\,\, \left(\!\!\!\begin{array}{c|c}U\cdot N\!& M(\phi)^\dagger\end{array}
  \!\!\!\right)\cdot V,\end{align*}
where the last equality follows from the nature of the block matrix $V$.
Claim (i) is then a consequence of the resulting equalities
\[ {\rm Nrd}_{Q(\G_\infty)}\bigl(\left(\!\!\!\begin{array}{c|c}U\cdot N\!& M(\phi)^\dagger\end{array}
  \!\!\!\right)\bigr) = {\rm Nrd}_{Q(\G_\infty)}(U)\cdot{\rm Nrd}_{Q(\G_\infty)}
  \bigl( \left(\!\!\!\begin{array}{c|c}N\!&M(\tilde\phi)^\dagger\end{array}\!\!\!\right)\bigr)\cdot{\rm Nrd}_{Q(\G_\infty)}(V)^{-1}\]
and the fact ${\rm Nrd}_{Q(\G_\infty)}(U)$ and ${\rm Nrd}_{Q(\G_\infty)}(V)$ are both units of $\xi_p(\KK_\infty/K) = \xi_p(\KK_\infty/K)^\#,$ where the last equality follows from Remark \ref{hash remark}(ii).

In a similar way, Remark \ref{hash remark}(ii) reduces the proof of claim (ii) to showing  that for every $N$ in ${\rm M}_{d,r}(\La(\G_\infty))$ one has
\[ {\rm Nrd}_{Q(\G_\infty)}\bigl(\left(\!\!\!\begin{array}{c|c}N\!&M(\phi)^\dagger\end{array}\!\!\!\right)\bigr) \in \lambda(\gamma_E)^{r'-r}\cdot \xi_p(\KK_\infty/K).\]

To show this we assume, as we may (under the hypothesis (\ref{sigma_S assumption})), that the place $v'$ of $S\setminus \Sigma_S(\KK_\infty)$ used in the constructions of Proposition 
 \ref{limitomac} does not belong to $\Sigma$. We then define a subset $J = J_{S,\Sigma,v'}$ of $[n]$ by setting 
 \begin{equation}\label{J def}  J := \{j \in [n]\setminus [r]:\text{the } j\text{-th element of } S\setminus \{v'\} \text{ belongs to } \Sigma\setminus \Sigma_S(\KK_\infty)\}.\end{equation}
 Then one has $|J| = r'-r$ and the nature of the map (\ref{infty map}) implies that for each $j$ in $J$, and every $i$ in $[d]$ there exist a (unique) element
 $c_{ij}$ of $\La(\G_\infty)$ with
 \[  M(\phi)_{ij} = c_{ij}(\gamma_E-1)\]
 and hence 
\begin{equation}\label{semi decomp} \left(\!\!\!\begin{array}{c|c}N\!&M(\phi)^\dagger\end{array}\!\!\!\right) =
  M'\cdot \Delta_E\end{equation}
 where $M'$ is the matrix in ${\rm M}_{d}(\La(\G_\infty))$  defined by 
 \begin{equation}\label{M' def} M'_{ij} = \begin{cases} c_{ij}, &\text{if $j\in J$ and $i \in [d]$}\\
                         (\!\!\!\begin{array}{c|c}N\!&M(\phi)^\dagger\end{array}\!\!\!)_{ij},
                         &\text{if $j\in [d]\setminus J$ and $i \in [d]$,}\end{cases}\end{equation}
and $\Delta_E$ is the diagonal matrix in ${\rm M}_{d}(\La(\G_\infty))$ with 
\[ \Delta_{E,ij} := \begin{cases} \gamma_E-1, &\text{if $i = j \in J$}\\
 1, &\text{if $i=j \in [d]\setminus J$}\\
 0, &\text{if $i \not= j$.}\end{cases}   \]
 %
 
 The required result is
 therefore true since
  ${\rm Nrd}_{Q(\G_\infty)}(\Delta_E) = \lambda(\gamma_E)^{r'-r}$, whilst (\ref{rn description}) implies
  ${\rm Nrd}_{Q(\G_\infty)}(M')$ belongs to $\xi_p(\KK_\infty/K)$.
\end{proof}

\subsubsection{} With the result of Lemma \ref{prel semi} in mind, we now introduce a restriction on the structure of the $\La(\G_\infty)$-module $\mathcal{S}_S^T(\KK_\infty)$ that will play an important role in the sequel.

For a character $\chi$ in ${\rm Ir}_p(\G_E)$ we define a prime ideal of
$\xi_p(\KK_\infty/K)$ by setting
\begin{equation}\label{wp def} \wp_\chi(\KK_\infty/K) := \ker\bigl( \xi_p(\KK_\infty/K) \to \zeta(\QQ_p^c[\G_E]) \xrightarrow{ x\mapsto x_\chi} \QQ_p^c\bigr),\end{equation}
where the unlabelled arrow is the natural projection.

\begin{definition}\label{semi def}{\em Fix $\chi$ in
${\rm Ir}_p(\G_E)$. Then the data $\KK_\infty/K, E$ and $S$ is said to be `semisimple at $\chi$' if $ {\rm Fit}_{\La(\G_\infty)}^r(\mathcal{S}_S^T(\KK_\infty))$ is not contained in $
I_E(\G_\infty)^{r'-r}\cdot \wp_\chi(\KK_\infty/K).$
}
\end{definition}

Before proceeding, we explain the motivation for our use of the word `semisimple' in this context. In particular, we note that the stated property of the module $Q$ in the following result implies that, after localizing at $\wp_\chi(\KK_\infty/K)$, it is  `semisimple at zero' in the sense relevant to Iwasawa-theoretic descent computations (cf. \cite{burns2}). 
%
%


\begin{lemma}\label{semisimplicity lemma} Fix $\chi$ in
${\rm Ir}_p(\G_E)$. Then, if the data $\KK_\infty/K, E$ and $S$ is semisimple at $\chi$, there exists an exact sequence of $\La(\G_\infty)$-modules 
\begin{equation}\label{semi ses}\La(\G_\infty)^r \to \mathcal{S}_S^T(\KK_\infty) \to Q \to 0\end{equation}
in which $Q$ has the following property: the natural map 
\[ Q^{\gamma_E = 1} \oplus (\gamma_E-1)Q \to Q\]
is bijective after localizing at $\wp_\chi(\KK_\infty/K)$. 
\end{lemma}

\begin{proof} 
Set $\wp := \wp_\chi(\KK_\infty/K)$. Then, under the stated  hypothesis, the argument of Lemma \ref{prel semi} implies, via the product decomposition (\ref{semi decomp}), that there exists a matrix $N$ in ${\rm M}_{d,r}(\La(\G_\infty))$ such that
\begin{equation}\label{non-inclusion} {\rm Nrd}_{Q(\G_\infty)}(M')\in \xi_p(\KK_\infty/K)\setminus \wp^\#,\end{equation} 
where the matrix $M'$ is defined in terms of $N$ as in (\ref{M' def}). 

For each matrix $M$ in ${\rm M}_d(\La(\G_\infty))$ we set $M^\ast := M^{{\rm tr},\#}$. We also identify the matrices $M(\phi)^\ast$, $M(\phi,N) := \left(\!\!\!\begin{array}{c|c}N\!&M(\phi)^\dagger\end{array}\!\!\!\right)^\ast$, $(M')^\ast$ and $\Delta_E^\ast$ with endomorphisms of $\La(\G_\infty)^d$ in the obvious way. Then the cokernel of $M(\phi)^\ast$ is isomorphic to $\mathcal{S}_S^T(\KK_\infty)$ and so there exists an exact sequence (\ref{semi ses}) in which $Q$ is the cokernel of $M(\phi,N)$. In addition, from the decomposition 
$M(\phi,N) = \Delta_E^\ast\cdot (M')^\ast$, one deduces that this choice of $Q$ lies in an exact sequence of $\La(\G_\infty)$-modules $Q'  \to Q \to {\rm cok}\bigl(\Delta_E^\ast\bigr) \to 0$, 
with $Q' := {\rm cok}\bigl((M')^{\ast}\bigr)$.

Since ${\rm cok}\bigl(\Delta_E^\ast\bigr)$ is isomorphic to $\La(\G_E)^{r'-r}$, to deduce that $Q$ has the claimed property it is therefore enough to show that $Q'_\wp$ vanishes. Now (\ref{non-inclusion}) implies that the reduced norm 
${\rm Nrd}_{Q(\G_\infty)}((M')^{\ast}) = {\rm Nrd}_{Q(\G_\infty)}((M'))^\#$ does not belong to $\wp$, and this implies $\gamma_E-1$ acts invertibly on $Q'_\wp$. Thus, since $\gamma_E-1$ belongs to $\wp$, Nakayama's Lemma implies that $Q'_\wp$ vanishes, as required.
\end{proof}

%
%
%
%
%
%
%


\subsubsection{}In the sequel, for any subfield $\KK$ of $\QQ^c$ we set
\[ A_S(\KK) := {\varprojlim}_{F}{\rm Cl}_S(F)_p \,\, \text{ and }\,\, A^T_S(\KK) := {\varprojlim}_{F}{\rm Cl}^T_S(F)_p\]
where in both limits $F$ runs over all finite extensions of $\QQ$ in $\KK$ and the
transition morphisms for $F\subset F'$ are the natural norm maps.

In the next result we describe the link between semisimplicity in the sense of Definition 
\ref{semi def} and the structural properties of modules of the form $A_S(\KK)$. 

In condition (ii) of this result we use the idempotent $e_{E/K,S,\Sigma}$ defined in (\ref{key idem def}) 
 (and we recall that, under the hypotheses of this section, one has $\Sigma\not= S$).


%
%

\begin{proposition}\label{abeliangross} The data $\KK_\infty/K, S$ and $E$ is semisimple at
 every character $\chi$ in ${\rm Ir}_p(\G_E)$ that satisfies all
 of the following conditions:
\begin{itemize}
\item[(i)] The space
$e_\chi\bigl(\QQ_p^c\otimes_{\ZZ_p} A_S(\KK_\infty)_{\mathcal{H}_\infty}\bigr)$ vanishes.
\item[(ii)] $e_\chi\cdot e_{E/K,S,\Sigma}\not= 0$.
\end{itemize}
\end{proposition}

\begin{proof} The essential idea of this argument is that, under the stated hypotheses, one can `reverse' the direction of the argument in Lemma \ref{semisimplicity lemma}. 

To make this precise, we 
 %
%
write $\xi'_p(\KK_\infty/K)$ for the subring of
$\zeta(Q(\G_\infty))$ generated over $\zeta(\La(\G_\infty))$ by the set $\{{\rm Nrd}_{Q(\G_\infty)}(M):
M \in \bigcup_{m > 0}{\rm M}_{m}(\La(\G_\infty))\}$. We also fix an open
normal subgroup $\mathcal{Z}$ of $\mathcal{H}_\infty$ that is central in $\G_\infty$ and
 contained in the decomposition subgroup in $\G_\infty$ of  
$w_{j,\infty}$ for each $j$ belonging to the set $J$ defined in (\ref{J def}) (so $v_j \in \Sigma\setminus \Sigma_S(\KK_\infty)$). We then write $E'$ for finite extension of $E$ that is obtained as the fixed field of $\mathcal{Z}$ in $\KK_\infty$.

We note $\xi'_p(\KK_\infty/K)$ is a
$\La(\mathcal{Z})$-order
  in $\zeta(Q(\G_\infty))$ (cf. the proof of Lemma \ref{relate xi zeta}) and also that
  the explicit description of reduced norm given in (\ref{rn description}) implies that the projection map
   $\zeta(\La(\G_\infty)) \to \zeta(\ZZ_p[\G_E])$ extends to a
   well-defined ring homomorphism $\varrho_E$ from $\xi'_p(\KK_\infty/K)$ to $\zeta(\QQ^c_p[\G_E])$. For any fixed character $\chi$ in ${\rm Ir}_p(\G_E)$ we can therefore define
    a prime ideal of $\xi'_p(\KK_\infty/K)$ by setting
   \[ \wp' = \wp'_\chi := \ker\bigl(\xi'_p(\KK_\infty/K) \xrightarrow{\varrho_E}\zeta(\QQ_p^c[\G_E]) \xrightarrow{x\mapsto x_\chi} \QQ_p^c\bigr).\]
  We write $\wp$ for the prime ideal $\wp'\cap
   \zeta(\La(\G_\infty))$ of $\zeta(\La(\G_\infty))$ and $M_{\wp}$ for any
   $\La(\G_\infty)$-module $M$ for the
$\La(\G_\infty)$-module obtained by localizing $M$ (as a $\zeta(\La(\G_\infty))$-module)
at $\wp$.

We now fix an
endomorphism $\phi$ as constructed in Proposition \ref{limitomac} with respect to a place $v'$ chosen in $S\setminus \Sigma$,  and set  
\[ \mathcal{E}_E(\KK_\infty) := \La(\G_\infty)^r\oplus
 \ZZ_p[\G_E]^{r'-r}.\]
We then claim that it suffices to prove that the given hypotheses on $\chi$ imply the existence of  a
  commutative diagram of $\La(\G_\infty)_{\wp}$-modules of the form
 \begin{equation}\label{infinite comp} \begin{CD} \La(\G_\infty)_{\wp}^d @> \phi_{\wp}>>
 \La(\G_\infty)_{\wp}^d @> \pi_{\infty,\wp} >>
 \mathcal{S}_S^T(\KK_\infty)^{\rm tr}_{\wp} @> >> 0\\
 @V\nu VV @\vert @V\nu' VV\\
 \La(\G_\infty)_\wp^d @> \varpi >> \La(\G_\infty)_{\wp}^d @>\varpi' >>
 \mathcal{E}_E(\KK_\infty)_{\wp} @> >> 0,\end{CD}\end{equation}
 in which $\nu$ is bijective and the rows are
 obtained by respectively localizing the
 upper row of (\ref{limit independence}) and the standard resolution of
 $\mathcal{E}_E(\KK_\infty)$. In particular, for each $i$ in $[d]$ one has
 \[ \varpi(b_i) = \begin{cases} 0 &\text{}\\
 (\gamma_E-1)(b_i) &\\
 b_i \end{cases} \quad\text{ and }\quad \varpi'(b_i) = \begin{cases} b_i, &\text{\hskip 0.3truein if $i \in [r]$,}\\
 b_{E,a}, &\text{\hskip 0.3truein if $i$ is the $a$-th place in $J$,}\\
 0, &\text{\hskip 0.3truein otherwise.}\end{cases}\]
 %

We assume for the moment such a diagram exists. We write
$U$ for the matrix of $\nu$ with respect to the
standard basis of $\La(\G_\infty)_\wp^d$ and $N$ for the matrix in
${\rm M}_{d,r}(\La(\G_\infty)_\wp)$ for which 
 $U^{-1}\cdot N$ is the transpose of the block matrix
 $\left(\!\!\!\begin{array}{c|c}I_r\!&0\end{array}\!\!\!\right)$. 
 
 Then
 the argument of Lemma \ref{prel semi} shows the
 commutativity of the first square in
 (\ref{infinite comp}) implies $U^{-1}\cdot \left(\!\!\!\begin{array}{c|c}N\!&M(\phi)^\dagger\end{array}\!\!\!\right)$
 is the $d\times d$ diagonal matrix $\Delta_E$ with $ii$-th entry equal to $\gamma_E-1$
 if $i \in J$ and equal to $1$ otherwise. Thus, if we
 choose an element $x$ of $\zeta(\La(\G_\infty))\setminus \wp$ for which $xN$ belongs to
  ${\rm M}_{d,r}(\La(\G_\infty))$, then one has
 \begin{align*} & \lambda(\gamma_E)^{r-r'}\cdot{\rm Nrd}_{Q(\G_\infty)}\bigl(\left(\!\!\!\begin{array}{c|c}xN\!&M(\phi)^\dagger\end{array}\!\!\!\right)\bigr)\\
 =\, &\lambda(\gamma_E)^{r-r'}\cdot{\rm Nrd}_{Q(\G_\infty)}(x)^r\cdot {\rm Nrd}_{Q(\G_\infty)}(\Delta_E)\cdot {\rm Nrd}_{Q(\G_\infty)}(U)\\
  =\, &{\rm Nrd}_{Q(\G_\infty)}(x)^r\cdot {\rm Nrd}_{Q(\G_\infty)}(U)\end{align*}
This implies the claimed semisimplicity at $\chi$ since
${\rm Nrd}_{Q(\G_\infty)}\bigl(\left(\!\!\!\begin{array}{c|c}xN\!&M(\phi)^\dagger\end{array}\!\!\!\right)\bigr)$ belongs to
${\rm Fit}_{\La(\G_\infty)}^r(\mathcal{S}_S^T(\KK_\infty))$ whilst neither
 ${\rm Nrd}_{Q(\G_\infty)}(x)$ nor ${\rm Nrd}_{Q(\G_\infty)}(U)$ belongs to $\wp_\chi'$.

It therefore suffices to prove the
given assumptions imply the existence of a diagram
of the form (\ref{infinite comp}). To do this we note that the second square of
 (\ref{infinite comp}) clearly commutes if we take $\nu'$ to be the
 composite homomorphism
 of $\La(\G_\infty)_\wp$-modules
\[ \nu': \mathcal{S}_S^T(\KK_\infty)^{\rm tr}_\wp\xrightarrow{\beta_\infty} (Y_{\infty,S'})_\wp \to
 {\bigoplus}_{v \in \Sigma}Y_{\infty,v,\wp} \to \mathcal{E}_E(\KK_\infty)_\wp,\]
where $S'$ denotes $S\setminus \{v'\}$, the second map is the natural projection and the third is induced by sending
$w_{i,\infty}$ for $i\in [r]$, respectively $i \in J$, to $b_i$,
respectively to $b_{E,i-r}$.

The key claim we make now is that $\nu'$ is bijective. To show this we take the limit over
 $L$ in $\Omega_E(\KK_\infty)$ of the exact sequence (\ref{selmer lemma seq2}) to obtain an exact sequence of
  $\La(\G_\infty)$-modules
\[ 0 \to A_S^T(\KK_\infty) \to \mathcal{S}_S^T(\KK_\infty)^{\rm tr} \to {\bigoplus}_{v \in S} Y_{\infty,v} \,\to \ZZ_p\to 0.\]
This sequence implies that $\nu'$ is bijective if all of the following conditions are satisfied: $A_S^T(\KK_\infty)_\wp$ vanishes; $Y_{\infty,v,\wp}$ vanishes for each $v \in S'\setminus \Sigma$; for every $j\in J$ the $\wp$-localization of the morphism
$Y_{\infty,v_j} \to \ZZ_p[\G_E]$ sending $w_{j,\infty}$ to $b_{E,j-r}$ is bijective; if $Y_{\infty,v',\wp}$ does not vanish, then the $\wp$-localization  of the natural projection  map $Y_{\infty,v'} \to \ZZ_p$ is bijective.

To verify these conditions we write $E(v)$ for each $v$ in $S\setminus \Sigma_S(\KK_\infty)$  for the maximal extension of $K$ in $\KK_\infty$ in which $w_{v,\infty}$ splits completely. We also write $E'(S)$ for the finite extension of $K$ in $\KK_\infty$ that is obtained as the compositum of $E'$ and the fields $E(v)$ and set $G' := \G_{E'(S)}$. 

We then note that \cite[Th. 3.5]{NickelDoc} implies $p^m\cdot \xi_p'(\KK_\infty/K)$ is contained in $\zeta(\La(\G_\infty))$ for any large enough integer $m$. This fact implies the existence of an element $t_\chi$
  of $\zeta(\La(\G_\infty))$ whose projection to $\zeta(\ZZ_p[G'])$
  is a $p$-power multiple of the
   idempotent $e_{(\chi)}$ of $\zeta(\QQ_p[G'])$ that corresponds to the
   irreducible $\QQ_p$-valued character of $G'$
    that contains $\chi$ as a component. It follows that
    $t_\chi\in \zeta(\La(\G_\infty))\setminus \wp$ and hence that for
    each place $v \in S \setminus 
    \Sigma_S(\KK_\infty)$ one has 
\begin{align}\label{exp description0} Y_{\infty,v,{\wp}} = 
 t_\chi\bigl( Y_{\infty,v,{\wp}}\bigr) =&\,  
 t_\chi\bigl(\QQ_p\otimes_{\ZZ_p}Y_{\infty,v}\bigr)_{\wp}\\
 =&\,  
 t_\chi\bigl(\QQ_p\cdot Y_{E'(S),
 \{v\}}\bigr)_{\wp} =
e_{(\chi)}(\QQ_p\cdot Y_{E,\{v\}})_\wp.\notag
\end{align}
In particular, since for each $v$ in $\Sigma$, the map 
\[ e_{(\chi)}(\QQ_p\cdot Y_{E,\{v\}})_\wp   \xrightarrow{ w_{v} \mapsto 1}
  t_\chi(\ZZ_p[\G_E]_\wp) = \ZZ_p[\G_E]_\wp\]
is bijective, we deduce that for every $j\in J$ the $\wp$-localization of the morphism
$Y_{\infty,v_j} \to \ZZ_p[\G_E]$ sending $w_{j,\infty}$ to $b_{E,j-r}$ is bijective, as required. 

By combining (\ref{exp description0}) with the result of Lemma \ref{idem lemma} and the assumed validity of condition (ii) we also derive the following consequences: the localisation $Y_{\infty,v,{\wp}}$ vanishes for all $v \in S\setminus \Sigma$; if $Y_{\infty,v',\wp}$ does not vanish, then $\chi$ is trivial, $S\setminus \Sigma= \{v'\}$ and the natural map 
$Y_{\infty,v',{\wp}} = e_{(\chi)}(\QQ_p\cdot Y_{E,\{v'\}})_\wp \to (\ZZ_p)_\wp$ is bijective. 
  
To complete the proof that $\nu'$ is bijective, it now suffices to show $A_S^T(\KK_\infty)_\wp$ vanishes. To do this we fix a topological
 generator $\gamma_{E'}$ of $\mathcal{Z}$
  and note $\La(\G_\infty)_\wp$ is an order over the discrete valuation
ring $\La(\mathcal{Z})_{\mathfrak{p}}$ obtained by localizing $\La(\mathcal{Z})$ at the prime
 ideal $\mathfrak{p}$ generated by $\gamma_{E'}-1$. Nakayama's Lemma therefore implies that $A_S^T(\KK_\infty)_\wp$ vanishes if the space
\[ t_\chi\bigl(\QQ_p\otimes_{\ZZ_p}A_S^T(\KK_\infty)_{\mathcal{Z}}\bigr) =
 e_{(\chi)}\bigl(\QQ_p\otimes_{\ZZ_p}A_S^T(\KK_\infty)_{\mathcal{Z}}\bigr) =
 e_{(\chi)}\bigl(\QQ_p\otimes_{\ZZ_p}A_S^T(\KK_\infty)_{\mathcal{H}_\infty}\bigr)\]
vanishes, and this follows directly from condition (i). 

At this stage we have established that the right hand square in
(\ref{infinite comp}) commutes when $\nu'$ is the isomorphism specified above. From this it follows that $\im(\phi_\wp)= \ker(\pi_{\infty,\wp})$
is equal to the free $\La(\G_\infty)_\wp$-submodule of
$\Lambda(\G_\infty)_\wp^d$ that has basis
\[ \{ (\gamma_E-1)(b_i)\}_{i \in J}\cup \{b_i\}_{i \in [d]\setminus J^\dagger},\]
where we set 
\begin{equation}\label{J def 2} J^\dagger := J \cup [r].\end{equation}

We may
therefore choose a section $\sigma$ to $\phi_\wp$ and thereby obtain an isomorphism of $\La(\G_\infty)_\wp$-modules
\[\La(\G_\infty)_\wp^d = \ker(\phi_\wp) \oplus \sigma(\im(\phi_\wp)) \cong \ker(\phi_\wp) \oplus \Lambda(\G_\infty)_\wp
^{d-r}.\]
Upon applying the Krull-Schmidt Theorem for the
$\La(\mathcal{Z})_\mathfrak{p}$-order $\La(\G_\infty)_\wp$ to this isomorphism, we
deduce that the $\La(\G_\infty)_\wp$-module $\ker(\phi_\wp)$ is free
 of rank $r$, and hence isomorphic to the kernel
 ${\bigoplus}_{i=1}^{i=r}\La(\G_\infty)_\wp\cdot b_i$ of $\varpi$.

In particular, if we fix an isomorphism of $\La(\G_\infty)_\wp$-modules $\nu_1: \ker(\phi_\wp) \cong
 \ker(\varpi)$, and write $\nu_2: \sigma(\im(\phi_\wp)) \to \La(\G_\infty)_\wp^d$
 for the map of $\La(\G_\infty)_\wp$-modules that sends the element 
\[ \begin{cases} \sigma((\gamma_E-1)(b_i)), &\text{if $i \in J$,}\\
\sigma(b_i), &\text{if $i \in [d]\setminus J^\dagger$}\end{cases}\]
to $b_i$, then the homomorphism $\nu = (\nu_1,\nu_2)$
  is an automorphism of $\La(\G_\infty)_\wp^d$ that makes the first square in (\ref{infinite comp})
 commute, as required to complete the proof.\end{proof}

\begin{remark}\label{gross-jaulent}{\em For a number field $E$, write $\Gamma_E$ for the Galois group over $E$ of its cyclotomic $\ZZ_p$-extension $E^{\rm cyc}$. Then it is conjectured by
Jaulent in \cite{jaulent} that, for every $E$,  the
 $\Gamma_E$-coinvariants $A_S(E^{\rm cyc})_{\Gamma_E}$ of $A_S(E^{\rm cyc})$ should be
 finite. In addition, if $E$ is a CM Galois extension
 of a totally real field $K$, then an observation of Kolster in
 \cite[Th. 1.14]{kolster}
  (where the result is attributed to Kuz'min \cite{kuzmin}) implies that the
 finiteness of $A_S(E^{\rm cyc})_{\Gamma_E}^-$ is equivalent to
  the earlier conjecture \cite[Conj. 1.15]{G0} of Gross and hence also, by
  \cite[Th. 5.2(ii)]{burns2}, to the validity of Gross's
  `Order of Vanishing Conjecture' \cite[Conj. 2.12a)]{G0} for all totally odd characters of
  $\G_{E}$. In particular, if $K$ contains at most one $p$-adic place that
  splits completely in $E/E^+$, then $A_S(E^{\rm cyc})_{\Gamma_E}^-$ is finite
  as a consequence of \cite[Prop. 2.13]{G0} (which itself relies
 Brumer's $p$-adic version of Baker's theorem). In general, if $\E$ is any
$\ZZ_p$-extension of a number field $E$,
then $A_S(\E)_{\Gal(\E/E)}$ is known to be finite in each of the following cases.
\begin{itemize}
\item[(i)] $E$ is abelian over $\QQ$ (cf. Greenberg \cite{greenberg0}).
\item[(ii)] $E$ is an abelian extension of an imaginary quadratic field and $\E = E^{\rm cyc}$ (cf. Maksoud \cite{maksoud}).
\item[(iii)] $\E = E^{\rm cyc}$ and
 $E$ has at most two $p$-adic places (cf. Kleine \cite{kleine}).
\item[(iv)] $\E$ is totally real and the Leopoldt conjecture is valid for $E$ at $p$
 (cf. Kolster \cite[Cor. 1.3]{kolster}).
\end{itemize}
}
\end{remark}

\section{A conjectural derivative formula for Rubin-Stark Euler systems}\label{IMRSsection}

In this section we define a notion of `the value of a higher derivative' of the Rubin-Stark non-commutative Euler system (from Definition \ref{ncrs def}) and formulate an explicit conjectural formula for such values.

We also show that this conjectural derivative formula specializes to recover the classical Gross-Stark
Conjecture (from \cite[Conj. 2.12b)]{G0}) and hence deduce its validity in an important family of examples.

Throughout the section we fix a rank one compact $p$-adic Lie extension $\KK_\infty$ of $K$ in $\KK$ that is ramified at only finitely many places. 

We fix sets of places $S$ and $T$ of $K$ as specified at the beginning of \S\ref{can res}, and use the abbreviations 
\[ \G := \G_\infty = G_{\KK_\infty/K},\,\,\, R_\infty := \La(\G) \quad \text{ and}\quad R_L := \ZZ_p[\G_L]\]
for each $L$ in $\Omega(\KK_\infty)$.

We also fix a normal subgroup
 $\mathcal{H}$ of $\G$ that is topologically isomorphic to $\ZZ_p$ and write $E$ for the fixed field of $\mathcal{H}$ in $\KK_\infty$. We then fix a subset $\Sigma$ of $\Sigma_S(E)$ that 
 satisfies the condition (\ref{sigma_S assumption})  and set 
\[ r := |\Sigma_S(\KK_\infty)|\quad \text{ and }\quad  r' := |\Sigma|\]
(so that $r \le r' < |S|$ since $\Sigma_S(\KK_\infty)\subseteq \Sigma \subsetneq S$).

\subsection{Derivatives of Rubin-Stark non-commutative Euler systems} 

For a natural number $t$ and non-negative integer $a$ we set
\[ {\bigcap}^a_{ R_\infty}R_\infty^t := {\varprojlim}_{L'}{\bigcap}_{R_{L'}}^a R_{L'}^t\]
and 

\[ {\bigcap}^a_{\CC_p\cdot R_\infty}(\CC_p\cdot R_\infty)^t := {\varprojlim}_{L'}\bigl(\CC_p\otimes_{\ZZ_p}{\bigcap}_{R_{L'}}^a R_{L'}^t\bigr),\]
where in both limits $L'$ runs over $\Omega(\KK_\infty)$ and the transition morphisms are induced by the natural projection maps $R_{L'}^t \to R_L^t$ for $L \subseteq L'$. For each $L$ in $\Omega(\KK_\infty)$ we also use the natural projection map 
\[ \pi_L^{a}: {\bigcap}^a_{\CC_p\cdot R_\infty}(\CC_p\cdot R_\infty)^t\to \CC_p\otimes_{\ZZ_p}{\bigcap}_{R_L}^a R_L^t.\]

%

We use the element $\varepsilon^{\rm RS}_{\KK_\infty,S,T}$ of ${\bigwedge}^{r}_{\CC_p\cdot\Lambda(\cG_\infty)}(\CC_p\cdot\co^\times_{\KK_\infty,S,T})$ defined in (\ref{limit RS element}).
%

\begin{proposition}\label{preGGS lemma} For each topological generator $\gamma$ of $\mathcal{H}$, there exists an  element 
\[ \partial^{\,r'-r}_{\gamma}(\varepsilon^{\rm RS}_{\KK_\infty,S,T}) \in \CC_p\cdot{\bigcap}^{r}_{R_E}\mathcal{O}_{E,S,T,p}^\times\]
that depends only on the data $\KK_\infty/K,\gamma, S$ and $T$ and has the following property: for every pair of embeddings $\iota_\infty: \mathcal{O}_{\KK_\infty,S,T}^\times \to R_\infty^d$ and $\iota_E: \mathcal{O}_{E,S,T,p}^\times \to R_E^d$ constructed as in Proposition \ref{limitomac}(iii), there exists an element $y(\gamma)$ of ${\bigcap}^{r}_{\CC_p\cdot R_\infty}(\CC_p\cdot R_\infty)^d$ such that both
\[ \iota_{\infty,*}(\varepsilon^{\rm RS}_{\KK_\infty,S,T}) = {\rm Nrd}_{Q(\G)}(\gamma-1)^{r'-r}\cdot y(\gamma)\]
and
\[ \iota_{E,*}\bigl(\partial^{\,r'-r}_{\gamma}(\varepsilon^{\rm RS}_{\KK_\infty,S,T})\bigr) = \pi^{r}_E(y(\gamma)).\]
\end{proposition}

\begin{proof} We fix a resolution $R_\infty^d\xrightarrow{\phi}R_\infty^d$ of $C_{\KK_\infty,S,T}$ as constructed in Proposition \ref{limitomac} with respect to a place $v'$ chosen in $S\setminus \Sigma$, and hence also an associated exact sequence as in the upper row of (\ref{limit independence}). This leads to fixed embeddings $\iota_\infty$ and $\iota_E$ of the stated form. In addition, for each $L$ in $\Omega_E(\KK_\infty)$ we obtain an induced resolution $R_L^d\xrightarrow{\phi_L}R_L^d$ of $C_{L,S,T,p}$. 

We note, in particular, that, for each $j$ belonging to the subset $J$ of $[d]$ defined in (\ref{J def}), the composite map $b_{E,j}^\ast\circ\phi_{E}$ is zero and so there exists a (unique) homomorphism 
\[ \widehat{\phi}_{j} =
 (\widehat{\phi}_{j,F})_{F\in \Omega_E(\KK_\infty)} \in \Hom_{R_\infty}(R_\infty^d,R_\infty)\]
with
 \begin{equation}\label{hat def} b_j^\ast\circ \phi = (\gamma-1)(\widehat{\phi}_{j}).\end{equation}
We claim that this implies an equality 
\begin{align}\label{divisibility relation} \bigl(\wedge_{j\in J}(b_{F,j}^\ast\circ\phi_{F})\bigr)
_{F} =&\,\,  \bigl(\wedge_{j\in J}
(\gamma-1)(\widehat{\phi}_{j,F})\bigr)_{F}\\
=&\,\,  \lambda(\gamma)^{r'-r}\cdot
\bigl(\wedge_{j\in J}\widehat{\phi}_{j,F}\bigr)
_{F},\notag\end{align}
where in each case $F$ runs over $\Omega_E(\KK_\infty)$. Here the first equality follows directly from the defining relations (\ref{hat def}) and, after taking account of (\ref{rn description}) (with $N$ taken to be the $1\times 1$ matrix
$(\gamma-1)$), the second equality can be verified by showing that, for each $F$ in $\Omega_E(\KK_\infty)$, one has 
\[ \wedge_{j\in J}\theta_{j,F} =  {\rm Nrd}_{\QQ_p[\cG_F]}(\gamma_{(F)}-1)^{r'-r}\cdot \bigl(\wedge_{j\in J}\theta_{j,F}'\bigr),\]
where $\gamma_{(F)}$ is the image of $\gamma$ in $\G_F$ and we set $\theta_{j,F}' := \widehat{\phi}_{j,F}$ and $\theta_{j,F} := (\gamma_{(F)}-1)(\theta_{j,F}')$. It is in turn enough to verify this last displayed equality after applying the projection functor $\zeta(A)\otimes_{\zeta(\QQ_p[\cG_F])}-$ for each simple Wedderburn component $A$ of $\QQ_p[\cG_F]$. Further, if we write $\theta_{j,A}$ and $\theta_{j,A}'$ for the corresponding projections of $\QQ_p\otimes_{\ZZ_p}\theta_{j,F}$ and $\QQ_p\otimes_{\ZZ_p}\theta'_{j,F}$, then the explicit definition (\ref{non comm ext power}) of reduced exterior products implies that the elements $\wedge_{j\in J}\theta_{j,A}$ and $\wedge_{j\in J}\theta'_{j,A}$ vanish unless the $A$-module $W$ that is generated by $\{\theta_{j,A}'\}_{j \in J}$ is free of rank $r'-r$. Then, in the latter case, the required equality follows by applying the general result of \cite[Lem. 4.13]{bses} with $\varphi$ taken to be the element of ${\rm End}_A(W)$ that sends each basis element $\theta_{j,A}'$ of $W$ to $\theta_{j,A} = (\gamma_{(F)}-1)(\theta_{j,A}')$.

For each $L$ in $\Omega_E(\KK_\infty)$ we consider the element 
\[ x_L := ({\wedge}_{j=r+1}^{j=d}(b_{L,j}^\ast\circ\phi_L))({\wedge}_{j\in [d]}b_{L,j})\in {\bigcap}_{R_L}^{r}R_L^d.\]
For each integer $j$ in $[d]\setminus [r]$ we also set 
\[ \theta_{L,j} := \begin{cases} \widehat{\phi}_{j,L}, &\text{if $j \in J$}\\
                                 b_{L,j}^\ast\circ\phi_L, &\text{otherwise.}\end{cases}\]

Then the family $x := (x_L)_L$ belongs to ${\bigcap}^r_{ R_\infty}R_\infty^d$ and, setting $\lambda(\gamma) = {\rm Nrd}_{Q(\G)}(\gamma-1)$, the relations (\ref{divisibility relation}) (as $F$ runs over $\Omega_E(\KK_\infty)$) imply that
\begin{equation}\label{dividing x} x = \lambda(\gamma)^{r'-r}\cdot x'\quad \text{with} \quad x' := \bigl(({\wedge}_{j=r+1}^{j=d}\theta_{L,j})({\wedge}_{j\in [d]}b_{L,j})\bigr)_L \in {\bigcap}^r_{ R_\infty}R_\infty^d.\end{equation}
%
%
%
This equality determines the element $x'$ uniquely since multiplication by $\lambda(\gamma)$ on ${\bigcap}^r_{ R_\infty}R_\infty^d$ is injective (see Remark \ref{uniqueness} below). 

For $L$ in $\Omega(\KK_\infty)$ we write $\iota_{L,\ast}$ for the injective homomorphism  ${\bigcap}^{r'}_{R_L}\mathcal{O}_{L,S,T,p}^\times \to {\bigcap}^{r'}_{R_L}R_L^d$ that is induced by our fixed resolution of $C_{L,S,T,p}$. Then the result of Proposition \ref{generate rubin}(ii) implies the existence of a unique element $z_L$ of $\CC_p[\G_L]e_L$ with  $\iota_{L,\ast}(\varepsilon^{\Sigma}_{L/K,S,T,p}) = z_L\cdot x_L$. In addition, the element $z := (z_L)_L$ belongs to $\varprojlim_L\CC_p[\G_L]$ and one has 
\[ \iota_{\infty,*}(\varepsilon^{\rm RS}_{\KK_\infty,S,T}) = (\iota_{L,\ast}(\varepsilon^{\Sigma}_{L/K,S,T,p}))_L = z\cdot x.\] 
The equality (\ref{dividing x}) therefore implies that 
\begin{equation}\label{der def 1}\iota_{\infty,*}(\varepsilon^{\rm RS}_{\KK_\infty,S,T}) = \lambda(\gamma)^{r'-r}\cdot y(\gamma) \quad \text{with}\quad y(\gamma) := z\cdot x' \in {\bigcap}^{r}_{\CC_p\cdot R_\infty}(\CC_p\cdot R_\infty)^d.\end{equation}
In addition, since $v_i\in \Sigma$ for all $i$ in the set $J^\dagger = J \cup [r]$ introduced  in (\ref{J def 2}), Proposition \ref{generate rubin}(i) implies that the element 
\[w_E := ({\wedge}_{j\in [d]\setminus J^\dagger}(b_{E,j}^\ast\circ\phi_{E}))({\wedge}_{j\in [d]}b_{E,j})\]
belongs to $\iota_{E,*}\bigl( {\bigcap}_{R_E}^{r'}\mathcal{O}_{E,S,T,p}^\times\bigr)$. This implies that the element 
\begin{align*}\pi_E^r(y(\gamma)) &=\,\, z_E\cdot x'_E \\
&=\,\, {\rm Nrd}_{\QQ_p[\G_E]}(-1)^{t_E}\cdot z_E\cdot (\wedge_{j\in J}\widehat{\phi}_{j,E})(w_E) \end{align*}
belongs to $\iota_{E,*}\bigl(\CC_p\otimes_{\ZZ_p} {\bigcap}_{R_E}^{r}\mathcal{O}_{E,S,T,p}^\times\bigr)$, where, following  \cite[Lem. 4.13]{bses}, the integer $t_E$ is fixed so that 
\[   {\wedge}_{j\in [d]\setminus [r]}\theta_{L,j} = {\rm Nrd}_{\QQ_p[\G_E]}((-1)^{t_E})\cdot ({\wedge}_{j\in [d]\setminus J^\dagger}(b_{E,j}^\ast\circ\phi_{E}))\wedge (\wedge_{j\in J}\widehat{\phi}_{j,E}).\]
It is therefore enough to show that the unique element $\partial^{\,r'-r}_{\gamma}(\varepsilon^{\rm RS}_{\KK_\infty,S,T})$ of $\CC_p\otimes_{\ZZ_p} {\bigcap}_{R_E}^{r}\mathcal{O}_{E,S,T,p}^\times$ that satisfies   
\begin{equation}\label{der def 2}\iota_{E,*}(\partial^{\,r'-r}_{\gamma}(\varepsilon^{\rm RS}_{\KK_\infty,S,T})) = \pi_E^r(y(\gamma))\end{equation}
is independent of the resolution of $C_{\KK_\infty,S,T}$  fixed above. To check this we can assume to be given a commutative diagram of the form (\ref{limit independence}) and we write $\tilde \iota_E$, $\widehat{\tilde\phi}_{j,E}$, $\tilde z$, $\tilde w_E$ and $\tilde y(\gamma)$ for the corresponding data that arises when making the above constructions with respect to the resolution given by the lower (rather than upper) row of this diagram. Then the commutativity of this diagram combines with the argument of Proposition \ref{generate rubin}(iv) to imply the existence of an element $\mu = (\mu_L)_L$ of $\xi(R_\infty)^\times$ such that both $\tilde z = \mu^{-1}\cdot z$ and 
\[ (\tilde\iota_{E,*})^{-1}\bigl((\wedge_{j\in J}\widehat{\tilde\phi}_{j,E})(\tilde w_E)\bigr) = (\iota_{E,*})^{-1}\bigl(\mu_E\cdot (\wedge_{j\in J}\widehat{\phi}_{j,E})(w_E)\bigr)\]
and hence also
\begin{align*} (\tilde\iota_{E,*})^{-1}(\tilde y(\gamma)) =&\, (\tilde\iota_{E,*})^{-1}(\tilde z_E\cdot (\wedge_{j\in J}\widehat{\tilde\phi}_{j,E})(\tilde w_E))\\ =&\, (\iota_{E,*})^{-1}(z_E\cdot (\wedge_{j\in J}\widehat{\phi}_{j,E})(w_E)) = (\iota_{E,*})^{-1}(y(\gamma)),\end{align*}
as required. \end{proof}

\begin{remark}\label{uniqueness} {\em To show that multiplication by $\lambda(\gamma)$ is injective on ${\bigcap}^r_{ R_\infty}R_\infty^d$ it suffices to take $z = (z_L)_L\in {\bigcap}^r_{ R_\infty}R_\infty^d$ with $\lambda(\gamma)\cdot z = 0$ and show that this implies $({\wedge}_{j\in [r]}\theta_j)(z_{F}) = 0$ for each $F$ in $\Omega_E(\KK_\infty)$ and each subset $\{\theta_j\}_{j \in [r]}$ of $\Hom_{R_F}(R_F^d,R_F)$. However, if one fixes a pre-image $(\theta_{j,L})_L$ of each $\theta_j$ under the surjection $\Hom_{R_\infty}(R_\infty^d,R_\infty) \to \Hom_{R_{F}}(R_F^d,R_F)$, then one has 
\[ \lambda(\gamma)\cdot (({\wedge}_{j\in [r]}\theta_{j,L})(z_{L}))_L = 
 (({\wedge}_{j\in [r]}\theta_{j,L})_L) \bigl( \lambda(\gamma) \cdot z) = 0\]
and so, since $(({\wedge}_{j\in [r]}\theta_{j,L})(z_{L}))_L$ belongs to $\xi(R_\infty)$, the argument in Proposition \ref{generate rubin}(i) implies that $(({\wedge}_{j\in [r]}\theta_{j,L})(z_{L}))_L= 0$ and hence also $({\wedge}_{j\in [r]}\theta_j)(z_{F}) = 0$, as required.}\end{remark}

%
%
 
Motivated by the result of Proposition \ref{preGGS lemma}, we now make the following key definition. 

\begin{definition}\label{derivative def}{\em The `$(r'-r)$-th order derivative at $\gamma$' of the Rubin-Stark element $\varepsilon^{\rm RS}_{\KK_\infty,S,T}$ is the element 
\[ \partial^{\,r'-r}_{\gamma}(\varepsilon^{\rm RS}_{\KK_\infty,S,T})\]
of $\CC_p\cdot{\bigcap}^{r}_{\ZZ_p[G]}\mathcal{O}_{E,S,T,p}^\times$.}
\end{definition}

%


\subsection{$\mathscr{L}$-invariant maps and the Generalized Gross-Stark Conjecture}
 Our aim is to formulate an explicit conjectural formula for the derivative element $\partial^{\,r'-r}_{\gamma}(\varepsilon^{\rm RS}_{\KK_\infty,S,T})$ introduced above. 
 
 For this purpose we shall need an appropriate
 generalization of the notion of 
`$\mathscr{L}$-invariant' that occurs in the classical Gross-Stark Conjecture and we
next prove a technical result that plays a key role in the construction of such a
 generalization.

\subsubsection{}We abbreviate $\G_E$ to $G$ and fix an (ordered) set of places 
\[ \Sigma' \subseteq \Sigma\setminus S_\infty^K\]
of $K$. Then, since each place $v$ in $\Sigma'$ is non-archimedean, we fix a place $w_v$ of $E$
above $v$ and write
$\phi^{\rm ord}_{v}$ for the map in
$\Hom_{\ZZ_p[G]}(\mathcal{O}_{E,S,p}^\times,\ZZ_p[G])$ that satisfies
\begin{equation}\label{phi ord def} \phi^{\rm ord}_{v}(u) := {\sum}_{g \in G}{\rm ord}_{w_v}(g^{-1}(u))
\cdot g\end{equation}
for each $u$ in $\mathcal{O}_{E,S,p}^\times$, where
 ${\rm ord}_{w_v}$ is the normalized additive valuation at $w_v$.

\begin{lemma}\label{ell-inv preparation} Write $e$ for the idempotent
$e_{E/K,S,\Sigma}$ defined in (\ref{key idem def}). Then the reduced exterior product
 ${\wedge}_{v\in \Sigma'}\phi^{\rm ord}_v$ induces an isomorphism of
 $\zeta(\QQ_p[G])$-modules
\[ \phi_{E,S,\Sigma'}^{\rm ord}: e\bigl(\QQ_p\cdot
{\bigcap}^{|\Sigma|}_{\ZZ_p[G]}\mathcal{O}_{E,S,p}^\times\bigr)
 \xrightarrow{\sim}
 e\bigl(\QQ_p\cdot {\bigcap}^{|\Sigma\setminus \Sigma'|}_{\ZZ_p[G]}
 \mathcal{O}_{E,S\setminus \Sigma',p}^\times\bigr). \]
 \end{lemma}

\begin{proof} Set $U := \mathcal{O}^\times_{E,S,p}$,
$U' := \mathcal{O}^\times_{E,S\setminus \Sigma',p}$, $r' := |\Sigma|$ and $r^* :=
|\Sigma\setminus \Sigma'| \le r'$.

Then, since each place in $\Sigma'$ splits completely in $E$, the definition of
$e$ implies that the
 $\QQ_p[G]e$-modules $e(\QQ_p\otimes_{\ZZ_p}U)\cong e(\QQ_p\otimes_{\ZZ}Y_{E,S})$ and
 $e(\QQ_p\otimes_{\ZZ_p}U')\cong e(\QQ_p\otimes_{\ZZ}Y_{E,S\setminus \Sigma'})$ are respectively free of ranks $r'$ and $r^*$ and so
  \cite[Th. 4.19(vi)]{bses} implies that the
$\zeta(\QQ_p[G])e$-modules $e\bigl(\QQ_p\cdot
{\bigcap}^{r'}_{\ZZ_p[G]}U\bigr)$ and
 $e\bigl(\QQ_p\cdot {\bigcap}^{r^*}_{\ZZ_p[G]}
 U'\bigr)$ are each free of rank one.

We now fix a representative of $C_{E,S,T,p}$ as in the upper row of
(\ref{very important diagram}) (with $L = E$ and $\Pi = S$) and use the subsets $J$ and $J^\dagger$ of $[d]$ defined in (\ref{J def}) and (\ref{J def 2}). Then Proposition \ref{generate rubin}(ii) implies that the element
\[ \varepsilon := ({\wedge}_{j\in [d]\setminus J^\dagger}(b_{E,j}^\ast\circ\phi))({\wedge}_{i\in [d]}b_{E,i})\]
generates $e\bigl(\QQ_p\cdot
{\bigcap}^{r'}_{\ZZ_p[G]}U\bigr)$ over $\zeta(\QQ_p[G])$ and so it is  enough to prove that
$\phi_{E,S,\Sigma'}^{\rm ord}(\varepsilon)$ is a generator of the
$\zeta(\QQ_p[G])$-modules $e\bigl(\QQ_p\cdot {\bigcap}^{r^*}_{\ZZ_p[G]}
 U'\bigr)$.

To do this we fix an isomorphism in $\Der^{\rm perf}(\ZZ_p[G])$ between
$C_{E,S,T,p}$ and the complex $P^\bullet$ given by
\[ \ZZ_p[G]^d\xrightarrow{\phi}\ZZ_p[G]^d,\]
where the first term is placed in degree zero and the cohomology groups are identified with those of $C_{E,S,T,p}$ by the maps in the upper row of
 (\ref{very important diagram}). We also write $P_{\Sigma'}^\bullet$ for the complex
\[ \ZZ_p[G]^{|\Sigma'|}\xrightarrow{0}\ZZ_p[G]^{|\Sigma'|},\]
where the first term is placed in degree zero, and we
%
choose a morphism of complexes of $\bz_p[G]$-modules $\alpha: P^\bullet \to P_{\Sigma'}^\bullet$ that represents the composite morphism in $\Der^{\rm perf}(\bz_p[G])$
\[ P^\bullet \cong C_{E,S,T,p} \to \ZZ_p\otimes _\ZZ \bigl({\bigoplus}_{w\in \Sigma'_E}\DR\Hom_\ZZ(\DR\Gamma((\kappa_w)_{\mathcal{W}},\ZZ),\ZZ)\bigr)[-1]
 \cong P_{\Sigma'}^\bullet ,  \]
where the first map is the fixed isomorphism, the second is induced by the exact
triangle in Lemma \ref{complex construction}(ii) and the third
is the canonical isomorphism induced by Remark \ref{exp EF} (and the fact that each place in $\Sigma'$ splits completely in $E$).

We write $J_{\Sigma'}$ for the subset of $J^\dagger$ comprising indices $j$ for which the $j$-th place of $S\setminus \{v'\}$ belongs to $\Sigma'$. 
Then there is a short exact sequence of complexes of $\ZZ_p[G]$-modules
(with horizontal differentials and the first term in the upper complex placed in
degree one)
\begin{equation}\label{exp diagram}\xymatrix{& & \ZZ_p[G]^{|\Sigma'|} \ar@{^{(}->}[d]_{(-{\rm id},\iota_1)}  \ar[rr]^{{\rm id}} & & \ZZ_p[G]^{|\Sigma'|}
 \ar[d]^{{\rm id}}\\
\ZZ_p[G]^d \ar[rr]^{\hskip -0.4truein(\alpha^0,\phi)} \ar[d]_{{\rm id}} & & \ZZ_p[G]^{|\Sigma'|}
\oplus \ZZ_p[G]^d \ar[rr]^{\hskip 0.2truein (0,\pi_1)} \ar@{->>}[d]^{(\iota_1,{\rm id})} & & \ZZ_p[G]^{|\Sigma'|}\\
\ZZ_p[G]^d \ar[rr]^{\phi_\alpha} & & \ZZ_p[G]^d. }\end{equation}
Here $\iota_1$ is the inclusion that sends the $i$-th element in the
standard basis of $\ZZ_p[G]^{|\Sigma'|}$ to the $k(i)$-th element in the standard basis of $\ZZ_p[G]^d$ where $k(i) \in J_{\Sigma'}$ is such that $v_{k(i)}$ is the $i$-th place in $\Sigma'$, and $\pi_1$ denotes the
corresponding projection. In addition, the endomorphism
$\phi_\alpha$ is such that for each $j$ in $[d]$ one has
\[ b_{E,j}^\ast \circ \phi_\alpha =
\begin{cases} b_{E,i}^\ast \circ \alpha^0, 
              &\text{if $j = k(i) \in J_{\Sigma'}$ for $i \in [|\Sigma'|]$,}\\
              0, &\text{if $J^\dagger\setminus J_{\Sigma'}$},\\
              b_{E,j}^\ast \circ \phi, &\text{if $j \in [d]\setminus J^\dagger.$}
                                                  \end{cases}\]

In particular, since $H^0(\alpha)$ coincides with the composite
\begin{equation}\label{ord isomorphism} \mathcal{O}^{\times}_{E,S,T,p} \xrightarrow{\epsilon\, \mapsto
{\sum}_{w\in \Sigma'_E}{\rm ord}_w(\epsilon)\cdot w}
Y_{E,\Sigma',p}\cong \ZZ_p[G]^{|\Sigma'|}\end{equation}
where the isomorphism  sends the ordered basis
$\{w_v\}_{v \in \Sigma'}$ to the standard basis of
 $\ZZ_p[G]^{|\Sigma'|}$, for every $j = k(i)\in J_{\Sigma'}$ one has $
 b_{E,i}^\ast\circ H^0(\alpha) =\phi_{v_{j}}^{\rm ord}$. Hence, for a suitable integer $a$, there are equalities
\begin{align*} \phi_{E,S,\Sigma'}^{\rm ord}(\varepsilon) =&\, ({\wedge}_{v \in \Sigma'}\phi_v^{\rm ord})(\varepsilon)\\
=&\, ({\wedge}_{j \in J_{\Sigma'}}(b_{E,j}^\ast \circ \phi_\alpha))(\varepsilon)\\
  =&\, ({\wedge}_{j\in J_{\Sigma'}}(b_{E,j}^\ast \circ \phi_\alpha))(({\wedge}_{j\in [d]\setminus J^\dagger}(b_{E,j}^\ast\circ\phi))({\wedge}_{i\in [d]}b_{E,i}))\\
  =&\, {\rm Nrd}_{\QQ_p[G]}((-1)^a)\cdot ({\wedge}_{j\in J_{\Sigma'}\cup ([d]\setminus J^\dagger)}(b_{E,j}^\ast\circ\phi_\alpha))({\wedge}_{i\in [d]}b_{E,i})\end{align*}
Since \cite[Prop. 4.21(i)]{bses} implies the latter element is a
generator of $e\bigl(\QQ_p\cdot{\bigcap}^{r^*}_{\ZZ_p[G]}\ker(\phi_\alpha)\bigr)$ over
$\zeta(\QQ_p[G])$, it is therefore enough to show $\ker(\phi_\alpha)$ is isomorphic to $U'$.
 This is in turn true since the first complex in the short exact sequence (\ref{exp diagram})
is acyclic and the second is equal to ${\rm Cone}(\alpha)[-1]$ and so the
 exact triangle in Lemma \ref{complex construction}(ii) induces an
 isomorphism in
 $\Der^{\rm perf}(\ZZ_p[G])$ between $C_{E,S\setminus\Sigma',T,p}$ and
 the third complex in (\ref{exp diagram}) and therefore also an isomorphism
 of $\ZZ_p[G]$-modules between $U' = H^0(C_{E,S\setminus\Sigma',T,p})$ and $\ker(\phi_\alpha)$.
 \end{proof}

\subsubsection{} We set 
\[ R_\infty := \La(\G) \quad \text{ and } \quad \mathcal{I}_E := R_\infty\cdot (\gamma-1).\]
We note that $\mathcal{I}_E$ is a (two-sided) ideal of $R_\infty$, that the assignment $x \mapsto x(\gamma-1)$ induces an isomorphism $R_\infty \cong \mathcal{I}_E$ of (left) $R_\infty$-modules and that there is a natural exact sequence of (left) $R_\infty$-modules
\[ 0 \to \mathcal{I}_E \xrightarrow{\subset} R_\infty \to R_E\to 0,\]
where the third arrow is the natural projection. This exact sequence
combines with (the limit over $L$ in $\Omega_E(\KK_\infty)$) of the isomorphism in Lemma \ref{complex construction}(iv) to give a canonical  exact triangle in $\Der(R_\infty)$
\begin{equation}\label{first bock triangle} \mathcal{I}_E\otimes^{\DL}_{R_\infty}C_{\KK_\infty,S,T} \to  C_{\KK_\infty,S,T} \xrightarrow{\theta} C_{E,S,T} \xrightarrow{\theta'}  (\mathcal{I}_E\otimes^{\DL}_{R_\infty}C_{\KK_\infty,S,T})[1]\end{equation}
and thereby also a composite Bockstein homomorphism
\begin{align}\label{first bock} H^0(C_{E,S,T}) \xrightarrow{H^0(\theta')}&\,  H^1(\mathcal{I}_E\otimes^{\DL}_{R_\infty}C_{\KK_\infty,S,T})\\
 \cong&\, \mathcal{I}_E\otimes_{R_\infty}H^1(C_{\KK_\infty,S,T})\notag\\
 \xrightarrow{ {\rm id}\otimes H^1(\theta)}&\, \mathcal{I}_E\otimes_{R_\infty}H^1(C_{E,S,T})\notag\\
 \cong&\, \bigl(\mathcal{I}_E/\mathcal{I}_E^2\bigr)\otimes_{R_E}H^1(C_{E,S,T})\notag\\
 \cong&\, H^1(C_{E,S,T}).\notag\end{align}
Here the first isomorphism is induced by the fact $C_{\KK_\infty,S,T}$ is acyclic in degrees greater than one, the second by the fact $\mathcal{I}_E$ acts trivially on $H^1(C_{E,S,T})$ and the third by the  isomorphism $R_E\cong \mathcal{I}_E/\mathcal{I}_E^2$ that sends $1$ to the class of $\gamma-1$. (For an alternative, and more direct, description of the composite (\ref{first bock}) as a Bockstein homomorphism see the proof of Lemma \ref{first exp bock comp} below.)

For each $v$ in $\Sigma$ the latter map induces a composite homomorphism of $\ZZ_p[G]$-modules
\[ \phi^{\rm Bock}_{\gamma,v} : \mathcal{O}_{E,S,T,p}^\times =
 H^0(C_{E,S,T}) \to H^1(C_{E,S,T}) = {\rm Sel}^T_S(E)_p^{\rm tr} \xrightarrow{\varrho_v} \ZZ_p[G].\]
Here the equalities come from Lemma \ref{complex construction}(ii) and $\varrho_v$ denotes the composite of the canonical projection ${\rm Sel}^T_S(E)_p^{\rm tr} \to \ZZ_p[G]\cdot w_v$ induced by (\ref{selmer lemma seq2}) and the homomorphism of $\ZZ_p[G]$-modules $\ZZ_p[G]\cdot w_v \to \ZZ_p[G]$ that sends $w_v$ to $1$ (and is well-defined since $v$ splits completely in $E$).

For any finite (ordered) subset $\Sigma'$ of $\Sigma$ the reduced exterior product of the maps $\phi^{\rm Bock}_{\gamma,v}$ over $v$ in $\Sigma'$
induces a homomorphism of $\zeta(\QQ_p[G])$-modules
\[ \phi_{\gamma,S,\Sigma}^{\rm Bock}: e_{E/K,S,\Sigma}\bigl(\QQ_p\cdot
{\bigcap}^{|\Sigma|}_{\ZZ_p[G]}\mathcal{O}_{E,S,p}^\times\bigr)
 \xrightarrow{}
 e_{E/K,S,\Sigma}\bigl(\QQ_p\cdot
 {\bigcap}^{|\Sigma\setminus \Sigma'|}_{\ZZ_p[G]}
 \mathcal{O}_{E,S,p}^\times\bigr) \]
where $e_{E/K,S,\Sigma}$ is the idempotent of $\zeta(\QQ[G])$ defined in (\ref{key idem def}).

Taking advantage of Lemma \ref{ell-inv preparation}, we can
now define a canonical `$\mathscr{L}$-invariant map'.

\begin{definition}\label{ell def} {\em The  $\mathscr{L}$-invariant map associated to $\gamma$, $S$ and a subset $\Sigma'$ of $\Sigma\setminus S_K^\infty$ is the homomorphism of $\zeta(\QQ_p[G])$-modules
\[  \mathscr{L}^{\Sigma'}_{\gamma,S} :
e_{E/K,S,\Sigma}\bigl(\QQ_p\cdot {\bigcap}^{|\Sigma\setminus\Sigma'|}_{\ZZ_p[G]}
 \mathcal{O}_{E,S\setminus \Sigma',p}^\times\bigr) \to e_{E/K,S,\Sigma}
 \bigl(\QQ_p\cdot {\bigcap}^{|\Sigma\setminus \Sigma'|}_{\ZZ_p[G]}
 \mathcal{O}_{E,S,p}^\times\bigr)\]
that is induced by the composite $\phi^{\rm Bock}_{\gamma,S,\Sigma'} \circ
(\phi_{E,S,\Sigma'}^{\rm ord})^{-1}$.
 }\end{definition}

The following result explains the significance of semisimplicity (in the sense of Definition \ref{semi def}) in this setting. 

\begin{lemma}\label{semi for ell-invariant} Assume that the set $\Sigma' := \Sigma\setminus \Sigma_S(\mathcal{K}_\infty)$ contains no archimedean places. Then, for any character $\chi$ in ${\rm Ir}_p(\G_E)$ at which $\KK_\infty/K, E$ and $S$ is semisimple, the $\chi$-component of the map $\mathscr{L}^{\Sigma'}_{\gamma,S}$ is injective. 
\end{lemma}

\begin{proof} We use the notation of the proof of Lemma \ref{prel semi}. In particular, one has  $r = |\Sigma_S(\mathcal{K}_\infty)|$, $\Sigma_S(\mathcal{K}_\infty) = \{v_i: i \in [r]\}$, $r' = |\Sigma|$,   $\Sigma' = \{v_j: j \in J\}$ and $J^\dagger := J \cup [r]$ (so $|\Sigma'| = r'-r$ and $|J^\dagger| = r'$).  We have also fixed a resolution $C(\phi)$ of $C_{\mathcal{K}_\infty,S,T}$ as in Remark \ref{exp interp hrncmc} and, for each $L \in \Omega(\mathcal{K}_\infty)$, we write $C(\phi_L)$ for the induced resolution of $C_{L,S,T,p}$. We abbreviate the notation $\gamma_E$ from loc. cit. to $\gamma$. 

Then, as Proposition \ref{generate rubin}(ii) implies that the element $\varepsilon_E$ of 
${\bigcap}^{r'}_{\ZZ_p[G]}\mathcal{O}_{E,S,p}^\times$ constructed by the argument of Proposition \ref{generate rubin} (with respect to the complex $C(\phi_E)$) generates the $\zeta(\QQ_p[G])$-module $e_{E/K,S,\Sigma}(\QQ_p\cdot {\bigcap}^{r'}_{\ZZ_p[G]}\mathcal{O}_{E,S,p}^\times\bigr)$, it is enough for us to show the given semisimplicity hypothesis on $\chi$ implies that 
\begin{equation}\label{required non-zero} e_\chi\bigl( \phi^{\rm Bock}_{\gamma,S,\Sigma'}(\varepsilon_E)\bigr) \not= 0.\end{equation}

The key point in showing this is that, under the given hypothesis, the argument of
Lemma \ref{prel semi} implies the existence of a matrix $N$ in
${\rm M}_{d,r}(R_\infty)$ for which one has 
\begin{equation}\label{key equality} {\rm Nrd}_{Q(\G)}
\bigl(\left(\!\!\!\begin{array}{c|c}N\!&M(\phi)^\dagger\end{array}\!\!\!\right)\bigr) = \lambda(\gamma)^{r'-r}\cdot \eta\,\,\text{ with }\,\, \eta \in \xi_p(\mathcal{K}_\infty/K) \setminus \wp_\chi(\KK_\infty/K).\end{equation}
To interpret this equality, we write
$\theta_i = (\theta_{i,F})_F$ for each index $i\in [r]$ for
the element of $\Hom_{R_\infty}(R_\infty^d,R_\infty)$ that corresponds to the $i$-th column of $N$ and, for each $F\in \Omega_E(\KK_\infty)$, we use the element 
\[ x_F := ({\wedge}_{i\in [d]\setminus J^\dagger}(b_{F,i}^\ast\circ\phi_{F}))({\wedge}_{j\in [d]} b_{F,j}) \in {\bigcap}_{\ZZ_p[\cG_F]}^{r'}\ZZ_p[\cG_F]^d.\]
In particular, we note that (\ref{rn description}) combines with \cite[Lem. 4.10]{bses} to
 imply that, for a suitable integer $t$, one has  
\begin{align*} &\bigl({\rm Nrd}_{Q(\G)}
\bigl(\left(\!\!\!\begin{array}{c|c}N\!&M(\phi)^\dagger\end{array}\!\!\!\right)\bigr)\\
=&\, {\rm Nrd}_{Q(\G)}((-1)^{t})\cdot\bigl(({\wedge}_{i\in [r]}\theta_{i,F})(({\wedge}_{j\in J}(b_{F,j}^\ast\circ\phi_{F}))
(x_F))\bigr)_{F\in \Omega_E(\KK_\infty)}\\
=&\, \lambda(\gamma)^{r'-r}\cdot{\rm Nrd}_{Q(\G)}((-1)^{t})\cdot\bigl(({\wedge}_{i\in [r]}\theta_{i,F})(({\wedge}_{j\in J}\widehat{\phi}_{j,F})(x_F))\bigr)_{F\in \Omega_E(\KK_\infty)}\end{align*}
where 
the second equality follows from (\ref{divisibility relation}). Upon comparing the latter equality with (\ref{key equality}), we deduce that 
\begin{align*} \bigl(({\wedge}_{i\in [r]}\theta_{i,F})(({\wedge}_{j\in J}\widehat{\phi}_{j,F})(x_F))\bigr)_{F\in \Omega_E(\KK_\infty)} =&\,\, {\rm Nrd}_{Q(\G)}((-1)^{t})\cdot\eta\\
 \in&\,\, \xi_p(\mathcal{K}_\infty/K) \setminus \wp_\chi(\KK_\infty/K).\end{align*}
In view of the explicit definition (\ref{wp def}) of the prime ideal $\wp_\chi(\KK_\infty/K)$, it follows that  
\[ e_\chi\bigl( ({\wedge}_{i\in [r]}\theta_{i,E}))(({\wedge}_{j\in J}\widehat{\phi}_{j,E})(x_E))\bigr) \not= 0.\]
 
To derive (\ref{required non-zero}) from here, it is clearly enough to prove that $({\wedge}_{j\in J}\widehat{\phi}_{j,E})(x_E)$ is equal to $\phi^{\rm Bock}_{\gamma,S,\Sigma'}(\varepsilon_E)$. To show this, we recall from Proposition \ref{generate rubin}(i) that 
\[ ({\wedge}_{j\in J}\widehat{\phi}_{j,E})(x_E) = ({\wedge}_{j\in J}(\widehat{\phi}_{j,E}\circ \iota_{E,*}))(\varepsilon_E),\]
where $\iota_{E,*}$ the injective map $\mathcal{O}_{E,S,T,p}^\times \to \ZZ_p[G]^d$ that is induced by the given resolution $C(\phi_E)$ of $C_{E,S,T,p}$. Given the explicit definition of $\phi^{\rm Bock}_{\gamma,S,\Sigma'}$, it is therefore enough for us to show that, for each $j\in J$ (so that $v_j \in \Sigma'$), one has 
\begin{equation}\label{bock equality} \widehat{\phi}_{j,E}\circ \iota_{E,*}= \phi_{\gamma,v_j}^{\rm Bock}.\end{equation} 

This identity can then be verified by an explicit, and straightforward, computation that combines the defining equality $\phi_j = (\gamma^{-1}-1)(\widehat{\phi}_j)$ of the homomorphism $\widehat{\phi}_j$ with the fact $\phi_{\gamma,v_j}^{\rm Bock}$ can be computed as the composite $\varpi\circ \theta$, with $\theta$ the connecting homomorphism $H^0(C_{E,S,T,p}) = \ker(\phi_E) \to {\rm cok}(\phi)$ associated to the short exact sequence of complexes (with vertical differentials)
\[
\xymatrix{ (R_\infty(\gamma-1))^{d}\,\, \ar[d]^{\phi}\ar@{^{(}->}[rr]^{\,\,\,\,\,\,\,\,\,\,\subset} & & R_\infty^{d} \ar[d]^{\phi} \ar@{->>}[r] & \ZZ_p[G]^{d} \ar[d]^{\phi_E}\\
(R_\infty(\gamma-1))^{d}\,\, \ar@{^{(}->}[rr]^{\,\,\,\,\,\,\,\,\,\,\subset} & & R_\infty^{d} \ar@{->>}[r] &\ZZ_p[G]^{d} }
\]
and $\varpi$ the composite map 
\[ {\rm cok}(\phi) \to {\rm cok}(\phi_E) = {\rm Sel}_S^T(E)^{\rm tr}_p \xrightarrow{\varrho_{v_j}}\ZZ_p[G]\]
in which the first arrow is the natural projection map.
 \end{proof}

\subsubsection{}We can now use the $\mathscr{L}$-invariant map from Definition \ref{ell def} to formulate an explicit conjectural formula for the $(r'-r)$-th order derivative at $\gamma$ of $\varepsilon^{\rm RS}_{\KK_\infty,S,T}$, as specified in Definition \ref{derivative def}.

\begin{conjecture}\label{GGSconjecture}{\em (Generalized Gross-Stark Conjecture)} Fix a subset $\Sigma$ of $\Sigma_S(E)$ as in (\ref{sigma_S assumption}) and set $r := |\Sigma_S(\KK_\infty)|$, $r' := |\Sigma|$ and $\Sigma' := \Sigma\setminus \Sigma_S(\KK_\infty)$ (so that $r \le r' < |S|$ and  $|\Sigma'| = r'-r$). Then, if no place in $\Sigma'$ is
archimedean, there is an equality 
 \[ \partial^{\,r'-r}_{\gamma}(\varepsilon^{\rm RS}_{\KK_\infty,S,T}) = \mathscr{L}^{\Sigma'}_{\gamma,S}(\varepsilon^{\Sigma_S(\KK_\infty)}_{E/K,S\setminus \Sigma',T})\]
in
 $\CC_p\cdot{\bigcap}^{r}_{\ZZ_p[\G_E]}\mathcal{O}_{E,S,p}^\times$. 
\end{conjecture}

In the sequel we write $\kappa_K$ for the homomorphism
$\chi_K\cdot\omega_K^{-1}: G_K \to \ZZ_p^\times$.

\begin{remark}\label{GGS independent}{\em Fix $a$ in $\ZZ_p^\times$. Then the first displayed equality in Lemma \ref{preGGS lemma} implies  $y(\gamma) = x_a^{r'-r}\cdot y(\gamma^a)$ with $x_a := {\rm Nrd}_{Q(\G)}\bigl((\gamma^a-1)/(\gamma-1)\bigr)$, and hence also
\[ \pi^{r'}_E\bigl(y(\gamma)\bigr) = \pi^{r'}_E\bigl( x_a^{r'-r}\cdot y(\gamma)\bigr) = \pi^0_E(x_a)^{r-r'}\cdot \pi^{r'}_E\bigl(y(\gamma^a)).\]
By a straightforward computation, one can also show that  $\mathscr{L}^{\Sigma'}_{\gamma,S} = \pi^0_E(x_a)^{r-r'}\cdot \mathscr{L}^{\Sigma'}_{\gamma^a,S}$ and hence that the validity of Conjecture \ref{GGSconjecture} is independent of the choice of $\gamma$.
In addition, one can explicitly compute $\pi^0_E(x_a)$ as follows. Fix a topological generator $\gamma_0$ of $\Gamma_K$ and write $n$ for the element of $\ZZ_p$ for which the projection of $\gamma$
to $\Gamma_K$ is equal to $\gamma_0^n$. Then, since $\gamma$ acts trivially on $V_\chi$ for each $\chi$ in ${\rm Ir}_p(G)$, an explicit computation of reduced norm (as in the proof of
 Corollary \ref{hrmc true}) shows that
\begin{equation}\label{eval gamma power} j_\chi({\rm Nrd}_{Q(\G)}(\gamma^a-1)) = \Phi_{\G,\chi}(\gamma^a-1)
 = ((1+t)^{na}-1)^{\chi(1)}\end{equation}
and hence that
\begin{align*}\pi^0_E(x_a)_\chi\, =&\, j_\chi (x_a)(0)\\
=& \, 
\left(\frac{(1+t)^{an}-1}{(1+t)^{n}-1}\right)^{\chi(1)}(0)\\ =& \, a^{\chi(1)}\\ =& \, \left(\frac{{\rm Nrd}_{\QQ_p[G]}({\rm log}_p(\kappa_K(\gamma^a)))}{{\rm Nrd}_{\QQ_p[G]}({\rm log}_p(\kappa_K(\gamma)))}\right)_\chi.\end{align*}
This equality implies, in particular, that each side of the equality in Conjecture \ref{GGSconjecture} becomes  independent of $\gamma$ after multiplying by the normalization factor ${\rm Nrd}_{\QQ_p[G]}\bigl({\rm log}_p(\kappa_K(\gamma))^{r'-r}$.
}\end{remark}

\begin{remark}{\em The argument of Theorem \ref{GGSvalid} below shows Conjecture \ref{GGSconjecture} recovers (in the relevant special case) the classical Gross-Stark Conjecture. For this reason, we refer to the derivative formula in Conjecture \ref{GGSconjecture} as the `Generalized Gross-Stark Conjecture' for the data $\KK_\infty/K,E,S$ and $T$. }\end{remark}

\subsection{Cyclotomic $\ZZ_p$-extensions} In this section we  investigate Conjecture \ref{GGSconjecture} in the special case $\KK_\infty = E^{\rm cyc}$.

\subsubsection{}In this case, for each $v$ in $\Sigma\setminus S_K^\infty$, the map $\phi^{\rm Bock}_{\gamma,v}$ has a more explicit description. To give this description we recall that in \cite[\S1]{G0} Gross defines for each place $w$ of $E$ a
local $p$-adic absolute value by means of the composite
%
\[ ||\cdot ||_{w,p}: E_w^\times\xrightarrow{r_w} G_{E_w^{\rm ab}/E_w}
\xrightarrow{\chi_{w}} \bz_p^\times \xrightarrow{x\mapsto x^{-1}} \ZZ_p^\times,\]
where $E_w^{\rm ab}$ is the maximal abelian extension of $E_w$ in $E_w^c$, 
$r_w$ is the local reciprocity map and $\chi_w = \chi_{E_w}$ is the cyclotomic character. We write $\phi_v^{\rm Gross}$ for the homomorphism of $\ZZ_p[G]$-modules $\mathcal{O}_{E,S,T,p}^\times \to \ZZ_p[G]$ that sends each  $u$ to the element 
\[  \phi^{\rm Gross}_{v}(u) = {\sum}_{g \in G}\mathrm{log}_p||g^{-1}(u)||_{w_v,p}
\cdot g.\]

\begin{lemma}\label{first exp bock comp} For each $v$ in $\Sigma\setminus S_K^\infty$ one has
\[ \phi_{\gamma,v}^{\rm Bock} = {\rm log}_p(\kappa_K(\gamma))^{-1}\cdot \phi^{\rm Gross}_{v}.\]\end{lemma}

\begin{proof} We set $C_\infty := C_{\KK_\infty,S,T}$ and $C := C_{E,S,T,p}$. We also fix a topological generator $\gamma_K$ of $\Gamma_K$, write $n$ for the element of $\ZZ_p$ such that the projection of $\gamma$ to $\Gamma_K$ is equal to $\gamma_K^n$ and set $T_n := {\sum}_{i=0}^{i=n-1}\gamma_K^i\in \La(\Gamma_K)$.

We consider the morphism of exact triangles in $\Der(R_\infty)$
\[
\minCDarrowwidth1em\begin{CD}
\mathcal{I}_E\otimes^\DL_{R_\infty} C_\infty @> >> C_\infty @> >> C @> >> (\mathcal{I}_E\otimes^\DL_{R_\infty} C_\infty)[1]\\
@V\theta_1 VV @V\theta_2 VV @\vert @V\theta_1[1] VV \\
 \La(\Gamma_K\times G)\otimes^\DL_{R_\infty}C_\infty @> \gamma_K-1>>  \La(\Gamma_K\times G)\otimes^\DL_{R_\infty}C_\infty @> >> C @> >>  (\La(\Gamma_K\times G)\otimes^\DL_{R_\infty}C_\infty)[1].\end{CD}\]
In this diagram the upper triangle is (\ref{first bock triangle}); in the lower triangle each term $\La(\Gamma_K\times G)$ is regarded as an $R_\infty$-bimodule via the natural diagonal injection $\G \to \Gamma_K\times G$ and $\gamma_K-1$ acts via left multiplication on the first factor in the tensor product and so the existence of the triangle is a consequence of the descent isomorphism $\ZZ_p[G]\otimes^\DL_{R_\infty}C_\infty \cong C$; the morphism $\theta_1$ sends each element $x(\gamma-1)\otimes c^i$ to
$(T_n\cdot\overline{x})\otimes c^i$, where $\overline{x}$ is the image of $x$ under the natural projection $R_\infty \to \La(\Gamma_K)\subseteq \La(\Gamma_K\times G)$, and $\theta_2$ sends each element $c^i$ to $1\otimes c^i$.

We abbreviate the map (\ref{first bock}) to $\beta$ and write $\beta'$ for the composite homomorphism
\[ H^0(C) \to H^1(\La(\Gamma_K\times G)\otimes^\DL_{R_\infty}C_\infty)\to H^1(C),\]
where the arrows are the maps induced by the lower triangle in the above diagram. Then, since the image of $T_n$ under the projection
$\La(\Gamma_K)\subseteq \La(\Gamma_K\times G) \to \ZZ_p[G]$ is equal to $n$, the commutativity of the above diagram implies (without assuming $\KK_\infty = E^{\rm cyc}$) that $\beta' = n\cdot\beta$.

We note next that, setting $c := {\rm log}_p(\kappa_K(\gamma_K)) = n^{-1}\cdot {\rm log}_p(\kappa_K(\gamma))$, the argument of \cite[Th. 5.7(i)]{burns2} proves $\phi^{\rm Gross}_{v} =  c\cdot (\varrho_v\circ \beta')$. (The difference in sign between this formula and that in loc. cit. is accounted for by the fact that the proof of \cite[Th. 5.7(i)]{burns2} uses the shifted complexes $C_\infty[1]$ and $C[1]$ in place of $C_\infty$ and $C$ and the differentials of $C_\infty[1]$ and $C_\infty$ differ by a sign.) This equality then implies that

\begin{align*}\phi^{\rm Gross}_{v} =&\,  c\cdot (\varrho_v\circ \beta')\\ =&\, c\cdot (\varrho_v\circ ( n\cdot \beta))\\
 =&\, (n\cdot c)\cdot (\varrho_v\circ \beta)\\ =&\, {\rm log}_p(\kappa_K(\gamma))\cdot \phi^{\rm Bock}_{\gamma,v},\end{align*}
as required. \end{proof}

\subsubsection{}\label{GGS comp section} In this section we explain the connection between
Conjecture \ref{GGSconjecture} and the classical Gross-Stark Conjecture and are thereby able to deduce its validity
in an important family of examples.

 To do this we further specialize to the case that $K$ is totally real, $E$ is CM and $\KK_\infty = E^{\rm cyc}$. We then write $\tau$ for the (unique) non-trivial element of $\Gal(E^{\rm cyc}/(E^{\rm cyc})^+)$ and
$e_{\pm}$ for the idempotent $(1\pm\tau)/2$ of $\La(\G)$. We identify the elements $\tau$ and $e_{\pm}$ with their respective images in $G$ and  $\ZZ_p[G]$.

We write ${\rm Ir}^{\pm}_p(G)$
for the subsets of ${\rm Ir}_p(G)$ comprising characters for which
$\chi(\tau) = \pm\chi(1)$. For any $\La(\G)$-module $M$ we write $M^\pm$ for the $\La(\G)$-submodule
$\{m \in M: \tau(m) = \pm  m\}$. For an element $m$ of $M$ we also often abbreviate $e_{\pm}(m)$ to $m^{\pm}$ and use a similar convention for homomorphisms.

We consider the homomorphism of $G$-modules
\[ \lambda^{\rm ord}_{E,S}:\mathcal{O}^{\times,-}_{E,S} \to Y^-_{E,S}\]
that sends each $u$ to
${\sum}_{w}{\rm ord}_w(u)\cdot w$, where in the sum $w$ runs over
all places of $E$ above those in $S\setminus S^K_\infty$, and also write
\begin{equation}\label{gross def} \lambda^{\rm Gross}_{E,S,p}:
\mathcal{O}^{\times,-}_{E,S,p} \to Y^-_{E,S,p}\end{equation}
for the map of $\bz_p[G]$-modules that sends
each $u$ in $\mathcal{O}^{\times,-}_{E,S}$ to ${\sum}_{w\in S_E}
\mathrm{log}_p||u||_{w,p} \cdot w$.

Then, since the scalar extension $\QQ_p\otimes_{\ZZ}\lambda^{\rm ord}_{E,S}$
is bijective, for each $\chi$ in ${\rm Ir}_p^-(\G_E)$ we can define
a $\CC_p$-valued `$\mathscr{L}$-invariant' by setting
\[ \mathscr{L}_S(\chi) := {\rm det}_{\bc_p}((\bc_p\otimes_{\bz_p}\lambda^{\rm Gross}_{E,S,p})
\circ (\bc_p\otimes_\ZZ\lambda^{\rm ord}_{E,S})^{-1}\mid \Hom_{\CC_p[G]}(
V_{\check{\chi}},\CC_p \cdot Y_{E,S,p}^-)).\]

In claim (ii) of the following result we refer to Gross's
`Order of Vanishing Conjecture' for $p$-adic Artin $L$-series, as discussed in Remark \ref{gross-jaulent} (and originally formulated by Gross in \cite[Conj. 2.12a)]{G}).

\begin{theorem}\label{GGSvalid} Let $E$ be a finite Galois CM extension of a totally real
field $K$, set $G := \G_E$ and $\KK_\infty := E^{\rm cyc}$ and fix a topological generator $\gamma$ of $\Gal(E^{\rm cyc}/E)$. Then no archimedean place of $K$ splits in $E$ and so $\Sigma_S(E)\not= S$, $\Sigma_S(\KK_\infty)=\emptyset$ and $|\Sigma_S(E)\setminus \Sigma_S(\KK_\infty)|$ is equal to $r' := |\Sigma_S(E)|$. In addition, the following claims are valid.

\begin{itemize}
\item[(i)] For every $\chi$ in ${\rm Ir}^-_p(\G_\infty)$ one has%
    \[ \partial^{\,r'}_{\gamma}(\varepsilon^{\rm RS}_{\KK_\infty,S,T})_\chi = {\rm log}_p(\kappa_K(\gamma))^{-r'\chi(1)}\cdot L^{r'\chi(1)}_{p,S}
(\check{\chi}\cdot\omega_K,0).\]
\item[(ii)] For every $\chi$ in ${\rm Ir}^-_p(\G_\infty)$ one has
\[ \mathscr{L}^{\Sigma_S(E)}_{\gamma,S}(\varepsilon^{\Sigma_S(\KK_\infty)}_{E/K,S\setminus \Sigma_S(E),T})_\chi
 = {\rm log}_p(\kappa_K(\gamma))^{-r'\chi(1)}\cdot \mathscr{L}_S(\chi)\cdot L_{S\setminus \Sigma_S(E)}(\check\chi,0).\]
\item[(iii)] The minus component of Conjecture \ref{GGSconjecture} is valid if Gross's Order of Vanishing Conjecture is valid for every
$\chi$ in ${\rm Ir}_p^-(G)$.
\end{itemize}
\end{theorem}

\begin{proof} The assertions concerning the sets $\Sigma_S(E)$ and $\Sigma_S(\KK_\infty)$ are clear. In particular, in this case the integer $r := r_{S,\KK_\infty}$ is equal to $0$ and no place in $\Sigma' := \Sigma_S(E)\setminus \Sigma_S(\KK_\infty) = \Sigma_S(E)$ is archimedean.

Remark \ref{r_S=0 example}(ii) therefore implies that both
\begin{equation}\label{explicit rs} (\varepsilon^{\rm RS}_{\KK_\infty,S,T})^- =
 \theta_{\KK_\infty,S,T}^-\quad \text{ and }\quad \varepsilon^{\Sigma_S(\KK_\infty)}_{E/K,S\setminus\Sigma',T}
= \theta_{E/K,S\setminus \Sigma_S(E),T}(0).\end{equation}

In particular, since $r = 0$ and $\theta_{\KK_\infty,S,T}^-\in \zeta(Q(\G))$ (see the proof of Corollary \ref{hrmc true}), the argument
 of Proposition \ref{preGGS lemma} implies 
\[ {\rm Nrd}_{Q(\G)}(\gamma-1)^{-r'}\cdot \theta_{\KK_\infty,S,T}^-\in \zeta(Q(\G)) \cap \varprojlim_L\zeta(\QQ_p[\G_L])\]
and that
\[ \partial_\gamma^{r'}(\varepsilon^{\rm RS}_{\KK_\infty,S,T})^- = \pi_E({\rm Nrd}_{Q(\G)}(\gamma-1)^{-r'}\cdot
 \theta_{\KK_\infty,S,T}^-).\]

To prove claim (i) it is thus enough to show that, for each $\chi$ in ${\rm Ir}^-_p(G)$ one has
\begin{equation}\label{tech result 2} \pi_E({\rm Nrd}_{Q(\G_\infty)}(\gamma-1)^{-r'}\cdot
 \theta_{\KK_\infty,S,T}^-)_\chi =
c(\gamma)^{-r'\chi(1)}\cdot L^{r'\chi(1)}_{p,S,T}
(\check{\chi}\cdot\omega_K,0),\end{equation}
with $c(\gamma) := {\rm log}_p(\kappa_K(\gamma))$. 

To prove this 
 we note that $\kappa := \kappa_K$ factors through the projection
 $G_{K} \to \Gamma_K$ and recall that
 for each $\psi$ in $A^-(\mathcal{G}_\infty)$
Deligne and Ribet have shown that there exists a unique element
$f_{S,T,\psi}$ in ${\bq_p^c}\otimes_{\bq_p}Q(\bz_p[[t]])$ for which
\[ L_{p,S,T}(\psi\cdot\omega_K,1-s) = f_{S,T,\psi}(\kappa(\gamma_K)^{s}-1)\]
(cf. \cite{greenberg}).

For every $\psi$ in $A(\Gamma_K)$ one has
\begin{align*} j_{\chi}(\theta_{\KK_\infty,S,T}^-)(\psi(\gamma_K)-1) =&\, j_{\chi\cdot\psi}(
\theta_{\KK_\infty,S,T}^-)(0)\\ =&\, L_{p,S,T}(\check\chi\cdot\check\psi\cdot\omega_K,0)\\ 
=&\,  f_{S,T,\check\chi\cdot\check\psi}(\kappa(\gamma_K)-1)\\
 =&\,  f_{S,T,\check\chi}(\check\psi(\gamma)\kappa(\gamma_K)-1)\\
 =&\,  \iota(f_{S,T,\check\chi})(\psi(\gamma_K)-1),\end{align*}
where $\iota$ is the ring automorphism of ${\bq_p^c}\otimes_{\bq_p}Q(\bz_p[[t]])$ that sends $t$ to $\kappa(\gamma_K)(1+t)^{-1}-1$. Here the first and fourth equalities follow from a general property of the maps $j_\chi$ and functions $f_{S,T,\check\chi}$ that are respectively established in \cite[Prop. 5 and (2), p. 563]{RitterWeiss}, the second follows from (\ref{stick interp}) and the other equalities are clear. Since the displayed equalities are true for every $\psi$ in $A(\Gamma_K)$ it follows that
$j_\chi(\theta_{\KK_\infty,S,T}) = \iota(f_{S,T,\check\chi})$.

The left hand side of (\ref{tech result 2}) is therefore equal to
\begin{align*}
 &\,\bigl(j_\chi({\rm Nrd}_{Q(\G)}
 (\gamma-1)^{-r'})\cdot j_\chi(\theta_{\KK_\infty,S,T}^-)\bigr)(0)\\
 =&\, \bigl(((1+t)^n-1)^{-r'\chi(1)}\cdot \iota(f_{S,T,\check\chi})\bigr)(0)\\
 =&\, n^{-r'\chi(1)}\cdot \bigl(t^{-r'\chi(1)}\cdot \iota(f_{S,T,\check\chi})\bigr)(0)\\
 =&\, n^{-r'\chi(1)}\cdot {\rm log}_p(\kappa(\gamma_K))^{-r'\chi(1)}\cdot L_{p,S,T}^{r'\chi(1)}(\tilde\chi\cdot\omega_K ,0)\\
 =&\, c(\gamma)^{-r'\chi(1)}\cdot L_{p,S,T}^{r'\chi(1)}(\tilde\chi\cdot\omega_K ,0),
\end{align*}
as required to complete the proof of claim (i). Here the second equality uses (\ref{eval gamma power}), the third is established by the argument of \cite[Lem. 5.9]{burns2} and all other equalities are clear.

To prove claim (ii) we write $\varrho$ for the natural projection map
$Y_{E,S,p}^- \to Y_{E,\Sigma_S(E),p}^-$ and note that the definition of the idempotent $e = e_{E/K,S,\Sigma_S(E)}$ ensures that the map $e(\QQ_p\otimes_{\ZZ_p}\varrho)$ is bijective.

For each $v$ in $\Sigma_S(E)$ we set 
\[ w_v^- := e_-(w_{v,E})\in Y_{E,\Sigma_S(E),p}^-.\]
We write
$\{(w_v^-)^\ast\}_{v\in \Sigma_S(E)}$ for the $\ZZ_p[G]^-$-basis of
$\Hom_{\ZZ_p[G]}(Y_{E,\Sigma_S(E),p}^-,\ZZ_p[G]^-)$ that is dual to the basis
$\{w_v^-\}_{v\in \Sigma_S(E)}$ of $Y_{E,\Sigma_S(E),p}^-$ and note that these
maps gives rise to an isomorphism of $\xi(\ZZ_p[G])$-modules
\[ {\wedge}_{v\in \Sigma_S(E)}(w_v^-)^\ast: {\bigcap}_{\ZZ_p[G]}^{r'}Y_{E,\Sigma_S(E),p}^-
\to {\bigcap}_{\ZZ_p[G]}^{0}Y_{E,\Sigma_S(E),p}^- = \xi(\ZZ_p[G])^-\]
that sends ${\wedge}_{v\in \Sigma_S(E)}w_v^-$ to
the identity element $e_-$ of the ring $\xi(\ZZ_p[G])^-$. 

In addition, for each $v$ in $\Sigma' =\Sigma_S(E)$ one has 
\[ (\phi_{v}^{\rm ord})^- = (w_v^-)^\ast\circ \varrho\circ\lambda^{\rm ord}_{E,S,p}\]
and hence also
\[ \phi_{E,S,\Sigma'}^{{\rm ord},-} = e\cdot (\QQ_p\cdot {\wedge}_{v\in \Sigma_S(E)}(w_v^-)^\ast)
\circ \bigl({\bigwedge}_{\QQ_p[G]}^{r'} \QQ_p\cdot\varrho\bigr) \circ
{\bigwedge}_{\QQ_p[G]}^{r'} (\QQ_p\otimes_{\ZZ_p}\lambda_{E,S,p}^{\rm ord}). \]

In a similar way, Lemma \ref{first exp bock comp} implies, for each $v$ in $\Sigma'$, that
\begin{align*} (\phi_{\gamma,v}^{\rm Bock})^{-} =&\, c(\gamma)^{-1}\cdot (\phi_{v}^{\rm Gross})^- \\
 =&\, c(\gamma)^{-1}\cdot (w_v^-)^\ast\circ \varrho\circ\lambda^{\rm Gross}_{E,S,p}\end{align*}
and hence also
\begin{multline*}\label{wedge bock description} \phi^{{\rm Bock},-}_{\gamma,S,\Sigma'} \\ =  {\rm Nrd}_{\QQ_p[G]}(c(\gamma))^{-r'}
e\cdot (\QQ_p\cdot {\wedge}_{v\in \Sigma_S(E)}(w_v^-)^\ast)
\circ \bigl({\bigwedge}_{\QQ_p[G]}^{r'} \QQ_p\cdot\varrho\bigr) \circ
{\bigwedge}_{\QQ_p[G]}^{r'} (\QQ_p\otimes_{\ZZ_p}\lambda_{E,S,p}^{\rm Gross}). \end{multline*}

In this case therefore, the $\mathscr{L}$-invariant map $\mathscr{L}_{\gamma,S}^{\Sigma',-} = \phi_{\gamma,S,\Sigma'}^{{\rm Bock},-}\circ (\phi_{E,S,\Sigma'}^{{\rm ord},-})^{-1}$ is equal to the endomorphism of
 $\zeta(\QQ_p[G])^-e$ that is given by multiplication by the element 
\begin{equation}\label{product element} {\rm Nrd}_{\QQ_p[G]}(c(\gamma)^{-r'})e\cdot{\rm Nrd}_{\QQ_p[G]}\bigl( (\QQ_p\otimes_{\ZZ_p}\lambda_{E,S,p}^{\rm Gross})\circ (\QQ_p\otimes_{\ZZ_p}\lambda_{E,S,p}^{\rm ord})^{-1}\bigr).\end{equation}
In particular, since 
\[ {\rm Nrd}_{\QQ_p[G]}(c(\gamma)^{-r'})_\chi = c(\gamma)^{-r'\chi(1)}\]
and $\mathscr{L}_S(\chi)$ is (by its very definition) equal to the $\chi$-component of the second reduced norm in the product (\ref{product element}), the equality in claim (ii) now follows from the second equality in (\ref{explicit rs}) and the fact that
\[ \theta_{E/K,S\setminus \Sigma_S(E),T}(0)_\chi = L_{S\setminus \Sigma_S(E),T}(\check{\chi},0).\]

Given claims (i) and (ii), 
 claim (iii) then follows directly from the fact that if
 Gross's Order of Vanishing Conjecture is valid for every $\chi$ in
${\rm Ir}^-_p(G)$, then the main result of Dasgupta, Kakde and Ventullo in \cite{dkv} implies that, for every $\chi$ in
${\rm Ir}^-_p(G)$, the explicit quantities in claims (i) and (ii) are equal (for details of this deduction see \cite[Prop. 2.6]{burns2}).
 \end{proof}


\section{The equivariant Tamagawa Number Conjecture for $\mathbb{G}_m$}\label{newsection6}



In this section we establish a precise connection between the Main Conjecture of Higher Rank
Non-commutative Iwasawa Theory (Conjecture \ref{hrncmc}), the Generalized Gross-Stark Conjecture (Conjecture \ref{GGSconjecture}) and the equivariant Tamagawa Number Conjecture for $\mathbb{G}_m$ over
  general Galois extensions of number fields.

  In this way we obtain a concrete strategy
   for obtaining new evidence in support of the latter conjecture, and thereby also 
    extend (to general Galois extensions) the main result of the Kurihara and the present authors in \cite{bks2}.

After reviewing the relevant case of the equivariant Tamagawa Number Conjecture in \S\ref{review subsection}, the main result of this section will be stated and proved in \S\ref{strategy}.

\subsection{Review of the conjecture}\label{review subsection} For the reader's convenience we shall first give an explicit statement of the equivariant Tamagawa Number Conjecture for $\mathbb{G}_m$ relative to an arbitrary finite Galois extension of number fields $L/K$ and clarify what is currently known about this conjecture.
 
 We set $G := \Gal(L/K)$ and fix a finite set of places $S$ of $K$ with 
 \[ S_K^\infty\cup S_{\rm ram}(L/K)\subseteq S\]
 and an auxiliary finite set of places $T$ of $K$ that is disjoint from $S$.

\subsubsection{}\label{Zeta elements and Tamagawa numbers} The leading term at $z=0$ of the Stickelberger function $\theta_{L/K,S,T}(z)$ defined in (\ref{0 stick def}) is 
$$\theta_{L/K,S,T}^\ast(0):={\sum}_{\chi\in \widehat G}L_{S,T}^\ast(\check \chi,0)e_\chi ,$$
where $L^\ast_{S,T}(\chi,0)$ is the leading term of $L_{S,T}(\chi,z)$ at $z=0$, and belongs to $\zeta(\RR[G])^\times.$

The equivariant Tamagawa Number Conjecture for $\mathbb{G}_m$ relative to $L/K$ is then an equality in the relative algebraic $\K_0$-group of the ring extension $\ZZ[G] \to \RR[G]$ that relates $\theta^*_{L/K,S,T}(0)$ to Euler characteristic invariants of the complex 
\[ C:= C_{L,S,T}\]
in $\Der^{{\rm lf},0}(\ZZ[G])$ that is constructed in Lemma \ref{complex construction}. In this section we interpret this conjectural equality in terms of the constructions made in \S\ref{rkt section}.

To do so we note that, just as in Definition \ref{pbe def} (and Remark \ref{full ltc2}), for each prime $p$ there exists a canonical subset ${\rm d}_{\ZZ_p[G]}(C_p)^{\rm pb}$ of the graded $\xi(\ZZ_p[G])$-module ${\rm d}_{\ZZ_p[G]}(C_p)$ comprising all primitive basis elements. 


We next note that the Dirichlet regulator isomorphism $R_{L,S}$ combines with the explicit descriptions of cohomology groups in Lemma \ref{complex construction}(i) (and the short exact sequence (\ref{selmer lemma seq2})) to induce a canonical isomorphism of graded $\zeta(\RR[G])$-modules
\begin{equation*}\label{lambda def} \lambda_{L,S}: {\rm d}_{\RR[G]}(\RR\cdot C) \to (\zeta(\RR[G]),0).\end{equation*}
%
%
From any fixed isomorphism of fields $j: \CC \cong\CC_p$ we obtain an induced ring embedding  $j_*: \zeta(\RR[G]) \to \zeta(\CC_p[G])$ and hence also, via scalar extension, an isomorphism of graded $\zeta(\CC_p[G])$-modules $\lambda^j_{L,S}:= \CC_p\otimes_{\RR, j}\lambda_{L,S}$. For each such $j$ we can therefore write 
\begin{equation}\label{zeta def 2} z_p  = z^j_{L/K,S,T}\end{equation}
for the zeta element associated (by Definition \ref{def zeta}) to the isomorphism $\lambda^j_{L,S}$ and element $j_*(\theta^*_{L/K,S,T}(0))$ of $\zeta(\CC_p[G])^\times$. 

It then follows from Theorem \ref{ltc2} and Remarks \ref{full ltc2} and \ref{bks remark} that the equivariant Tamagawa Number Conjecture for $\mathbb{G}_m$ relative to the extension $L/K$ is valid if and only if the following conjecture is valid for all primes $p$.

\begin{conjecture}[{TNC$_p(L/K)$}]\label{zeta conj} For each isomorphism of fields $j: \CC\cong \CC_p$ one has 
%
%
%
\[ {\rm Nrd}_{\QQ_p[G]}(\K_1(\ZZ_p[G]))\cdot z_p  = {\rm d}_{\ZZ_p[G]}(C_p)^{\rm pb}.\]
\end{conjecture}

\begin{remark}{\em If Tate's formulation \cite[Chap. I, Conj. 5.1]{tate} of Stark's principal conjecture is valid for $L/K$, then the same approach as in Remark \ref{signsandintegral}(ii) shows that the zeta element $z_p$, and hence also each side of the above conjectural equality, is independent of the choice of isomorphism $j$. In addition, the independence of this equality from the choices of both $S$ and $T$ follows in a straightforward fashion from the properties of the complex $C$ described in Lemma \ref{complex construction}(ii) and (iii) respectively.}\end{remark}





\subsubsection{}\label{pm evidence} In \cite[\S4]{JN2}, Johnston and Nickel provide a clear and comprehensive overview of evidence in support of Conjecture TNC$_p(L/K)$ circa 2015.  
%
%
In this section we assume that $K$ is totally real and $L$ is CM and recall several more recent developments concerning the conjecture in this case. 

In the sequel we shall write `TNC$_p(L/K)^\pm$ is valid' to denote that the displayed 
equality in Conjecture TNC$_p(L/K)$ is valid after applying the exact scalar extension functor $\xi(\ZZ_p[G])e_{\pm}\otimes_{\xi(\ZZ_p[G])}-$, where the idempotent
 $e_{\pm}$ of $\zeta(\ZZ_p[G])$ is as defined in \S\ref{GGS comp section} (with $E$ replaced by $L$).

We first recall what is known about Conjecture TNC$_p(L/K)^+$. Before stating the result, we recall that the `$p$-adic Stark Conjecture at
$s=1$' is discussed by Tate in \cite[Chap. VI, \S5]{tate}, where it is attributed to Serre \cite{serre1} (see also \cite[Rem. 4.1.7]{dals}), and predicts an explicit formula for the leading term at $z=1$ of $L_{p,S,T}(\psi,z)$ for each $\psi$ in ${\rm Ir}_p^+(G)$.

\begin{proposition}\label{iwasawa theory+} TNC$_p(L/K)^+$ is valid if all of the following conditions are satisfied.
\begin{itemize}
\item[(i)] $p$ is prime to $|G|$ or $\mu_p(L)$ vanishes.
\item[(ii)] Breuning's Local Epsilon Constant Conjecture is valid for all extensions obtained by $p$-adically completing $L/K$.
\item[(iii)] The $p$-adic Stark Conjecture at
$s=1$  is valid for all characters in ${\rm Ir}_p^+(G)$.
\end{itemize}
\end{proposition}

\begin{proof} It is shown in \cite[\S9.1]{dals} that TNC$_p(L/K)^+$  is valid provided that all of the following conditions are satisfied: the $p$-adic Stark Conjecture at
$s=1$ is valid for all characters in ${\rm Ir}_p(G)$; if $p$ divides $|G|$, then $\mu_p(L)$ vanishes; the $p$-component of a certain element $T\Omega^{\rm loc}(\QQ(0)_L,\ZZ[G])$ of $\K_0(\ZZ[G],\RR[G])$ vanishes.

The stated claim is therefore valid since the result \cite[Th. 4.1]{breuning} of Breuning simplies that the $p$-component of $T\Omega^{\rm loc}(\QQ(0)_L,\ZZ[G])$ vanishes if the `Local Epsilon Constant Conjecture' formulated in \cite{breuning} is valid for all extensions obtained by $p$-adically completing $L/K$.
\end{proof}
%


\begin{remark}\label{bleycobbe}{\em Breuning's Local Epsilon Constant Conjecture has been shown to be valid for all tamely ramified extensions of local fields \cite[Th. 3.6]{breuning} and also for certain classes of wildly ramified extensions (cf. Bley and Cobbe \cite{bc} and the references contained therein). All such results lead to more explicit versions of Proposition \ref{iwasawa theory+}. }\end{remark}

Turning to Conjecture TNC$_p(L/K)^-$, we first record an unconditional result. This result relies crucially on the recent verification by Dasgupta and Kakde \cite{dk} of the Strong Brumer-Stark Conjecture. 
%
%


\begin{proposition}\label{abelian sylow} TNC$_p(L/K)^-$ is valid if the Sylow $p$-subgroups of $G$ are abelian. 
\end{proposition}

\begin{proof} 


We write $M$ for the motive $h^0({\rm Spec}(L))$, regarded as defined over $K$ and with coefficients $\QQ[G]$. Then, in terms of the notation of Remark \ref{remark etnc}(i), TNC$_p(L/K)$ asserts the vanishing of the element of $\K_0(\ZZ_p[G],\CC_p[G])$ given by   
\[ \delta_{\ZZ_p[G],\CC_p}(L^\ast(M,0)) - \chi_{\ZZ_p[G],\CC_p}(C(M),t).\]
We must therefore show the stated conditions imply that the image $T\Omega(L/K)^-_p$ of this element under the natural projection $\K_0(\ZZ_p[G],\CC_p[G])\to K_0(\ZZ_p[G]e_-, \CC_p[G]e_-)$ vanishes. 

To do this, we first combine a result of Nickel \cite[Th. 1]{Nickel} with \cite[Rem. 6.1.1(iii)]{dals} to deduce that $T\Omega(L/K)^-_p$ has finite order. From \cite[Prop. 6.2]{NickelIntegrality}, we can then deduce that $T\Omega(L/K)^-_p$ vanishes if and only $T\Omega(L'/K')^-_p$ vanishes for every intermediate Galois CM extension $L'/K'$ of $L/K$ whose Galois group is either $p$-elementary or a direct product of a $p$-elementary group with $\{1,\tau\}$. In addition, by the argument of \cite[\S 3]{Nickel}, the assumption that the Sylow $p$-subgroups of $G$ are abelian implies that every $p$-elementary subquotient of $G$ is abelian.

To complete the proof, it is therefore enough to prove that $T\Omega(L'/K')^-_p$ vanishes for every intermediate Galois extension $L'/K'$ of $L/K$ in which $L'$ is CM, $K'$ is totally real and $\Gal(L'/K')$ is abelian. For any such extension $L'/K'$, however, the vanishing of $T\Omega(L'/K')^-_p$ is derived from the seminal results of Dasgupta and Kakde in \cite{dk} by Bullach, Daoud, Seo and the first  author in \cite[Th. B (a)]{scarcity}  (and see also the related works of 
Nickel \cite{Nickel}, of Atsuta and Kataoka \cite{ak} and of Dasgupta, Kakde and Silliman \cite{dks}).
\end{proof}

In the general case, there is also the following result of the first author \cite[Cor. 3.8(ii)]{burns2}.

\begin{proposition}\label{iwasawa theory} TNC$_p(L/K)^-$ is valid if $\mu_p(L)=0$ and Gross's Order of Vanishing Conjecture \cite[Conj. 2.12a)]{G0} is valid for every totally odd character of $G$. 
\end{proposition}

\begin{remark}{\em If one sets $S := S_K^\infty\cup S_K^p$, then Remark \ref{gross-jaulent} implies that the conditions in Proposition \ref{iwasawa theory} are satisfied if and only if the $\ZZ_p$-module $A_S(L^{\rm cyc})$ is finitely generated and its quotient $A_S(L^{\rm cyc})_{\Gamma_L}^-$ is finite.}\end{remark}

\begin{remark}\label{Nickel6.8}{\em In addition, Nickel has proved in \cite[Th. 6.8]{NickelIntegrality} that if $\mu_p(L)=0$ and $p$ is `non-exceptional' in a certain technical sense (see \cite[Def. 6.5]{NickelIntegrality}), then TNC$_p(L/K)^-$ is valid. We recall, in particular, from loc. cit. that for any given extension $L/K$ there can only be finitely many `exceptional' primes $p$. }\end{remark}

We next explain how Proposition \ref{iwasawa theory} leads to unconditional
evidence in support of Conjecture TNC$_p(L/K)^-$ in the technically most difficult case that
the Sylow $p$-subgroups of $G$ are non-abelian (thereby complementing Proposition \ref{abelian sylow}) and the relevant
$p$-adic $L$-series possess trivial zeroes. 


\begin{corollary}\label{fourth thm} TNC$_p(L/K)^-$ is valid if all of the following conditions are satisfied.
\begin{itemize}
\item[(i)] $G$ is the semi-direct product of an abelian group $A$ by a supersolvable group.
\item[(ii)] $L^A$ is totally real and has at most one $p$-adic place that  splits in $L/L^+$.
\item[(iii)] $\mu_p(L^P)$ vanishes where $P$ is any given subgroup of $G$ of $p$-power order.
\end{itemize}
\end{corollary}

\begin{proof} It suffices to show that the three given conditions imply the validity of the conditions stated in Proposition \ref{iwasawa theory}.

Condition (i) implies that for every $\rho$ in ${\rm Ir}_p(G)$ there exists a subgroup $A_\rho$ of $G$ that contains $A$ (and hence $\tau$) and a linear $\QQ_p^c$-valued character $\rho'$ of $A_\rho$ such that $\rho = {\rm Ind}_{A_\rho}^G(\rho')$ (for a proof of this fact see \cite[II-22, Exercice]{serre} and the argument of \cite[II-18]{serre}). It is also clear that if $\rho$ belongs to ${\rm Ir}^-_p(G)$, then the field $L^{\ker(\rho')}$ is CM.

In particular, since the functorial properties of $p$-adic Artin $L$-series
under induction and inflation imply that the conjecture \cite[Conj. 2.12a)]{G0} is valid
for $\rho$ if and only if it is valid for $\rho'$ we can assume
(after replacing $L/K$ by $L^{\ker(\rho')}/L^{A_\rho}$ and $\rho$ by $\rho'$)
that $\rho$ is linear. In view of condition (ii) we can also assume that
$K$ has at most one $p$-adic place that splits in $L/L^+$. We then recall that, if all of these hypotheses are satisfied, then \cite[Conj. 2.12a)]{G0} is verified by Gross in
\cite[Prop. 2.13]{G}.

It is therefore enough to note that condition (iii) implies $\mu_p(F)=0$. This is because if $\mu_p(E)=0$ for some number field $E$, then Nakayama's Lemma implies $\mu_p(E')=0$ for any $p$-power degree Galois extension $E'$ of $E$.
\end{proof}

\begin{example}\label{exp ex}{\em If the field $L^P$ in Corollary \ref{fourth thm}(iii) is abelian over $\mathbb{Q}$, then $\mu_p(L^P)$ vanishes by Ferrero-Washington \cite{fw}. Hence, if in such a case the field $L^A$ has only one $p$-adic place, then Corollary \ref{fourth thm} implies the unconditional validity of Conjecture TNC$_p(L/K)^-$. It is straightforward to describe families of non-abelian extensions satisfying  these hypotheses.

\noindent{}(i) Let $F$ be a real quadratic field in which $p$ does not split and assume that $L$ is a CM abelian extension of $F$ of exponent $2p^n$ for some natural number $n$. One can then set $K=\mathbb{Q}$ and $A=G_{L/F}$ and assume that $L/K$ is Galois with (generalized) dihedral Galois group. Then $L^A = F$, $P$ is normal in $G$ and the quotient group $G/P$ is abelian, as required.

\noindent{}(ii) Let $E$ be a totally real $A_4$ extension of $\mathbb{Q}$ with the property that $3$ does not split in its unique cubic subfield. Then for any imaginary quadratic field $F$ the field $L := EF$ is a CM Galois extension of $\mathbb{Q}$ and $G_{L/\mathbb{Q}}$ is of the form $A\rtimes\mathbb{Z}/3$ with $A := \mathbb{Z}/2\times\mathbb{Z}/2\times \mathbb{Z}/2$ (where   $\mathbb{Z}/3$ acts trivially on one copy of $\mathbb{Z}/2$ and cyclically permutes the non-trivial elements in the remaining factor $\mathbb{Z}/2\times\mathbb{Z}/2$) and so is abelian-by-cyclic. The field $L^A$ is then the unique cubic subfield of $L$ and so is totally real with only one $3$-adic place and the field $E_1:= L^AF$ is abelian over $\mathbb{Q}$ so $\mu_3(E_1)$ vanishes. One can also show that if $\mu_3(E_2)$ vanishes for any given quadratic extension $E_2$ of $E_1$ in $L$, then $\mu_3(L)$ also vanishes and so Corollary \ref{fourth thm} applies to $L/\mathbb{Q}$.}
\end{example}
%

Finally, to end this section, we take the opportunity to clarify an aspect of some results in \cite{burns2}. To be specific, we give a concrete example to show that the above approach also allows one to describe situations in which the hypotheses of \cite[Cor. 3.3]{burns2} are satisfied by characters that are both faithful and of arbitrarily large degree. For all such examples one thus obtains a $p$-adic analytic construction of $p$-units that generate non-abelian Galois extensions of totally real fields and also encode explicit structural information about ideal class groups, thereby extending and refining the $p$-adic analytic approach to Hilbert's twelfth problem this is described for linear $p$-adic characters by Gross in \cite[Prop. 3.14]{G}. In the same way one deduces these examples verify a natural $p$-adic analogue of a question of Stark in \cite{stark} and a conjecture of Chinburg in \cite{chin} that were both formulated in the setting of characters of degree two.

\begin{example}{\em  Fix a totally real field $E$ and a cyclic CM extension $E'$ of $E$ in which precisely one $p$-adic place $v$ of $E$ splits completely and no other place of $E$ that ramifies in $E/\bq$ splits completely. We let $k$ be any subfield of $E$ for which the restriction of $v$ has absolute degree one and write $F$ for the Galois closure of $E'$ over $k$. Then $F$ is a CM field and for any faithful linear character $\psi'$ of $G_{E'/E}$ the character 
\[
 \psi:= {\rm Ind}_{G_{F/E}}^{G_{F/k}}({\rm Inf}_{G_{E'/E}}^{G_{F/E}}(\psi'))\] 
of $G_{F/k}$ is irreducible, totally odd, faithful and of degree $[E:k]$. Further, the functoriality of $p$-adic $L$-functions under induction and inflation combines with the result  \cite[Prop. 2.13]{G} of Gross to imply $\psi$ validates the hypotheses of \cite[Cor. 3.3]{burns2} with $S$ taken to be $S_k^\infty\cup S_k^p \cup S_{\rm ram}(E/k)$ and $v_1$ the place of $k$ below $v$.
}\end{example}


%

\subsection{Main conjectures, derivative formulas and Tamagawa numbers}\label{strategy}

\subsubsection{Statement of the main result} We now assume to be given a $\ZZ_p$-extension $K^\infty$ of a number field $K$ and for each
 extension $L$ of $K$ set 
 \[ L^\infty := LK^\infty.\]

We also assume to be given a finite Galois extension $E$ of $K$, a finite set $S$ of places of $K$ such that  
\begin{equation}\label{inclusion hyp}   S_K^\infty\cup S_{\rm ram}(E^\infty/K)\subseteq S\end{equation}
and an auxiliary
 finite set of places $T$ of $K$ that is disjoint from $S$ and such that
 Hypothesis \ref{tf hyp} is satisfied with $\KK_\infty = E^\infty$. We then fix a place 
\begin{equation}\label{v' fix} v'\in S \setminus \Sigma_S(E^\infty)\end{equation}
(noting that this is always possible under the hypothesis (\ref{inclusion hyp}) since $S_{\rm ram}(E^\infty/K)$ is both disjoint from $\Sigma_S(E^\infty)$ and non-empty). 

We set 
\[ G := \G_E\]
and for each $\chi$ in ${\rm Ir}_p(G)$ we write
$E_\chi$ for the fixed field of $\ker(\chi)$ in $E$ and set
\[ G_\chi := \G_{E_\chi}\,\,\text{ and }\,\, \Sigma_\chi := \Sigma_{S\setminus \{v'\}}(E_\chi)\] 
(so that $\Sigma_S(E^\infty)\subseteq \Sigma_\chi,$ $\Sigma_S(E)\setminus \{v'\} \subseteq \Sigma_\chi$ and $\Sigma_\chi \not= S$).

We now introduce a convenient restriction on characters of $G$. We note that condition (ii) of this hypothesis is stated in terms of the idempotent $ e_{E/K,S,\Sigma_\chi}$ of $\zeta(\QQ[G])$ that is defined in (\ref{key idem def}) (so that a more explicit interpretation of the condition can be obtained via the equivalence of the properties (i) and (v) in Lemma \ref{idem lemma}). 

\begin{hypothesis}\label{infinity hyp} $\chi$ is a character in ${\rm Ir}_p(G)$ that has all of the following properties. 
\begin{itemize}
\item[(i)] $\Sigma_S(E_\chi^\infty) = \Sigma_S(E^\infty)$ and $(\Sigma_{\chi}\setminus \Sigma_S(E^\infty))\cap S_K^\infty = \emptyset$.
\item[(ii)] $e_\chi\cdot e_{E/K,S,\Sigma_\chi} \neq 0$.
\item[(iii)] The space $e_\chi\bigl(\QQ_p^c\otimes_{\ZZ_p}A_S(E_\chi^\infty)_{\Gal(E_\chi^\infty/E_\chi)}\bigr)$ vanishes.
\item[(iv)] The Generalized Gross-Stark Conjecture (Conjecture \ref{GGSconjecture}) for the data $E_\chi^\infty/K, E_\chi, S$ and $T$ is valid after multiplication by $e_\chi$. 
\end{itemize}
\end{hypothesis}

\begin{definition} {\em We define idempotents of $\zeta(\QQ_p[G])$ by setting 
\[ e_{E,1}^* := {\sum}_{\chi}e_\chi \quad\text{and}\quad  e_{E,2}^* := {\sum}_{\chi}e_\chi,\]
where the sums are taken over all characters $\chi$ in ${\rm Ir}_p(G)$ that satisfy both of the conditions (i) and (ii), respectively (iii) and (iv), in Hypothesis \ref{infinity hyp}. We then define an idempotent
\[ e_{E}^* := e_{E,1}^*e_{E,2}^*\]
of $\zeta(\QQ_p[G])$.}
\end{definition}

\begin{remark}\label{satisfy remark} {\em Under mild restrictions, the conditions (i) and (ii) in Hypothesis \ref{infinity hyp} are satisfied by every  $\chi\in {\rm Ir}_p(G)$ so that one has 
\[ e_{E,1}^* =1\quad\text{and hence}\quad e_{E}^* = e_{E,2}^*.\] 
For example, if $K^\infty$ is the cyclotomic $\ZZ_p$-extension of $K$, then $\Sigma_S(E_\chi^\infty)\subseteq S_K^\infty$ and so $\Sigma_S(E^\infty)= \Sigma_S(E_\chi^\infty)$ if $S_K^\infty\cap S_{\rm ram}(L/K) = \emptyset$. Hence, if the latter condition is satisfied (as is the case, for example, if $|G|$ is odd), then Hypothesis \ref{infinity hyp}(i) is satisfied by all  $\chi\in {\rm Ir}_p(G)$. In addition, the equivalence of the conditions (i) and (ii) in Lemma \ref{idem lemma} implies $\chi\in {\rm Ir}_p(G)$ validates Hypothesis \ref{infinity hyp}(ii) if and only if one has $e_\chi(\QQ^c\otimes_\ZZ X_{L,S\setminus \Sigma_\chi}) = 0$, or equivalently (by (\ref{tate fe}))  ${\rm ord}_{z=0}L_{S}(\chi,z) = \chi(1)\cdot|\Sigma_\chi|$. It is easily checked that this condition is satisfied by every linear $\chi$ and also, as a consequence of Frobenius reciprocity, by any non-linear (irreducible)  $\chi$ whose restriction to $G_v$ for each $v \in S\setminus S_K^\infty$ does not contain the trivial character of $G_v$. We note, in particular, that  the latter condition is satisfied by any $\chi$ of the form ${\rm Ind}_H^G\phi$ with $H$ a proper subgroup of $G$ and $\phi$ a linear character of $H$ for which one has $H\cap G_v \not\subseteq \ker(\phi)$ for all $v \in S\setminus S_K^\infty$. 

These observations imply that (even if $G$ is non-abelian) Hypothesis \ref{infinity hyp}(ii) is satisfied by every $\chi$ in ${\rm Ir}_p(G)$ under a variety of explicit conditions. It can be checked, for example, that this situation arises in each of the following concrete cases: \

\noindent{}(i) $G$ is abelian;\ 

\noindent{}(ii)  $G_v = G$ for all $v \in S\setminus S_K^\infty$;\

\noindent{}(iii) $G$ is a Frobenius group whose Frobenius complement is abelian and Frobenius kernel is contained in $G_v$ for all $v \in S\setminus S_K^\infty$;\ 

\noindent{}(iv) $G$ is non-abelian of order $p^3$ and its (non-trivial) centre is contained in $G_v$ for all $v \in S\setminus S_K^\infty$.
}
\end{remark}


%

In the next result we establish a concrete link between the Main Conjecture of Higher Rank Non-Commutative Iwasawa Theory (given by Conjecture \ref{hrncmc}), the Generalized Gross-Stark Conjecture (given by Conjecture \ref{GGSconjecture}) and the equivariant Tamagawa Number Conjecture for $\mathbb{G}_m$. 

We observe in particular that, after taking account of Remark \ref{satisfy remark}(i), this result  generalizes to arbitrary finite Galois extensions the main result of Kurihara and the present authors in \cite{bks2} concerning  abelian extensions (see Remark \ref{last remark}). 

\begin{theorem}\label{strategytheorem} The validity of the Main Conjecture of Higher Rank Non-Commutative Iwasawa Theory for $E^\infty/K, S$ and $T$ implies the validity of the equivariant Tamagawa Number
Conjecture for the pair $(h^0({\rm Spec}(E)),\ZZ_p[G]e_E^*)$.
\end{theorem}


Our proof of this result will rely on the interpretation of the equivariant Tamagawa Number Conjecture in terms of the concept of locally-primitive bases, as given explicitly by Conjecture \ref{zeta conj}. 

\subsubsection{The proof of Theorem \ref{strategytheorem}} At the outset, we set 
\[ r := |\Sigma_S(E^\infty)|,\,\, \G := \G_{E^\infty}\quad \text{and}\quad R_\infty := \La(\G).\]
We also fix an endomorphism
$\phi$ of $R_\infty^d$ that is constructed by the method of 
 Proposition \ref{limitomac} with respect to the place $v'$ fixed in (\ref{v' fix}). 
 
 Then, under the assumed
 validity of Conjecture \ref{hrncmc}, Remark
  \ref{exp interp hrncmc} implies the existence of an element $u$ of
  $\K_1(R_\infty)$ that validates the equality 
\begin{equation}\label{fund eq3}  \iota_{\infty,*}(\varepsilon^{\rm RS}_{E^\infty,S,T}) =
{\rm Nrd}_{Q(\G)}(u)\cdot\bigl( (\wedge_{i=r+1}^{i=d}(b_{L,i}^\ast\circ\phi_{L}))(\varpi_L)\bigr)_{L \in \Omega(E^\infty)} \in {\bigcap}_{R_\infty}^rR_\infty^d
 \end{equation}  
from (\ref{fund eq}) with $\KK_\infty = E^\infty$, where for each $L \in \Omega(E^\infty)$ we set 
\[ R_L := \ZZ_p[\cG_L] \quad \text{and}\quad \varpi_L := {\wedge}_{j\in [d]}b_{L,j}\in {\bigcap}_{R_L}^dR_L^d.\]

We next set 
\[ n = |S\setminus \{v'\}|\] 
and use the ordering 
\[ S\setminus \{v'\} = \{v_j: j \in [n]\}\]
of $S\setminus \{v'\}$ that is induced by (\ref{ordering}) and (\ref{ordering 2}) for the field $\mathcal{K} = E^\infty$. 

We then fix a character $\chi$ in $\rm{Ir}_p(G)$ that satisfies Hypothesis \ref{infinity hyp} and use the following convenient notation   

\begin{align*} \Sigma_0 :=&\, \Sigma_S(E^\infty) = \Sigma_S(E_\chi^\infty) = \{v_j : j \in [r]\}\\
r_\chi :=&\, |\Sigma_\chi|\,\,\, (\text{so that  } r_\chi \ge r \text{ since } \Sigma_0\subseteq  \Sigma_\chi)\\
\Sigma_\chi ' :=&\, \Sigma_\chi\setminus \Sigma_0 \,\,\, (\text{so that } |\Sigma_\chi'| = r_\chi -r \,\,\text{ and }\,\, \Sigma_\chi' \cap S_K^\infty = \emptyset) \\
J_\chi :=&\, \{ i \in [n]: v_i \in \Sigma_\chi '\}\\
J_\chi^\dagger :=&\, J_\chi\cup [r] = \{ i \in [n]: v_i \in \Sigma_\chi\}\\
S_\chi ':=&\, S\setminus \Sigma_\chi ' = (S\setminus \Sigma_\chi)\cup \Sigma_0\\
\G_\chi :=&\, \cG_{E_\chi^\infty}\\
\mathcal{H}_\chi :=&\, \Gal(E_\chi^\infty/E_\chi)\\
G_\chi :=&\, \cG_{E_\chi} \cong \G_\chi/\mathcal{H}_\chi\\
R_{\chi,\infty} :=&\, \Lambda(\G_\chi)\\
R_\chi :=&\, \ZZ_p[G_\chi].\end{align*}
We note that, with this notation, the idempotent 
\[ e_{(\chi)} := e_{E_\chi/K,S,\Sigma_\chi} \in \zeta(\QQ_p[G_\chi])\]
defined in (\ref{key idem def}) coincides with $e_{E_\chi/K,S_\chi',\Sigma_0}$.

We write $\Omega_\chi$ for the subset $\Omega_{E_\chi}(E_\chi^\infty)$ of
$\Omega(E_\chi^\infty)$ comprising fields that contain $E_\chi$. For  $F$
in $\Omega_\chi$ we write $u_F$ for the projection of $u$ to $\K_1(R_F)$ and set 
\[ \varepsilon_{F,\chi} := {\rm Nrd}_{\QQ_p[\G_F]}((-1)^{m_\chi})\cdot {\rm Nrd}_{\QQ_p[\G_F]}(u_F)\cdot
((\wedge_{i\in [d]\setminus J_\chi^\dagger}(b_{F,i}^\ast\circ\phi_{F}))
 (\varpi_F)) \in
 {\bigcap}_{R_F}^{r_\chi}R_F^d,\]
with 
\begin{equation}\label{sign term 2} m_\chi := |J_\chi|\cdot |[d]\setminus J_\chi^\dagger| = (r_\chi-r)(d-r_\chi).\end{equation}
In the case $F = E_\chi$ we further abbreviate $u_F$, $\varepsilon_{F,\chi}$ and $\varpi_F$ to $u_\chi$, $\varepsilon_\chi$ and $\varpi_\chi$ respectively. 

We now fix a topological generator $\gamma_\chi$ of $\mathcal{H}_\chi$. Then, noting that $\Sigma_S(E^\infty) = \Sigma_S(E_\chi^\infty)$, we can compare the projection to ${\bigcap}_{R_{\chi,\infty}}^rR_{\chi,\infty}^d$ of the equality (\ref{fund eq3})  
%
to the explicit construction of $\partial_{\gamma_\chi}^{r_\chi-r}\bigl(\varepsilon^{\rm RS}_{E_\chi^\infty,S,T}\bigr)$ in Proposition \ref{preGGS lemma} via the equalities (\ref{der def 1}) and (\ref{der def 2}) (with $\mathcal{K}_\infty/K$ and $\gamma$ taken to be $E_\chi^\infty/K$ and $\gamma_\chi$). 

In this way, we find that the element $z$ in the latter formulas is, in the present context, equal to ${\rm Nrd}_{Q(\G)}(u)$, and hence that there are equalities
\begin{align*} \iota_{E_\chi,*}(\partial_{\gamma_\chi}^{r_\chi-r}\bigl(\varepsilon^{\rm RS}_{E_\chi^\infty,S,T}\bigr)) =&\, {\rm Nrd}_{\QQ_p[G_\chi]}((-1)^{m_\chi + t_\chi})\cdot (\wedge_{j\in J_\chi}(\widehat{\phi}_{\chi,j,E_\chi}))(\varepsilon_\chi)\\
=&\,{\rm Nrd}_{\QQ_p[G_\chi]}((-1)^{m_\chi + t_\chi})\cdot (\wedge_{j\in J_\chi}(\widehat{\phi}_{\chi,j,E_\chi}\circ \iota_{\chi,*}'))(\hat\varepsilon_\chi)\\
=&\,{\rm Nrd}_{\QQ_p[G_\chi]}((-1)^{m_\chi + t_\chi})\cdot \phi^{\rm Bock}_{\gamma_\chi,S,\Sigma_\chi'}(\hat\varepsilon_\chi).
\end{align*}
Here the integer $t_\chi$ is fixed so that 
\begin{equation}\label{sign term} {\rm Nrd}_{\QQ_p[G_\chi]}((-1)^{t_\chi}) \cdot {\wedge}_{j \in [d]\setminus [r]}b_{E_\chi,j} =  \bigl({\wedge}_{j \in [d]\setminus J^\dagger_\chi}b_{E_\chi,j}\bigr)\wedge\bigl({\wedge}_{j \in J_\chi}b_{E_\chi,j}\bigr)\end{equation}
and each homomorphism $\widehat{\phi}_{\chi,i} = (\widehat{\phi}_{\chi,i,F})_F$ is as defined in (\ref{hat def}) (with $\gamma$ replaced by $\gamma_\chi$). 
Further, setting $U_\chi := \mathcal{O}_{E_\chi,S,T,p}^\times$, we have written $\iota_{\chi,*}'$ for the embedding ${\bigcap}_{R_\chi}^{r_\chi}U_\chi \to {\bigcap}_{R_\chi}^{r_\chi}R_\chi^d$ induced by our fixed resolution of $C_{E_\chi,S,T,p}$ and, following Proposition \ref{generate rubin}, we write $\hat\varepsilon_\chi$ for the unique element of $e_{(\chi)}({\bigcap}_{R_\chi}^{r_\chi}U_\chi)$ that is sent by $\wedge_{\QQ_p[G_\chi]}^{r_\chi}(\QQ_p\otimes_{\ZZ_p}\iota_{\chi,*}')$ to $\varepsilon_\chi$. Finally, we note that the last displayed equality follows from the corresponding case of the equality (\ref{bock equality}). 

Now, under Hypothesis \ref{infinity hyp}(iv), the above displayed formula combines with the validity of the $\chi$-component of Conjecture \ref{GGSconjecture} for the data $E_\chi^\infty/K, E_\chi, S$ and $T$ to imply that 
\begin{align*} {\rm Nrd}_{\QQ_p[G_\chi]}((-1)^{m_\chi +t_\chi})\cdot e_\chi\bigl(\phi^{\rm Bock}_{\gamma_\chi,S,\Sigma_\chi'}(\hat\varepsilon_\chi)\bigr) =&\, e_\chi\bigl(\mathscr{L}^{\Sigma_\chi '}_{\gamma_\chi,S}(\varepsilon^{\Sigma_0}_{E_\chi/K,S_\chi ',T})\bigr)\\
 =&\, e_\chi\bigl(\phi^{\rm Bock}_{\gamma_\chi,S,\Sigma_\chi '}\bigl( (\phi^{\rm ord}_{E_\chi,S,\Sigma_\chi '})^{-1}(\varepsilon^{\Sigma_0}_{E_\chi/K,S_\chi ',T})\bigr)\bigr).\end{align*}

In addition, since Hypothesis \ref{infinity hyp}(ii) implies $e_\chi\cdot e_{(\chi)} \not= 0$, the conditions in Hypothesis \ref{infinity hyp}(i), (ii) and (iii) imply that the hypotheses of Proposition \ref{abeliangross} are satisfied by the data $E_\chi^\infty/K, E_\chi$ and $S$, and hence that the relevant case of Lemma \ref{semi for ell-invariant} implies that the $\chi$-component of the map $\phi^{\rm Bock}_{\gamma_\chi,S,\Sigma_\chi '}$ is injective. 

From the last displayed equality, we can therefore deduce that 
\begin{equation}\label{almost third step} {\rm Nrd}_{\QQ_p[G_\chi]}((-1)^{m_\chi+t_\chi})\cdot e_\chi(\hat\varepsilon_\chi) = e_\chi\bigl( (\phi^{\rm ord}_{E_\chi,S,\Sigma_\chi '})^{-1}(\varepsilon^{\Sigma_0}_{E_\chi/K,S_\chi ',T})\bigr).\end{equation}

In the next result, we describe the link between the elements $\varepsilon^{\Sigma_0}_{E_\chi/K,S_\chi ',T}$ and $\varepsilon^{\Sigma_\chi}_{E_\chi/K,S,T}$. In the case that $E_\chi/K$ is abelian, this result was first proved by Rubin in \cite[Prop. 5.2]{R}.

\begin{lemma}\label{third key step} In $\CC_p\otimes_{\ZZ_p}{\bigcap}_{R_\chi}^{r}\mathcal{O}_{E_\chi,S_\chi ',T,p}^\times$ one has 
\[ \varepsilon^{\Sigma_0}_{E_\chi/K,S_\chi ',T} = \phi^{\rm ord}_{E_\chi,S,\Sigma_\chi '}\bigl(\varepsilon^{\Sigma_\chi}_{E_\chi/K,S,T}\bigr).\]\end{lemma}

\begin{proof} We set $\tilde w_i := w_{v_i,E_\chi}-w_{v',E_\chi}$ for each index $i$ in $J_\chi^\dagger$. We then consider the free $R_\chi$-modules 
\[ Y := {\bigoplus}_{i \in [r]}R_\chi\cdot \tilde w_{i}\quad \text{ and }\quad Y':= {\bigoplus}_{i \in J_\chi^\dagger}R_\chi\cdot \tilde w_{i},\]
and use the isomorphism 
\[ \nu: {\bigcap}_{R_\chi}^{r_\chi}Y' \cong {\bigcap}_{R_\chi}^{r}Y\]
of (free, rank one) $\xi(R_\chi)$-modules that sends ${\wedge}_{j \in J_\chi^\dagger}\tilde w_j$ to ${\wedge}_{j \in [r]}\tilde w_{j}$.

We note that 
%
%
 $\theta^{r}_{E_\chi/K,S_\chi ',T}(0)e_{(\chi)}$ belongs to $(\zeta(\RR[G_\chi])e_{(\chi)})^\times$. Hence, after recalling the explicit definition (in Definition \ref{rubin-stark context}) of the elements $\varepsilon^{\Sigma_0}_{E/K,S_\chi ',T}$ and $\varepsilon^{\Sigma_\chi}_{E_\chi/K,S,T}$, and setting 
\[ \eta := \theta^{r_\chi}_{E_\chi/K,S,T}(0)\cdot (\theta^{r}_{E_\chi/K,S_\chi ',T}(0)e_{(\chi)})^{-1},\]
the verification of the claimed equality is reduced to proving commutativity of the following diagram

\[
\begin{CD} e_{(\chi)}(\CC_p\otimes_{\ZZ_p}{\bigcap}_{R_\chi}^{r_\chi}\mathcal{O}_{E_\chi,S,T,p}^\times) @> \CC_p\otimes_{\RR}\lambda >> e_{(\chi)}(\CC_p\otimes_{\ZZ_p}{\bigcap}_{R_\chi}^{r_\chi}Y')\\
@V \phi_{E_\chi,S,\Sigma_\chi'}^{\rm ord}VV @VV \eta \times(\CC_p\otimes_{\ZZ_p}\nu)V\\
e_{(\chi)}(\CC_p\otimes_{\ZZ_p}{\bigcap}_{R_\chi}^{r}\mathcal{O}_{E_\chi,S'_\chi,T,p}^\times) @> \CC_p\otimes_{\RR}\lambda' >> e_{(\chi)}(\CC_p\otimes_{\ZZ_p}{\bigcap}_{R_\chi}^{r}Y),\end{CD}\]
where we abbreviate the maps $\lambda^{r_\chi}_{E_\chi,S}$ and $\lambda^{r}_{E_\chi,S'_\chi}$ to $\lambda$ and $\lambda'$ respectively. 

Write ${\rm ord}_{S,\Sigma_\chi}$ for the map defined as in (\ref{ord isomorphism}) (with $E$ and $\Sigma$ replaced by $E_\chi$ and $\Sigma_\chi)$ Then the commutativity of the above diagram is itself verified by noting that the identity (\ref{changeS}) (with $F'=F=E_\chi$ and $S'$ and $S(F)$ replaced by $S$ and $S_\chi '$ respectively) implies an equality 
\[ \theta^{r_\chi}_{E_\chi/K,S,T}(0)e_{(\chi)} = \theta^{r}_{E_\chi/K,S_\chi',T}(0)e_{(\chi)} \cdot {\prod}_{j \in J_\chi}{\rm log} ({\rm N}w_{j,E_\chi}), \]
whilst the endomorphism 
\[ e_{(\chi)}((\CC_p\otimes_{\ZZ_p}{\rm ord}_{S,\Sigma_\chi})\circ(\CC_p\otimes_{\RR}R_{E_\chi,S})^{-1})\]
of $e_{(\chi)}(\CC_p\otimes_{\ZZ_p}Y')$ is represented, with respect to the ordered $\CC_p[\G_\chi]e_{(\chi)}$-basis 
$\{e_{(\chi)}(\tilde w_j)\}_{j \in J_\chi^\dagger}$, by a block matrix of the form 
\begin{equation*} \left(\begin{array}{cc} M & 0\\
 * & \Delta \end{array}\right).
\end{equation*}
Here $M$ is the matrix of the endomorphism 
\[ e_{(\chi)}((\CC_p\otimes_{\ZZ_p}{\rm ord}_{S_\chi',\Sigma})\circ(\CC_p\otimes_{\RR}R_{E_\chi,S_\chi'})^{-1})\]
of $e_{(\chi)}(\CC_p\otimes_{\ZZ_p}Y)$ with respect to the ordered 
$\CC_p[\G_\chi]e_{(\chi)}$-basis $\{e_{(\chi)}(\tilde w_j)\}_{j \in [r]}$, and $\Delta$ is the diagonal $(r_\chi-r)\times (r_\chi-r)$ matrix with $j$-th entry equal to $({\rm log}( {\rm N}w_{v(j),E_\chi}))^{-1}$, where $v(j)$ here denotes the $j$-th place in the ordered set $J_\chi$.
\end{proof}

Upon substituting the result of Lemma \ref{third key step} into the equality (\ref{almost third step}) we deduce that, for every character $\chi$ that satisfies Hypothesis \ref{infinity hyp}, there is an equality  
\begin{equation}\label{fourth step} {\rm Nrd}_{\QQ_p[G_\chi]}((-1)^{m_\chi +t_\chi})\cdot e_\chi(\hat\varepsilon_\chi) = e_\chi (\varepsilon^{\Sigma_\chi}_{E_\chi/K,S,T}).\end{equation}

To interpret these equalities we write $C_E$ for the complex $C_{E,S,T,p}$. Then, upon unwinding the definition
  of the isomorphism $\lambda_{E,S}^j$ that is discussed in \S\ref{Zeta elements and Tamagawa numbers}, and setting $A := \CC_p[G]$, one obtains a commutative diagram of $\zeta(A)e_{(\chi)}$-modules 
 %
 \[\begin{CD}
 e_{(\chi)}\bigl(\CC_p\cdot {\rm d}_{R_E}(C_E)\bigr)
 @> \lambda_{E,S}^{j,{\rm u}} >> \zeta(A)e_{(\chi)}\\
 @V \CC_p\otimes_{\QQ_p}\Theta_{E,S,T,p}^{\Sigma_\chi}  VV @A AA\\
 e_{(\chi)}\bigl(\CC_p\cdot {\bigcap}^{r_\chi}_{R_E}U_E\bigr) @>
 {\bigwedge}_{A}^{r_\chi}(\CC_p\otimes_{\RR,j}R_{E,S}) >>
 e_{(\chi)}(\CC_p\cdot {\bigcap}^{r_\chi}_{R_E}X_{E,S,p}).\end{CD}
 \]
Here $\lambda_{E,S}^{j,{\rm u}}$ denotes the ungraded part of $\lambda_{E,S}^j$, the projection map $\Theta_{E,S,T,p}^{\Sigma_\chi}$ is as defined in (\ref{Theta'}) and the right hand side vertical arrow is
induced by combining the natural identification 
\[ e_{(\chi)}\bigl(\CC_p\cdot
{\bigcap}^{r_\chi}_{R_E}X_{E,S,p}\bigr)
 = e_{(\chi)}\bigl(\CC_p\cdot
 {\bigcap}^{r_\chi}_{R_\chi}Y_{E_\chi,\Sigma_{\chi},p}\bigr)\]
  together with  the isomorphism of $R_\chi$-modules 
 $Y_{E_\chi,\Sigma_{\chi},p} \cong R_\chi^{r_\chi}$ that is 
 induced by the basis $\{w_{v,E_\chi}: v \in \Sigma_{\chi}\}$
  and the canonical isomorphism of $\xi(R_\chi$)-modules 
  ${\bigcap}^{r_\chi}_{R_\chi}R_\chi^{r_\chi} \cong \xi(R_\chi)$.

To proceed we now use the primitive basis
 \[ \beta_E := \bigl(({\wedge}_{i\in [d]}b_{E,i})\otimes ({\wedge}_{i\in [d]}b^\ast_{E,i}),0\bigr)\]
 of ${\rm d}_{R_E}(C_E)$ that is obtained when one represents
  $C_E$ by the complex
\[ R_E^d\xrightarrow{\phi_E}R_E^d,\]
with the first term placed in degree zero. We observe, in particular, that the explicit description of the map $\Theta_{E,S,T,p}^{\Sigma_\chi}$
   given in (\ref{exp proj}) implies that 
 \begin{equation}\label{mult equal} {\rm Nrd}_{\QQ_p[G_\chi]}((-1)^{m_\chi+t_\chi})\cdot e_\chi(\hat\varepsilon_\chi) =
   e_\chi(\Theta_{E,S,T,p}^{\Sigma_\chi}(\beta_E)).\end{equation}
 Note that the scalar multiplying factor ${\rm Nrd}_{\QQ_p[G_\chi]}((-1)^{m_\chi+t_\chi})$ occurs here since the equality  (\ref{sign term}) and definition (\ref{sign term 2}) combine to imply  
 \begin{align*} &\,\, {\rm Nrd}_{\QQ_p[G_\chi]}((-1)^{m_\chi+t_\chi})\cdot {\wedge}_{j \in [d]\setminus [r]}b_{E_\chi,j} \\
 =&\,\,
  {\rm Nrd}_{\QQ_p[G_\chi]}((-1)^{m_\chi})\cdot \bigl({\wedge}_{j \in [d]\setminus J^\dagger_\chi}b_{E_\chi,j}\bigr)\wedge \bigl({\wedge}_{j \in J_\chi}b_{E_\chi,j}\bigr)\\
 =&\,\, \bigl({\wedge}_{j \in J_\chi}b_{E_\chi,j}\bigr)\wedge \bigl({\wedge}_{j \in [d]\setminus J^\dagger_\chi}b_{E_\chi,j}\bigr)
 \end{align*}
 and because  the verification of the description (\ref{exp proj})  in the present setting assumes that $J_\chi = [r_\chi]\setminus [r].$

Upon combining the equality (\ref{mult equal}) with (\ref{fourth step}), the definition of
   $\varepsilon^{\Sigma_\chi}_{E_\chi/K,S,T}$ and the commutativity of the above diagram, one finds that, for every $\chi$ validating Hypothesis \ref{infinity hyp}, there are equalities 
  \[ e_\chi(\lambda_{E,S}^{j,{\rm u}}(\beta_E)) = e_\chi (j_*(\theta_{E_\chi/K,S,T}^{r_\chi}(0))) =
   e_\chi (j_*(\theta^*_{E/K,S,T}(0)))\]
in $\zeta(\CC_p[G])e_\chi$. 

These equalities  in turn imply that the zeta element $z_p$ defined in (\ref{zeta def 2}) satisfies 
 \[ e_E^*(z_p) =  e_E^*\bigl((\lambda^j_{E,S})^{-1}\bigl((j_*(\theta^*_{E/K,S,T}(0)),0)\bigr)\bigr) = e_E^*(\beta_E).\]
 In particular, since $e_E^*(\beta_E)$ is a primitive basis of the
 $\xi(R_E)e_E^*$-module that is generated by ${\rm d}_{R_E}(C_E)$, this equality implies the validity of
 the equivariant Tamagawa Number Conjecture for the pair
 $(h^0({\rm Spec}(E)),R_Ee_E^*)$. 
 
 This completes the proof of Theorem \ref{strategytheorem}.

\begin{remark}\label{last remark}{\em The result of Theorem \ref{strategytheorem} combines with the observations in Remark \ref{satisfy remark} to present a strategy for obtaining supporting evidence for Conjecture ${\rm TNC}_p(L/K)$ beyond the case of CM-extensions of totally real fields discussed in \S\ref{pm evidence}. In particular, Remark \ref{satisfy remark}(i) implies that, upon specialisation to abelian extensions $L/K$, this strategy recovers that developed by Kurihara and the present authors in \cite[\S5]{bks2}. In addition, for non-abelian Galois extensions, it already gives a simpler proof of existing results such as Proposition \ref{iwasawa theory}. To explain the latter point, we assume that $K$ is totally real, $L$ is CM and $K^\infty$ is the cyclotomic $\ZZ_p$-extension of $K$. Then, for each $\chi\in {\rm Ir}^-_p(G)$, one has $\Sigma_S(E_\chi^\infty) = \Sigma_S(E^\infty) = \emptyset$ and $\Sigma_\chi \cap S_K^\infty = \emptyset$ so Hypothesis \ref{infinity hyp}(i) is satisfied, whilst (\ref{tate fe}) implies Hypothesis \ref{infinity hyp}(ii) is satisfied (with the place $v'$ in (\ref{v' fix}) taken to be archimedean) if ${\rm ord}_{z=0}L_{S}(\chi,z) = \chi(1)\cdot|\Sigma_\chi|$ and Remark \ref{gross-jaulent} combines with Theorem \ref{GGSvalid}(iii) to imply the conditions (iii) and (iv) in Hypothesis \ref{infinity hyp} are equivalent. Thus, if we write $e_*$ for the idempotent of $\zeta(\QQ_p[G])$ that is obtained by summing $e_\chi$ over all $\chi$ in ${\rm Ir}_p^-(G)$ for which one has both 
\[ e_\chi\bigl(\QQ_p^c\otimes_{\ZZ_p}A_S(E_\chi^\infty)_{\Gal(E_\chi^\infty/E_\chi)}\bigr) = 0 \quad \text{ and }\quad {\rm ord}_{z=0}L_{S}(\chi,z) = \chi(1)\cdot|\Sigma_\chi|,\]
then Theorem \ref{strategytheorem} combines with Corollary \ref{hrmc true} to imply the validity, modulo the assumed vanishing of $\mu_p(L)$, of the image of the equality in Conjecture ${\rm TNC}_p(L/K)$ under the functor $\ZZ_p[G]e_*\otimes_{\ZZ_p[G]}-$. We note, in particular, that this argument avoids the rather extensive descent computations that are required for the proof of Proposition \ref{iwasawa theory} given in \cite{burns2}.} \end{remark}

\end{document}